\documentclass[11pt]{article}

\usepackage{bbm,verbatim,latexsym,amsfonts,amsmath,amssymb,graphicx,fancyhdr,array,booktabs, listings, makecell, cases, empheq, datetime2, arydshln 
}
\usepackage{latexsym,amsfonts,amsmath,amssymb,graphicx,amsthm,soul,verbatim,authblk,xcolor, enumitem, mathrsfs}
\usepackage[framemethod=tikz]{mdframed}
\allowdisplaybreaks[4]

\usepackage{xcolor}
\usepackage{listings}

\usepackage{xcolor}
\usepackage{listings}
\usepackage[colorlinks=true,linkcolor=blue,citecolor=blue,urlcolor=blue]{hyperref}
\usepackage[title]{appendix}

\lstdefinelanguage{Mathematica}{
  morekeywords={Module,Sum,Binomial,Pochhammer,FunctionExpand,Together,ClearAll,Assuming,Simplify,Factor,Min},
  sensitive=true,
  morecomment=[l]{(*},
  morecomment=[s]{(*}{*)},
  morestring=[b]"
}

\lstdefinestyle{mma}{
  language=Mathematica,
  basicstyle=\ttfamily\footnotesize,    
  keywordstyle=\bfseries\color{blue!70!black},
  commentstyle=\itshape\color{green!40!black},
  stringstyle=\color{orange!70!black},
  columns=fullflexible,
  keepspaces=true,
  showstringspaces=false,
  tabsize=2,
  frame=none,
  breaklines=true,
  breakatwhitespace=false,
  breakautoindent=true,
  breakindent=0pt,
  postbreak=\mbox{\textcolor{gray}{$\hookrightarrow$}\space},
  xleftmargin=0pt,
  xrightmargin=0pt,
  upquote=true
}

\newcommand{\EE}{\mathbb{ E}}
\newcommand{\PP}{\mathbb{P}}

\newcommand{\R}{\mathbb{R}}
\newcommand{\C}{\mathbb{C}}

\newcommand{\HH}{\mathbb{H}}
\newcommand{\N}{\mathbb{N}}

\newcommand{\Z}{\mathbb{Z}}

\newcommand{\A}{\mathcal{A}}

\newcommand{\pa}{\partial}

\newcommand{\F}{{\mathcal F}}
\newcommand{\G}{{\mathcal G}}
\newcommand{\HF}{\mathcal{F}^{[3]}}

\newcommand{\vt}{\vartheta}
\newcommand{\err}{\mathscr{E}}

\def\eps{\varepsilon}
\def\til{\widetilde}
\def\ha{\widehat}
\def\sem{\setminus}
\def\lin{\overline}
\def\ulin{\underline}
\def\u{\underline}
\def\vphi{\varphi}

\def\L{{\mathcal L}}

\newcommand{\ind}{\mathbbm{1}}

\DeclareMathOperator{\dist}{dist} 
 
\DeclareMathOperator{\Imm}{Im } \DeclareMathOperator{\Ree}{Re }

 \DeclareMathOperator{\diag}{diag}
 
 \DeclareMathOperator{\LP}{LP}
\DeclareMathOperator{\med}{med}

\newenvironment{subproof}[1][\proofname] {\begin{proof}[#1]} {\end{proof}}

\newtheorem{Theorem}{Theorem}[section]
\newtheorem{Lemma}[Theorem]{Lemma}
\newtheorem{Corollary}[Theorem]{Corollary}
\newtheorem{Proposition}[Theorem]{Proposition}
\newtheorem{Conjecture}[Theorem]{Conjecture}

\theoremstyle{definition}

\newtheorem{Construction}[Theorem]{Construction}

\theoremstyle{remark}
\newtheorem{Remark}[Theorem]{Remark}
\newtheorem{Example}[Theorem]{Example}

\numberwithin{equation}{section}
\newcommand{\BGE}{\begin{equation}}
\newcommand{\BGEN}{\begin{equation*}}
\newcommand{\EDE}{\end{equation}}
\newcommand{\EDEN}{\end{equation*}}
\newcommand{\hyper}[5]{%
  {}_{#1}F_{#2}\left(\genfrac{}{}{0pt}{}{#3}{#4};#5\right)%
}

\date{}
\title{Trivariate Hypergeometric Series Formulas\\ for Pure Partition Functions of Multiple $3$-SLE$_\kappa$}
\author{Dapeng Zhan}
\affil{Michigan State University}
\date{\empty}

\begin{document}
\maketitle

\begin{center}
\emph{Dedicated to Gregory F.\ Lawler on the occasion of his 70th birthday.}
\end{center}

\begin{abstract}
Pure partition functions of multiple SLE are characterized by null-state partial differential equations, M\"obius covariance, and boundary asymptotics. After quotienting by M\"obius covariance, the case of three curves is the first genuinely multivariable one: the moduli space has three independent variables, naturally represented by the three unoriented cross-ratios of the three pairs of links.

We solve this M\"obius-normalized three-variable problem for the two basic link-pattern types of multiple $3$-SLE$_\kappa$, namely the rainbow and neighbor patterns. Writing $\beta=4/\kappa$, we construct explicit trivariate hypergeometric-series normal forms and identify them with the corresponding pure partition functions  {for all $\beta>1/2$ in both cases, equivalently for all $\kappa\in(0,8)$}.

The proof is analytic. The null-state PDEs and M\"obius covariance yield recursion relations for the trivariate coefficient arrays. In the rainbow case, coefficient estimates give convergence and boundary regularity on the closed cube. In the neighbor case, Pfaff systems continue the local power series to a neighborhood of $[0,1)^3$, while  {boundary analysis and corner propagation} give continuity on $[0,1]^3$ {throughout the full range $\beta>1/2$}. The two-dimensional boundary degenerations are classical Appell $F_1$ and Horn $G_2$ functions.

The probabilistic identification uses SLE martingale arguments and It\^o calculus, together with positivity and boundary regularity. We also discuss boundary degenerations, including heuristic connections with boundary Green's functions.
\end{abstract}



\newpage
\tableofcontents

\section{Introduction}

Schramm--Loewner evolution (SLE) provides a canonical description of conformally invariant random planar interfaces and plays a central role in the scaling-limit theory of two-dimensional critical lattice models; see, e.g., \cite{Law-SLE,RS}. In models with alternating boundary conditions, multiple interfaces emerge simultaneously, which leads to the theory of multiple SLE. In the multiple-SLE framework, the law of a collection of noncrossing curves is encoded by a partition function satisfying a system of second-order partial differential equations together with M\"obius covariance and suitable asymptotic conditions. This point of view goes back to the work of Bauer, Bernard, and Kyt\"ol\"a \cite{BBK05} (rooted in the classical Belavin--Polyakov--Zamolodchikov (BPZ) equations of conformal field theory \cite{BPZ84}) and was developed further by Dub\'edat through commutation relations and related holonomic systems \cite{Julien-commu,D06}. Pure partition functions, indexed by link patterns, are the basic building blocks in this theory. For a prescribed pairing of the marked boundary points, they are expected to encode, and in several frameworks do encode, the extremal multiple-SLE measures connecting the paired points.

From the probabilistic point of view, these functions encode the interaction among the curves; in Loewner-type descriptions, they determine the drift terms of the driving processes. In some works they also arise through probabilistic weights---for instance, through constructions involving Brownian loop measure---rather than being introduced \emph{a priori} as smooth solutions to the PDE system. In the present paper, however, we mainly adopt the analytic point of view: we take the PDE, M\"obius covariance, and asymptotic conditions as the starting point and establish explicit formulas from them.

The existence problem for pure partition functions has been studied from several complementary perspectives. A quantum-group approach to pure partition functions was developed in \cite{KP16}. In a different direction, Peltola and Wu developed a probabilistic framework for local and global multiple SLEs for $\kappa\in(0,4]$ and studied the associated pure partition functions \cite{Global}. The subsequent uniqueness theory of global multiple SLEs was established in \cite{VPW21}. In the case of two curves, Wu proved a conformal Markov characterization of hypergeometric SLE and obtained explicit hypergeometric formulas for the corresponding partition functions (\cite{Wu-hyper}). Further explicit formulas are known in special cases: for $\kappa=4$, through the relation with Gaussian free field level lines \cite{Global}, and for $\kappa=2$, through boundary correlation functions for loop-erased random walk and uniform spanning tree \cite{LERW-partition}. Altogether, these results strongly suggest that explicit special-function formulas should exist well beyond the already understood cases.

There is also a closely related but methodologically quite different line of work based on
Coulomb gas integrals and contour-integral representations. Starting from commuting SLEs,
Dub\'edat obtained Euler integral representations for solutions of the associated holonomic
systems \cite{D06}. Flores and Kleban subsequently developed a systematic Coulomb gas
approach to the relevant null-state PDE systems and constructed explicit contour-integral
solutions, while Flores, Simmons, and Kleban later derived explicit connectivity weights for
several polygonal configurations, including the hexagon case \cite{FSK15a,FSK15b,FSK15c, FSK15d, FSK22}.
These results are explicit at the level of contour integrals, but they do not yield the M\"obius-normalized trivariate hypergeometric-series normal forms developed in the present paper. A related recent preprint is
\cite{FLPW24}, where Feng, Liu, Peltola, and Wu study SLE partition functions and pure
partition functions for $\kappa\in(0,8)$ via Coulomb gas integrals.

More generally, the partition-function viewpoint also appears in other SLE-type processes. For instance, in \cite{Reversal-kappa-rho}, multivariable Appell-Lauricella functions arise in the study of the time-reversal of multiple-force-point SLE$_\kappa(\u \rho)$ with all force points on the same side.


Against this background, the present paper studies the first genuinely multivariable instance of the pure partition-function problem from a direct analytic viewpoint. We consider multiple \(3\)-SLE\(_\kappa\), or equivalently multiple SLE with \(3\) curves. After quotienting by M\"obius covariance, the case \(N=2\) leaves one independent variable, while the case \(N=3\) leaves three. These three moduli are naturally encoded by the three unoriented cross-ratios associated with the three pairs of links in the link pattern. This exact match between the dimension of the M\"obius-quotient moduli space and the number of symmetric link cross-ratios is special to \(N=3\), and it is the mechanism that makes a symmetric trivariate hypergeometric normal form possible.

The novelty here is not merely the existence of explicit expressions.  Contour-integral and Coulomb-gas representations provide a powerful and important approach to the same class of null-state equations.  The present work develops a different, M\"obius-normalized three-variable framework.  Starting directly from the partition-function PDEs and covariance, we obtain recursion relations for the coefficient arrays, construct trivariate hypergeometric-series solutions, and then establish the analytic continuation and boundary regularity needed for the probabilistic identification.  Thus the probabilistic problem is reduced to, and solved through, a concrete singular analytic system.

Starting from a link pattern $\alpha\in \LP_3$, we perform an explicit normalization and a
conformally covariant change of variables to three unoriented cross-ratios
$(r_1,r_2,r_3)\in(0,1)^3$, thereby reducing the problem to the study of a trivariate power
series
\[
F(r_1,r_2,r_3)=\sum_{n_1,n_2,n_3\ge 0}
A^{(\beta)}_{n_1,n_2,n_3}\,r_1^{n_1}r_2^{n_2}r_3^{n_3},
\qquad
\beta=\frac{4}{\kappa}\in\Big(\frac12,\infty\Big).
\]
This reduction yields a M\"obius-normalized trivariate hypergeometric-series normal form in the
bulk. Unlike the Appell--Lauricella functions appearing in \cite{Reversal-kappa-rho}, the
coefficient array $A^{(\beta)}=(A^{(\beta)}_{n_1,n_2,n_3})$ is determined recursively by the
partition-function PDE rather than given in closed form. At the same time,
lower-dimensional boundary pieces lead to classical Appell and Horn functions. For $N=4$, by
contrast, quotienting by M\"obius covariance would leave five independent variables, whereas
the analogous pairwise construction produces six cross-ratios, leading to redundancy; retaining
only five would destroy the symmetry that is crucial in the present paper. Thus the coordinate mechanism exploited here is already special to $N=3$.  {The associated analysis remains delicate, and the boundary regularity arguments take rather different forms in the rainbow and neighbor cases, as described below.}

In the rainbow case, a hidden symmetry appears: although the planar configuration is not manifestly symmetric in the three curves, the reduced coefficient array is symmetric in its three indices.  This symmetry is accompanied by strong global estimates.  We prove absolute summability of the coefficient array for all \(\beta>1/2\), obtain a continuous extension to the closed cube, and verify the required boundary asymptotics.  After undoing the normalization, the resulting functions satisfy the null-state PDEs, M\"obius covariance, positivity, and the pure partition-function asymptotics.  {This supplies the analytic input needed for the probabilistic identification of the rainbow formulas.}

In the neighbor case, the local construction is only the beginning. The main difficulty is global boundary regularity near the singular side faces and the triple corner. We derive several Pfaff systems, use them to continue the local solution to a neighborhood of $[0,1)^3$, and then prove continuity on $[0,1]^3$ {throughout the full range $\beta>1/2$}, by side-face analysis, normal-direction estimates, and corner-propagation arguments.  {This supplies the boundary
regularity needed for the probabilistic identification of the neighbor
formulas.}

The final step returns to probability.  The analytic construction produces functions satisfying the null-state PDEs, M\"obius covariance, positivity, and the required boundary behavior.  These properties make the corresponding SLE observables local martingales; using It\^o calculus and the positivity result proved in Section~\ref{subsection-positivity}, we identify the normalized formulas with the desired pure partition functions  {for all $\kappa\in(0,8)$}.

\vspace{1em}
\noindent \textbf{Organization of the paper.}
The paper is organized as follows. In Section~\ref{Section:partition}, we introduce the  $3$-SLE$_\kappa$ partition functions, reduce the problem to the study of the auxiliary function \(F\), and prove positivity by probabilistic means under natural boundary assumptions on \(F\). In Section~\ref{section:recursions}, we derive the recursion relations for the coefficient array and construct formal power-series solutions. In Section~\ref{section:two-dim}, we analyze the two-variable degenerations leading to Appell and Horn functions; Section~\ref{section-degenerate} concludes with a heuristic discussion of degenerate partition functions and Green's functions. 
Section \ref{section:analytic} treats the rainbow case and develops the Pfaff-system continuation for the neighbor case. Section \ref{section:continuous} proves the neighbor boundary regularity on $[0,1]^3$.
In Section~\ref{section:integer}, we discuss special integer values of \(\beta\) and their algebraic structure. Finally, Appendix~\ref{section:appendix} records a technical ODE result used in the Pfaff-system analysis.

\section{Multiple SLE Partition Functions and Their Reduction} \label{Section:partition}
In this section, we recall the definition of multiple SLE partition functions, then specialize to the case $N=3$ and reduce the problem to an auxiliary function of three cross-ratios. We also formulate a positivity theorem under natural boundary assumptions on this auxiliary function. We largely adopt the setup and notation of Peltola and Wu \cite{Global}.

For $n\in\N$, we write $[n]:=\{j\in\N:j\le n\}$.
Let $\kappa\in(0,8)$ and $h=h(\kappa)=\frac{6-\kappa}{2\kappa}$. Let $N\in\N$. An $N$-SLE$_\kappa$ partition function is a positive smooth function $\mathcal Z$ on
$$\mathfrak{X}_{2N}:=\{(x_1,\dots,x_{2N})\in\R^{2N}:x_j<x_k\text{ if }j<k\}$$
satisfying the following two properties:
\begin{itemize}
  \item [(PDE)] {\textit{Partial differential equations of second order:}}  For any $i\in [2N]$,
  \BGE \Big[ \frac\kappa 2 \pa_i^2+\sum_{j\in [2N]\sem\{ i\}} \Big(\frac{2}{x_j-x_i}\pa_j-\frac{2h}{(x_j-x_i)^2}\Big)\Big]{\mathcal Z}=0 .\label{PDE}\EDE
\item [(COV)] {\textit{M\"obius covariance}}: For all M\"obius maps $\vphi$ of $\HH:=\{z\in\C:\Imm z>0\}$ such that $\vphi(x_1)<\vphi(x_2)<\cdots<\vphi(x_{2N})$,
\BGE {\mathcal Z}(x_1,\dots,x_{2N})=\prod_{i=1}^{2N} \vphi'(x_i)^h \cdot {\mathcal Z}(\vphi(x_1),\dots,\vphi(x_{2N})).\label{COV}\EDE
\end{itemize}
These governing PDEs are essentially the Belavin-Polyakov-Zamolodchikov (BPZ) equations arising from the null-state decoupling in Conformal Field Theory (\cite{BPZ84}).

For a set $S$, we denote by $\binom{S}{k}$  the set of all $k$-element subsets of $S$.
A link pattern of $N$ links is a partition $\{\{a_1,b_1\},\dots,\{a_N,b_N\}\}$ of $[2N]$ such that for any $\{j,k\}\in \binom{[N]}{2}$, $(a_j-a_k)(a_j-b_k)(b_j-a_k)(b_j-b_k)>0$, which is equivalent to the existence of mutually disjoint curves $\gamma_1,\dots,\gamma_N$ in $\HH$ such that each $\gamma_j$ joins $a_j$ to $b_j$.  The set of link patterns of N links is denoted by $\LP_N$. It has  $C_N$ elements, where $C_N$ is the $N$-th Catalan number: $\frac 1{N+1}\binom{2N}N$.

$\LP_1$ has one element: $\{\{1,2\}\}$; $\LP_2$ has two elements: $\{\{1,2\},\{3,4\}\}$ and $\{\{1,4\},\{2,3\}\}$; and $\LP_3$ has five elements: $$ \{\{1,6\},\{2,5\},\{3,4\}\},\quad \{\{1,4\},\{2,3\},\{5,6\}\},\quad \{\{1,2\},\{3,6\},\{4,5\}\},$$ $$\{\{1,2\},\{3,4\},\{5,6\}\},\quad\{\{1,6\},\{2,3\},\{4,5\}\}.$$
Among the five patterns, the first three are those for which one of the three links separates the other two links; these are called rainbow patterns. The remaining two are called neighbor patterns.

The \textit{pure partition functions} $Z_\alpha$ are indexed by link patterns $\alpha\in \bigcup_{N=0}^\infty \LP_N$. They are positive solutions to (PDE) (\ref{PDE}) and (COV) (\ref{COV}) singled out by boundary conditions given in terms of their asymptotic behavior:
\begin{itemize}
  \item [(ASY)] \textit{Asymptotics}: For all $N\in\N$, $\alpha\in \LP_N$, $j\in[ 2N-1]$, and $\xi\in\R$, if $x_1<\cdots<x_{j-1}<\xi<x_{j+2}<\cdots<x_{2N}$, then for $\ulin x=(x_1,\dots,x_{2N})\in \mathfrak{X}_{2N}$,
  \BGE \lim_{x_j,x_{j+1}\to\xi} \frac{{\mathcal Z}_\alpha(\u x)}{(x_{j+1}-x_j)^{-2h}}=\begin{cases} 0,&\text{if }\{j,j+1\}\not\in \alpha,\\
  {\mathcal Z}_{\alpha/\{j,j+1\}}(\ulin x/\{j,j+1\}),&\text{if }\{j,j+1\}\in \alpha,
  \end{cases}\label{ASY}\EDE
where $\alpha/\{j,j+1\}\in \LP_{N-1}$ is obtained from $\alpha$ by removing  $\{ j, j + 1\}$ and relabeling the remaining indices by $[2N-2]$ in increasing order; and $\ulin x{/\{j,j+1\}}\in \mathfrak{X}_{2(N-1)}$ is obtained by removing $x_j,x_{j+1}$ from $\ulin x$ and keeping the remaining numbers in increasing order.
\end{itemize}
Here we understand that $\emptyset\in \LP_0$ and require that ${\mathcal Z}_\emptyset=1$.

\subsection{Reduction of pure partition functions when $N=3$}

The pure partition functions ${\mathcal Z}_\alpha$ are known for   $N\le 2$. When $N=1$,
$${\mathcal Z}_{\{\{1,2\}\}}(x_1,x_2)=(x_2-x_1)^{-2h}.$$
When $N=2$, by the explicit hypergeometric formulas in \cite{Wu-hyper}, after normalization,  for $\alpha =\{\{a_1,b_1\},\{a_2,b_2\}\}\in \LP_2$ and $\ulin x=(x_1,x_2,x_3,x_4)\in \mathfrak{X}_4$,
\BGE {\mathcal Z}_{\alpha}(\ulin x)=\prod_{j=1}^2 |x_{a_j}-x_{b_j}|^{-2h}\cdot \frac{(C^{(\alpha)}_{\u x})^{\frac 2\kappa}\cdot {}_2F_1 (1-\frac 4\kappa,\frac 4\kappa;\frac 8\kappa; C^{(\alpha)}_{\u x} )}{{}_2F_1 (1-\frac 4\kappa,\frac 4\kappa;\frac 8\kappa; 1 )},\label{Z2}\EDE
where ${}_2F_1 (1-\frac 4\kappa,\frac 4\kappa;\frac 8\kappa; \cdot )$ is a Gaussian hypergeometric function, and
$C^{(\alpha)}_{\u x}=C_{\{x_{a_1},x_{b_1}\},\{x_{a_2},x_{b_2}\}}$ is the \textit{unoriented cross ratio} of $\{x_{a_1},x_{b_1}\}$ and $\{x_{a_2},x_{b_2}\}$, which is defined by
\BGE   C_{\{p_1,q_1\},\{p_2,q_2\}}:=\min\Big\{\frac{|p_1-p_2||q_1-q_2|}{|p_1-q_2||q_1-p_2|}, \frac{|p_1-q_2||q_1-p_2|}{|p_1-p_2||q_1-q_2|}\Big\}. \label{unCR}\EDE
Note that $C_{\{p_1,q_1\},\{p_2,q_2\}}=C_{\{p_2,q_2\},\{p_1,q_1\}}\in (0,1)$ if $\{\{p_1,q_1\},\{p_2,q_2\}\}$ forms a link pattern, i.e., there are two disjoint curves $\gamma_1,\gamma_2$ in $\HH$ such that $\gamma_j$ joins $p_j$ to $q_j$, $j=1,2$.

For $a,b,c\in\C$ and $c\not\in \Z_{\le 0}$, the Gaussian hypergeometric function  ${}_2F_1(a,b;c ;x)$  is defined (\cite[Chapter 15]{NIST:DLMF}) by
$${}_2F_1(a,b;c ;z)=\sum_{n=0}^\infty \frac{(a)_n(b)_n}{(1)_n(c)_n} \cdot z^n,\quad |z|<1,$$
where $(q)_n$ is the Pochhammer symbol:
$$(q)_n=\begin{cases}
  1,&\text{if }n=0;\\
  q(q+1)\cdots(q+n-1),&\text{if }n\in\N.
\end{cases}$$
Furthermore, if $\Ree(c-a-b)>0$, 
 Gauss's identity (\cite[Eq.~15.4.20]{NIST:DLMF}) states that the power series defining ${}_2F_1(a,b;c ;\cdot)$ also converges absolutely at $1$, and
\BGE {}_2F_1\Big(a,b;c; 1\Big)=\frac{\Gamma(c)\Gamma(c-a-b)}{\Gamma(c-a)\Gamma(c-b)}.\label{Gauss}\EDE

Besides the cases $N\le 2$,  explicit formulas for pure partition functions of $N$-SLE$_\kappa$ are known for $\kappa=4$ (\cite{Global}) and $\kappa=2$ (\cite{LERW-partition})  through the relations of the corresponding multiple SLE partition functions with GFF level lines and loop-erased walks, respectively.

We now specialize to the case $N=3$. Let $\alpha=\{\{a_{\ell},b_{\ell}\}:1\le {\ell}\le 3\}\in \LP_3$.  Let $I_\alpha(\ulin x)=\prod_{{\ell}=1}^3 |x_{a_{\ell}}-x_{b_{\ell}}|^{-2h}$. Suppose ${\mathcal Z}_\alpha= I_\alpha {\mathcal Y}_\alpha$.
Then ${\mathcal Z}_\alpha$ satisfies (COV) (\ref{COV}) iff ${\mathcal Y}_\alpha$ satisfies M\"obius invariance;
and  ${\mathcal Z}_\alpha$ satisfies (PDE) (\ref{PDE}) iff for all $i\in[6]$,   ${\mathcal M}^\alpha_i {\mathcal Y}_\alpha=0$, where
\BGE {\mathcal M}^\alpha_i:=\frac\kappa2 \pa_i^2+\frac{\kappa -6}{x_i-x_{\phi_\alpha(i)}} \pa_i+\sum_{j\in[6]\sem\{i\}} \frac 2{x_j-x_i} \pa_j-2h\sum_{{\ell}\in[3]:i\not\in \{a_{\ell},b_{\ell}\}} \Big(\frac 1{x_{a_{\ell}}-x_i}-\frac 1{x_{b_{\ell}}-x_i}\Big)^2,\label{Mi}\EDE
and $\phi_\alpha(i)$ is the unique $j$ such that $\{i,j\}\in \alpha$.

After factoring out $I_\alpha$, we now choose one of the three links as a reference and introduce coordinates adapted to that choice.
Fix $i_0\in [6]$. Let $i_\infty\in[6]$ and ${\ell}_0\in [3]$ be such that  $\{i_0,i_\infty\}=\{a_{{\ell}_0},b_{{\ell}_0}\}$. Then $i_\infty=\phi_\alpha(i_0)$.
Define   $y_j=\frac{1}{x_{i_\infty}-x_{j}}$ for $j\in [6]\sem \{i_\infty\}$, and $w=y_{i_0}$. Then for any ${\ell}\in [3]\sem \{{\ell}_0\}$, $y_{a_{\ell}}$ and $y_{b_{\ell}}$ lie on the same side of $w$.
Relabel $y_{a_{\ell}}$ and $y_{b_{\ell}}$ as $u_{\ell}$ and $v_{\ell}$ such that $u_{\ell}$ is the one closer to $w$, and $v_{\ell}$ is the other, i.e., $\{u_{\ell},v_{\ell}\}=\{y_{a_{\ell}},y_{b_{\ell}}\}$ and $|w-u_{\ell}|<|w-v_{\ell}|$. Recall the unoriented cross ratio defined by (\ref{unCR}). Then for ${\ell}\in [3]\sem\{{\ell}_0\}$,
\BGE C_{\{x_{a_{{\ell}_0}},x_{b_{{\ell}_0}}\},\{x_{a_{\ell}},x_{b_{\ell}}\}}=\frac{w-u_{\ell}}{w-v_{\ell}}.\label{CR1}\EDE
Let $\sigma  = 1$ or $-1$ depending on whether $\alpha$ is a rainbow pattern or neighbor pattern, respectively.  Let $\{{\ell}_1,{\ell}_2\}=[3]\sem \{{\ell}_0\}$. Then
\BGE C_{\{x_{a_{{\ell}_1}},x_{b_{{\ell}_1}}\},\{x_{a_{{\ell}_2}},x_{b_{{\ell}_2}}\}}=\Big(\frac{(u_{{\ell}_1}-u_{{\ell}_2})(v_{{\ell}_1}-v_{{\ell}_2}) }{(u_{{\ell}_1}-v_{{\ell}_2})(v_{{\ell}_1}-u_{{\ell}_2})} \Big)^{\sigma}. \label{CR2}\EDE

We now write $Y_\alpha$ in the coordinates $(w,u_{\ell},v_{\ell})_{{\ell}\in[3]\setminus\{{\ell}_0\}}$ and denote the resulting function by $Y$. Then ${\mathcal M}^\alpha_{i_0} {\mathcal Y}_\alpha=0$ iff ${\mathcal N} {\mathcal Y}=0$, where
\BGE {\mathcal N} := \pa_w^2+   \sum_{{\ell}\in \{{\ell}_1,{\ell}_2\}}\Big[ \frac {-\frac 4\kappa}{w-u_{\ell}} \pa_{u_{\ell}}+ \frac {-\frac 4\kappa}{w-v_{\ell}} \pa_{v_{\ell}} -\frac 4\kappa h  \Big(\frac{1}{w-u_{\ell}}-\frac{1}{w-v_{\ell}}\Big)^2\Big]. \label{N}\EDE
Define $${\mathcal X}={\mathcal Y}\cdot  \prod_{1\le j<k\le 3} C_{\{x_{a_{j}},x_{b_{j}}\},\{x_{a_{k}},x_{b_{k}}\}}^{-\frac 2\kappa}.$$
By (\ref{CR1}) and (\ref{CR2}), we have
$$\frac{\pa_w^2 {\mathcal Y}}{\mathcal Y}=\frac{\pa_w^2 {\mathcal X}}{\mathcal X}+2\sum_{{\ell}\in\{{\ell}_1,{\ell}_2\}} \Big(\frac{\frac 2\kappa}{w-u_{\ell}}+\frac{-\frac 2\kappa}{w-v_{\ell}}\Big)\frac{\pa_w{\mathcal X}}{\mathcal X}+\sum_{{\ell}\in \{{\ell}_1,{\ell}_2\}}\Big(\frac{-\frac 2\kappa}{(w-u_{\ell})^2}+\frac{\frac 2\kappa}{(w-v_{\ell})^2}\Big)$$
$$+\sum_{j\in\{{\ell}_1,{\ell}_2\}}\sum_{k\in\{{\ell}_1,{\ell}_2\}} \Big(\frac{\frac 2\kappa}{w-u_{j}}-\frac{\frac 2\kappa}{w-v_{j}}\Big)\Big(\frac{\frac 2\kappa}{w-u_{k}}-\frac{\frac 2\kappa}{w-v_{k}}\Big);$$
and for $\{j,k\}=\{{\ell}_1,{\ell}_2\}$,
$$\frac{\pa_{u_{j}} {\mathcal Y}}{\mathcal Y}=\frac{\pa_{u_{j}} {\mathcal X}}{\mathcal X}+\frac{-\frac 2\kappa}{w-u_{j}}+\sigma \Big(\frac{\frac 2\kappa }{u_{j}-u_{k}}-\frac{\frac 2\kappa }{u_{j}-v_{k}}\Big),$$
$$\frac{\pa_{v_{j}} {\mathcal Y}}{\mathcal Y}=\frac{\pa_{v_{j}} {\mathcal X}}{\mathcal X}+\frac{\frac 2\kappa}{w-v_{j}}+\sigma\Big(\frac{\frac 2\kappa  }{v_{j}-v_{k}}-\frac{\frac 2\kappa  }{v_{j}-u_{k}}\Big).$$

Introduce a new parameter $\beta=\frac 4\kappa$.
Since $\kappa\in (0,8)$, $\beta\in (\frac 12,\infty)$.
Substituting the above identities into $\mathcal N \mathcal Y=0$, we collect the mixed terms using the following partial fraction identity:
\begin{align}
  &\sum_{s=1}^2 \Big(\frac{1}{w-u_{{\ell}_s}}\Big(\frac{1}{u_{{\ell}_s}-u_{{\ell}_{3-s}}}- \frac{1}{u_{{\ell}_s}-v_{{\ell}_{3-s}}}\Big)+\frac{1}{w-v_{{\ell}_s}}\Big(\frac{1}{v_{{\ell}_s}-v_{{\ell}_{3-s}}}- \frac{1}{v_{{\ell}_s}-u_{{\ell}_{3-s}}}\Big)\Big)\nonumber\\
  =&  \Big(\frac 1{w-u_{{\ell}_1}}-\frac 1{w-v_{{\ell}_1}}\Big) \Big(\frac 1{w-u_{{\ell}_2}}-\frac 1{w-v_{{\ell}_2}}\Big).\label{combine}
\end{align}
It follows that ${\mathcal N} {\mathcal Y}=0$ iff ${\mathcal O}{\mathcal X}=0$, where
$${\mathcal O}=\pa_w^2+\beta ^2  \frac{1-\sigma }2 \Big(\frac 1{w-u_{{\ell}_1}}-\frac 1{w-v_{{\ell}_1}}\Big) \Big(\frac 1{w-u_{{\ell}_2}}-\frac 1{w-v_{{\ell}_2}}\Big)$$
\BGE
+\sum_{{\ell}\in \{{\ell}_1,{\ell}_2\}}\Big(\frac{\beta \pa_w}{w-u_{\ell}}+\frac{-\beta \pa_w }{w-v_{\ell}}+\frac{-\beta \pa_{u_{\ell}}}{w-u_{\ell}}+\frac{-\beta \pa_{v_{\ell}}}{w-v_{\ell}}\Big)-\beta (1-\beta )\sum_{{\ell}\in[3]\sem\{{\ell}_0\}} \frac 1{w-v_{\ell}}\cdot \Big(\frac 1{w-u_{{\ell} }}-\frac 1{w-v_{{\ell} }}\Big).\label{O}\EDE

We write $C_{{\ell}_1,{\ell}_2}$ for $C_{\{x_{a_{{\ell}_1}},x_{b_{{\ell}_1}}\},\{x_{a_{{\ell}_2}},x_{b_{{\ell}_2}}\}}$, and $C_{\ell}$ for $C_{\{x_{a_{{\ell}_0}},x_{b_{{\ell}_0}}\},\{x_{a_{{\ell}}},x_{b_{{\ell}}}\}}$, ${\ell}\in\{{\ell}_1,{\ell}_2\}$. By (\ref{CR1}), $C_{{\ell}}$ is a function of $w,u_{{\ell}},v_{{\ell}}$, ${\ell}\in\{{\ell}_1,{\ell}_2\}$; and by (\ref{CR2}), $C_{{\ell}_1,{\ell}_2}$ is a function of $u_{{\ell}_1},v_{{\ell}_1},u_{{\ell}_2},v_{{\ell}_2}$. Moreover, for $s=1,2$,
$$\pa_w C_{{\ell}_s}=\frac{C_{{\ell}_s}}{w-u_{{\ell}_s}}-\frac{C_{{\ell}_s}}{w-v_{{\ell}_s}},\quad \pa_{u_{{\ell}_s}} C_{{\ell}_s}=-\frac{C_{{\ell}_s}}{w-u_{{\ell}_s}},\quad \pa_{v_{{\ell}_s}} C_{{\ell}_s}=\frac{C_{{\ell}_s}}{w-v_{{\ell}_s}};$$
$$\pa_{u_{{\ell}_s}} C_{{\ell}_1,{\ell}_2}=\frac{\sigma  C_{{\ell}_1,{\ell}_2}}{u_{{\ell}_s}-u_{{\ell}_{3-s}}}- \frac{\sigma C_{{\ell}_1,{\ell}_2}}{u_{{\ell}_s}-v_{{\ell}_{3-s}}},\quad \pa_{v_{{\ell}_s}} C_{{\ell}_1,{\ell}_2}=\frac{\sigma C_{{\ell}_1,{\ell}_2}}{v_{{\ell}_s}-v_{{\ell}_{3-s}}}- \frac{\sigma  C_{{\ell}_1,{\ell}_2}}{v_{{\ell}_s}-u_{{\ell}_{3-s}}} . $$

Suppose ${\mathcal X}=F(C_{{\ell}_1,{\ell}_2},C_{{\ell}_1},C_{{\ell}_2})$. Then by the above formulas,
$$\pa_w {\mathcal X}=\sum_{{\ell} \in  \{{\ell}_1,{\ell}_2\}} \Big(\frac{1}{w-u_{{\ell} }}-\frac{1}{w-v_{{\ell} }}\Big)C_{{\ell} }  \pa_{C_{{\ell}}} F;$$
$$\pa_w^2 {\mathcal X}= \sum_{j\in\{{\ell}_1,{\ell}_2\}}\sum_{k \in \{{\ell}_1,{\ell}_2\}} \Big(\frac{1}{w-u_{j}}-\frac{1}{w-v_{j}}\Big) \Big(\frac{1}{w-u_{k}}-\frac{1}{w-v_{k}}\Big)C_{j} C_{k} \pa_{C_{j}}\pa_{C_{k}} F$$ $$+\sum_{{\ell}\in \{{\ell}_1,{\ell}_2\}}\Big[\Big(\frac{1}{w-u_{\ell}}-\frac{1}{w-v_{\ell}}\Big)^2-\frac {1}{(w-u_{\ell})^2}+\frac{1}{(w-v_{\ell})^2}\Big]C_{\ell} \pa_{C_{\ell}}  F;$$
and for  $\{j,k\}=\{{\ell}_1,{\ell}_2\}$,
$$\pa_{u_{j}}{\mathcal X}=  \Big(\frac{\sigma C_{{\ell}_1,{\ell}_2}  }{u_{j}-u_{k}}- \frac{\sigma C_{{\ell}_1,{\ell}_2}  }{u_{j}-v_{k}}\Big) \pa_{C_{{\ell}_1,{\ell}_2}} F-\frac{C_{j}}{w-u_{j}} \pa_{C_{j}}F;$$
$$\pa_{v_{j}}{\mathcal X}= \Big(\frac{\sigma  C_{{\ell}_1,{\ell}_2}}{v_{j}-v_{k}}- \frac{\sigma C_{{\ell}_1,{\ell}_2}}{v_{j}-u_{k}}\Big)\pa_{C_{{\ell}_1,{\ell}_2}} F+\frac{C_{j}}{w-v_{j}} \pa_{C_{j}}F.$$
Thus,
\begin{align}
  {\mathcal O}{\mathcal X}&=\sum_{j\in\{{\ell}_1,{\ell}_2\}}\sum_{k \in  \{{\ell}_1,{\ell}_2\}} \Big(\frac{1}{w-u_{j}}-\frac{1}{w-v_{j}}\Big) \Big(\frac{1}{w-u_{k}}-\frac{1}{w-v_{k}}\Big)C_{j} C_{k} \pa_{C_{j}}\pa_{C_{k}} F\label{Line1}\\
  &+\sum_{{\ell}\in \{{\ell}_1,{\ell}_2\}}\Big[\Big(\frac{1}{w-u_{\ell}}-\frac{1}{w-v_{\ell}}\Big)^2-\frac {1}{(w-u_{\ell})^2}+\frac{1}{(w-v_{\ell})^2}\Big]C_{{\ell}} \pa_{C_{{\ell}}}  F\label{Line2}\\
  &+\beta\sum_{j\in\{{\ell}_1,{\ell}_2\}}\sum_{k \in  \{{\ell}_1,{\ell}_2\}}  \Big(\frac{1}{w-u_{j}}-\frac{1}{w-v_{j}}\Big) \Big(\frac{1}{w-u_{k}}-\frac{1}{w-v_{k}}\Big) C_{k} \pa_{C_{k}} F\label{Line3}\\
  &+\sum_{(j,k)\in\{(1,2),(2,1)\}} \frac{-\sigma \beta }{w-v_{{\ell}_j}}  \Big(\frac 1{v_{{\ell}_j}-v_{{\ell}_{k}}}-\frac 1{v_{{\ell}_j}-u_{{\ell}_{k}}}\Big)   C_{{\ell}_1,{\ell}_2} \pa_{C_{{\ell}_1,{\ell}_2}} F\label{Line4}\\
 & +\sum_{(j,k)\in\{(1,2),(2,1)\}}  \frac{-\sigma \beta }{w-u_{{\ell}_j}}  \Big(\frac 1{u_{{\ell}_j}-u_{{\ell}_{k}}}-\frac 1{u_{{\ell}_j}-v_{{\ell}_{k}}}\Big)  C_{{\ell}_1,{\ell}_2} \pa_{C_{{\ell}_1,{\ell}_2}} F\label{Line5}\\
  &+ \sum_{{\ell} \in  \{{\ell}_1,{\ell}_2\}} \Big(\frac \beta{(w-u_{\ell})^2}-\frac \beta{(w-v_{\ell})^2}\Big) C_{{\ell}} \pa_{C_{{\ell}}} F\label{Line6}\\
  & -\sum_{{\ell}\in \{{\ell}_1,{\ell}_2\}} \frac {\beta (1-\beta )}{w-v_{\ell}}\cdot \Big(\frac 1{w-u_{{\ell} }}-\frac 1{w-v_{{\ell} }}\Big)F\label{Line7}\\
& +\beta ^2 \cdot \frac{1-\sigma }2 \Big(\frac 1{w-u_{{\ell}_1}}-\frac 1{w-v_{{\ell}_1}}\Big) \Big(\frac 1{w-u_{{\ell}_2}}-\frac 1{w-v_{{\ell}_2}}\Big)F.\label{Line8}
\end{align}
Decomposing the sums in (\ref{Line1}) and (\ref{Line3}) respectively into cases $j=k$ and $j\ne k$ and
using (\ref{CR1}) and (\ref{combine}), we find that ${\mathcal O}{\mathcal X}={\mathcal P}F$, where
$$ {\mathcal P}:=  \frac{(1-C_{{\ell}_1})(1-C_{{\ell}_2})}{(w-u_{{\ell}_1})(w-u_{{\ell}_2})} \Big(\underbrace{2C_{{\ell}_1} C_{{\ell}_2} \pa_{C_{{\ell}_1}}\pa_{C_{{\ell}_2}}}_{(\ref{Line1}),\, j\ne k}  \underbrace{+\beta C_{{\ell}_1} \pa_{C_{{\ell}_1}} +\beta C_{{\ell}_2} \pa_{C_{{\ell}_2}}}_{(\ref{Line3}),\,j\ne k} \underbrace{-\beta \sigma C_{{\ell}_1,{\ell}_2}\pa_{C_{{\ell}_1,{\ell}_2}}}_{(\ref{Line4})\,\&\,(\ref{Line5})}\underbrace{+\beta ^2\cdot \frac{1-\sigma }2}_{(\ref{Line8})} \Big)  $$
$$+\sum_{{\ell}\in \{{\ell}_1,{\ell}_2\}} \frac{1-C_{{\ell}}}{(w-u_{\ell})(w-v_{\ell})} \Big( \underbrace{C_{{\ell}}(1-C_{{\ell}}) \pa_{C_{{\ell}}}^2}_{(\ref{Line1}),\,j=k} \underbrace{+\beta (1-C_{\ell}) \pa_{C_{\ell}}}_{(\ref{Line3}),\,j= k}\underbrace{+\beta(1+C_{\ell})\pa_{C_{\ell}}}_{(\ref{Line6})}\underbrace{ -2C_{{\ell}}  \pa_{C_{{\ell}}}}_{(\ref{Line2})} \underbrace{ -\beta (1-\beta )}_{(\ref{Line7})}\Big) .$$
The underbraces show the contribution of (\ref{Line1})-(\ref{Line8}) to each term of ${\mathcal P}$. The computation of the term contributed by (\ref{Line4}) and (\ref{Line5}) involves the equality (\ref{combine}).


Thus,  ${\mathcal O} {\mathcal X}=0$ is equivalent to
\begin{align}
&\Big(2C_{{\ell}_1} C_{{\ell}_2} \pa_{C_{{\ell}_1}}\pa_{C_{{\ell}_2}} +\beta C_{{\ell}_1} \pa_{C_{{\ell}_1}} +\beta C_{{\ell}_2} \pa_{C_{{\ell}_2}} -\beta  \sigma C_{{\ell}_1,{\ell}_2}\pa_{C_{{\ell}_1,{\ell}_2}}+\beta ^2\cdot \frac{1-\sigma }2\Big)F \nonumber \\
+& \sum_{s=1}^2  \frac 1{1-C_{{\ell}_{3-s}}} \frac {w-u_{{\ell}_{3-s}}}{w-v_{{\ell}_s}} \Big( C_{{\ell}_s}(1-C_{{\ell}_s}) \pa_{C_{{\ell}_s}}^2 +(2\beta -2C_{{\ell}_s}) \pa_{C_{{\ell}_s}} -\beta (1-\beta ) \Big)F=0 
\label{l1l2}
\end{align}

By (\ref{CR1}) and (\ref{CR2}), the latter yielding $C_{{\ell}_1,{\ell}_2}^{\sigma }=\frac{(u_{{\ell}_1}-u_{{\ell}_2})(v_{{\ell}_1}-v_{{\ell}_2}) }{(u_{{\ell}_1}-v_{{\ell}_2})(v_{{\ell}_1}-u_{{\ell}_2})}$, we get
$$\frac{w-u_{{\ell}_1}}{w-v_{{\ell}_2}}+\frac{w-u_{{\ell}_2}}{w-v_{{\ell}_1}} =\frac{(C_{{\ell}_1}+C_{{\ell}_2})-C_{{\ell}_1,{\ell}_2}^{\sigma }(1+C_{{\ell}_1}C_{{\ell}_2})} {1-C_{{\ell}_1,{\ell}_2}^{\sigma }}.$$
Thus, a sufficient condition for (\ref{l1l2}) to hold is that there exists a function $G=G(C_{{\ell}_1,{\ell}_2},C_{{\ell}_1},C_{{\ell}_2})$ such that, for $j\in[2]$,
$$ ( C_{{\ell}_j}(1-C_{{\ell}_j}) \pa_{C_{{\ell}_j}}^2 +(2\beta -2C_{{\ell}_j}) \pa_{C_{{\ell}_j}} -\beta (1-\beta )  )F = (1-C_{{\ell}_{3-j}})(1-C_{{\ell}_1,{\ell}_2}) G;$$
and
$$ (2C_{{\ell}_1} C_{{\ell}_2} \pa_{C_{{\ell}_1}}\pa_{C_{{\ell}_2}} +\beta C_{{\ell}_1} \pa_{C_{{\ell}_1}} +\beta C_{{\ell}_2} \pa_{C_{{\ell}_2}} -\beta \sigma C_{{\ell}_1,{\ell}_2}\pa_{C_{{\ell}_1,{\ell}_2}}+\beta ^2\cdot \frac{1-\sigma }2 )F$$
$$=\begin{cases}
   ({C_{{\ell}_1,{\ell}_2} (1+C_{{\ell}_1}C_{{\ell}_2})-(C_{{\ell}_1}+C_{{\ell}_2})})   G,&\text{if }\sigma =1;\\
   (C_{{\ell}_1,{\ell}_2} (C_{{\ell}_1}+C_{{\ell}_2})-(1+C_{{\ell}_1}C_{{\ell}_2}))  G,&\text{if }\sigma =-1.
\end{cases}$$

In the theorem below we relabel the three cross-ratios by
$$r_{\ell}=C_{\{x_{a_j},x_{b_j}\},\{x_{a_k},x_{b_k}\}},\quad \{j,k,{\ell}\}=[3].$$

\begin{Theorem}
Let $\kappa\in (0,8)$, and $\beta=4/\kappa\in (1/2,\infty)$.
  \begin{enumerate}
    \item [(i)] Let $\alpha=\{\{a_1,b_1\},\{a_2,b_2\},\{a_3,b_3\}\}\in \LP_3$ be a rainbow pattern. Let $F,G\in C^\infty((0,1)^3)$. Suppose that for any $j,k,{\ell}$ such that $\{j,k,{\ell}\}=[3]$, we have
        \BGE (r_{\ell}(1-r_{\ell})\pa_{r_{\ell}}^2+(2\beta -2r_{\ell})\pa_{r_{\ell}}-\beta (1-\beta ))F=(1-r_j)(1-r_k) G;\label{rainbow1}\EDE
        \BGE (2r_jr_k\pa_{r_j}\pa_{r_k}+\beta r_j\pa_{r_j}+\beta r_k\pa_{r_k}-\beta r_{\ell}\pa_{r_{\ell}}) F =(r_{\ell}(1+r_jr_k)-(r_j+r_k))G.\label{rainbow2}\EDE
        Let $H(r_1,r_2,r_3)=(r_1r_2r_3)^{\frac 2\kappa} F(r_1,r_2,r_3)$.
        Let $c_\alpha \in (0,\infty)$. Define ${\mathcal Z}_\alpha$ on $\mathfrak{X}_6$   by
        \BGE {\mathcal Z}_\alpha(\u x)=c_\alpha \prod_{j=1}^3 |x_{a_j}-x_{b_j}|^{-2h}   H\Big(C_{\{x_{a_1},x_{b_1}\},\{x_{a_2},x_{b_2}\}}, C_{\{x_{a_2},x_{b_2}\},\{x_{a_3},x_{b_3}\}}, C_{\{x_{a_3},x_{b_3}\},\{x_{a_1},x_{b_1}\}}\Big).\label{Z3}\EDE
        Then ${\mathcal Z}_\alpha$ satisfies (PDE) (\ref{PDE}) and (COV) (\ref{COV}).
    \item [(ii)] Let $\alpha=\{\{a_1,b_1\},\{a_2,b_2\},\{a_3,b_3\}\}\in \LP_3$ be a neighbor pattern. Let $F$ and $G$ be smooth functions of $(r_1,r_2,r_3)\in(0,1)^3$. Suppose that for any $j,k,{\ell}$ such that $\{j,k,{\ell}\}=[3]$, we have (\ref{rainbow1}) and
        \BGE (2r_jr_k\pa_{r_j}\pa_{r_k}+\beta r_j\pa_{r_j}+\beta r_k\pa_{r_k}+\beta r_{\ell}\pa_{r_{\ell}}+\beta ^2) F =(r_{\ell}(r_j+r_k)-(1+r_jr_k))G.\label{neighbor2}\EDE
        Let
\(
H(r_1,r_2,r_3)=(r_1r_2r_3)^{2/\kappa}F(r_1,r_2,r_3)
\).
Let $c_\alpha\in(0,\infty)$.
        Then the  ${\mathcal Z}_\alpha$ defined by (\ref{Z3})
        satisfies (PDE) (\ref{PDE}) and (COV) (\ref{COV}).
    \item [(iii)] Suppose the $F$ in (i) (resp.\ (ii)) extends continuously to $[0,1]^3$ and satisfies \BGE F(r,1,1)=F(1,r,1)=F(1,1,r)=c_F\cdot {}_2F_1 (1-\beta,\beta;2\beta; r ),\quad r\in (0,1),\label{boundary-F}\EDE
for some $c_F\in (0,\infty)$.
Then the ${\mathcal Z}_\alpha$ defined in (i) (resp.\ (ii)) with  $c_\alpha:=  c_F^{-1}\cdot {}_2F_1 (1-\beta,\beta;2\beta;1)^{-1}$ satisfies (ASY) (\ref{ASY}) for $N=3$.
\item [(iv)] Suppose the $F$ in (i) (resp.\ (ii))  extends continuously to $[0,1]^3$ and is positive on $\pa [0,1]^3$. Then the ${\mathcal Z}_\alpha$ defined in (i) (resp.\ (ii)) is positive for all $\kappa\in (0,8)$ (i.e., $\beta>1/2$).
  \end{enumerate}
\label{Prop1}
\end{Theorem}
\begin{proof}
(i) and (ii) follow from the computation before this theorem.

(iii) Let $j\in [5]$. Fix $x_1<\cdots<x_{j-1}<\xi<x_{j+1}<\cdots<x_6$ and let $x_j,x_{j+1}\to \xi$ in $(x_{j-1},x_{j+2})$ with $x_j<x_{j+1}$, where $x_0:=-\infty$ and $x_7:=+\infty$. First, suppose $\{j,j+1\}\not\in\alpha$. Then for each $k\in[3]$, $|x_{a_k}-x_{b_k}|$ tends to a finite positive number. Let $k, {\ell}\in[3]$ be distinct such that $j\in \{a_k,b_k\}$ and $j+1\in \{a_{\ell},b_{\ell}\}$. Then as $x_j,x_{j+1}\to \xi$,
\(
C_{\{x_{a_k},x_{b_k}\},\{x_{a_{\ell}},x_{b_{\ell}}\}} =O(x_{j+1}-x_j),
\)
and if $m\in [3]\setminus \{k,{\ell}\}$, then both $C_{\{x_{a_k},x_{b_k}\},\{x_{a_m},x_{b_m}\}}$ and $C_{\{x_{a_{\ell}},x_{b_{\ell}}\},\{x_{a_m},x_{b_m}\}}$ converge to numbers in $(0,1)$.  Since $F$ extends continuously to $[0,1]^3$, it follows that
\[
F\big(C_{\{x_{a_1},x_{b_1}\},\{x_{a_2},x_{b_2}\}},
C_{\{x_{a_2},x_{b_2}\},\{x_{a_3},x_{b_3}\}},
C_{\{x_{a_3},x_{b_3}\},\{x_{a_1},x_{b_1}\}}\big)\text{ is bounded}.
\]
Recalling that
\(
H(r_1,r_2,r_3)=(r_1r_2r_3)^{2/\kappa}F(r_1,r_2,r_3)
\),
we obtain
\[
H\big(C_{\{x_{a_1},x_{b_1}\},\{x_{a_2},x_{b_2}\}},
C_{\{x_{a_2},x_{b_2}\},\{x_{a_3},x_{b_3}\}},
C_{\{x_{a_3},x_{b_3}\},\{x_{a_1},x_{b_1}\}}\big)
=O\big((x_{j+1}-x_j)^{2/\kappa}\big).
\]
Therefore,
\(
{\mathcal Z}_{\alpha}(\u x)=O\big((x_{j+1}-x_j)^{ \frac 2\kappa}\big)
\).
Hence, since  $\kappa\in (0,8)$,
\[
\frac{{\mathcal Z}_{\alpha}(\u x)}{(x_{j+1}-x_j)^{-2h}}
=O\big((x_{j+1}-x_j)^{\frac 2\kappa+2h}\big)
=O\big((x_{j+1}-x_j)^{\frac{8-\kappa}\kappa}\big)\to 0.
\]

Second, suppose $\{j,j+1\}=\{a_{{\ell}_0},b_{{\ell}_0}\}\in\alpha$. Let $\{{\ell}_1,{\ell}_2\}=[3]\sem \{{\ell}_0\}$. Then, as $x_j,x_{j+1}\to \xi$, $C_{\{x_{a_{{\ell}_1}},x_{b_{{\ell}_1}}\},\{x_{a_{{\ell}_2}},x_{b_{{\ell}_2}}\}}$ and $|x_{a_{{\ell}_s}}-x_{b_{{\ell}_s}}|$, $s=1,2$, remain unchanged, and $C_{\{x_{a_{{\ell}_s}},x_{b_{{\ell}_s}}\},\{x_{a_{{\ell}_0}},x_{b_{{\ell}_0}}\}}$, $s=1,2$, both tend to $1$. Since $F$ has a continuous extension to $[0,1]^3$ and satisfies (\ref{boundary-F}), from (\ref{Z2}) we see that  $\frac{{\mathcal Z}_\alpha(\u x)}{|x_{j+1}-x_j|^{-2h}}\to {\mathcal Z}_{\alpha/\{j,j+1\}}(\u x/\{j,j+1\})$. Thus, ${\mathcal Z}_\alpha$ satisfies (ASY) (\ref{ASY}).

(iv) We postpone the proof of this part to the next subsection.
\end{proof}

 {
\begin{Remark}
For the functions $F$ constructed later, Corollaries~\ref{victory@R}
and~\ref{victory@N-full} show that the relevant hypotheses of
Theorem~\ref{Prop1} are satisfied for every $\beta>1/2$ in both the
rainbow and neighbor cases. Equivalently, the corresponding
identifications hold for every $\kappa\in(0,8)$.
\end{Remark}}

\subsection{Proof of  positivity} \label{subsection-positivity}
In this subsection, we prove  Theorem \ref{Prop1} (iv), thereby completing the proof of Theorem~2.1. 

We denote the \( F \), \( H \), and \( {\mathcal Z}_\alpha \) from (i) and (ii) respectively by
\( \bigl( F^{[\mathrm{R}]}, H^{[\text{R}]}, {\mathcal Z}_\alpha^{[\text{R}]}  \bigr)\)
and
\(\bigl( F^{[\mathrm{N}]}, H^{[\text{N}]}, {\mathcal Z}_\alpha^{[\text{N}]} \bigr)\).
Let $\LP^{[\text{R}]}_3$ (resp. $\LP^{[\text{N}]}_3$) denote the set of rainbow (resp. neighbor) patterns in $\LP_3$.
Let $\sigma^{[\text{R}]}=1$ and $\sigma^{[\text{N}]}=-1$. We need to show that, for $q\in\{\text{R},\text{N}\}$, if $\alpha=\{\{a_j,b_j\}:j\in[3]\}\in \LP^{[q]}_3$, then ${\mathcal Z}^{[q]}_\alpha >0$ on $\mathfrak{X}_6$.

We may pick $j_0\in [3]$ such that $|a_{j_0}-b_{j_0}|=1$ since every link pattern contains at least one adjacent link. By relabeling elements of $\alpha$ we may assume that $j_0=3$ and $a_3<b_3$. Let $i_0=a_3$ and $i_\infty=b_3$. Let $\u x=(x_1,\dots,x_6)\in \mathfrak{X}_6$. Let $w=\frac1{x_{i_\infty}-x_{i_0}}$ and $y_{a_j}=\frac1{x_{i_\infty}-x_{a_j}}$, $y_{b_j}=\frac1{x_{i_\infty}-x_{b_j}}$, $j=1,2$. Then all $y_{a_j}$ and $y_{b_j}$ lie on the left side of $w$. By relabeling the two remaining links, and if necessary exchanging the two endpoints within each unordered pair, we further assume that $0<w-y_{a_j}<w-y_{b_j}$, $j=1,2$, and $w-y_{a_1}<w-y_{a_2}$. This relabeling does not change the type \(q\in\{\text{R},\text{N}\}\). Let $u_j=y_{a_j}$ and $v_j=y_{b_j}$, $j=1,2$. Note that if $q=\text{R}$, $w>u_1>u_2>v_2>v_1$, and if $q=\text{N}$, $w>u_1>v_1>u_2>v_2$. So we define $$\mathfrak{E}^{[\text{R}]}_5=\{(w,u_1,v_1,u_2,v_2)\in \R^5:w>u_1>u_2>v_2>v_1\};$$
   $$\mathfrak{E}^{[\text{N}]}_5=\{(w,u_1,v_1,u_2,v_2)\in \R^5:w>u_1>v_1>u_2>v_2\}.$$
For $q\in\{\text{R},\text{N}\}$, let
\BGE {\mathcal X}^{[q]}(w,u_1,v_1,u_2,v_2)=F^{[q]}\Big(\frac{w-u_1}{w-v_1}, \frac{w-u_2}{w-v_2}, \Big(\frac{(u_1-u_2)(v_1-v_2)}{(u_1-v_2)(v_1-u_2)}\Big)^{\sigma^{[q]}}\Big);\label{X}\EDE
\BGE {\mathcal Y}^{[q]}(w,u_1,v_1,u_2,v_2)=H^{[q]}\Big(\frac{w-u_1}{w-v_1},\frac{w-u_2}{w-v_2}, \Big(\frac{(u_1-u_2)(v_1-v_2)}{(u_1-v_2)(v_1-u_2)}\Big)^{\sigma^{[q]}}\Big).\label{Y}\EDE
To prove that ${\mathcal Z}^{[q]}_\alpha>0$ on $\mathfrak{X}_6$, it suffices to show that, equivalently, ${\mathcal X}^{[q]}>0$ or ${\mathcal Y}^{[q]}>0$ on $\mathfrak{E}^{[q]}_5$.


Since ${\mathcal Z}^{[q]}_\alpha$ satisfies (PDE) (\ref{PDE}), ${\mathcal X}^{[q]}$ satisfies ${\mathcal O}{\mathcal X}^{[q]}=0$, where $\mathcal O$ is as in (\ref{O}), and $\mathcal Y^{[q]}$ satisfies $\mathcal N \mathcal Y^{[q]}=0$, where $\mathcal N$ is the operator defined in (\ref{N}).
Equivalently, after multiplying by \(\kappa/2\), these equations read
$$\frac\kappa 2 \pa_w^2 {\mathcal X}^{[q]}+\sum_{j=1}^2 \Big(\frac {2\pa_{w} {\mathcal X}^{[q]}} {w-u_j} -\frac {2\pa_{w} {\mathcal X}^{[q]}} {w-v_j}+\frac {2\pa_{u_j} {\mathcal X}^{[q]}} {u_j-w} +\frac {2\pa_{v_j} {\mathcal X}^{[q]}} {v_j-w}-\Big(2-\frac 8\kappa\Big)\frac { {\mathcal X}^{[q]}}{w-v_j}\Big(\frac 1{w-u_j}-\frac 1{w-v_j}\Big)\Big)$$
\BGE +\frac 4\kappa(1-\sigma^{[q]})\Big(\frac 1{w-u_1}-\frac 1{w-v_1}\Big)\Big(\frac 1{w-u_2}-\frac 1{w-v_2}\Big){\mathcal X}^{[q]} =0;\label{X-PDE}\EDE
   \BGE \frac\kappa 2 \pa_w^2 {\mathcal Y}^{[q]}+\sum_{j=1}^2 \Big(\frac {2\pa_{u_j} {\mathcal Y}^{[q]}} {u_j-w} +\frac {2\pa_{v_j} {\mathcal Y}^{[q]}} {v_j-w}-2h\Big(\frac 1{w-u_j}-\frac 1{w-v_j}\Big)^2 {\mathcal Y}^{[q]}\Big)=0.\label{Y-PDE}\EDE

To complete the proof, we now use a probabilistic input. For the reader's convenience, we briefly recall the facts about chordal SLE that will be used below.  We use the overline to denote the complex conjugate.  Let $W_t$, $0\le t<T$, be a real valued continuous function, where $T\in (0,\infty]$. For each $z\in\C$, let $g_t(z)$  be the solution of the Loewner equation:
 \BGE \pa_t g_t(z)=\frac 2{g_t(z)-W_t}, \quad g_0(z)=z.\label{Loewner}\EDE
The solution exists for $t\in [0,\tau_z)$ for some time $\tau_z\in [0,T]$. So for each $t\in [0,T)$, $g_t$ is defined on $\{z\in \C:\tau_z>t\}$. It turns out that each $g_t$ is a conformal map from a domain symmetric about the real axis onto the complement of a real interval
and satisfies  $g_t(\lin z)=\lin{g_t(z)}$ and $g_t(z)=z+\frac {2t}z+O(\frac 1{z^2})$ as $z\to \infty$. When $\tau_z<\infty$, $g_t(z)\to W_{\tau_z}$ as $t\to \tau_z^-$. Let $K_t=\{z\in\HH:\tau_z\le t\}$. Then $g_t$ maps $\HH\sem K_t$ conformally onto $\HH$. The $g_t$ and $K_t$ are called chordal Loewner maps and hulls, respectively, driven by $W_t$.

If for every $0\le t<T$, $g_t^{-1}|_{\HH}$ extends continuously to $\HH\cup\R$, and $\eta(t):= g_t^{-1}(W_t)$, $0\le t<T$, is continuous, then we call $\eta$ the chordal Loewner curve driven by $W$. When such $\eta$ exists, it is a continuous curve in $\HH\cup\R$ started from $W_0$, and for each $t\in [0,T)$,   $\HH\sem K_t=\{z\in\HH:\tau_z>t\}$ is the unbounded component of $\HH\sem \eta([0,t])$, and $\{z\in\C:\tau_z>t\}$ is the unbounded component of $\C\sem (\eta([0,t])\cup \lin{\eta([0,t])})$.
In the case that $W_t=w+\sqrt\kappa B_t$, where $w\in \R$, $\kappa>0$, and $B_t$ is a standard Brownian motion, then the chordal Loewner curve $\eta$ driven by $W_t$ exists and satisfies $\lim_{t\to\infty} \eta(t)=\infty$ (\cite{RS}), and  is called a chordal SLE$_\kappa$ curve in $\HH$ from $w$ to $\infty$. 

We now begin the positivity proof. We first treat the range $\kappa\in(0,4]$. 
Let $(w,u_1,v_1,u_2,v_2)\in \mathfrak{E}^{[q]}_5$. Let $W_t=w+\sqrt\kappa B_t$. Let $g_t$ and $K_t$ be the chordal Loewner maps and hulls, respectively, driven by $W_t$, and let $\eta$ be the corresponding curve.  Then $\eta$ is a simple curve and intersects $\R$ only at $w$.  
So $U^j_t:=g_t(u_j)$ and $V^j_t:=g_t(v_j)$, $j=1,2$, are well defined for all $t\ge 0$ and satisfy the ODEs: $dU^j_t=\frac {2dt}{U^j_t-W_t}$ and $dV^j_t=\frac{2dt}{V^j_t-W_t}$. Moreover,  $(W_t,U^1_t,V^1_t,U^2_t,V^2_t)\in \mathfrak{E}^{[q]}_5$ for all $t\ge 0$.  Let
\BGE M^{[q]}_t={\mathcal Y}^{[q]}(W_t,U^1_t,V^1_t,U^2_t,V^2_t)\exp\Big(-2h\sum_{j=1}^2\int_0^t \Big(\frac 1{W_s-U^j_s} -\frac 1{W_s-V^j_s}\Big)^2 ds\Big).\label{Mt}\EDE
Since the generator of \((W_t,U_t^1,V_t^1,U_t^2,V_t^2)\) is
\(
\frac{\kappa}{2}\partial_w^2
+\sum_{j=1}^2\left(\frac{2}{u_j-w}\partial_{u_j}
+\frac{2}{v_j-w}\partial_{v_j}\right)
\),
equation (\ref{Y-PDE})  and It\^s formula imply that \(M^{[q]}\) is a local martingale. Since $F$ extends continuously to $[0,1]^3$, $F$ is bounded, and so is $H$ as $|H|\le |F|$. From $\kappa\le 4$ we get $h\ge 0$. So $M^{[q]}$ is uniformly bounded, which implies that $M^{[q]}$ is a closable true martingale, i.e., $M^{[q]}_\infty:=\lim_{t\to \infty} M^{[q]}_t$ almost surely exists, and $\EE[M^{[q]}_\infty]=M^{[q]}_0={\mathcal Y}^{[q]}(w,u_1,v_1,u_2,v_2)$. Thus, in order to show that ${\mathcal Y}^{[q]}(w,u_1,v_1,u_2,v_2)>0$, it suffices to show that a.s.\ $M^{[q]}_\infty> 0$.

Let $r=\max\{w-v_1,w-v_2\}>0$. Let $T=\sup\{t\ge 0:|\eta(t)-w|\le 2r\}<\infty$, and $R_t=|\eta(t)-w|$. Then  $R_t\to\infty$ as $t\to \infty$.  For $j=1,2$ and $t\ge T$, the semi-annulus $A_{r,R_t}:=\{z\in\HH:r<|z-w|<R_t\}$ separates the real interval $[v_j,u_j]$ from the union of the right side of $\eta([0,t])$ and $[w,\infty)$ in $\HH\sem K_{t}$, namely, any curve in $\HH\sem K_{t}$ that joins a point in $[v_j,u_j]$ to a point that lies either on the right side of $\eta([0,t])$ or on $[w,\infty)$ must cross this semi-annulus. By the comparison principle of extremal length (\cite{Ahl}), the extremal distance between $[v_j,u_j]$ and the union of the right side of $\eta([0,t])$ and $[w,\infty)$ in $\HH\sem K_{t}$ is at least the extremal distance between the two semi-circles: $\{z\in\HH:|z-w|=r\}$ and  $\{z\in\HH:|z-w|=R_t\}$ in $A_{r,R_t}$, which is known to be $\log(R_t/r)/\pi$. Note that $g_{t}$ maps $\HH\sem K_{t}$ conformally onto $\HH$, and by the continuous extension, sends $[v_j,u_j]$ to $[V^j_{t}, U^j_{t}]$, and sends  the union of the right side of $\eta([0,t])$ and $[w,\infty)$ to $[W_{t},\infty)$. By conformal invariance of the extremal length, the extremal distance between $[V^j_{t}, U^j_{t}]$ and $[W_{t},\infty)$ in $\HH$, denoted by $L_t$, is at least $\log(R_t/r)/\pi$. Thus, $L_t\to\infty$ as $t\to\infty$. On the other hand, $L_t$ is the unique positive number such that there is a conformal map from $\HH$ onto the rectangle $(0,L_t)\times (0,1)$, which sends $(V^j_t, U^j_t)$ and $[W_t,\infty)$ setwise to the two vertical sides: $\{0\}\times (0,1)$ and $\{L\}\times (0,1)$, respectively. For two boundary intervals in \(\mathbb H\), their extremal distance tends to
\(\infty\) only when the corresponding cross-ratio degenerates. In the present
normalization, from $L_t\to \infty$ we get
\BGE \lim_{t\to\infty} \frac{W_{t}-U^j_{t}}{W_{t}-V^j_{t}}=1,\quad  j=1,2.\label{limi-Rj}\EDE

Next, we consider the limit of $C_t:=\Big(\frac{(U^1_{t}-U^2_{t})(V^1_{t}-V^2_{t})} {(U^1_{t}-V^2_{t})(V^1_{t}-U^2_{t})} \Big)^{\sigma^{[q]}}$ as $t\to \infty $. 
Define
\BGE \til g_t(z)=\frac{g_t(z)-g_t(u_1)}{g_t'(u_1)},\quad t\ge 0.\label{til-gt}\EDE
Then $\til g_t$ is the conformal map from $\Omega_t:=\C\sem (\eta([0,t])\cup \lin{\eta([0,t])})$ onto $\C\sem \til I_t$ for some compact real interval $\til I_t$, which satisfies $\til g_t(u_1)=0$ and $\til g_t'(u_1)=1$. Let $\Omega_\infty$ denote the  connected component of $\C\sem (\eta\cup \overline \eta)$ that contains $u_1$. As $t\to \infty$, $\til g_t$ converges locally uniformly in $\Omega_\infty$ to the conformal map $\til g_\infty$ from $\Omega_\infty$ onto $\C\sem [a,\infty)$ for some $a>0$ that satisfies $\til g_\infty(u_1)=0$ and $\til g_\infty'(u_1)=1$. Let $\til U^j_t=\til g_t(u_j)$ and $\til V^j_t=\til g_t(v_j)$, $t\ge 0$, $j\in[2]$. Then, as $t\to\infty$, $\til U^j_t\to \til U^j_\infty$ and $\til V^j_t\to \til V^j_\infty$, and so \BGE C_t = \til C_t:=\Big(\frac{(\til U^1_{t}-\til U^2_{t})(\til V^1_{t}-\til V^2_{t})} {(\til U^1_{t}-\til V^2_{t})(\til V^1_{t}-\til U^2_{t})} \Big)^{\sigma^{[q]}}\to \Big(\frac{(\til U^1_{\infty}-\til U^2_{\infty})(\til V^1_{\infty}-\til V^2_{\infty})} {(\til U^1_{\infty}-\til V^2_{\infty})(\til V^1_{\infty}-\til U^2_{\infty})} \Big)^{\sigma^{[q]}}=:\til C_\infty .\label{limi-R3}\EDE
The four boundary points remain distinct in the limiting component
\(\Omega_\infty\), and their images under \(\widetilde g_\infty\) are therefore
four distinct real points. Thus,  $\til C_\infty\in (0,1)$.

Combining the results in the above two paragraphs with  (\ref{Y}) and the positivity of $F$ on $\pa [0,1]^3$, we get
$$\lim_{t\to \infty} {\mathcal Y}^{[q]}(W_t,U^1_t,V^1_t,U^2_t,V^2_t)=\til C_\infty^{\frac 2\kappa} F(1,1,\til C_\infty )>0.$$
In order to show that a.s.\ $M^{[q]}_\infty>0$, we need to show that, for $j\in [2]$, a.s.\
\BGE \int_0^\infty \Big(\frac 1{W_s-U^j_s} -\frac 1{W_s-V^j_s}\Big)^2 ds<\infty.\label{boundary-cr}\EDE
We derive the inequality by understanding the geometric meaning of the integral.

Since $U^j_t=g_t(u_j)$ and $V^j_t=g_t(v_j)$, from (\ref{Loewner}) we get
$$ \frac{d}{dt} \log(U^j_t-V^j_t)= \frac{-2 }{(W_t-U^j_t)(W_t-V^j_t)},\quad j\in [2].$$ 
Differentiating (\ref{Loewner}) w.r.t.\ $z$ and then substituting $z=u_j$ and $z=v_j$, we get
$$ \frac d{dt} \log(g_t'(u_j))=\frac{-2 }{(W_t-U^j_t)^2},\quad \frac d{dt} \log(g_t'(v_j))=\frac{-2 }{(W_t-V^j_t)^2},\quad j\in [2].$$ 
Combining the above two displayed formulas, we get, for $t\ge 0$,
$$\int_0^t \Big(\frac 1{W_s-U^j_s} -\frac 1{W_s-V^j_s}\Big)^2 ds =\frac 12 \log\Big(\frac{(U^j_t-V^j_t)^2} {g_t'(u_j)g_t'(v_j) }\Big)\Big|_0^t =\frac 12 \log\Big(\frac{(  U^j_t-  V^j_t)^2}{ g_t'(u_j) g_t'(v_j)(u_j-v_j)^2 }\Big) $$
$$=\frac 12 \log\Big(\frac{(\til U^j_t-\til V^j_t)^2}{\til g_t'(u_j)\til g_t'(v_j)(u_j-v_j)^2 }\Big) \to  \frac 12 \log\Big(\frac{(\til U^j_\infty-\til V^j_\infty)^2}{\til g_\infty'(u_j)\til g_\infty'(v_j)(u_j-v_j)^2 }\Big)\in\R,\quad \text{as }t\to \infty .$$
This proves the desired (\ref{boundary-cr}).

Finally, assume $\kappa\in (4,8)$. 
We work on a different Loewner process. Let $W_t$ and $U^j_t,V^j_t$, $j=1,2$, $0\le t<\tau$, be the solution, defined on its maximal interval \([0,\tau)\), of the following SDE/ODE system:
\BGE dW_t=\sqrt\kappa dB_t+\sum_{j=1}^2 \Big(\frac 2{W_t-U^j_t}-\frac 2{W_t-V^j_t}\Big)dt,\quad W_0=w;\label{dWt-2}\EDE
\BGE dU^j_t=\frac {2dt}{U^j_t-W_t},\quad U^j_0=u_j;\quad dV^j_t=\frac {2dt}{V^j_t-W_t},\quad V^j_0=v_j.\label{dUV}\EDE
Let $g_t$ and $K_t$ be the chordal Loewner maps and hulls, respectively, driven by $W_t$. Then $U^j_t=g_t(u_j)$ and $V^j_t=g_t(v_j)$, $j\in [2]$. The corresponding Loewner curve, denoted by $\eta$, exists, which is called a chordal SLE$_\kappa(2,-2,2,-2)$ curve in $\HH$ started from $w$ with the corresponding force points $u_1,v_1,u_2,v_2$ (\cite{LSW-8/3,MS1}).

We denote $\rho_{u_j}=2$ and $\rho_{v_j}=-2$, $j\in [2]$. In the rainbow case, $W_t>U^1_t>U^2_t>V^2_t>V^1_t$ on $[0,\tau)$. Since $\rho_{u_1}=2>\frac\kappa 2-2$, $\rho_{u_1}+\rho_{u_2}=4>\frac\kappa 2-2$, $\rho_{u_1}+\rho_{u_2}+\rho_{v_2}=2>\frac\kappa2-2$, and $\rho_{u_1}+\rho_{u_2}+\rho_{v_2}+\rho_{v_1}=0\in (\frac\kappa 2-4,\frac\kappa 2-2)$, by the continuation-threshold description for
SLE\(_\kappa(\underline\rho)\) processes; see \cite[Section 2]{MS1}, a.s.\ $\eta(\tau):=\lim_{t\to \tau^-} \eta(t)\in (-\infty,v_1)$. In the neighbor case, $W_t>U^1_t>V^1_t>U^2_t>V^2_t$ on $[0,\tau)$. Since $\rho_{u_1}=2>\frac\kappa 2-2$, $\rho_{u_1}+\rho_{v_1}=0>\frac\kappa 2-4$, $\rho_{u_1}+\rho_{v_1}+\rho_{u_2}=2>\frac\kappa2-2$, and $\rho_{u_1}+\rho_{v_1}+\rho_{u_2}+\rho_{v_2}=0\in (\frac\kappa 2-4,\frac\kappa 2-2)$, by \cite[Section 2]{MS1}, a.s.\ $\eta(\tau ):=\lim_{t\to \tau^-} \eta(t)\in (-\infty,v_2)\cup (u_2,v_1)$, and $\PP[\eta(\tau)\in (-\infty,v_2)]>0$. In both cases, $\tau<\infty$, and  \(W_t\) and the quantities \(g_t(y)\), for
\(y\in\{u_j,v_j:j=1,2\}\), have finite limits as \(t\uparrow\tau\).
We denote these limits by \(W_\tau\) and \(g_\tau(y)\). We need a couple of lemmas.

\begin{Lemma}
For any $j\in [2]$,
\BGE \lim_{t\to \tau^-} \frac{W_t-U^j_t}{W_t-V^j_t}=\begin{cases}
  1,&\text{if }\eta(\tau)\in (-\infty,v_j);\\
  0,&\text{if }\eta(\tau)\in (v_j,u_j).
\end{cases}\label{cases}\EDE \label{lemma-middle-1}
\end{Lemma}
\begin{proof}
  First, suppose $\eta(\tau)\in (-\infty,v_j)$. The proof follows a similar extremal length argument. Let $r=v_j-\eta(\tau)>0$. Let $T=\sup\{t\in [0,\tau]: |\eta(t)-\eta(\tau)|\ge r/2\}$. Let $\eps_t=|\eta(t)-\eta(\tau)|$. Then  $\eps_t\to 0$ as $t\to \tau^-$. For $t\in [T,\tau)$, the semi-annulus $A_{\eps_t}:=\{z\in\HH:\eps_t<|z-\eta(\tau)|<r\}$ separates the real interval $[v_j,u_j]$ from the union of the right side of $K_t$ and $[w,\infty)$ in $\HH\sem K_{t}$. So the extremal distance between $[v_j,u_j]$ and the union of the right side of $\eta([0,t])$ and $[w,\infty)$ in $\HH\sem K_{t}$ is at least  $\log(r/\eps_t)/\pi$, which tends to $\infty$ as $t\to\tau^-$. By conformal invariance, the extremal distance between $[g_t(v_j),g_t(u_j)]=[V^j_t,U^j_t]$  and $[W_t,\infty)$ in $\HH$ tends to $\infty$, as $t\to \tau^-$. This implies the limit in first case of (\ref{cases}).

  Second, suppose $\eta(\tau)\in (v_j,u_j)$. In this case, as $t\to\tau^-$, $W_t-g_t(u_j)\to 0$   and $ W_t-g_t(v_j)$ converges to a nonzero real number. 
This implies the limit in the second case of (\ref{cases}).
\end{proof}

\begin{Lemma}
Suppose $\{y_1,y_2,y_3,y_4\}=\{u_1,v_1,u_2,v_2\}$ and   $y_4<y_3<y_2<y_1$.  Then
  $$\lim_{t\to \tau^-}\frac{(g_t(y_1)-g_t(y_2))(g_t(y_3)-g_t(y_4))} {(g_t(y_1)-g_t(y_3))(g_t(y_2)-g_t(y_4))}\begin{cases}
    \in (0,1) &\text{if } \eta(\tau)\in(-\infty,y_4);\\
    \in (0,1) &\text{if } \eta(\tau)\in(y_4,y_3);\\
    =0 &\text{if } \eta(\tau)\in(y_3,y_2);\\
    \in (0,1) &\text{if } \eta(\tau)\in(y_2,y_1).\\
  \end{cases}$$
  By the elementary cross-ratio identity
\(
\frac{(a-b)(c-d)}{(a-c)(b-d)}
+
\frac{(a-d)(b-c)}{(a-c)(b-d)}
=1
\),
the complementary cross-ratio has limit equal to one minus the limit above.
\label{lemma-middle-2}
\end{Lemma}
\begin{proof}
For $0\le t<\tau$, let $\til g_t(z)=\frac{g_t(z)-g_t(y_1)}{g_t'(y_1)}$. Let $\Omega$ denote the connected component of $\C\sem (\eta([0,\tau])\cup \lin{\eta([0,\tau])})$ that contains $y_1$. Then as $t\to \tau^-$, $\til g_t$ converges locally uniformly in $\Omega$ to the conformal map $\til g_\tau$ from $\Omega$ onto $\C\sem [a,\infty)$ for some $a>0$ that sends the prime ends $y_1$ and $\eta(\tau)$ respectively to $0$ and $\infty$ and satisfies $\til g_\tau'(y_1)=1$.

If $\eta(\tau)\in(-\infty,y_4)$, then as $t\to \tau^-$, $\til g_t(y_j)\to \til g_\tau(y_j)$, $1\le j\le 4$, and so
 $$ \frac{(g_t(y_1)-g_t(y_2))(g_t(y_3)-g_t(y_4))} {(g_t(y_1)-g_t(y_3))(g_t(y_2)-g_t(y_4))} =  \frac{(\til g_t(y_1)-\til g_t(y_2))(\til g_t(y_3)-\til g_t(y_4))} {(\til g_t(y_1)-\til g_t(y_3))(\til g_t(y_2)-\til g_t(y_4))}$$
\BGE \to  \frac{(\til g_\tau(y_1)-\til g_\tau(y_2))(\til g_\tau(y_3)-\til g_\tau(y_4))} {(\til g_\tau(y_1)-\til g_\tau(y_3))(\til g_\tau(y_2)-\til g_\tau(y_4))}\in (0,1).\label{repeated1}\EDE

If $\eta(\tau)\in (y_4, y_3)$, then as $t\to \tau^-$, $\til g_t(y_j)\to \til g_\tau(y_j)$, $1\le j\le 3$, and $\til g_t(y_4)\to -\infty$. To see that $\til g_t(y_4)$ diverges to $-\infty$, we just need to note that the extremal distance between $(-\infty,y_4]$ and $[y_2,y_1]$ in $\HH\sem K_t$ tends to $\infty$ as $t\to \tau^-$, which by conformal invariance equals the extremal distance between $(-\infty,\til g_t(y_4)]$ and $[\til g_t(y_2),\til g_t(y_1)]$ in $\HH$. Since the extremal distance tends to $\infty$, it follows that $(\til g_t(y_1)-\til g_t(y_2))/(\til g_t(y_1)-\til g_t(y_4))\to 0$. Also recall that $\til g_t(y_1)$ and $\til g_t(y_2)$ converge to two distinct numbers.  Thus, (\ref{repeated1}) holds with the limit replaced by $ \frac{\til g_\tau(y_1)-\til g_\tau(y_2) } {\til g_\tau(y_1)-\til g_\tau(y_3) }$, which again lies in $(0,1)$.

If $\eta(\tau)\in (y_3, y_2)$, then as $t\to \tau^-$, the extremal distance between $(y_4,y_3)$ and $(y_2,y_1)$ in $\HH\sem K_t$ tends to $\infty$, which implies that the extremal distance between $(g_t(y_4),g_t(y_3))$ and $(g_t(y_2),g_t(y_1))$ in $\HH$ tends to $\infty$. So in this case the cross ratio tends to $0$.

If $\eta(\tau)\in (y_2,y_1)$, then as $t\to \tau^-$, $g_t(y_1)\to W_\tau>g_\tau(y_2)>g_\tau(y_3)>g_\tau(y_4)$, and so
$$ \frac{(g_t(y_1)-g_t(y_2))(g_t(y_3)-g_t(y_4))} {(g_t(y_1)-g_t(y_3))(g_t(y_2)-g_t(y_4))} \to  \frac{(W_\tau-g_\tau(y_2))(g_\tau(y_3)-g_\tau(y_4))} {(W_\tau-g_\tau(y_3))(g_\tau(y_2)-g_\tau(y_4))}\in (0,1).$$
\end{proof}

Now we continue the proof of positivity. For $0\le t<\tau$ and $q\in \{\text{R},\text{N}\}$, define
$$ N^{[q]}_t={\mathcal X}^{[q]}(W_t,U^1_t,V^1_t,U^2_t,V^2_t)\exp\Big(\sum_{j=1}^2 \Big(\frac 8\kappa-2\Big) \int_0^t \frac 1{W_s-V^j_s}\Big(\frac 1{W_s-U^j_s}-\frac 1{W_s-V^j_s}\Big)ds $$
\BGE +\frac 4\kappa(1-\sigma^{[q]})\int_0^t\Big(\frac 1{W_s-U^1_s}-\frac 1{W_s-V^1_s}\Big) \Big(\frac 1{W_s-U^2_s}-\frac 1{W_s-V^2_s}\Big)ds\Big).\label{Nt}\EDE
By (\ref{dWt-2}) , (\ref{dUV}, (\ref{X-PDE}) and It\^o's formula, $N^{[q]}_t$ is a local martingale. Since $F^{[q]}$ is bounded on $[0,1]^3$, ${\mathcal X}^{[q]}(\cdot)$ is uniformly bounded. If $q=\text{R}$, the exponent in (\ref{Nt}) is nonpositive since  $\frac 8\kappa-2<0$ and $\frac 1{W_s-U^j_s}>\frac 1{W_s-V^j_s}>0$, and $1-\sigma^{[q]}=0$. Thus, $N^{[q]}$ is uniformly bounded. In the case $q=\text{N}$, we define $\ha R_t=\frac{(U^1_t-V^2_t)(V^1_t-U^2_t)}{(U^1_t-U^2_t)(V^1_t-V^2_t)}$. Then $\ha R_t\in (0,1)$ on $[0,\tau)$. Using (\ref{dUV}) we find that
$$\frac d{dt} \log(\ha R_t)=2\Big(\frac 1{W_t-U^1_t}-\frac 1{W_t-V^1_t}\Big)\Big(\frac 1{W_t-U^2_t}-\frac 1{W_t-V^2_t}\Big)>0.$$
Thus, since $\sigma^{[\text{N}]}=-1$,
$$\exp\Big(\frac 4\kappa(1-\sigma^{[\text{N}]})\int_0^t\Big(\frac 1{W_s-U^1_s}-\frac 1{W_s-V^1_s}\Big) \Big(\frac 1{W_s-U^2_s}-\frac 1{W_s-V^2_s}\Big)ds\Big)=\Big(\frac{\ha R_t}{\ha R_0}\Big)^{\frac 4\kappa}\le \ha R_0^{-\frac 4\kappa}.$$
Thus,  $N^{[\text{N}]}$ is also uniformly bounded. So $N^{[q]}$ is a closable true martingale, i.e., $N^{[q]}_\tau:=\lim_{t\to \tau^-} N^{[q]}_t$ a.s.\ exists, and $\EE[N^{[q]}_\tau]=N^{[q]}_0 ={\mathcal X}^{[q]}(w,u_1,v_1,u_2,v_2)$.

 By Lemmas \ref{lemma-middle-1}  and   \ref{lemma-middle-2}, as $t\to\tau^-$, the triple
\BGE \Big(\frac{W_t-U^1_t}{W_t-V^1_t}, \frac{W_t-U^2_t}{W_t-V^2_t}, \Big(\frac{(U^1_t-U^2_t)(V^1_t-V^2_t)}{(U^1_t-V^2_t)(V^1_t-U^2_t)}\Big)^{\sigma^{[q]}}\Big)\label{triple}\EDE
converges to a boundary point of $[0,1]^3$, since at least one of its three components tends to $1$ or $0$.  Since $F^{[q]}$ extends continuously to $[0,1]^3$ and is positive on $\pa [0,1]^3$, ${\mathcal X}^{[q]}(W_t,U^1_t,V^1_t,U^2_t,V^2_t)$ tends to a positive number. Thus, $N^{[q]}_\tau\ge 0$.  It remains to show that $\PP[N^{[q]}_\tau>0]>0$. For this purpose, it suffices to show that
\BGE\PP\left[\forall j\in[2], \int_0^\tau \frac 1{W_s-V^j_s}\Big(\frac 1{W_s-U^j_s}-\frac 1{W_s-V^j_s}\Big)ds<\infty \right]>0 . \label{boundary-cr*}\EDE

Using an argument similar to that used in the proof of (\ref{boundary-cr}), we obtain that, on the positive probability event that $E:=\{\eta(\tau)\in (-\infty,v_1\wedge v_2)\}$, for $j\in[2]$,
$$\int_0^\tau \frac 1{W_s-V^j_s} \Big(\frac 1{W_s-U^j_s} -\frac 1{W_s-V^j_s}\Big)ds=\frac 12\log\Big(\frac{\til g_\tau'(v_j)(u_j-v_j) }{ \til U^j_\tau-\til V^j_\tau }\Big) \in\R, $$
where $\til g_\tau$ is the conformal map from the connected component of $\C\sem (\eta([0,\tau])\cup \lin{\eta([0,\tau])})$, which contains $u_1$ and hence, on this event, contains all four points
\(u_1,v_1,u_2,v_2\), onto $\C\sem [a,\infty)$ for some $a>0$, such that $\til g_\tau(u_1)=0$ and  $\til g_\tau'(u_1)=1$. So we get (\ref{boundary-cr*}) and thus complete the proof of Theorem \ref{Prop1} (iv).

\section{Recursions and Formal Power Series Solutions} \label{section:recursions}
\subsection{Recursion relations for the coefficient array}
In this section we apply the power-series method to the two PDE systems arising from
Theorem~\ref{Prop1}: the rainbow system  (\ref{rainbow1}) and (\ref{rainbow2}) and the neighbor system
(\ref{rainbow1}) and (\ref{neighbor2}). This yields formal power-series solutions in the two
cases. Although the SLE range corresponds to \(\beta>1/2\), the coefficient
recursions below are algebraic and will be written for all \(\beta>0\).

Assume that $F$ and $G$ are analytic on $[0,1)^3$ and their power series expansions at the origin are
\BGE F(\u r)=\sum_{\u n\in\N_0^3}  A_{\u n} \u r^{\u n},\quad G(\u r)=\sum_{\u n\in\N_0^3} B_{\u n} \u r^{\u n},\label{FG} \EDE
where $\N_0:=\N\cup\{0\}$, and $\u r^{\u n}:=r_1^{n_1}r_2^{n_2}r_3^{n_3}$.
We view $A$ and $B$ as functions on $\N_0^3$ and extend them to $\Z^3$ by zero. Let $e_j$ be the $j$-th unit vector in $\Z^3$.  Then the PDE (\ref{rainbow1}), (\ref{rainbow2}) and (\ref{neighbor2}) are respectively equivalent to the following linear equations for the coefficients $A$ and $B$:
\BGE (n_{\ell}+1)(n_{\ell}+2\beta ) A_{\u n+e_{\ell}}-(n_{\ell}+\beta )(n_{\ell}+1-\beta ) A_{\u n}=B_{\u n}-B_{\u n-e_j}-B_{\u n-e_k} +B_{\u n-e_j-e_k},\quad \u n\in \Z^3;\label{rainbow1'}\EDE
\BGE (2n_jn_k+\beta n_j+\beta n_k-\beta n_{\ell})A_{\u n}=B_{\u n-e_{\ell}}+B_{\u n-e_{\ell}-e_j-e_k}-B_{\u n-e_j} -B_{\u n-e_k},\quad \u n\in \Z^3;\label{rainbow2'}\EDE
\BGE (2n_jn_k+\beta n_j+\beta n_k+\beta n_{\ell}+\beta ^2)A_{\u n}=B_{\u n-e_{\ell}-e_j} +B_{\u n-e_{\ell}-e_k}-B_{\u n}- B_{\u n-e_j-e_k},\quad \u n\in \Z^3.\label{neighbor2'}\EDE

The rest of this section is devoted to solving two different systems of linear equations for the coefficients. The first consists of (\ref{rainbow1'}) and (\ref{rainbow2'}); and the second consists of (\ref{rainbow1'}) and (\ref{neighbor2'}).

We identify  $\R^{\N_0^3}$ with the subset of $\R^{\Z^3}$ consisting of functions that vanish on $\Z^3\sem \N_0^3$. The value of a function $A\in \R^{\Z^3}$ at $\u n=(n_1,n_2,n_3)\in\Z^3$ will be written as  $A_{\u n}$ or $A_{(n_1,n_2,n_3)}$.
A function  $A\in\R^{\Z^3}$  is called symmetric if, for any $(n_1,n_2,n_3)\in\Z^3$ and permutation $\sigma\in {\mathfrak S}_{[3]}$, $A_{(n_1,n_2,n_3)}=A_{(n_{\sigma(1)},n_{\sigma(2)},n_{\sigma(3)})}$.

For $j\in [3]$, define
\BGE P^j_{\u n}=n_j\quad \text{and}\quad Q^j_{\u n}=n_j+\beta,\label{PQ}\EDE
when $\u n=(n_1,n_2,n_3)$.  Define $f_\pm: \Z\to \R$ by
\BGE f_+(n)=n(n-1+2\beta ),\quad f_-(n)=(n+\beta )(n+1-\beta ).\label{fpm}\EDE
For $j\in[3]$, define operators $\nabla_j$, $\L _j$ and ${\mathcal H}_j$ on $\R^{\Z^3}$ such that for $\u n=(n_1,n_2,n_3)\in\Z^3$,
\BGE (\nabla_j  A)_{\u n}=A_{\u n}-A_{\u n-e_j},\quad (\L _j A)_{\u n}=f_+(n_j+1) A_{\u n+e_j}-f_-(n_j) A_{\u n};\label{paDj}\EDE
\BGE ({\mathcal H}_j A)_{\u n}= (\nabla_j  \L _j A)_{\u n}=f_+(n_j+1) A_{\u n+e_j}-(f_+(n_j)+f_-(n_j))A_{\u n}+f_-(n_j-1) A_{\u n-e_j}.\label{L=}\EDE

We observe that $\nabla_j  $ and $\L _j$ are operators on $\R^{\N_0^3}$, i.e., if $A\in \R^{\N_0^3}$, then $\nabla_j  A,\L _j A\in \R^{\N_0^3}$. To see this, let $A\in \R^{\N_0^3}$. If $\u n\in\Z^3\sem \N_0^3$, then $\u n-e_j\in \Z^3\sem \N_0^3$, and so $(\nabla_j  A)_{\u n}=A_{\u n}-A_{\u n-e_j}=0$. This shows that $\nabla_j  A \in \R^{\N_0^3}$. For $\L _j A$, we note that if $\u n\in\Z^3\sem \N_0^3$,  either $\u n+e_j\in \Z^3\sem \N_0^3$ or $n_j=-1$. In the first case,  $(\L _j A)_{\u n}$ equals $0$ as it is a linear combination of $A_{\u n+e_j}$ and $A_{\u n}$, which are both $0$. In the second case, since $P^j_{\u n+e_j}=0$, we get $(\L _j A)_{\u n}=0$ using $(f_+(P^j))_{\u n+e_j}=f_+(0)=0$. This shows that $\L _j A\in \R^{\N_0^3}$. So ${\mathcal H}_j=\nabla_j  \L _j$ is also an operator on $\R^{\N_0^3}$.

Using the symbols $\nabla_j  $ and $\L _j$ we may rewrite (\ref{rainbow1'}) as
 \BGE (\L _lA)_{\u n}=(\nabla_j \nabla_k  B)_{\u n},\quad \u n\in \Z^3,\quad \{j,k,{\ell}\}=[3]. \label{AB1@R}\EDE
It follows directly from the definition that, for any distinct $j, k\in[3]$, $\nabla_j $ and $\L _j$ commute with $\nabla_k $ and $\L _k$, and so ${\mathcal H}_j=\nabla_j \L _j$ commutes with ${\mathcal H}_k=\nabla_k \L _k$.
Thus, (\ref{AB1@R}) implies that, for all $j,k\in [3]$,
\BGE  ({\mathcal H}_j A)_{\u n}=({\mathcal H}_k A)_{\u n},\quad \u n\in\Z^3,\quad j,k\in [3]. \label{Ljk=}\EDE

\begin{Lemma}
\begin{itemize}
  \item [(i)]   For any sequence $(a_m)_{m=0}^\infty$, there is a unique $A\in \R^{\N_0^2}$ that satisfies the
two-dimensional analogue of (\ref{Ljk=}) for all  $j,k\in [2]$, and  $A_{m e_1}=a_m$ for all $m\in\N_0$. Moreover, such $A$ is symmetric, and each $A_{\u n}$ is a linear combination of the $a_m$'s. In other words, there are $C_{\u n}^{(m)}\in\R$, $\u n\in\N_0^2$, $m\in\N_0$, depending only on $\beta$ such that, for each $\u n\in\N_0^2$, there are only finitely many $m$ such that $C_{\u n}^{(m)}\ne 0$, and $A_{\u n}=\sum_{m=0}^\infty C_{\u n}^{(m)} a_m$.
  \item [(ii)] The above statement holds with $ {\N_0^3}$ and $[3]$ in place of $ {\N_0^2}$ and $[2]$, respectively.
\end{itemize}
\label{construction}
\end{Lemma}
\begin{proof}
 (i) We construct $A_{(n_1,n_2)}$ using an induction on $n_2$. We always extend \(A\) by zero outside \(\mathbb N_0^2\). When \(n_2=0\),
we define \(A_{(n_1,0)}=a_{n_1}\) for \(n_1\in\mathbb N_0\), and set
\(A_n=0\) for \(n\notin\mathbb N_0^2\).
For $N\in \N_0$, let $P_N^2$ denote the proposition that $A_{\u n}$ is defined for every $\u n\in \Z\times \Z_{\le N}$, that each such $A_{\u n}$ is a linear combination of the
$a_m$'s, and that ${\mathcal H}_1A={\mathcal H}_2A$ holds on $\Z\times \Z_{\le N-1}$, where
$\Z_{\le m}:=\{n\in\Z:n\le m\}$. At the beginning, $\text{P}^2_0$ holds since ${\mathcal H}_1 A={\mathcal H}_2 A=0$ on $\Z^2\sem \N_0^2$.
  Suppose at some step $\text{P}^2_N$ holds. In order to show $\text{P}^2_{N+1}$, we construct $A_{(n_1,N+1)}$ for $n_1\in\Z$. Expanding $({\mathcal H}_1 A)_{\u n}=({\mathcal H}_2 A)_{\u n}$ for $\u n=(n_1,N)$, we get a linear equation involving only $A_{(n_1,N+1)}$, $A_{(n_1,N)}$, $A_{(n_1-1,N)}$, $A_{(n_1+1,N)}$, $A_{(n_1,N-1)}$, which have been constructed at this step except for $A_{(n_1,N+1)}$. We also see that the coefficient of $A_{(n_1,N+1)}$ is the positive number $f_+(N+1)=(N+1)(N+2\beta)$. Thus, there is a unique way to define $A_{(n_1,N+1)}$ such that $({\mathcal H}_1 A)_{\u n}=({\mathcal H}_2 A)_{\u n}$ holds at $(n_1,N)$. Moreover, since $A_{(n_1,N)}$, $A_{(n_1-1,N)}$, $A_{(n_1+1,N)}$, $A_{(n_1,N-1)}$ are all linear combinations of $a_n$'s, so is $A_{(n_1,N+1)}$.
  This finishes the induction step. By induction, we can construct $A_{\u n}$ for all $\u n\in\N^2$, such that ${\mathcal H}_1 A={\mathcal H}_2 A$, and each $A_{\u n}$ is a linear combination of $a_m$'s. We have the uniqueness of $A$ since the construction in each step is uniquely determined. The symmetry of this two-dimensional array will follow by restricting the symmetric three-dimensional array constructed in part (ii) to the plane \(n_3=0\).

  (ii) We construct $A_{(n_1,n_2,n_3)}$ using an induction on $n_3$. When $n_3=0$, we define  $A_{(n_1,n_2,0)}$ to be the $A_{(n_1,n_2)}$ from (i),  and set
\(A_n=0\) for \(n\notin\mathbb N_0^3\). Then $A_{n_1 e_1}=a_{n_1}$ for each $n_1\in\N_0$, each $A_{(n_1,n_2,0)}$ is a linear combination of $a_m$'s;  ${\mathcal H}_1 A={\mathcal H}_2 A={\mathcal H}_3 A$ hold on $\Z^2\times \Z_{\le -1}$; and ${\mathcal H}_1 A={\mathcal H}_2 A$ holds on $\Z^2\times \{0\}$. For $N\in\N_0$, let $\text{P}^3_N$ denote the proposition that $A_{\u n}$ have been constructed as linear combinations of $a_n$'s for all $\u n\in \Z^2\times \Z_{\le N}$ such that ${\mathcal H}_3 A={\mathcal H}_1 A = {\mathcal H}_2 A $ on $\Z^2\times \Z_{\le N-1}$ and  ${\mathcal H}_1 A={\mathcal H}_2 A$ on $\Z^2\times \{N\}$. Now we know that $\text{P}^3_0$ holds. Suppose $\text{P}^3_N$ holds for some $N\in\N_0$. In order to show $\text{P}^3_{N+1}$, we construct $A_{(n_1,n_2,N+1)}$ for $(n_1,n_2)\in\N_0^2$. Expanding $({\mathcal H}_3 A)_{\u n}=({\mathcal H}_1 A)_{\u n}$ for $\u n=(n_1,n_2,N)$, we get a linear equation involving only $A_{(n_1,n_2,N+1)}$, $A_{(n_1,n_2,N)}$, $A_{(n_1-1,n_2,N)}$, $A_{(n_1+1,n_2,N)}$, $A_{(n_1,n_2,N-1)}$, which have been constructed as linear combinations of $a_n$'s at this step except for $A_{(n_1,n_2,N+1)}$. We also see that the coefficient of $A_{(n_1,n_2,N+1)}$ is the positive number $f_+(N+1)$. Thus, there is a unique way to define $A_{(n_1,n_2,N+1)}$ as a linear combination of $a_m$'s such that $({\mathcal H}_3 A)_{\u n}=({\mathcal H}_1 A)_{\u n}$ holds at $\u n=(n_1,n_2,N)$. Combining with the induction hypothesis, we see that   $({\mathcal H}_3 A)_{\u n}=({\mathcal H}_1 A)_{\u n}=({\mathcal H}_2 A)_{\u n}$ at such $\u n$. To complete the induction step, it remains to show that ${\mathcal H}_1 A={\mathcal H}_2 A$ on $\N_0^2\times \{N+1\}$. Let $\u n=(n_1,n_2,N)\in\N_0^2\times\{N\}$. By the  commutation relations between ${\mathcal H}_j$'s,
  $$({\mathcal H}_3{\mathcal H}_1 A)_{\u n}=({\mathcal H}_1{\mathcal H}_3 A)_{\u n}=({\mathcal H}_1{\mathcal H}_2 A)_{\u n}=({\mathcal H}_2{\mathcal H}_1 A)_{\u n}=({\mathcal H}_2{\mathcal H}_3 A)_{\u n}=({\mathcal H}_3{\mathcal H}_2 A)_{\u n}.$$
On the other hand, for $j=1,2$, by (\ref{L=}),
$$({\mathcal H}_3{\mathcal H}_j A)_{\u n}=f_+(N+1) ({\mathcal H}_j A)_{\u n+e_3}-(f_+(N)+f_-(N))({\mathcal H}_j A)_{\u n}+f_-(N-1) ({\mathcal H}_j A)_{\u n-e_3}.$$
Since $({\mathcal H}_1 A)_{\u n}=({\mathcal H}_2 A)_{\u n}$ and $({\mathcal H}_1 A)_{\u n-e_3}=({\mathcal H}_2 A)_{\u n-e_3}$, by the above two displayed formulas and induction hypothesis, we get $({\mathcal H}_1 A)_{(n_1,n_2,N+1)}=({\mathcal H}_2 A)_{(n_1,n_2,N+1)}$. This  finishes the induction step. So we get the existence of $A$. Since the construction of $A$ in every step is determined, the uniqueness of $A$ follows.

Define $A'$ on $\N_0^3$ by $A'(n_1,n_2,n_3)=A(n_1,n_3,n_2)$, $n_1,n_2,n_3\in\N_0$. Since the system (\ref{Ljk=})  is invariant under permutations of the coordinates, $A'$ also satisfies (\ref{Ljk=}) for all $j,k\in[3]$. Moreover, $A'(n,0,0)=A(n,0,0)=a_n$ for all $n\in\N_0$. By the uniqueness of $A$ we get $A'=A$. Define $A''$ on $\N_0^3$ by $A''(n_1,n_2,n_3)=A(n_3,n_2,n_1)$, $n_1,n_2,n_3\in\N_0$. Then $A''$ also satisfies (\ref{Ljk=}) for all  $j,k\in[3]$, and agrees with $A$ on $\{(0,n,0):n\in\N_0\}$.

We claim that a solution to (\ref{Ljk=}) is also determined by its values on
$\{(0,n,0):n\in\N_0\}$. Indeed, if $U$ and $V$ both satisfy (\ref{Ljk=}) and
$U_{(0,n,0)}=V_{(0,n,0)}$ for all $n\in\N_0$, define
\[
\widetilde U_{(n_1,n_2,n_3)}:=U_{(n_2,n_1,n_3)},\qquad
\widetilde V_{(n_1,n_2,n_3)}:=V_{(n_2,n_1,n_3)},\qquad n_1,n_2,n_3\in\N_0.
\]
Then $\widetilde U$ and $\widetilde V$ also satisfy (\ref{Ljk=}), since this system is
invariant under swapping the first two coordinates, and they agree on
$\{(n,0,0):n\in\N_0\}$. By the uniqueness proved above, $\widetilde U=\widetilde V$,
and hence $U=V$. Applying this to $U=A''$ and $V=A$, we obtain $A''=A$. Since
$A'=A$ and $A''=A$, the array $A$ is invariant under the transpositions $(23)$ and
$(13)$, which generate $S_3$. Therefore $A$ is symmetric. Finally, the symmetry
statement in \textup{(i)} follows by restricting this three-dimensional array to
$\N_0^2\times\{0\}$.
\end{proof}

\subsection{Formal solutions for the rainbow patterns}
We first treat the rainbow patterns. More precisely, we solve the recursion system
corresponding to (\ref{rainbow1'}) and (\ref{rainbow2'}) under the normalization
$A_{\u 0}=1$. The goal is to construct the unique symmetric coefficient array and to
record the boundary data that will later be identified with an Appell $F_1$ function.

For $\u n=(n_1,n_2,n_3)$, write $|\u n|=n_1+n_2+n_3$. We continue to use the quantities
$P^j$ and $Q^j$ from (\ref{PQ}), and we now define $M^j,T:\Z^3\to\R$ by
\BGE
M^j_{\u n}=|\u n|-2n_j+1-\beta,\quad T_{\u n}=|\u n|-1+3\beta.
\label{DT@R}
\EDE

\begin{Theorem}
Let $\beta>0$.  There exists a unique pair of $A,B\in\R^{\N_0^3}$ such that $A_{\u 0}=1$, and, for any $\u n=(n_1,n_2,n_3)\in\Z^3$ and   $\{j,k,{\ell}\}=[3]$, equations (\ref{AB1@R}) and (\ref{rainbow2'}) hold. As before, \(A\) and \(B\) are extended by zero outside \(\mathbb N_0^3\).
Moreover, $A$ and $B$ satisfy the following additional properties.
\begin{enumerate} [label=\textup{(\roman*)},ref=\textup{\roman*} ]
  \item \label{ABR:i}  $A$ and $B$ are symmetric.
  \item \label{ABR:ii}  Formula (\ref{Ljk=}) holds for any $j,k\in[3]$.
  \item  \label{ABR:iii} The boundary coefficients on the face \(n_3=0\) are given by: for any $n_1,n_2\in\N_0$,
  \BGE A_{(n_1,n_2,0)}=\frac{(\beta)_{n_1}(\beta)_{n_2}(1-\beta )_{n_1+n_2}}{(1)_{n_1}(1)_{n_2}(3\beta )_{n_1+n_2}},\quad B_{(n_1,n_2,0)}=-\frac{(\beta )_{n_1+1}(\beta )_{n_2+1}(1-\beta )_{n_1+n_2+1}}{(1)_{n_1}(1)_{n_2}(3\beta )_{n_1+n_2+1}}.
  \label{AB=@R}\EDE
 \item  \label{ABR:iv}
For any $\u n\in\Z^3$, if $\{j,k,{\ell}\}= [3]$, we have
\begin{align}
  (P^jM^j A)_{\u n+e_j}-(P^kM^k A)_{\u n+e_k}&= -((P^j-P^k)M^{\ell} A)_{\u n};\label{++@R}\\
  (P^jT A)_{\u n+e_j}-(Q^kM^{\ell} A)_{\u n-e_k}&=((P^k+Q^j)M^k A)_{\u n}  ; \label{+-@R}\\
 (Q^jM^k A)_{\u n-e_j}-(Q^kM^j A)_{\u n-e_k}&=  ((P^j-P^k)T A)_{\u n}.\label{--@R}
\end{align}
\item \label{ABR:v} If $\{j,k,{\ell}\}= [3]$, using the zero-extension convention, define functions $X^{j,\pm} $  on $\N_0^3$ by
 \begin{align}
   X^{j,+}_{\ulin n}:= &(P^j M^j T A)_{\ulin n+e_j} + ((P^j+2\beta -\frac 12)   M^kM^{\ell} A )_{\ulin n},\label{Xj+@R}\\
   X^{j,-}_{\ulin n}:= & (Q^j   M^k M^lA )_{\ulin n-e_j} + ((P^j-\beta +\frac 12) M^j T A)_{\ulin n}.\label{Xj-@R}
 \end{align}
 Then the six functions $X^{j,+}$ and $X^{j,-}$, $j\in [3]$, are all equal.
 \item \label{ABR:vi} For any $\u n\in\Z^3$, if $\{j, k, {\ell}\}= [3]$, $$2(\nabla_{\ell} B)_{\u n} =(P^jM^j A)_{\u n+e_j}+ (Q^jM^{\ell} A)_{\u n}=(P^kM^k A)_{\u n+e_k}+ (Q^kM^{\ell} A)_{\u n}$$
\BGE =(P^jT A)_{\u n+e_j+e_k}-(Q^jM^k A)_{\u n+e_k}= (P^kT A)_{\u n+e_j+e_k}-(Q^kM^j A)_{\u n+e_j}.\label{all=B@R}\EDE
\item \label{ABR:vii}  If $\beta \in\N$, then $A_{\u n}=0$ when $\max\{M^j_{\u n}:j\in[3]\}\ge 1$, and $B_{\u n}=0$ when $\max\{M^j_{\u n}:j\in[3]\}\ge 0$. So there are only finitely many $\u n$ such that $A_{\u n}\ne 0$ or $B_{\u n}\ne 0$.
\end{enumerate}
\label{Theorem-AB@R}
\end{Theorem}


\begin{Remark} 
We call the restriction of $A$ to $\N_0^2\times\{0\}$ its initial values.
We may use (\ref{ABR:ii}) to derive recursive formulas for $A$  and combine them with the initial values (\ref{AB=@R}) of $A$ and an induction on $n_3$  to compute all values of $A$.

Since $T_{\u n+e_j}>0$ for all $\u n\in\N_0^3$, we may use (\ref{+-@R}) to get another set of recursive formulas for $A$. In the case $\beta\not\in\Z$, $M^j_{\u n}$ never vanish, and we may also use (\ref{++@R}) and the equality $X^{3,+}=X^{3,-}$ respectively to derive   two further sets of recursive formulas. 

We do not have a closed-form formula for $A$. The following natural guess, however, is not correct:
$$A_{(n_1,n_2,n_3)}=\frac{(\beta)_{n_1}(\beta)_{n_2}(\beta)_{n_3}(1-\beta)_{n_1+n_2+n_3}}{(1)_{n_1}(1)_{n_2}(1)_{n_3} (3\beta)_{n_1+n_2+n_3}}.$$
\end{Remark}

\begin{Example} 
When $\beta \in \N$, the formal power series $F$ and $G$ in (\ref{FG}) are polynomials by Theorem \ref{Theorem-AB@R} (vii). We list below the polynomial expressions for $F$ in the cases $\beta=1$ and $\beta=2$:
\[
F_{\beta=1}=1,\qquad
F_{\beta=2}=1-\frac13(r_1+r_2+r_3)+\frac16 r_1r_2r_3.
\]
After choosing the normalization as in Theorem~\ref{Prop1}(iii), these polynomial
solutions give, via Theorem~\ref{Prop1}, the \(N=3\) rainbow pure partition functions
for \(\kappa=4\) and \(\kappa=2\), up to multiplicative constants, coincide with the $N=3$ rainbow pure partition functions in \cite{Global} and \cite{LERW-partition} for $\kappa=4$ and $2$, respectively.
%
\end{Example}

\medskip
\noindent\textit{Proof of Theorem~\ref{Theorem-AB@R}.}
We organize the proof as follows. We first construct the array $A$ and verify some of its
properties. We then introduce the cross rule and square rule for the recurrence relations.
After establishing further properties of $A$, we construct the array $B$ using flows on
$\Z^3$. Finally, we treat the case $\beta\in\N$.

\begin{subproof}[Construction  and initial properties of $A$]
We first derive the values of $A_{(n,0,0)}$ assuming that  $A$ and $B$ satisfy (\ref{AB1@R}), (\ref{rainbow2'}) and $A_{\u 0}=1$. 
Setting ${\ell}=1$, $j=2$, $k=3$, and $\u n=(n,0,0)$ in (\ref{AB1@R}) and  ${\ell}=1$, $j=2$, $k=3$, and $\u n=(n+1,0,0)$ in (\ref{rainbow2'}), we get $$f_+(n+1)A_{(n+1,0,0)}-f_-(n) A_{(n,0,0)}=B_{(n,0,0)},\quad -\beta (n+1)A_{(n+1,0,0)}=B_{(n,0,0)},$$
which together imply $A_{(n+1,0,0)}=\frac{(n+\beta )(n+1-\beta )}{(n+1)(n+3\beta )}\cdot A_{(n,0,0)}$.
Since $A_{\u 0}=1$, we get
\BGE A_{(n,0,0)}=\frac{(\beta )_n(1-\beta )_n}{(1)_n (3\beta )_n},\quad n\in \N_0.\label{ABn00@R}\EDE

Since $A$ satisfies (\ref{AB1@R}), it also satisfies (\ref{Ljk=}). By Lemma \ref{construction}, there is a unique $A\in\R^{\N_0^3}$ that satisfies  (\ref{ABn00@R}) and (\ref{Ljk=}) for all $j,k\in [3]$, and such $A$ is symmetric.
This gives the uniqueness of $A$ and properties (\ref{ABR:i}) and (\ref{ABR:ii}) for $A$.  Since (\ref{AB=@R}) contains (\ref{ABn00@R}), by the uniqueness part of Lemma \ref{construction}, to prove that $A$ satisfies (\ref{ABR:iii}) it suffices to show that the $A$ defined by (\ref{AB=@R}) satisfies (\ref{Ljk=}) for all $\u n\in\N_0^2\times \{0\}$ and $j,k\in[2]$.  Let $A^0$ on $\N_0^2$ be defined by (\ref{AB=@R}). Then for $\u n=(n_1,n_2)$ and distinct $j,k\in[2]$,
$$f_+(n_j+1) A^0_{\u n+e_j}=\frac{(2\beta +n_j)(\beta +n_j)(1-\beta +n_j+n_k)}{(3\beta +n_j+n_k)} A^0_{\u n};$$
and if, in addition, $n_j+n_k-\beta \ne 0$,
 $$f_-(n_j-1) A^0_{\u n-e_j}=\frac{(n_j-\beta )n_j(3\beta -1+n_j+n_k)}{(-\beta +n_j+n_k)} A^0_{\u n}.$$
 The two formulas respectively imply
\BGE f_+(n_1+1)A^0_{\u n+e_1}-f_+(n_2+1)A^0_{\u n+e_2}=(n_1-n_2)(1-\beta +n_1+n_2) A^0_{\u n};\label{f+f+@R}\EDE
\BGE f_-(n_1-1)A^0_{\u n-e_1}-f_-(n_2-1)A^0_{\u n-e_2}=(n_1-n_2)(3\beta -1+n_1+n_2)A^0_{\u n}.\label{f-f-@R}\EDE
We now prove (\ref{f-f-@R}) assuming $n_1+n_2-\beta =0$. In this case, $\beta\in\N$, and so $(1-\beta )_{n_1+n_2}=0$, which implies that $A^0_{\u n}=0$. If $n_1,n_2\ge 1$, then
$$f_-(n_j-1)A^0_{\u n-e_j}=(n_j-\beta )(n_j-1+\beta ) \frac{(\beta )_{n_j-1}(\beta )_{n_k}(1-\beta )_{n_j+n_k-1}}{(1)_{n_j-1}(1)_{n_k}(3\beta )_{n_j+n_k-1}}$$ $$=- \frac{(\beta )_{n_j}(\beta )_{n_k}(1-\beta )_{n_j+n_k-1}}{(1)_{n_j-1}(1)_{n_k-1}(3\beta )_{n_j+n_k-1}},$$
where we used $n_j-\beta=-n_k$. The last expression is symmetric in \(j\) and \(k\). Hence $f_-(n_1-1)A^0_{\u n-e_1}=f_-(n_2-1) A^0_{\u n-e_2}$. So (\ref{f-f-@R}) holds. If $n_1=0$, the two terms on the LHS of (\ref{f-f-@R}) are both $0$ since (a)  $A^0_{\u n-e_1}=0$ because $\u n-e_1\not\in\N_0^2$; (b) $f_-(n_2-1)=0$ because $n_2=\beta$. The same holds if $n_2=0$ by symmetry. Thus,   (\ref{f-f-@R}) holds if any of $n_1,n_2$ is $0$. Combining (\ref{f+f+@R}), (\ref{f-f-@R}) and (\ref{L=}), we get $({\mathcal H}_1 A)_{\u n}-({\mathcal H}_2 A)_{\u n}=0$, as desired. So $A$ satisfies (\ref{ABR:iii}).
\end{subproof}

Now we use \(E_{\u n} (\begin{smallmatrix}
  + & + \\ j & k
\end{smallmatrix} )\) or \(E_{\u n} (\begin{smallmatrix}
  + & + \\ k & j
\end{smallmatrix} )\) to denote (\ref{++@R}); use \(E_{\u n}(\begin{smallmatrix}
  + & - \\ j & k
\end{smallmatrix} )\) or \(E_{\u n} (\begin{smallmatrix}
  - & + \\ k & j
\end{smallmatrix} )\) to denote (\ref{+-@R}); and use  \(E_{\u n}(\begin{smallmatrix}
  - & - \\ j & k
\end{smallmatrix} )\) or \(E_{\u n} (\begin{smallmatrix}
  - & - \\ k & j
\end{smallmatrix} )\) to denote (\ref{--@R}).

We observe that  \(E_{\u n}(\begin{smallmatrix}  \sigma & \tau \\ j & k \end{smallmatrix} )\) is a linear equation involving $A_{\u m}$, $\u m\in S:=\{\u n,\u n+\sigma e_j,\u n +\tau e_k\}$. Let $A\in \R^{\N_0^3}$. If $S\subset\Z^3\sem \N_0^3$, then $A_{\u m}=0$ for all $\u m\in S$ by the zero-extension convention, and so \(E_{\u n}(\begin{smallmatrix}  \sigma & \tau \\ j & k \end{smallmatrix} )\) trivially holds.  If  $S\cap \N_0^3=\{\u n'\}$,  then either $\u n'=\u n$, $n_j=n_k=0$, and $\sigma=\tau=-$, or $\u n'=\u n+e_j$, $\sigma=+$, and $n'_j=0$, or $\u n'=\u n+e_k$, $\tau=+$, and $n'_k=0$. In either case the coefficient of $A_{\u n'}$ is $0$, and so \(E_{\u n}(\begin{smallmatrix}\sigma & \tau \\ j & k \end{smallmatrix} )\) also holds. We call the property that \(E_{\u n}(\begin{smallmatrix} \sigma & \tau \\ j & k \end{smallmatrix} )\) holds in the case that $S$ intersects $\N_0^3$ at zero or one element the \textit{exterior triviality} of $E$.

\begin{Lemma}
Let $\u n=(n_1,n_2,n_3)\in \N_0^3$ and $\sigma,\tau\in \{+,-\}$. Let $\{j, k, {\ell}\}=[3]$. If $E_{\u n}(\begin{smallmatrix} \sigma & \tau \\ j & k \end{smallmatrix} ) $ holds, then $X^{j,\sigma}_{\u n}=X^{k,\tau}_{\u n}$. Conversely, if $X^{j,+}_{\u n}=X^{k,+}_{\u n}$, then $E_{\u n}(\begin{smallmatrix}  + & + \\ j & k \end{smallmatrix} ) $ holds; if $X^{j,-}_{\u n}=X^{k,-}_{\u n}$ and $M^{\ell}_{\u n}\ne 1$, then $E_{\u n}(\begin{smallmatrix} - & - \\ j & k \end{smallmatrix} ) $ holds; if $X^{j,+}_{\u n}=X^{k,-}_{\u n}$ and $M^j_{\u n}\ne 1$, then $E_{\u n}(\begin{smallmatrix} + & - \\ j & k \end{smallmatrix} ) $ holds.
\label{RX@R}
\end{Lemma}
\begin{proof}
   Case 1. $\sigma=\tau=+$. Multiplying (\ref{++@R}) ($E_{\u n}(\begin{smallmatrix}
  + & + \\ j & k
\end{smallmatrix})$)  by $T_{\u n+e_j}=T_{\u n+e_k}=T_{\u n}+1$, we get
$$(TP^jM^jA)_{\u n+e_j}-(TP^kM^k  A)_{\u n+e_k} = - ((T+1)(P^j-P^k)M^{\ell} A)_{\u n}.$$
A direct calculation gives
$$(P^j+2\beta -\frac 12) M^k-(P^k+2\beta -\frac 12) M^j= (T+1)(P^j-P^k) .$$
Combining the above two formulas, we get  $X^{j,+}_{\u n}=X^{k,+}_{\u n}$. On the other hand, if  $X^{j,+}_{\u n}=X^{k,+}_{\u n}$, we may reverse the above argument to get (\ref{++@R}) since $T_{\u n}+1>0$.

Case 2. $\sigma=\tau=-$.
Multiplying (\ref{--@R}) ($E_{\u n}(\begin{smallmatrix}
  - & - \\ j & k
\end{smallmatrix})$) by $M^{\ell}_{\u n-e_j}=M^{\ell}_{\u n-e_k} =M^{\ell}_{\u n}-1$, we get
$$(M^lQ^jM^k A)_{\u n-e_j}-(M^lQ^kM^j  A)_{\u n-e_k} =((M^{\ell}-1)(P^j-P^k)T A)_{\u n}.$$
A direct calculation gives
$$(P^j-\beta +\frac 12)M^j-(P^k-\beta +\frac 12)M^k=-(M^{\ell}-1)(P^j-P^k).$$
Combining the above two formulas, we get $X^{j,-}_{\u n}=X^{k,-}_{\u n}$. Conversely, if  $X^{j,-}_{\u n}=X^{k,-}_{\u n}$, we may reverse the above argument to get (\ref{--@R}) when $M^{\ell}_{\u n}\ne 1$.

Case 3. $\sigma=+$ and $\tau=-$. Multiplying (\ref{+-@R}) ($E_{\u n}(\begin{smallmatrix}
  + & - \\ j & k
\end{smallmatrix})$) by $M^j_{\u n+e_j}=M^j_{\u n-e_k}=M^j_{\u n}-1$, we get
$$(M^jP^jT A)_{\u n+e_j}-(M^jQ^kM^{\ell} A)_{\u n-e_k} =((M^j-1)(P^k+Q^j) M^k A)_{\u n} . $$
A direct calculation gives
$$(P^j+2\beta -\frac 12)M^{\ell}-(P^k-\beta +\frac 12)T=-(M^j-1)(P^k+Q^j).$$
Combining the above two formulas, we get $X^{j,+}_{\u n}=X^{k,-}_{\u n}$. Conversely, if  $X^{j,+}_{\u n}=X^{k,-}_{\u n}$, we may reverse the above argument to get (\ref{+-@R}) when $M^j_{\u n}\ne 1$.

Case 4. $\sigma=-$ and $\tau=+$. This follows from Case 3 after interchanging \(j\) and \(k\).
\end{proof}

\begin{Lemma} [Cross Rule]
  For any $\u n=(n_1,n_2,n_3)\in \N_0^3$ and distinct $j,k\in [3]$,  $E_{\u n}(\begin{smallmatrix}
  - & - \\ j & k
\end{smallmatrix} ) $ and $E_{\u n}(\begin{smallmatrix}
  +& - \\ j & k
\end{smallmatrix} ) $ together imply $E_{\u n}(\begin{smallmatrix}
  -& + \\ j & k
\end{smallmatrix} ) $ and $E_{\u n}(\begin{smallmatrix}
  + & + \\ j & k
\end{smallmatrix} ) $.
\label{cross@R}
\end{Lemma}
\begin{proof}
We first show that $E_{\u n}(\begin{smallmatrix}
  - & - \\ j & k
\end{smallmatrix} ) $ and $E_{\u n}(\begin{smallmatrix}
  +& - \\ j & k
\end{smallmatrix} ) $ together imply $E_{\u n}(\begin{smallmatrix}
  -& + \\ j & k
\end{smallmatrix} ) $. Let ${\ell}\in [3]\sem \{j,k\}$. Swapping $j$ and $k$ in (\ref{+-@R}) ($E_{\u n}(\begin{smallmatrix}
  + & - \\ j & k
\end{smallmatrix})$), we see that $E_{\u n}(\begin{smallmatrix}
  - & + \\ j & k
\end{smallmatrix})$ is equivalent to
\BGE (P^kT A)_{\u n+e_k}-(Q^jM^{\ell} A)_{\u n-e_j}=((P^j+Q^k)M^jA)_{\u n}.\label{target@R}\EDE 
Multiplying (\ref{--@R}) ($E_{\u n}(\begin{smallmatrix}
  - & - \\ j & k
\end{smallmatrix})$), (\ref{+-@R}) ($E_{\u n}(\begin{smallmatrix}
  + & - \\ j & k
\end{smallmatrix})$) and (\ref{target@R}) respectively by $n_j+n_k+\beta $, $n_j+2\beta $ and $n_k+2\beta $, we get
$$ ((P^j+P^k+1+\beta )Q^jM^k  A)_{\u n-e_j}-((P^j+P^k+1+\beta )Q^kM^j  A)_{\u n-e_k}$$ \BGE =((P^j+P^k+\beta )(P^j-P^k)TA)_{\u n},\label{--change@R}\EDE
\BGE (f_+(P^j)T A)_{\u n+e_j}-((P^j+2\beta )Q^kM^{\ell}  A)_{\u n-e_k}=( (P^j+2\beta )(P^k+Q^j)M^kA)_{\u n},\label{+-change@R}\EDE
\BGE (f_+(P^k)T A)_{\u n+e_k}-((P^k+2\beta )Q^jM^lA)_{\u n-e_j}=((P^k+2\beta )(P^j+Q^k)M^jA)_{\u n}.\label{-+change@R}\EDE

Multiplying the equation $({\mathcal H}_j A)_{\u n}-({\mathcal H}_k A)_{\u n}=0$ by $n_1+n_2+n_3+3\beta $, we get
$$ (Tf_+(P^j)  A)_{\u n+e_j}-( T f_+(P^k) A)_{\u n+e_k}+((T+2)f_-(P^j)A)_{\u n-e_j}-((T+2)f_-(P^k)A)_{\u n-e_k}$$
\BGE = ((T+1)2(P^j-P^k)(P^j+P^k+\beta )A)_{\u n}. \label{fjk@R}\EDE

We claim that the linear combination (\ref{--change@R}) $+$ (\ref{+-change@R}) $-$ (\ref{-+change@R}) $-$ (\ref{fjk@R}) reduces to a trivial identity. First, the coefficients of $A_{\u n+e_j}$ and $A_{\u n+e_k}$ are zero. Second, the coefficient of $A_{\u n-e_j}$ is
$$(n_j-1+\beta )[(n_j+n_k+\beta )(n_j+n_{\ell}-n_k-\beta )+(n_k+2\beta )(n_j+n_k-n_{\ell}-\beta )-(n_j+n_k+n_{\ell}+3\beta )(n_j-\beta )]=0;$$
and symmetrically, the coefficient of $A_{\u n-e_k}$ is also $0$. Third, the coefficient of $A_{\u n}$ is
$$(n_j+n_k+\beta )[(n_j-n_k)(n_j+n_k+n_{\ell}-1+3\beta )+(n_j+2\beta ) (n_j+n_{\ell}-n_k+1-\beta )$$
$$-(n_k+2\beta )(n_k+n_{\ell}-n_j+1-\beta )-
2(n_j-n_k) (n_j+n_k+n_{\ell}+3\beta )]=0.$$
This proves the claim. Since the above linear combination is a trivial identity,  equation (\ref{-+change@R}) is a consequence of (\ref{--change@R}), (\ref{+-change@R}) and (\ref{fjk@R}). Since (\ref{--change@R}) and (\ref{+-change@R}) follow from $E_{\u n}(\begin{smallmatrix}
  - & - \\ j & k
\end{smallmatrix} ) $ and $E_{\u n}(\begin{smallmatrix}
  +& - \\ j & k
\end{smallmatrix} ) $,  and (\ref{fjk@R}) holds by (\ref{Ljk=}), we see that $E_{\u n}(\begin{smallmatrix}
  - & - \\ j & k
\end{smallmatrix} ) $ and $E_{\u n}(\begin{smallmatrix}
  +& - \\ j & k
\end{smallmatrix} ) $ together imply (\ref{-+change@R}), which further implies (\ref{target@R}), i.e., $E_{\u n}(\begin{smallmatrix}
  -& + \\ j & k
\end{smallmatrix} ) $, since  the factor $n_k+2\beta $ by which (\ref{target@R}) is multiplied  is positive.

It remains to show that $E_{\u n}(\begin{smallmatrix}
  + & - \\ j & k
\end{smallmatrix} ) $, $E_{\u n}(\begin{smallmatrix}
  - & - \\ j & k
\end{smallmatrix} ) $ and $E_{\u n}(\begin{smallmatrix}
  - & + \\ j & k
\end{smallmatrix} ) $ together imply $E_{\u n}(\begin{smallmatrix}
  +& + \\ j & k
\end{smallmatrix} ) $. By Lemma \ref{RX@R}, the three conditions imply that $X^{j,+}_{\u n}=X^{k,-}_{\u n}=X^{j,-}_{\u n}=X^{k,+}_{\u n}$. So we have $X^{j,+}_{\u n}=X^{k,+}_{\u n}$. By Lemma \ref{RX@R} again we get $E_{\u n}(\begin{smallmatrix}
  +& + \\ j & k
\end{smallmatrix} ) $, as desired.
\end{proof}

\begin{Lemma} [Square Rule]
   For any $\u n=(n_1,n_2,n_3)\in \N_0^3$ and distinct $j, k\in [3]$,  $E_{\u n}(\begin{smallmatrix}
  + & + \\ j & k
\end{smallmatrix} ) $ and $E_{\u n+e_j}(\begin{smallmatrix}
  - & + \\ j & k
\end{smallmatrix} ) $ together imply $E_{\u n+e_k}(\begin{smallmatrix}
  + & - \\ j & k
\end{smallmatrix} ) $ and $E_{\u n+e_j+e_k}(\begin{smallmatrix}
  - & - \\ j & k
\end{smallmatrix} ) $.
\label{square@R}
\end{Lemma}
\begin{proof}
Replacing $\u n$ by $\u n+e_j$ in (\ref{target@R}) ($E_{\u n}(\begin{smallmatrix}
  - & + \\ j & k
\end{smallmatrix})$), we see that $E_{\u n+e_j}(\begin{smallmatrix}
  - & + \\ j & k
\end{smallmatrix} ) $ is equivalent to
\BGE(P^kT A)_{\u n+e_j+e_k}- (Q^jM^{\ell} A)_{\u n}=((P^j+Q^k)M^j A)_{\u n+e_j}. \label{square2@R}\EDE
Swapping $j$ and $k$ in (\ref{square2@R}), we see that $E_{\u n+e_k}(\begin{smallmatrix}
  + & - \\ j & k
\end{smallmatrix} ) $ is equivalent to
\BGE (P^jT A)_{\u n+e_j+e_k}-(Q^kM^{\ell} A)_{\u n}=((P^k+Q^j)M^k A)_{\u n+e_k}. \label{square1-@R}\EDE
Replacing $\u n$ by $\u n+e_j+e_k$ in (\ref{--@R}) ($E_{\u n}(\begin{smallmatrix}
  - & - \\ j & k
\end{smallmatrix})$), we see that $E_{\u n+e_j+e_k}(\begin{smallmatrix}
  - & - \\ j & k
\end{smallmatrix} ) $ is equivalent to
\BGE (Q^jM^k A)_{\u n+e_k}-(Q^kM^j A)_{\u n+e_j}=  ((P^j-P^k)T A)_{\u n+e_j+e_k}. \label{square2-@R}\EDE

Multiplying (\ref{++@R}) ($E_{\u n}(\begin{smallmatrix}
  + & + \\ j & k
\end{smallmatrix} ) $) and (\ref{square2@R}) ($E_{\u n+e_j}(\begin{smallmatrix}
  - & + \\ j & k
\end{smallmatrix} ) $) respectively by $n_j+n_k+\beta +1$ and $n_j+1$, and arranging terms,
we get
\BGE ((Q^j+P^k)P^jM^j  A)_{\u n+e_j}-((Q^j+P^k)P^kM^k  A)_{\u n+e_k}=-((Q^j+P^k+1)(P^j-P^k)M^{\ell} A)_{\u n},\label{square3@R}\EDE
\BGE( P^jP^kTA)_{\u n+e_j+e_k}-( P^j(P^j+Q^k)M^j A)_{\u n+e_j} = ((P^j+1)Q^jM^{\ell}  A)_{\u n}. \label{square4@R}\EDE
Summing (\ref{square3@R}) and (\ref{square4@R}) and observing that $Q^j+P^k=P^j+Q^k$ and
$$-(n_j+n_k+1+\beta )(n_j-n_k)+(n_j+1)(n_j+\beta )=(n_k+1)(n_k+\beta ),$$
we get
$$(n_k+1)(P^jT A)_{\u n+e_j+e_k}-(n_k+1)((Q^j+P^k)M^k A)_{\u n+e_k}=(n_k+1)(Q^kM^{\ell} A)_{\u n}.$$
Since \(n_k+1>0\) for \(n\in\mathbb N_0^3\), we may cancel the common factor
\(n_k+1\) to get (\ref{square1-@R}), i.e., $E_{\u n+e_k}(\begin{smallmatrix}
  + & - \\ j & k
\end{smallmatrix} ) $.

Multiplying (\ref{++@R}) ($E_{\u n}(\begin{smallmatrix}
  + & + \\ j & k
\end{smallmatrix} ) $) and (\ref{square2@R})  ($E_{\u n+e_j}(\begin{smallmatrix}
  - & + \\ j & k
\end{smallmatrix} ) $) respectively by $n_j+\beta $ and $n_j-n_k$, and arranging terms,
we get
\BGE ((Q^j-1) P^jM^j A)_{\u n+e_j}-(Q^j  P^kM^k A)_{\u n+e_k}+(Q^j(P^j-P^k)M^{\ell}  A)_{\u n}=0,\label{square5@R}\EDE
\BGE((P^j-P^k) Q^jM^{\ell} A)_{\u n}+((P^j-P^k-1)(P^j+Q^k)M^j A)_{\u n+e_j}=((P^j-P^k)P^k  TA)_{\u n+e_j+e_k}. \label{square6@R}\EDE
Subtracting (\ref{square5@R}) from (\ref{square6@R}) and observing that
$$(n_j+n_k+1+\beta )(n_j-n_k)-(n_j+1)(n_j+\beta )=-(n_k+1)(n_k+\beta ),$$
we get
$$(n_k+1)(Q^jM^k A)_{\u n +e_k}-(n_k+1)(Q^kM^j A)_{\u n+e_j }=(n_k+1)((P^j-P^k)TA)_{\u n+e_j+e_k}.$$
We may then cancel the common nonzero factor $(n_k+1)$ to get (\ref{square2-@R}), i.e.,  $E_{\u n+e_j+e_k}(\begin{smallmatrix}
  - & - \\ j & k
\end{smallmatrix} ) $.
\end{proof}

\begin{Remark}
We use the names ``cross rule''   and ``square rule''  because the indices that appear in Lemma \ref{cross@R} are $\u n$, $\u n\pm e_j$, and $\u n\pm e_k$, which form a cross, and the indices that appear in Lemma \ref{square@R}  are $\u n$, $\u n+e_j$, $\u n+e_k$, and $\u n+e_j+e_k$, which form a square.
\end{Remark}

\begin{subproof} [Further properties of $A$]
In the \(E\)-notation, Theorem~\ref{Theorem-AB@R} (\ref{ABR:iv}) is the assertion that, for any
\(n\in\mathbb Z^3\), distinct \(j,k\in[3]\), and
\(\sigma,\tau\in\{+,-\}\), the relation
\(E_n({}^{\sigma}_{j}{}^{\tau}_{k})\) holds.
We prove (\ref{ABR:iv}) using an induction on $n_3$.
For $N\in\N_0$, let $\text{P}^E_N$ denote the proposition that for all $\u n\in \Z^2\times \Z_{\le N-1}$ and $\sigma,\tau\in \{+,-\}$,  $E_{\u n}(\begin{smallmatrix}
  \sigma & \tau \\ 1 & 3
\end{smallmatrix} ) $ holds, and for all $\u n\in \Z^2\times \{N\}$ and $\sigma\in \{+,-\}$,   $E_{\u n}(\begin{smallmatrix}
  \sigma & - \\ 1 & 3
\end{smallmatrix} ) $ holds.

We first prove $\text{P}^E_0$. By the exterior  triviality of $E$, it suffices to show that $E_{\u n}(\begin{smallmatrix}
  \sigma & - \\ 1 & 3
\end{smallmatrix} ) $ holds for all $\u n\in \N_0^2\times \{0\}$  and $\sigma\in \{+,-\}$. For $\u n=(n_1,n_2,0)$, $E_{\u n}(\begin{smallmatrix}
  + & - \\ 1 & 3
\end{smallmatrix} ) $ (\ref{+-@R})  and $E_{\u n}(\begin{smallmatrix}
  - & - \\ 1 & 3
\end{smallmatrix} ) $ (\ref{--@R})  are respectively equivalent to
\begin{align*}
    (n_1+1)(n_1+n_2+3\beta ) A_{\u n+e_1}&=(n_1+\beta )(n_1+n_2+1-\beta ) A_{\u n}  ; \\
 (n_1-1+\beta )(n_1+n_2-\beta ) A_{\u n-e_1}&=  n_1(n_1+n_2-1+3\beta ) A_{\u n}.
\end{align*}
By (\ref{AB=@R}), the above formulas hold for all $\u n\in \N_0^2\times \{0\}$. Thus, $\text{P}^E_0$ holds.

By the symmetry of $A$, we also know that, for any distinct $j, k\in [3]$ and $\sigma\in\{+,-\}$, $E_{\u n}(\begin{smallmatrix}
  \sigma & - \\ j & k
\end{smallmatrix} ) $ holds for all $\u n\in\Z^3$ with $n_k=0$. We call this property the \textit{boundary triviality} of $E$.

Suppose we have proved $\text{P}^E_N$ for some $N\in\N_0$. By the cross rule applied to $\u n\in \N_0^2\times \{N\}$ and $(j,k)=(1,3)$ and the exterior  triviality of $E$ applied to $\u n\in (\Z^2\sem \N_0^2)\times \{N\}$, we see that for all $\u n\in \Z^2\times \{N\}$ and $\sigma\in \{+,-\}$,  $E_{\u n}(\begin{smallmatrix}  \sigma & + \\ 1 & 3 \end{smallmatrix} ) $ holds.  By the square rule applied to $\u n\in \N_0^2\times \{N\}$ and $(j,k)=(1,3)$  and the exterior triviality and boundary triviality of $E$ applied to $\u n\in (\Z^2\sem \N_0^2)\times \{N\}$, we then know that for all $\u n\in \Z^2\times \{N+1\}$  and $\sigma\in \{+,-\}$,  $E_{\u n}(\begin{smallmatrix} \sigma & - \\ 1 & 3 \end{smallmatrix} ) $ holds. Here the cases in which the relevant set of indices meets \(\mathbb N_0^3\)
in at most one point are covered by exterior triviality, while the remaining
boundary cases reduce to \(n_k=0\) and are covered by boundary triviality. Thus, $\text{P}^E_{N+1}$ holds.
By induction, $\text{P}^E_N$ holds for all $N\in\N_0$, which implies that $E_{\u n}(\begin{smallmatrix}
  \sigma & \tau \\ 1 & 3
\end{smallmatrix} ) $ holds for all $\u n\in \Z^3$ and $\sigma,\tau\in \{+,-\}$. Finally, by the symmetry of $A$, the above statement holds with $1,3$ replaced by $j$ and $k$ respectively, for  any distinct $j, k\in [3]$. This completes the proof of (\ref{ABR:iv}). By Lemma \ref{RX@R}, we then get (\ref{ABR:v}).
\end{subproof}

\begin{subproof}[Construction  and main properties of $B$]
Fix $\u n\in\Z^3$ and $\{j, k, {\ell}\}= [3]$. From  \(E_{\u n}(\begin{smallmatrix} + & + \\ j & k \end{smallmatrix} )\) (\ref{++@R}), \(E_{\u n+e_k}(\begin{smallmatrix}  + & - \\ j & k\end{smallmatrix} )\) (\ref{square1-@R}) and \(E_{\u n+e_j+e_k}(\begin{smallmatrix}  - & - \\ j & k\end{smallmatrix} )\) (\ref{square2-@R}), we get
$$-(P^jM^j A)_{\u n+e_j}- (Q^jM^{\ell} A)_{\u n}=-(P^kM^k A)_{\u n+e_k}- (Q^kM^{\ell} A)_{\u n}$$
\BGE =(Q^jM^k A)_{\u n+e_k}-(P^jT A)_{\u n+e_j+e_k}= (Q^kM^j A)_{\u n+e_j}-(P^kT A)_{\u n+e_j+e_k}.\label{all=@R}\EDE
Now we define a flow $\phi$ on $\Z^3$ such that for any $\u n\in\Z^3$ and ${\ell}\in[3]$, $\phi(\u n-e_{\ell},\u n)$ equals   the common value in (\ref{all=@R}), and $\phi(\u n,\u n-e_{\ell})=-\phi(\u n-e_{\ell},\u n)$.

Consider a square with vertices $\u n$, $\u n-e_j$, $\u n-e_j-e_k$, $\u n-e_k$, where $\u n\in\Z^3$ and $j, k\in [3]$ are distinct. Let ${\ell}\in [3]\sem \{j,k\}$. Using the definition of \(\phi\) and applying (\ref{all=@R}) to the four directed edges of this square, we obtain
  \begin{align}
 \phi(\u n,\u n-e_j)&=  (P^lM^{\ell} A)_{\u n+e_{\ell}}+ (Q^lM^j A)_{\u n};\label{phi1@R}\\
 \phi(\u n-e_j,\u n-e_j-e_k)&=  (P^lT A)_{\u n+e_{\ell}}-  (Q^lM^j A)_{\u n};\label{phi2@R}\\
 \phi(\u n-e_j-e_k,\u n-e_k)&= - (P^lT A)_{\u n+e_{\ell}}+  (Q^lM^k A)_{\u n};\label{phi3@R}\\
 \phi(\u n-e_k,\u n)&=- (P^lM^{\ell} A)_{\u n+e_{\ell}}- (Q^lM^k A)_{\u n}\nonumber.
  \end{align}
 Summing the four displayed formulas, we get
$$ \phi(\u n,\u n-e_j)+\phi(\u n-e_j,\u n-e_j-e_k)+\phi(\u n-e_j-e_k,\u n-e_k)+\phi(\u n-e_k,\u n)=0.$$ 
Since every finite nearest-neighbor closed lattice loop in \(\mathbb Z^3\)
is generated by elementary squares. Thus \(\phi\) is conservative. Hence there exists a potential
\(B:\mathbb Z^3\to\mathbb R\), unique up to an additive constant, such that
\(
\frac12\phi(u,v)=B(v)-B(u)
\)
for every nearest-neighbor edge \((u,v)\). Note that $\phi(u,v)=0$ when $u,v\in\Z^3\sem \N_0^3$. Indeed, if an edge lies entirely outside \(\mathbb N_0^3\), then in the
appropriate expression in (\ref{all=@R}) all \(A\)-terms either have indices outside
\(\mathbb N_0^3\), or are multiplied by a factor \(P^m=0\). Since \(\mathbb Z^3\setminus\mathbb N_0^3\) is connected and \(\phi=0\) on
all edges contained in this set, \(B\) is constant there; we choose the
additive constant so that this constant is \(0\), and so $B\in\R^{\N_0^3}$. Then we get (\ref{ABR:vi})  using (\ref{all=@R}) and the definitions of $\phi$ and $B$.

  Summing  (\ref{phi1@R}) and (\ref{phi3@R}) and using the relation between $\phi$ and $B$, we get
 $$-2(\nabla_j\nabla_k B)_{\u n}=-(P^{\ell}(T-M^{\ell})A)_{\u n+e_{\ell}}+(Q^{\ell}(M^j+M^k)A)_{\u n}$$
 $$=-(P^{\ell}(2P^{\ell}-2+4\beta)A)_{\u n+e_{\ell}}+((P^{\ell}+\beta)(2 P^{\ell}+2-2\beta)A)_{\u n}=-2 (\L_{\ell} A)_{\u n},$$
 which implies (\ref{AB1@R}).
Summing (\ref{phi1@R}) and (\ref{phi2@R}) and then replacing $\u n$ by $\u n-e_{\ell}$, we get
\BGE 2B_{\u n-e_j-e_k-e_{\ell}}-2B_{\u n-e_{\ell}}=(P^{\ell}(M^{\ell}+T)A)_{\u n}=n_{\ell}(2n_j+2n_k+2\beta ) A_{\u n}.\label{diag1@R}\EDE
Swapping ${\ell}$ with $j$ and ${\ell}$ with $k$, respectively, we get
\BGE 2B_{\u n-e_j-e_k-e_{\ell}}-2B_{\u n-e_j} =n_j(2n_{\ell}+2n_k+2\beta ) A_{\u n}.\label{diag2@R}\EDE
\BGE 2B_{\u n-e_j-e_k-e_{\ell}}-2B_{\u n-e_k} =n_k(2n_{\ell}+2n_j+2\beta) A_{\u n}.\label{diag3@R}\EDE
Applying $($(\ref{diag2@R})$+$(\ref{diag3@R})$-$(\ref{diag1@R})$)/2$, we get (\ref{rainbow2'}).

For each fixed value of the third coordinate, (\ref{AB1@R}) determines
\(\nabla_1\nabla_2B\). Since \(B=0\) outside \(\mathbb N_0^3\), iterated
summation in the first two coordinates then determines \(B\) uniquely.  This shows the uniqueness of $B$.
Since the defining formula (\ref{all=@R}) for \(\phi\) is invariant under simultaneous
permutations of the coordinate labels, the symmetry of \(A\) implies the
corresponding symmetry of \(\phi\), and hence of \(B\).
Thus, $B$ satisfies (\ref{ABR:i}).

To check the boundary formula for \(B\) in (\ref{ABR:iii}), it suffices to verify that
the expression for \(B\) in (\ref{AB=@R}), with zero extension outside
\(\mathbb N_0^2\times\{0\}\), satisfies the first-order relations obtained
from (\ref{AB1@R}) on this face; these relations determine \(B\) uniquely by summation.
We need to show that, for any $\u n=(n_1,n_2,0)$ and distinct $j, k\in [2]$, $(\L _j A)_{\u n}=(\nabla_k  B)_{\u n}$. Indeed, using (\ref{AB=@R}) and the definitions of \(\mathcal L_j\) and \(\nabla_k\),
we compute
$$(\L _j A)_{\u n}=(n_j+2\beta )\cdot \frac{(\beta )_{n_j+1}(\beta )_{n_k} (1-\beta )_{n_j+n_k+1}}{(1)_{n_j}(1)_{n_k}(3\beta )_{n_j+n_k+1}}-(n_j+1-\beta )\cdot \frac{(\beta )_{n_j+1}(\beta )_{n_k} (1-\beta )_{n_j+n_k}}{(1)_{n_j}(1)_{n_k}(3\beta )_{n_j+n_k}}$$
$$=\frac{(\beta )_{n_j+1}(\beta )_{n_k} (1-\beta )_{n_j+n_k}}{(1)_{n_j}(1)_{n_k}(3\beta )_{n_j+n_k+1}}\cdot ((n_j+2\beta )(1-\beta +n_j+n_k)-(n_j+1-\beta )(3\beta +n_j+n_k));$$
$$(\nabla_k  B)_{\u n}=-\frac{(\beta )_{n_j+1}(\beta )_{n_k+1}(1-\beta )_{n_j+n_k+1}} {(1)_{n_j}(1)_{n_k}(3\beta )_{n_j+n_k+1}}+ \frac{n_k(\beta )_{n_j+1}(\beta )_{n_k}(1-\beta )_{n_j+n_k}} {(1)_{n_j}(1)_{n_k}(3\beta )_{n_j+n_k}}$$
$$=\frac{(\beta )_{n_j+1}(\beta )_{n_k} (1-\beta )_{n_j+n_k}}{(1)_{n_j}(1)_{n_k}(3\beta )_{n_j+n_k+1}}\cdot (-(\beta +n_k)(1-\beta +n_j+n_k)+n_k(3\beta +n_j+n_k)). $$
The two bracketed factors are equal after simplification, and hence $(\L _j A)_{\u n}=(\nabla_k  B)_{\u n}$.
\end{subproof}

\begin{subproof}[Integer values of $\beta$]
Finally, suppose $\beta \in\N$.
Using the zero-extension convention, we prove by induction on \(n_3\) that
for \(n\in\mathbb Z^3\),  $M^3_{\u n}\ge 1$, equivalently, $n_1+n_2\ge n_3+\beta$, implies that $A_{\u n}=0$. The proposition trivially holds if $n_3<0$. In the case $n_3=0$, the proposition follows from (\ref{ABn00@R}) and the fact that $(1-\beta)_{n_1+n_2}=0$ when $n_1+n_2\ge \beta$. Suppose for some $N\in\N_0$, the proposition holds for $n_3=N$. Suppose $M^3_{(n_1,n_2,N+1)}\ge 1$, i.e., $n_1+n_2\ge N+1+\beta$. By the induction hypothesis, $A_{(n_1-1,n_2,N)}=A_{(n_1,n_2,N)}=0$. Applying (\ref{+-@R}) with $\u n=(n_1,n_2,N)$, $n_1,n_2\in\N_0$, $j=3$, $k=1$, and ${\ell}=2$, we get
$$ (N+1)(n_1+n_2+N+3\beta) A_{(n_1,n_2,N+1)}=(Q^1M^2 A)_{(n_1-1,n_2,N)}+((P^1+Q^3)M^1A)_{(n_1,n_2,N)}=0.$$
Since $N+1>0$ and $n_1+n_2+N+3\beta>0$, we get $A_{(n_1,n_2,N+1)}=0$ if $(n_1,n_2)\in\N_0^2$. When $(n_1,n_2)\not\in\N_0^2$, $A_{(n_1,n_2,N+1)}=0$ by the zero-extension convention.  Thus, \(M_n^3\ge1\) implies \(A_n=0\). By the symmetry of $A$, $M^1_{\u n}\ge 1$ or $M^2_{\u n}\ge 1$ also implies that $A_{\u n}=0$.

Now we prove that, for any $j\in[3]$, $M^j_{\u n}\ge 0$ implies that $B_{\u n}=0$. By symmetry we assume $j=3$. Suppose $\u n\in\Z^3$ is such that $M^3_{\u n}\ge 0$. Applying (\ref{all=B@R}) with ${\ell}=3$ and $j=1$, we get
$$2B_{\u n}-2B_{\u n-e_3}=(P^1M^1 A)_{\u n+e_1}+(Q^1M^3 A)_{\u n}.$$
The first term on the right-hand side equals $0$ because $M^3_{\u n+e_1}=M^3_{\u n}+1\ge 1$ implies $A_{\u n+e_1}=0$. The second term on the right-hand side equals $0$ because either $M^3_{\u n}=0$ or $M^3_{\u n}\ge 1$, and the latter implies $A_{\u n}=0$. So we get $B_{\u n}=B_{\u n-e_3}$. Since $M^3_{\u n-e_3}=M^3_{\u n}+1\ge 0$, we further get $B_{\u n-e_3}=B_{\u n-2e_3}$.
Iterating the same argument for $m$ steps such that $n_3-m\not\in \N_0$ and using the zero-extension convention, we get $B_{\u n}=B_{\u n-m e_3}=0$.

If $n_1\ge \beta$ and $n_2\ge n_3$, then $M^3_{(n_1,n_2,n_3)}=n_1+n_2-n_3+1-\beta\ge 1$, and so $A_{\u n}=0$. Thus, by the symmetry of $A$, if $\max\{n_1,n_2,n_3\}\ge \beta$, then   \(A_n=0\).
The same argument, with the threshold \(M_n^j\ge0\), gives
\(B_n=0\) whenever \(\max\{n_1,n_2,n_3\}\ge\beta-1\).  Together with the zero-extension convention, this implies that there are only finitely many $\u n$ such that $A_{\u n}$ or $B_{\u n}$ is not zero. This finishes the proof of (\ref{ABR:vii}).
\end{subproof}


\subsection{Formal solutions for the neighbor patterns}
We now turn to the neighbor patterns. The strategy is parallel to that of the rainbow case, but the relevant auxiliary quantities and boundary values are different. We continue to use $P^j$ and $Q^j$ from (\ref{PQ}), while redefining $M^j$ and $T$ by
\BGE M^j_{\u n}= |\u n|-2n_j-1+2\beta ,\quad T_{\u n}=|\u n|+1.\label{DT@N}\EDE

\begin{Theorem}
Let $\beta>0$.  There exists a unique pair of $A,B\in\R^{\N_0^3}$ such that $A_{\u 0}=1$, where $\u 0=(0,0,0)$, and for any $\u n=(n_1,n_2,n_3)\in\Z^3$ and $\{j, k, {\ell}\}=[3]$, (\ref{AB1@R}) and (\ref{neighbor2'}) hold.
Moreover, $A$ and $B$ satisfy the following additional properties.
\begin{enumerate} [label=\textup{(\roman*)},ref=\textup{\roman*} ]
  \item \label{ABN:i} $A$ and $B$ are symmetric.
  \item \label{ABN:ii}  Formula (\ref{Ljk=}) holds for any $j,k\in[3]$.
  \item \label{ABN:iii}  For any $n_1,n_2\in\N_0$,
  \BGE A_{(n_1,n_2,0)}=\frac{(\beta )_{n_1}(\beta )_{n_2}(1-2\beta )_{|n_1-n_2|}}{(1)_{n_1}(1)_{n_2}(2\beta )_{|n_1-n_2|}},\quad B_{(n_1,n_2,0)}=-\frac{(\beta )_{n_1+1}(\beta )_{n_2+1}(1-2\beta )_{|n_1-n_2|}}{(1)_{n_1}(1)_{n_2}(2\beta )_{|n_1-n_2|}}.
  \label{AB=@N}\EDE
 \item \label{ABN:iv}  For any $\u n\in\Z^3$, if $\{j,k,{\ell}\}= [3]$, we have
\begin{align}
  (P^jM^kA)_{\u n+e_j}-(P^kM^jA)_{\u n+e_k}&=((P^j-P^k)TA)_{\u n};\label{++@N}\\
  (P^jM^{\ell} A)_{\u n+e_j}-(Q^kTA)_{\u n-e_k} &=-((P^k+Q^j)M^jA)_{\u n}; \label{+-@N}\\
  (Q^jM^j A)_{\u n-e_j}-(Q^kM^k A)_{\u n-e_k} &= -((P^j-P^k)M^lA)_{\u n}.\label{--@N}
\end{align}
\item \label{ABN:v} If $\{j, k, {\ell}\}=[3]$, define functions $X^{j,\pm} $  on $\N_0^3$ by
 \begin{align}
   X^{j,+}_{\u n}&=(P^jM^kM^lA)_{\u n+e_j}+((P^j+2\beta -\frac 12)M^j TA)_{\u n};\label{Xj+@N}\\
  X^{j,-}_{\u n}&=(Q^jM^j TA)_{\u n-e_j}+((P^j-\beta +\frac 12)M^kM^{\ell} A)_{\u n};\label{Xj-@N}
 \end{align}
 Then the six functions $X^{j,+}$ and $X^{j,-}$, $j\in[3]$, are all equal.
 \item \label{ABN:vi} For any $\u n\in\Z^3$ and $\{j, k, {\ell}\}=[3]$,
 $$2(\nabla_{\ell} B)_{\u n}=(P^jM^k A)_{\u n+e_j}-(Q^j TA)_{\u n}=(P^kM^j A)_{\u n+e_k}-(Q^k TA)_{\u n}$$
\BGE=-(P^jM^{\ell} A)_{\u n+e_j+e_k}-(Q^j M^j A)_{\u n+e_k} =-(P^kM^{\ell} A)_{\u n+e_j+e_k}-(Q^k M^k A)_{\u n+e_j}.\label{all=B@N}\EDE
\item \label{ABN:vii} If $2\beta \in\N$, then $A_{\u n}=B_{\u n}=0$ when $\min\{M^j_{\u n}:j\in[3]\}\le -1$.
\end{enumerate}
\label{Theorem-AB@N}
\end{Theorem}

\begin{Remark}
We may use (ii) to derive recursive formulas for $A$  and combine them with the initial values (\ref{AB=@N}) of $A$ and an induction on $n_3$  to compute all values of $A$.

In the case $2\beta\notin\Z$, the quantities $M^j_{\u n}$, $j\in[3]$, never vanish, and we may also use  (\ref{++@N}), (\ref{+-@N}), and the identity $X^{3,+}=X^{3,-}$ to derive three further sets of recursive formulas.
\end{Remark}

\medskip
\noindent\textit{Proof of Theorem~\ref{Theorem-AB@N}.}   The overall scheme of this proof is the same as in the rainbow case,
but we give full details whenever the argument is not identical.



\begin{subproof} [Construction  and initial properties of $A$]
 We first derive the values of $A_{(n,0,0)}$ assuming that $A$ and $B$ satisfy (\ref{AB1@R}) and (\ref{neighbor2'}).
Substituting ${\ell}=1$, $j=2$, $k=3$, and $\u n=(n,0,0)$ into (\ref{AB1@R}) and   (\ref{neighbor2'}), we get $$f_+(n+1)A_{(n+1,0,0)}-f_-(n) A_{(n,0,0)}=B_{(n,0,0)},\quad (\beta n+\beta ^2)A_{(n,0,0)}=-B_{(n,0,0)},$$
which together imply $A_{(n+1,0,0)}=\frac{(n+\beta )(n+1-2\beta )}{(n+1)(n+2\beta )}\cdot A_{(n,0,0)}$. Since $A_{\u 0}=1$, we get
\BGE A_{(n,0,0)}=\frac{(\beta )_n(1-2\beta )_n}{(1)_n (2\beta )_n},\quad n\in \N_0.\label{ABn00@N}\EDE

Since \(A\) satisfies (\ref{AB1@R}), applying the difference operators and using
commutativity of the discrete differences gives (\ref{Ljk=}). By Lemma \ref{construction}, there is a unique $A\in\R^{\N_0^3}$ that satisfies (\ref{ABn00@N}) and (\ref{Ljk=}) for all $j,k\in [3]$, and such $A$ is symmetric. The remainder of the proof is dedicated to showing that such $A$ satisfies Theorem \ref{Theorem-AB@N}.
This gives the uniqueness of the coefficient array \(A\), as well as the
\(A\)-parts of (i) and (ii).
Since the specialization \(n_2=0\) of the \(A\)-formula in (\ref{AB=@N}) is exactly
(\ref{ABn00@N}),   to prove that $A$ satisfies (\ref{ABN:iii}), by the uniqueness part of Lemma \ref{construction},
it suffices to show that the two-dimensional array defined by the right-hand
side of the \(A\)-formula in (\ref{AB=@N}) satisfies the two-dimensional analogue of
(\ref{Ljk=}) for \(j,k\in[2]\).   Let $A^0$ on $\N_0^2$ be defined by (\ref{AB=@N}). Then for $\u n=(n_1,n_2,0)\in\N_0^2\times\{0\}$, if $2\beta +n_1-n_2\ne 0$,
\BGE f_+(n_1+1) A^0_{\u n+e_1}=\frac{(2\beta +n_1)(\beta +n_1)(1-2\beta +n_1-n_2)}{2\beta +n_1-n_2} A^0_{\u n};\label{f+A@N}\EDE
and
\BGE f_-(n_2-1)A^0_{\u n-e_2}=\frac{(n_2-\beta )n_2(2\beta -1+n_2-n_1)}{-2\beta +n_2-n_1} A^0_{\u n}.\label{f-A@N}\EDE
Here we used the fact that, if $2\beta+n_1-n_2\ne 0$,
\BGE \frac{(1-2\beta )_{|(n_1+1)-n_2|}}{(2\beta )_{|(n_1+1)-n_2|}}= \frac{(1-2\beta )_{|n_1-(n_2-1)|}}{(2\beta )_{|n_1-(n_2-1)|}} =\frac{1-2\beta +n_1-n_2}{2\beta +n_1-n_2} \cdot \frac{(1-2\beta )_{|n_1-n_2|}}{(2\beta )_{|n_1-n_2|}},\label{fact-Horn}\EDE
no matter whether $n_1\ge n_2$ or $n_1<n_2$.  Thus, if $2\beta +n_1-n_2\ne 0$,
$$f_+(n_1+1) A^0_{\u n+e_1}-f_-(n_2-1)A^0_{\u n-e_2}=\frac{1-2\beta +n_1-n_2}{2\beta +n_1-n_2}((2\beta +n_1)(\beta +n_1)-(n_2-\beta )n_2)A^0_{\u n}$$
\BGE =(1-2\beta +n_1-n_2)(n_1+n_2+\beta )A^0_{\u n}.\label{f+-A1@N}\EDE
Suppose now $2\beta +n_1-n_2=0$. Then $2\beta\in\N$ and $(1-2\beta )_{|n_1-n_2|}=0$, which implies $A^0_{\u n}=0$. From $n_2=n_1+2\beta \ge 1$ we get
$$f_+(n_1+1)A^0_{\u n+e_1}=(n_1+1)n_2\cdot\frac{(\beta )_{n_1+1}(\beta )_{n_2}(1-2\beta )_{2\beta -1}}{(1)_{n_1+1}(1)_{n_2} (2\beta )_{2\beta -1}}
=\frac{(\beta )_{n_1+1}(\beta )_{n_2}(1-2\beta )_{2\beta -1}}{(1)_{n_1}(1)_{n_2-1} (2\beta )_{2\beta -1}};$$
$$f_-(n_2-1) A^0_{\u n-e_2}=(n_2-1+\beta )(n_1+\beta )\cdot \frac{(\beta )_{n_1}(\beta )_{n_2-1}(1-2\beta )_{2\beta -1}}{(1)_{n_1}(1)_{n_2-1} (2\beta )_{2\beta -1}}
=\frac{(\beta )_{n_1+1}(\beta )_{n_2}(1-2\beta )_{2\beta -1}}{(1)_{n_1}(1)_{n_2-1} (2\beta )_{2\beta -1}}.$$
So we have $f_+(n_1+1)A^0_{\u n+e_1}-f_-(n_2-1)A^0_{\u n-e_2}=0$. Since $A^0_{\u n}=0$, (\ref{f+-A1@N}) also holds in the case $2\beta +n_1-n_2= 0$. Symmetrically,
\BGE f_+(n_2+1) A^0_{\u n+e_2}-f_-(n_1-1)A^0_{\u n-e_1}=(1-2\beta +n_2-n_1)(n_1+n_2+\beta )A^0_{\u n}. \label{f+-A2@N}\EDE
Subtracting (\ref{f+-A2@N}) from (\ref{f+-A1@N}), we get
$$(f_+(n_1+1)A^0_{\u n+e_1}+f_-(n_1-1)A^0_{\u n-e_1})-(f_+(n_2+1)A^0_{\u n+e_2}+f_-(n_2-1)A^0_{\u n-e_2})$$
$$=2(n_1-n_2)(n_1+n_2+\beta )A^0_{\u n}=((f_+(n_1)+f_-(n_1))-(f_+(n_2)+f_-(n_2)))A^0_{\u n},$$
which implies the desired equality  $({\mathcal H}_1 A^0)_{\u n}-({\mathcal H}_2 A^0)_{\u n}=0$, as desired. So $A$ satisfies the $A$-formula in (\ref{ABN:iii}).
\end{subproof}

Now we use \(E_{\u n} (\begin{smallmatrix}
  + & + \\ j & k \end{smallmatrix} )\) or \(E_{\u n} (\begin{smallmatrix} + & + \\ k & j \end{smallmatrix} )\) to denote (\ref{++@N}); use \(E_{\u n}(\begin{smallmatrix} + & - \\ j & k \end{smallmatrix} )\) or \(E_{\u n} (\begin{smallmatrix}  - & + \\ k & j \end{smallmatrix} )\) to denote (\ref{+-@N}); and use  \(E_{\u n}(\begin{smallmatrix}  - & - \\ j & k \end{smallmatrix} )\) or \(E_{\u n} (\begin{smallmatrix}   - & - \\ k & j \end{smallmatrix} )\) to denote (\ref{--@N}).
The argument in the paragraph  before Lemma \ref{RX@R} can be applied here to show that the present  $E$-relation also satisfies the \textit{exterior triviality}: \(E_{\u n} (\begin{smallmatrix}
  \sigma & \tau \\ j & k
\end{smallmatrix} )\) holds whenever $\{\u n, \u n+\sigma e_j,\u n+\tau e_k\}$   intersects $\N_0^3$ at zero or one element.

\begin{Lemma}
Let $\u n=(n_1,n_2,n_3)\in \N_0^3$ and $\sigma,\tau\in \{+,-\}$. Let $\{j, k, {\ell}\}=[3]$. If $E_{\u n}(\begin{smallmatrix}
  \sigma & \tau \\ j & k
\end{smallmatrix} ) $ holds, then $X^{j,\sigma}_{\u n}=X^{k,\tau}_{\u n}$. Conversely, if $X^{j,+}_{\u n}=X^{k,+}_{\u n}$, then $E_{\u n}(\begin{smallmatrix}
  + & + \\ j & k
\end{smallmatrix} ) $ holds if $M^{\ell}_{\u n}\ne -1$; if $X^{j,-}_{\u n}=X^{k,-}_{\u n}$, then $E_{\u n}(\begin{smallmatrix}
  - & - \\ j & k
\end{smallmatrix} ) $ always holds; if $X^{j,+}_{\u n}=X^{k,-}_{\u n}$, then $E_{\u n}(\begin{smallmatrix}
  + & - \\ j & k
\end{smallmatrix} ) $ holds if $M^k_{\u n}\ne -1$.
\label{RX@N}
\end{Lemma}
\begin{proof}
   Case 1. $\sigma=\tau=+$. Multiplying (\ref{++@N}) by $M^{\ell}_{\u n+e_j}=M^{\ell}_{\u n+e_k}=M^{\ell}_{\u n}+1$, we get
$$ (M^lP^jM^kA)_{\u n+e_j}-(M^lP^kM^jA)_{\u n+e_k}=((M^{\ell}+1)(P^j-P^k)TA)_{\u n}.$$
A direct calculation gives
$$(P^j+2\beta -\frac 12) M^j-(P^k+2\beta -\frac 12) M^k=- (M^{\ell}+1)(P^j-P^k) .$$
Combining the above two formulas, we get  $X^{j,+}_{\u n}=X^{k,+}_{\u n}$. On the other hand, if  $X^{j,+}_{\u n}=X^{k,+}_{\u n}$, we may reverse the above argument to get (\ref{++@N}) in the case $M^{\ell}_{\u n}\ne -1$.

Case 2. $\sigma=\tau=-$.
Multiplying (\ref{--@N}) by $T_{\u n-e_j}=T_{\u n-e_k} =T_{\u n}-1$, we get
 $$ ( TQ^jM^jA)_{\u n-e_j}-(TQ^kM^k A)_{\u n-e_k} = -((T-1)(P^j-P^k)M^lA)_{\u n}.$$
A direct calculation gives
$$(P^j-\beta +\frac 12)M^k-(P^k-\beta +\frac 12)M^j=(T-1)(P^j-P^k).$$
Combining the above two formulas, we get $X^{j,-}_{\u n}=X^{k,-}_{\u n}$. On the other hand, if  $X^{j,-}_{\u n}=X^{k,-}_{\u n}$, we may reverse the above argument to get (\ref{--@N}) in the case $T_{\u n}-1\ne 0$, which is true unless $\u n=\u 0$. In the case $\u n=\u 0$, (\ref{--@N}) is trivial since both sides equal $0$.

Case 3. $\sigma=+$ and $\tau=-$. Multiplying (\ref{+-@N}) by $M^k_{\u n+e_j}=M^k_{\u n-e_k}=M^k_{\u n}+1$, we get
$$ (M^kP^jM^{\ell} A)_{\u n+e_j}-(M^kQ^kTA)_{\u n-e_k} =-((M^k+1)(P^k+Q^j)M^jA)_{\u n}.$$
A direct calculation gives
$$(P^j+2\beta -\frac 12)T-(P^k-\beta +\frac 12)M^{\ell} =(M^k+1)(P^k+Q^j).$$
Combining the above two formulas, we get $X^{j,+}_{\u n}=X^{k,-}_{\u n}$. On the other hand, if  $X^{j,+}_{\u n}=X^{k,-}_{\u n}$, we may reverse the above argument to get (\ref{+-@N}) when $M^k_{\u n}\ne -1$.

Case 4. $\sigma=-$ and $\tau=+$. This follows from Case 3 after interchanging \(j\) and \(k\).
\end{proof}

\begin{Lemma} [Cross Rule]
  For any $\u n=(n_1,n_2,n_3)\in \N_0^3$ and distinct $j, k\in [3]$,  $E_{\u n}(\begin{smallmatrix}
  - & - \\ j & k
\end{smallmatrix} ) $ and $E_{\u n}(\begin{smallmatrix}
  +& - \\ j & k
\end{smallmatrix} ) $ together imply $E_{\u n}(\begin{smallmatrix}
  -& + \\ j & k
\end{smallmatrix} ) $ and $E_{\u n}(\begin{smallmatrix}
  + & + \\ j & k
\end{smallmatrix} ) $.
\label{cross@N}
\end{Lemma}
\begin{proof}
First, we show that $E_{\u n}(\begin{smallmatrix}  - & - \\ j & k \end{smallmatrix} ) $ and $E_{\u n}(\begin{smallmatrix} +& - \\ j & k \end{smallmatrix} ) $ together imply $E_{\u n}(\begin{smallmatrix}  -& + \\ j & k \end{smallmatrix} ) $. Let ${\ell}\in[3]\sem\{j,k\}$. Swapping $j$ and $k$ in (\ref{+-@N}), we see that  $E_{\u n}(\begin{smallmatrix}  -& + \\ j & k \end{smallmatrix} ) $ is equivalent to:
\BGE   (P^kM^{\ell} A)_{\u n+e_k}-(Q^jTA)_{\u n-e_j} =-((P^j+Q^k)M^kA)_{\u n}.\label{target1@N}\EDE
Multiplying (\ref{--@N}) ($E_{\u n}(\begin{smallmatrix}  -& - \\ j & k \end{smallmatrix} ) $), (\ref{+-@N}) ($E_{\u n}(\begin{smallmatrix}  + & - \\ j & k \end{smallmatrix} ) $) and (\ref{target1@N}) respectively by $n_j+n_k+\beta $, $n_j+2\beta $ and $n_k+2\beta $, we get
$$ ((P^j+P^k+1+\beta ) Q^jM^j A)_{\u n-e_j}-((P^j+P^k+1+\beta )Q^kM^k  A)_{\u n-e_k}$$ \BGE= -((P^j+P^k+\beta )(P^j-P^k)M^lA)_{\u n},\label{--change@N}\EDE
\BGE (f_+(P^j)M^{\ell} A)_{\u n+e_j}-((P^j+2\beta ) Q^kTA)_{\u n-e_k} =-((P^j+2\beta )(P^k+Q^j)M^jA)_{\u n};\label{+-change@N}\EDE
\BGE (f_+(P^k)M^{\ell} A)_{\u n+e_k}-((P^k+2\beta ) Q^jT A)_{\u n-e_j}=-((P^k+2\beta ) (P^j+Q^k)M^kA)_{\u n}.\label{-+change@N}\EDE

Multiplying the equation $({\mathcal H}_j A)_{\u n}-({\mathcal H}_k A)_{\u n}=0$ by $n_j+n_k-n_{\ell}+2\beta $, we get
$$ (M^lf_+(P^j) A)_{\u n+e_j}-(M^lf_+(P^k)  A)_{\u n+e_k}+((M^{\ell} +2)f_-(P^j)A)_{\u n-e_j}-(f_-(P^k)(M^{\ell}+2)A)_{\u n-e_k}$$
\BGE = ((M^{\ell}+1)2(P^j-P^k)(P^j+P^k+\beta )A)_{\u n}. \label{fjk@N}\EDE

We claim that the linear combination  (\ref{--change@N}) $-$ (\ref{+-change@N}) $+$ (\ref{-+change@N}) $+$ (\ref{fjk@N}) reduces to a trivial identity. First, the coefficients of $A_{\u n+e_j}$ and $A_{\u n+e_k}$ are zero. Second, the coefficient of $A_{\u n-e_j}$ is
$$(n_j-1+\beta )[(n_j+n_k+\beta )(n_k+n_{\ell}-n_j+2\beta )-(n_k+2\beta )(n_j+n_k+n_{\ell})+(n_j+n_k-n_{\ell}+2\beta )(n_j-\beta )]=0;$$
and symmetrically, the coefficient of $A_{\u n-e_k}$ is also $0$. Third, the coefficient of $A_{\u n}$ is
$$(n_j+n_k+\beta )[-(n_j-n_k)(n_j+n_k-n_{\ell}-1+2\beta )+ (n_j+2\beta )(n_{\ell}+n_k-n_j-1+2\beta )$$
$$-(n_k+2\beta ) (n_{\ell}+n_j-n_k-1+2\beta )+2(n_j+n_k-n_{\ell}+2\beta )(n_j-n_k)]=0.$$
Since the above linear combination is a trivial identity, equation (\ref{-+change@N})
is a consequence of   (\ref{--change@N}), (\ref{+-change@N}) and (\ref{fjk@N}). Since (\ref{--change@N}) and (\ref{+-change@N}) follow from $E_{\u n}(\begin{smallmatrix}  - & - \\ j & k \end{smallmatrix} ) $ and $E_{\u n}(\begin{smallmatrix}  +& - \\ j & k \end{smallmatrix} ) $,  and (\ref{fjk@N}) holds by (\ref{Ljk=}), we see that $E_{\u n}(\begin{smallmatrix}  - & - \\ j & k \end{smallmatrix} ) $ and $E_{\u n}(\begin{smallmatrix}  +& - \\ j & k \end{smallmatrix} ) $ together imply (\ref{-+change@N}), which further implies (\ref{target1@N}) and then $E_{\u n}(\begin{smallmatrix}  -& + \\ j & k \end{smallmatrix} ) $ since  the factor $n_k+2\beta $ by which (\ref{target1@N}) is multiplied  is positive.

Second, we show that $E_{\u n}(\begin{smallmatrix}  - & - \\ j & k \end{smallmatrix} ) $ and $E_{\u n}(\begin{smallmatrix}   +& - \\ j & k \end{smallmatrix} ) $ together imply $E_{\u n}(\begin{smallmatrix}  + & + \\ j & k \end{smallmatrix} ) $. Multiplying (\ref{--@N}) ($E_{\u n}(\begin{smallmatrix}  - & - \\ j & k \end{smallmatrix} ) $), (\ref{+-@N}) ($E_{\u n}(\begin{smallmatrix} + & - \\ j & k \end{smallmatrix} ) $) and (\ref{++@N}) ($E_{\u n}(\begin{smallmatrix}  + & + \\ j & k\end{smallmatrix} ) $) respectively by $n_j-\beta $, $n_k-n_j$ and $n_k+2\beta $, we obtain, respectively,
\BGE (f_-(P^j)M^j A)_{\u n-e_j}-((P^j-\beta )Q^kM^k A)_{\u n-e_k} = -((P^j-\beta )(P^j-P^k)M^lA)_{\u n};\label{--change2@N}\EDE
\BGE  ((P^k-P^j+1) P^jM^{\ell}  A)_{\u n+e_j}-((P^k-P^j+1) Q^kTA)_{\u n-e_k} =-((P^k-P^j)(P^k+Q^j)M^jA)_{\u n}; \label{+-change2@N}\EDE
\BGE ((P^k+2\beta ) P^jM^kA)_{\u n+e_j}-(f_+(P^k)M^jA)_{\u n+e_k}=((P^k+2\beta ) (P^j-P^k)TA)_{\u n}.\label{++change2@N}\EDE

Multiplying the equation $({\mathcal H}_j A)_{\u n}-({\mathcal H}_k A)_{\u n}=0$ by $n_k+n_{\ell}-n_j+2\beta $, we get
$$ ((M^j+2)f_+(P^j) A)_{\u n+e_j}-(M^j f_+(P^k)  A)_{\u n+e_k}+(M^j f_-(P^j) A)_{\u n-e_j}-((M^j+2)f_-(P^k)A)_{\u n-e_k}$$
\BGE = ((M^j+1) 2(P^j-P^k)(P^j+P^k+\beta )A)_{\u n}. \label{fjk2@N}\EDE

We claim that the linear combination
(\ref{--change2@N}) $+$ (\ref{+-change2@N}) $+$ (\ref{++change2@N}) $-$ (\ref{fjk2@N}) reduces to a trivial identity. First, the coefficients of $A_{\u n-e_j}$ and $A_{\u n+e_k}$ are zero. Second, the coefficient of $A_{\u n+e_j}$ is
$$(n_j+1)[(n_k-n_j)(n_j+n_k-n_{\ell}+2\beta )+(n_k+2\beta )(n_j+n_{\ell}-n_k+2\beta )-(n_k+n_{\ell}-n_j+2\beta )(n_j+2\beta )]=0.$$
Third, the coefficient of $A_{\u n-e_k}$ is
$$(n_k-1+\beta )[-(n_j-\beta )(n_j+n_{\ell}-n_k+2\beta )-(n_k-n_j)(n_j+n_k+n_{\ell})+(n_k+n_{\ell}-n_j+2\beta )(n_k-\beta )]=0.$$
Finally, the coefficient of $A_{\u n}$ is
$$(n_j-n_k)[-(n_j-\beta )(n_j+n_k-n_{\ell}-1+2\beta )+(n_j+n_k+\beta )(n_k+n_{\ell}-n_j-1+2\beta )$$ $$+(n_k+2\beta )(n_j+n_k+n_{\ell}+1)- 2(n_k+n_{\ell}-n_j+2\beta)(n_j+n_k+\beta )]=0.$$
Since the above linear combination is a trivial identity, equation (\ref{++change2@N}) is a consequence of  (\ref{--change2@N}), (\ref{+-change2@N}) and (\ref{fjk2@N}). Since (\ref{--change2@N}) and (\ref{+-change2@N}) follow from $E_{\u n}(\begin{smallmatrix}
  - & - \\ j & k
\end{smallmatrix} ) $ and $E_{\u n}(\begin{smallmatrix}
  +& - \\ j & k
\end{smallmatrix} ) $,  and (\ref{fjk2@N}) always holds, we see that $E_{\u n}(\begin{smallmatrix}
  - & - \\ j & k
\end{smallmatrix} ) $ and $E_{\u n}(\begin{smallmatrix}
  +& - \\ j & k
\end{smallmatrix} ) $ together imply (\ref{++change2@N}), which further implies $E_{\u n}(\begin{smallmatrix}
  +& + \\ j & k
\end{smallmatrix} ) $ since  the factor $n_k+2\beta $ by which (\ref{++@N}) is multiplied is positive.
\end{proof}

\begin{Lemma} [Square Rule]
   For any $\u n=(n_1,n_2,n_3)\in \N_0^3$ and distinct $j, k\in [3]$,  $E_{\u n}(\begin{smallmatrix}
  + & + \\ j & k
\end{smallmatrix} ) $ and $E_{\u n+e_j}(\begin{smallmatrix}
  - & + \\ j & k
\end{smallmatrix} ) $ together imply $E_{\u n+e_k}(\begin{smallmatrix}
  + & - \\ j & k
\end{smallmatrix} ) $ and $E_{\u n+e_j+e_k}(\begin{smallmatrix}
  - & - \\ j & k
\end{smallmatrix} ) $.
\label{square@N}
\end{Lemma}
\begin{proof}
Replacing $\u n$ by $\u n+e_j$ in (\ref{target1@N}), we see that $E_{\u n+e_j}(\begin{smallmatrix}
  - & + \\ j & k
\end{smallmatrix} ) $ is equivalent to
\BGE(P^kM^{\ell} A)_{\u n+e_j+e_k}- (Q^jT A)_{\u n}=-((P^j+Q^k)M^k A)_{\u n+e_j}. \label{square2@N}\EDE
Swapping $j$ and $k$ in (\ref{square2@N}), we see that $E_{\u n+e_k}(\begin{smallmatrix}
  + & - \\ j & k
\end{smallmatrix} ) $ is equivalent to
\BGE (P^jM^{\ell} A)_{\u n+e_j+e_k}- (Q^kT A)_{\u n}=-((P^k+Q^j)M^j A)_{\u n+e_k}. \label{square1-@N}\EDE
Replacing $\u n$ by $\u n+e_j+e_k$ in (\ref{--@N}) ($E_{\u n}(\begin{smallmatrix}
  - & - \\ j & k
\end{smallmatrix})$), we see that $E_{\u n+e_j+e_k}(\begin{smallmatrix}
  - & - \\ j & k
\end{smallmatrix} ) $ is equivalent to
\BGE (Q^jM^j A)_{\u n+e_k}-(Q^kM^k A)_{\u n+e_j} = -((P^j-P^k)M^lA)_{\u n+e_j+e_k}. \label{square2-@N}\EDE

Multiplying (\ref{++@N}) ($E_{\u n}(\begin{smallmatrix}
  + & + \\ j & k
\end{smallmatrix} ) $) and (\ref{square2@N}) respectively by $n_j+n_k+\beta +1$ and $n_j+1$, and arranging terms, we get
\BGE ((Q^j+P^k)P^jM^k  A)_{\u n+e_j}-( (Q^j+P^k) P^kM^j A)_{\u n+e_k}=((Q^j+P^k+1)(P^j-P^k)T A)_{\u n},\label{square3@N}\EDE
\BGE(P^jP^kM^{\ell} A)_{\u n+e_j+e_k}+(P^j (P^j+Q^k)M^k A)_{\u n+e_j}= ((P^j+1) Q^jT A)_{\u n}. \label{square4@N}\EDE
Subtracting (\ref{square3@N}) from (\ref{square4@N}) and observing that $Q^j+P^k=P^j+Q^k$ and
$$(n_j+1)(n_j+\beta )-(n_j+n_k+\beta +1)(n_j-n_k) =(n_k+1)(n_k+\beta ),$$
we get
$$(n_k+1)(P^jM^{\ell} A)_{\u n+e_j+e_k}+(n_k+1)((Q^j+P^k)M^j A)_{\u n+e_k}=(n_k+1)(Q^kT A)_{\u n}.$$
We may then cancel the common nonzero factor $(n_k+1)$ to get (\ref{square1-@N}), i.e., $E_{\u n+e_k}(\begin{smallmatrix}
  + & - \\ j & k
\end{smallmatrix} ) $.

Multiplying (\ref{++@N}) ($E_{\u n}(\begin{smallmatrix}
  + & + \\ j & k
\end{smallmatrix} ) $) and (\ref{square2@N}) respectively by $n_j+\beta $ and $n_j-n_k$, and arranging terms, we get
\BGE ( Q^j P^kM^jA)_{\u n+e_k}+( Q^j (P^j-P^k)TA)_{\u n}=((Q^j-1) P^jM^kA)_{\u n+e_j} ,\label{square5@N}\EDE
\BGE((P^j-P^k) P^kM^{\ell} A)_{\u n+e_j+e_k}-((P^j-P^k)Q^jT A)_{\u n}=-( (P^j-P^k-1)(P^j+Q^k)M^kA)_{\u n+e_j} . \label{square6@N}\EDE
Summing (\ref{square5@N}) and (\ref{square6@N}), and observing that
$$(n_j+1)(n_j+\beta )-(n_j-n_k)(n_j+n_k+\beta +1)=(n_k+1)(n_k+\beta ),$$
we get
$$(n_k+1)(Q^jM^j A)_{\u n +e_k}+(n_k+1)((P^j-P^k)M^lA)_{\u n+e_j+e_k}=(n_k+1)(Q^kM^k A)_{\u n+e_j }.$$
We may then cancel the common factor $(n_k+1)$ to get (\ref{square2-@N}), i.e.,  $E_{\u n+e_j+e_k}(\begin{smallmatrix}
  - & - \\ j & k
\end{smallmatrix} ) $.
\end{proof}

\begin{subproof}[Further properties of $A$]
  The proof of  (\ref{ABN:iv}) here is very similar to that of  Theorem \ref{Theorem-AB@R} (\ref{ABN:iv}).
 More precisely, we use the same induction scheme as in the proof of
Theorem \ref{Theorem-AB@R} (\ref{ABN:iv}), with Lemmas \ref{cross@N} and \ref{square@N} replacing Lemmas \ref{cross@R} and \ref{square@R}.
For the induction base, we need to verify $E_{\u n}(\begin{smallmatrix}
  \sigma & - \\ 1 & 3
\end{smallmatrix} ) $, $\u n\in \Z^2\times \{0\}$, $\sigma\in \{+,-\}$. This follows from (\ref{AB=@N}). In fact, for $\u n=(n_1,n_2,0)$, $E_{\u n}(\begin{smallmatrix}
  + & - \\ 1 & 3
\end{smallmatrix} ) $  and $E_{\u n}(\begin{smallmatrix}
  - & - \\ 1 & 3
\end{smallmatrix} ) $  are respectively equivalent to
\begin{align*}
    (n_1+1)(n_1-n_2+2\beta ) A_{\u n+e_1}&=-(n_1+\beta )(n_2-n_1-1+2\beta ) A_{\u n},\\
 (n_1-1+\beta )(n_2-n_1+2\beta ) A_{\u n-e_1}&= - n_1(n_1-n_2-1+2\beta ) A_{\u n}.
\end{align*}
Both identities follow by taking the ratio of the two neighboring values
given in (\ref{AB=@N}), together with the zero-extension convention. Once  (\ref{ABN:iv}) is established,  Lemma \ref{RX@N} then  yields (\ref{ABN:v}).
\end{subproof}

\begin{subproof}[Construction and main properties of $B$]  Fix $\u n\in\Z^3$ and $\{j, k, {\ell}\}= [3]$. From  \(E_{\u n}(\begin{smallmatrix}
  + & + \\ j & k
\end{smallmatrix} )\) (\ref{++@N}), \(E_{\u n+e_k}(\begin{smallmatrix}
  + & - \\ j & k
\end{smallmatrix} )\) (\ref{square1-@N}) and \(E_{\u n+e_j+e_k}(\begin{smallmatrix}
  - & - \\ j & k
\end{smallmatrix} )\) (\ref{square2-@N}), we get
 $$(Q^j TA)_{\u n}-(P^jM^k A)_{\u n+e_j}=(Q^k TA)_{\u n}-(P^kM^j A)_{\u n+e_k}$$
\BGE=(P^jM^{\ell} A)_{\u n+e_j+e_k}+(Q^j M^j A)_{\u n+e_k} =(P^kM^{\ell} A)_{\u n+e_j+e_k}+(Q^k M^k A)_{\u n+e_j}.\label{all=@N}\EDE
Now we define a flow $\phi$ on $\Z^3$ such that, for $\u n\in\Z^3$ and ${\ell}\in[3]$, $\phi(\u n,\u n-e_{\ell})$ equals   the common value in (\ref{all=@N}), and $\phi(\u n-e_{\ell},\u n)=-\phi(\u n,\u n-e_{\ell})$.

Consider a square with vertices $\u n$, $\u n-e_j$, $\u n-e_j-e_k$, $\u n-e_k$, where $\u n\in\Z^3$ and $j,k\in [3]$ are distinct. Let ${\ell}\in [3]\sem \{j,k\}$.
Using the definition of \(\phi\) and applying (\ref{all=@N}) to the four directed
edges of this square, we obtain
  \begin{align}
 \phi(\u n,\u n-e_j)&=   (Q^lT A)_{\u n}-(P^lM^k A)_{\u n+e_{\ell}};\label{phi1@N}\\
 \phi(\u n-e_j,\u n-e_j-e_k)&=  (P^lM^k A)_{\u n+e_{\ell}}+  (Q^lM^{\ell} A)_{\u n};\label{phi2@N}\\
 \phi(\u n-e_j-e_k,\u n-e_k)&=  -(P^lM^j A)_{\u n+e_{\ell}}-  (Q^lM^{\ell} A)_{\u n};\label{phi3@N}\\
 \phi(\u n-e_k,\u n)&= (P^lM^j A)_{\u n+e_{\ell}}-(Q^lT A)_{\u n}\nonumber.
  \end{align}
 Summing  the four displayed formulas, we get
$$ \phi(\u n,\u n-e_j)+\phi(\u n-e_j,\u n-e_j-e_k)+\phi(\u n-e_j-e_k,\u n-e_k)+\phi(\u n-e_k,\u n)=0.$$ 
 This shows that $\phi$ is conservative, and so there exists a potential \(B:\mathbb Z^3\to\mathbb R\), unique up to an
additive constant, such that
\(
\frac12\phi(u,v)=B(v)-B(u)
\)
for every oriented nearest-neighbor edge \((u,v)\).
 Note that $\phi(u,v)=0$ when $u,v\in\Z^3\sem \N_0^3$. Indeed, on such an edge all relevant \(A\)-terms vanish by the zero-extension convention, except possibly terms killed by a factor \(P^m=0\).
Therefore we may further assume that $B=0$ on $\Z^3\sem \N_0^3$, and so $B\in\R^{\N_0^3}$. Then we get (vi) of Theorem \ref{Theorem-AB@N} using (\ref{all=@N}) and the definitions of $\phi$ and $B$.

Summing  (\ref{phi1@N}) and (\ref{phi3@N}), and using $\phi/2=dB$, we get (\ref{AB1@R}).
Summing (\ref{phi1@N}) and (\ref{phi2@N}),  and using $\phi/2=dB$, we get
\BGE 2B_{\u n-e_j-e_k}-2B_{\u n }=(Q^{\ell}(M^{\ell}+T)A)_{\u n}=(n_{\ell}+\beta )(2n_j+2n_k+2\beta ) A_{\u n}.\label{diag1@N}\EDE
Interchanging the roles of \({\ell}\) and \(j\), and then the roles of \({\ell}\) and
\(k\), respectively, gives
\BGE 2B_{\u n-e_{\ell}-e_k }-2B_{\u n } =(n_j+\beta )(2n_{\ell}+2n_k+2\beta ) A_{\u n},\label{diag2@N}\EDE
\BGE 2B_{\u n-e_{\ell}-e_j}-2B_{\u n } =(n_k+\beta )(2n_{\ell}+2n_j+2\beta ) A_{\u n}.\label{diag3@N}\EDE
Combining (\ref{diag2@N}), (\ref{diag3@N}) and (\ref{diag1@N})  in the form $($(\ref{diag2@N})$+$(\ref{diag3@N})$-$(\ref{diag1@N})$)/2$, we obtain (\ref{neighbor2'}).

From (\ref{AB1@R}), $A$ determines $\nabla_1 \nabla_2 B$, which further determines $B$ since we assume $B=0$ on $\Z^3\sem \N_0^3$. This shows the uniqueness of $B$. The symmetry of $A$ implies the symmetry of the flow $\phi$, which further implies the symmetry of $B$. Thus, $B$ satisfies (\ref{ABN:i}).

To check the $B$-formula in (\ref{ABN:iii}), it suffices to show that, if $B$ is defined by (\ref{AB=@N}), then (\ref{AB1@R}) holds for any $\u n\in\N_0^2\times \{0\}$ and ${\ell}\in [2]$. We need to show that, for any $\u n=(n_1,n_2,0)$ and distinct $j, k\in [2]$, $(\L_j A)_{\u n}=(\nabla_k  B)_{\u n}$.  By   (\ref{f+A@N}) and (\ref{fact-Horn}), when $n_1-n_2+2\beta \ne 0$,
$$(\L _1 A)_{\u n}=\Big[(\beta +n_1)(2\beta +n_1)\cdot\frac{1-2\beta +n_1-n_2}{2\beta +n_1-n_2}-(\beta +n_1) (1-\beta +n_1)\Big]\cdot A_{\u n};$$
$$(\nabla_2  B)_{\u n}=\Big[-(\beta +n_1)(\beta +n_2)+(\beta +n_1)n_2\cdot \frac{1-2\beta +n_1-n_2}{2\beta +n_1-n_2}\Big]\cdot A_{\u n}.$$
Comparing the above two formulas, we get $(\L _1 A)_{\u n}=(\nabla_2  B)_{\u n}$. If $n_1-n_2+2\beta =0$, then $2\beta\in\N$. From $(1-2\beta )_{|n_1-n_2|}=(1-2\beta )_{2\beta }=0$ we know that $A_{\u n}=B_{\u n}=0$. From $n_2=n_1+2\beta $, we get
$$f_+(n_1+1)A_{\u n+e_1}=(n_1+1)(n_1+2\beta )\cdot \frac{(\beta )_{n_1+1} (\beta )_{n_2} (1-2\beta )_{2\beta -1}}{(1)_{n_1+1}(1)_{n_2} (2\beta )_{2\beta -1}}
$$ $$= \frac{(\beta)_{n_1+1} (\beta )_{n_2} (1-2\beta )_{2\beta -1}}{(1)_{n_1}(1)_{n_2-1} (2\beta )_{2\beta -1}}=-B_{\u n-e_2},$$
which also implies $(\L _1 A)_{\u n}=(\nabla_2  B)_{\u n}$. By symmetry, we also have $(\L _2 A)_{\u n}=(\nabla_1  B)_{\u n}$.
\end{subproof}

\begin{subproof}[Integer and half-integer values of $\beta$]
Finally, suppose $2\beta \in\N$. We prove by an induction on $n_3$ that if $M^1_{\u n}\le -1$, i.e., $n_1\ge n_2+ n_3+2\beta$, then $A_{\u n}=0$. When $n_3<0$, this holds by the zero-extension convention.  When $n_3=0$, this holds by (\ref{AB=@N}) and the zero-extension convention. Suppose  the statement holds when $n_3= N\in \N_0$. Consider a point $(n_1,n_2,N+1)$ with $n_1\ge n_2+ N+1+2\beta $. Applying (\ref{++@N}) with $\u n=(n_1,n_2,N)$, $j=3$ and $k=1$, we get
$$ (P^3 M^1 A)_{(n_1,n_2,N+1)}=(P^1M^3 A)_{(n_1+1,n_2,N)}+((P^3-P^1)TA)_{(n_1,n_2,N)}=0. $$
The last ``$=$'' holds by the induction hypothesis,
since
\(
M^1_{(n_1+1,n_2,N)}\le
M^1_{(n_1,n_2,N)}\le -1
\).
Since $P^3_{(n_1,n_2,N+1)}=N+1\ge 1$ and $M^1_{(n_1,n_2,N+1)} =n_2+N+2\beta -n_1\le -1$, the above formula implies that $A_{(n_1,n_2,N+1)}=0$. This completes the induction step. So the statement is true. By the symmetry of $A$, we also get $A_{\u n}=0$ if $M^2_{\u n}\le -1$ or $M^3_{\u n}\le -1$.

Finally, we prove that, for any $j\in[3]$, $M^j_{\u n}\le -1$ implies that $B_{\u n}=0$. By symmetry we assume $j=1$. Suppose $\u n\in\N_0^3$ is such that $M^1_{\u n}\le -1$. Applying (\ref{all=B@N}) with ${\ell}=3$, $j=1$ and $k=2$, we get
$$2B_{\u n}-2B_{\u n-e_3}=(P^1M^2 A)_{\u n+e_1}-(Q^1T A)_{\u n}=0$$
because $A_{\u n+e_1}=A_{\u n}=0$, both of which follow from  $M^1_{\u n+e_1}\le M^1_{\u n}\le -1$. So we get $B_{\u n}=B_{\u n-e_3}$, which further implies $B_{\u n}=B_{\u n-m e_3}$ for any $m\in\N$ since $M^1_{\u n-m e_3}\le -1$ for any $m\in\N$. For $m$ big enough, we have $\u n-me_3\not\in\N_0^3$, and so $B_{\u n}=B_{\u n-m e_3}=0$ by the zero-extension convention. This finishes the proof of (\ref{ABN:vii}).
\end{subproof}

\section{Two-Dimensional Degenerations} \label{section:two-dim}
In this section, we study the restrictions of the trivariate series to the coordinate faces
\(\mathcal S_{\ell}^0:=\{r\in[0,1)^3:r_{\ell}=0\}, \qquad {\ell}\in[3]\). By symmetry it suffices to consider the case ${\ell}=3$.
Let $A^0_{(n_1,n_2)}=A_{(n_1,n_2,0)}$ and $B^0_{(n_1,n_2)}=B_{(n_1,n_2,0)}$, $n_1,n_2\in\N_0$, where $A$ and $B$ are as in Theorem \ref{Theorem-AB@R} or \ref{Theorem-AB@N}. The problem reduces to the convergence of $\sum_{\u n\in\N_0^2} A^0_{\u n}\u r^{\u n}$ and $\sum_{\u n\in\N_0^2} B^0_{\u n}\u r^{\u n}$, where $\u r\in [0,1)^2$ and $\u r^{\u n}:=r_1^{n_1}r_2^{n_2}$.

The recurrence relations  (\ref{rainbow1'}), (\ref{rainbow2'}), and (\ref{neighbor2'}) imply respectively that,
\BGE (n_j+1)(n_j+2\beta) A^0_{\u n+e_j}-(n_j+\beta)(n_j+1-\beta) A^0_{\u n}=B^0_{\u n}-B^0_{\u n-e_k},\label{rainbow1''}\EDE
\BGE (\beta n_j-\beta n_k) A^0_{\u n}=-(B^0_{\u n-e_j}-B^0_{\u n-e_k}), \label{rainbow2''}\EDE
\BGE (\beta n_j+\beta n_k+\beta^2) A^0_{\u n}= -(B^0_{\u n}-B^0_{\u n-e_j-e_k}), \label{neighbor2''}\EDE
for distinct $j, k\in [2]$ and all $\u n\in\N_0^2$, where $e_1:=(1,0)$ and $e_2:=(0,1)$.

By Stirling's formula and the identity $(q)_n=\Gamma(q+n)/\Gamma(q)$ when $q\in\R\sem \Z_{\le 0}$, we see that for any $a,b\in\R\sem \Z_{\le 0}$, there is $C_{a,b}\in (0,\infty)$ depending only on $a$ and $b$ such that
\BGE \Big| \frac{(a)_n}{(b)_n }\Big|\le C_{a,b}(n+1)^{a-b},\quad \forall n\in\N_0.\label{Stirling}\EDE
The estimate also hold if $a\in \Z_{\le 0}$ and $b\in\R\sem \Z_{\le 0}$ since in that case $(a)_n\ne 0$ for only finitely many $n$. In this section, we allow the implicit constants to depend on \(\beta\).

We will also use the following positivity property. For any $\beta>0$,
\BGE f_\beta(r):= {}_2F_1(1-\beta,\beta;2\beta;r)>0,\quad \forall r\in (-1,1),\label{2F1>0}\EDE
which follows from Euler's integral representation for ${}_2F_1$  (see \cite[\S 15.6]{NIST:DLMF}):
$$f_\beta(z)=\frac{\Gamma(2\beta)}{\Gamma(\beta)^2} \int_0^1 t^{\beta-1}(1-t)^{\beta-1}(1-zt)^{\beta-1} dt,\quad z\in \C\sem [1,\infty).$$
In the case $\beta>1/2$,  by Gauss's identity, $f_\beta(1)=\frac{\Gamma(2\beta)\Gamma(2\beta-1)}{\Gamma(\beta) \Gamma(3\beta-1)}>0$.
We denote by
\BGE
        \mathcal L_{x}
        :=
        x(1-x)\partial_{x}^2
        +(2\beta-2x)\partial_{x}
        -\beta(1-\beta)
\label{hyperF21}\EDE
the Gauss hypergeometric operator in the variable \(x\) with parameters $(\beta,1-\beta;2\beta)$. Then $\L_r f_\beta(r)=0$.

\subsection{The rainbow case: Appell functions}
We begin with the rainbow case. On the face $r_3=0$, the resulting two-variable series
can be identified with an Appell $F_1$ function. Let $A$ and $B$ be as in  Theorem \ref{Theorem-AB@R}. Then $A^0$ and $B^0$ satisfy (\ref{rainbow1''}) and (\ref{rainbow2''}), and their formulas are given by (\ref{AB=@R}). By (\ref{Stirling}) we have
\BGE |A^0_{(n_1,n_2)}|\lesssim (n_1+1)^{\beta-1} (n_2+1)^{\beta-1} (n_1+n_2+1)^{1-4\beta}
\lesssim   (n_1\wedge n_2+1)^{\beta-1} (n_1\vee n_2+1)^{-3\beta},\label{A0@R<*}\EDE
\BGE |B^0_{(n_1,n_2)}|\lesssim (n_1+1)^\beta (n_2+1)^\beta(n_1+n_2+1)^{1-4\beta} \le (n_1+n_2+1)^{1-2\beta}, \label{B0@R<*}\EDE
where in (\ref{A0@R<*}) we use $n_1\vee n_2+1\asymp n_1+n_2+1$ and write the product in terms of $n_1\wedge n_2$ and $n_1\vee n_2$.

\begin{Proposition}
(i) For any $\beta>0$, both series $\sum_{\u n\in\N_0^2} A^0_{\u n} \u r^{\u n}$ and $\sum_{\u n\in\N_0^2} B^0_{\u n} \u r^{\u n}$ converge locally uniformly in $ (-1,1)^2$. Let $F_0$ and $G_0$ denote the limit functions, respectively. Then they are analytic and symmetric and satisfy the PDE (for any distinct $j, k\in [2]$)
\BGE r_j(1-r_j)\pa_j^2 F_0+(2\beta-2r_j)\pa_j F_0-\beta(1-\beta) F_0=(1-r_k)G_0, \label{rainbow1'''} \EDE
\BGE \beta r_j\pa_j F_0-\beta r_k\pa_k F_0=-(r_j-r_k) G_0. \label{rainbow2'''}\EDE
Moreover, $F_0$ coincides with the Appell $F_1$ function with parameters $(1-\beta;\beta,\beta;3\beta)$ (cf.\ \cite{Appell}):
\BGE F_0(r_1,r_2)= F_1^{\mathrm{Appell}}(1-\beta ;\beta ,\beta ;3\beta ;r_1,r_2),\label{F1-Appell}\EDE
and is positive in $(-1,1)^2$.

(ii) Suppose $\beta>1/3$. Then for any $L\in (0,1)$, $\sum_{\u n\in\N_0^2} |A^0_{\u n}|L^{n_2}<\infty$, and   $F_0$ is continuous and positive on $[0,1]^2\sem \{(1,1)\}$.
Moreover,   for any $n_1\in\N_0$, the following sum is absolutely convergent and
\BGE A^{\Sigma}_{n_1}:=\sum_{n_2=0}^\infty A^0 _{(n_1,n_2)}=\frac{\Gamma(3\beta)\Gamma(3\beta-1)}{\Gamma(2\beta)\Gamma(4\beta-1)}\cdot \frac{(\beta)_{n_1}(1-\beta)_{n_1}}{(1)_{n_1}(2\beta)_{n_1}},\label{A1@R}\EDE
which implies that $F_0$ has boundary values
\BGE  F_0(x,1)=F_0(1,x)=
\frac{\Gamma(3\beta)\Gamma(3\beta-1)}{\Gamma(2\beta)\Gamma(4\beta-1)} \cdot f_\beta (x),\quad x\in [0,1).\label{F1@R}\EDE

(iii) Suppose $\beta>1/2$. Then $\sum_{\u n\in\N_0^2} |A^0_{\u n}|<\infty$, and   $F_0$ is continuous and positive on $[0,1]^2$ with
\BGE F_0(1,1)=\frac{\Gamma(3\beta)\Gamma(2\beta-1)}{\Gamma(\beta)\Gamma(4\beta-1)}.\label{F11@R}\EDE
\label{Appell}
\end{Proposition}
\begin{proof}
(i) By (\ref{A0@R<*}) and (\ref{B0@R<*}), for any $L\in (0,1)$, $\sum |A^0_{\u n}| L^{|\u n|}<\infty$ and $\sum |B^0_{\u n}| L^{|\u n|}<\infty$,  which implies the local uniform convergence of  $\sum_{\u n\in\N_0^2} A^0_{\u n} \u r^{\u n}$ and $\sum_{\u n\in\N_0^2} B^0_{\u n} \u r^{\u n}$ in $(-1,1)^2$. Therefore the limit functions $F_0$ and $G_0$ are analytic. The symmetry of $F_0$ and $G_0$ follows from that of $A^0$ and $B^0$. Equations (\ref{rainbow1'''}) and (\ref{rainbow2'''}) follow from (\ref{rainbow1''}) and (\ref{rainbow2''}).
Identity (\ref{F1-Appell}) follows from (\ref{AB=@R}) and the definition of the Appell $F_1$ function (cf.\ \cite{Appell}). Since \(\beta>0\), the Euler-type integral representation for Appell \(F_1\) in \cite[Formula (18)]{Schlosser-Appell}
applies and gives:
\BGE F_0(r_1,r_2)=\frac{\Gamma(3\beta)}{\Gamma(\beta)^3} \int\!\!\int_{\triangle} u^{\beta-1} v^{\beta-1} (1-u-v)^{\beta-1} (1-u r_1 -v r_2)^{\beta-1} dudv, \label{F0-integral}
\EDE
where $\triangle$ is the triangular region: $\{(u,v)\in\R^2:u> 0,v> 0,u+v< 1\}$. For \((r_1,r_2)\in(-1,1)^2\), we have \(1-ur_1-vr_2>0\) on \(\triangle\), so the integrand is positive.

(ii)  By (\ref{A0@R<*}), for any $L\in (0,1)$,
$$\sum_{\u n\in\N_0^2} |A^0_{\u n}| L^{n_2}\lesssim \sum_{n_1\le n_2} (n_1+1)^{\beta-1} (n_2+1)^{-3\beta} L^{n_2}+\sum_{n_2\le n_1} (n_2+1)^{\beta-1} (n_1+1)^{-3\beta} L^{n_2}$$
\BGE \lesssim \sum_{n_2=0}^\infty (n_2+1)^\beta (n_2+1)^{-3\beta} L^{n_2}+ \sum_{n_2=0 }^\infty (n_2+1)^{\beta-1} (n_2+1)^{1-3\beta} L^{n_2}<\infty,\label{sumA0L}\EDE
where in the second ``$\lesssim$'' we used $-3\beta<-1$. By Weierstrass $M$-test, $\sum_{\u n\in\N_0^2} A^0{\u n} \u r^{\u n}$ converges uniformly on $[0,1]\times [0,L]$, and so $F_0$ is continuous on this set. Since $L\in (0,1)$ is arbitrary, $F_0$ is continuous on $[0,1]\times [0,1)$. By the symmetry of $F_0$, $F_0$ is also continuous on $[0,1)\times [0,1]$. So it is continuous on $[0,1]^2\sem \{(1,1)\}$.

To prove (\ref{A1@R}), we use Gauss's identity (\ref{Gauss}):
  $$A^\Sigma_{n_1}=\sum_{n_2=0}^\infty A^0 _{(n_1,n_2)}=\frac{(\beta)_{n_1}(1-\beta)_{n_1}}{(1)_{n_1} (3\beta)_{n_1}} \sum_{n_2=0}^\infty \frac{(\beta)_{n_2}(1-\beta+n_1)_{n_2}}{(1)_{n_2}(3\beta+n_1)_{n_2}}$$
  $$=\frac{(\beta)_{n_1}(1-\beta)_{n_1}}{(1)_{n_1} (3\beta)_{n_1}} \cdot \frac{\Gamma(3\beta+n_1) \Gamma(3\beta-1)}{\Gamma(2\beta+n_1)\Gamma(4\beta-1)} =\frac{\Gamma(3\beta)\Gamma(3\beta-1)}{\Gamma(2\beta)\Gamma(4\beta-1)}\cdot \frac{(\beta)_{n_1}(1-\beta)_{n_1}}{(1)_{n_1}(2\beta)_{n_1}}.$$
Since $\sum A_{\u n} \u r^{\u n}$ converges absolutely on $[0,1)\times [0,1]$, we have  $F_0(x,1)=\sum_{n=0}^\infty A^{\Sigma}_n x^n$. By the above displayed formula and the symmetry of $F_0$ we get (\ref{F1@R}). By part (i), $F_0>0$ on $[0,1)^2$. Together with (\ref{2F1>0}) and (\ref{F1@R}), this implies that $F_0>0$ on $[0,1]^2\setminus\{(1,1)\}$.

(iii) Since $-2\beta<-1$,  (\ref{sumA0L}) holds for $L=1$, i.e., $\sum|A_{\u n}|<\infty$. By Weierstrass $M$-test, $\sum_{\u n\in\N_0^2} A^0{\u n} \u r^{\u n}$ converges uniformly on $[0,1]^2$, and so $F_0$ is continuous on $[0,1]^2$. By (\ref{F1@R}) and Gauss's identity, we get (\ref{F11@R}), which implies that $F_0$ is also positive at $(1,1)$.
\end{proof}

\subsection{The neighbor case: Horn functions}
We now turn to the neighbor case. The corresponding two-variable degeneration is
governed by a Horn $G_2$ function, and its boundary behavior is subtler than in the rainbow
case. Let $A$ and $B$ be as in Theorem \ref{Theorem-AB@N}. Then $A^0$ and $B^0$ satisfy (\ref{rainbow1''}) and (\ref{neighbor2''}), and their formulas are given by (\ref{AB=@N}). By (\ref{Stirling}) we have
$$ |A^0_{(n_1,n_2)}|\lesssim (n_1+1)^{\beta-1} (n_2+1)^{\beta-1} (|n_1-n_2|+1)^{1-4\beta}$$
 \BGE \le (n_1\wedge n_2+1)^{\beta-1} (n_1\vee n_2+1)^{\beta-1}(n_1+n_2+1)^{0\vee (1-4\beta)},\label{A0@N<*}\EDE
\BGE |B^0_{(n_1,n_2)}|\lesssim (n_1+1)^\beta (n_2+1)^\beta(|n_1-n_2|+1)^{1-4\beta}  \le (n_1+n_2+1)^{2\beta+(0\vee (1-4\beta))}. \label{B0@N<*}\EDE

\begin{Proposition}
 Both  $\sum_{\u n\in\N_0^2} A^0_{\u n} \u r^{\u n}$ and $\sum_{\u n\in\N_0^2} B^0_{\u n} \u r^{\u n}$ converge locally uniformly in $ (-1,1)^2$. Let $F_0$ and $G_0$ denote the limit functions, respectively. Then they are analytic and symmetric and satisfy equation (\ref{rainbow1'''}) for any distinct $j,k\in [2]$ and
\BGE \beta r_1\pa_1 F_0+\beta r_2\pa_2 F_0+\beta^2 F_0=-(1-r_1r_2) G_0. \label{neighbor2'''}\EDE
Moreover, $F_0>0$ on $(-1,1)^2$ and is given by the   Horn $G_2$ function (\cite{Horn}) with parameters $(\beta,\beta;1-2\beta,1-2\beta)$:
\BGE F_0(r_1,r_2)= G_2^{\mathrm{Horn}}(\beta ,\beta ;1-2\beta ,1-2\beta;-r_1,-r_2).\label{F0-Horn}\EDE
In addition, the following hold.
\begin{enumerate}
  \item [(i)] Suppose $\beta>1/3$. Then for any $n_1\in\N_0$, $A^{\Sigma}_{n_1}:=\sum_{n_2=0}^\infty A^0_{(n_1,n_2)}$ converges absolutely, and
\BGE A^{\Sigma}_{n_1} =\frac{\Gamma(2\beta)\Gamma(3\beta-1)}{\Gamma(\beta)\Gamma(4\beta-1)}\cdot \frac{(\beta)_{n_1}(1-\beta)_{n_1}}{(1)_{n_1}(2\beta)_{n_1}}.\label{A1@N}\EDE
Moreover, for any $L\in (0,1)$,
\BGE \sum_{n_1=0}^\infty \sum_{n_2=0}^\infty |A^0_{(n_1,n_2)}| L^{n_2}=\sum_{n_1=0}^\infty \sum_{n_2=0}^\infty |A^0_{(n_1,n_2)}| L^{n_1}<\infty.\label{AL}\EDE
In this case, $F_0 $ is continuous and positive on $[0,1]^2\sem\{(1,1)\}$ with
\BGE  F_0(x,1)=F_0(1,x)=
\frac{\Gamma(2\beta)\Gamma(3\beta-1)}{\Gamma(\beta)\Gamma(4\beta-1)} \cdot f_\beta (x),\quad x\in [0,1).\label{F1@N}\EDE
\item [(ii)]  If $\beta>1/2$, then $F_0$ extends continuously to a positive function on  $[0,1]^2$ with
\BGE F_0(1,1) = \frac{\Gamma(2\beta)^2\Gamma(2\beta-1)}{\Gamma(\beta)^2 \Gamma(4\beta-1)}.\label{F11@N}\EDE
\end{enumerate}
\label{Horn}
\end{Proposition}
\begin{proof}
By (\ref{A0@N<*}) and (\ref{B0@N<*}), for any $L\in (0,1)$, $\sum |A^0_{\u n}| L^{|\u n|}<\infty$ and $\sum |B^0_{\u n}| L^{|\u n|}<\infty$,  which implies the local uniform convergence of  $\sum_{\u n\in\N_0^2} A^0_{\u n} \u r^{\u n}$ and $\sum_{\u n\in\N_0^2} B^0_{\u n} \u r^{\u n}$ in $(-1,1)^2$. Therefore the limit functions $F_0$ and $G_0$ are analytic. The symmetry of $F_0$ and $G_0$ follows from that of $A^0$ and $B^0$. The PDE (\ref{rainbow1'''}) and (\ref{neighbor2'''}) follow from (\ref{rainbow1''}) and (\ref{neighbor2''}).
Identity (\ref{F0-Horn}) follows from the definition of the Horn $G_2$ function (\cite{Horn}) and that \BGE A^0_{(n_1,n_2)}=\frac{(\beta)_{n_1} (\beta)_{n_2} (1-2\beta)_{n_1-n_2} (1-2\beta)_{n_2-n_1}}{(1)_{n_1}(1)_{n_2}(-1)^{n_1+n_2}},\label{Horn-formula}\EDE
where if $n=-m\in \Z_{<0}$ and $q\not\in [m]$,  $(q)_n:=(q-1)^{-1}(q-2)^{-1}\cdots (q-m)^{-1}$.
The positivity of $F_0$ on $(-1,1)^2$ follows from the integral representation of the Horn $G_2$ function (\cite[Formula (3.10)]{G2-integral}), which together with (\ref{F0-Horn}) implies
\BGE F_0(\u r)=\frac{\Gamma(2\beta)^2}{\Gamma(\beta)^4} \int\!\!\int_{(0,1)^2} K(\u s,\u r)d\u s,\label{F=int-K}\EDE
where
$$K (\u s, \u r):=\underbrace{(s_1s_2(1-s_1)(1-s_2))^{\beta-1}}_{=:W(\u s)} \underbrace{(1-s_1r_1)^{2\beta-1} (1-s_2r_2)^{2\beta-1} (1-s_1s_2r_1r_2)^{1-4\beta}}_{=:\Phi(\u s,\u r)} .$$


(i) Suppose $\beta>1/3$. We now prove (\ref{A1@N}). Consider two cases. Case 1: $n_1+1-2\beta\not\in\N$. In this case, $(1-2\beta)_{\pm n_1}\ne 0$, and by (\ref{Horn-formula}),
  $$\sum_{n_2=0}^\infty A^0_{(n_1,n_2)}= \frac{(\beta)_{n_1}(1-2\beta)_{n_1}(1-2\beta)_{-n_1}}{(1)_{n_1}(-1)^{n_1}} \sum_{n_2=0}^\infty \frac{(\beta)_{n_2}(1-2\beta+n_1)_{-n_2}(1-2\beta-n_1)_{n_2}}{(1)_{n_2} (-1)^{n_2}}$$
  $$=\frac{(\beta)_{n_1}(1-2\beta)_{n_1}}{(1)_{n_1}(2\beta)_{n_1}} \sum_{n_2=0}^\infty \frac{(\beta)_{n_2}(1-2\beta-n_1)_{n_2}}{(1)_{n_2} (2\beta-n_1)_{n_2}}.$$
  By (\ref{Stirling}), the summand in the last series is   $O_{n_1}(n_2^{-3\beta}) $. Since
\(
(2\beta-n_1)-\beta-(1-2\beta-n_1)=3\beta-1>0
\),
Gauss's identity applies and gives
$$A^{\Sigma}_{n_1}= \frac{(\beta)_{n_1}(1-2\beta)_{n_1}}{(1)_{n_1}(2\beta)_{n_1}} \cdot \frac{\Gamma(2\beta-n_1)\Gamma(3\beta-1)}{\Gamma(\beta-n_1)\Gamma(4\beta-1)}= \frac{(\beta)_{n_1}(1-\beta)_{n_1}}{(1)_{n_1}(2\beta)_{n_1}}  \cdot \frac{\Gamma(2\beta )\Gamma(3\beta-1)}{\Gamma( \beta )\Gamma(4\beta-1)},$$
  where we used $(-1)^n(1-x)_{n}=1/ (x)_{-n} =\Gamma(x)/\Gamma(x-n) $. So we get (\ref{A1@N}).

Case 2: $N:=n_1+1-2\beta\in\N$. Then $2\beta\in\N$ and $1-2\beta =N-n_1\in\Z_{\le 0}$, which implies that $A^0_{(n_1,n_2)}=0$ if $|n_2-n_1|\ge n_1-N+1$. Thus, $\sum_{n_2=0}^\infty A^0_{(n_1,n_2)}$ converges absolutely, and
$$A^{\Sigma}_{n_1}=\sum_{n_2=N}^\infty \frac{(\beta)_{n_1}(\beta)_{n_2}(1-2\beta)_{n_2-n_1} }{(1)_{n_1}(1)_{n_2} (2\beta)_{n_2-n_1}}$$
$$=\frac{(\beta)_{n_1}(\beta)_N(1-2\beta)_{N-n_1} }{(1)_{n_1}(1)_N (2\beta)_{N-n_1}} \sum_{n_2=N}^\infty  \frac{(\beta+N)_{n_2-N} (1-2\beta+N-n_1)_{n_2-N} }{(1+N)_{n-N_2} (2\beta-n_1+N)_{n-N_2} } $$
$$=\frac{(\beta)_{n_1}(\beta)_{N}(1-2\beta)_{1-2\beta} }{(1)_{n_1}(1)_N (2\beta)_{1-2\beta}} \sum_{m=0}^\infty  \frac{(\beta+N)_m (2-4\beta)_m }{(1+N)_m (1)_{m} } $$$$=\frac{(\beta)_{n_1}(\beta)_{N}(1-2\beta)_{1-2\beta} }{(1)_{n_1}(1)_N (2\beta)_{1-2\beta}} \cdot \frac{\Gamma(N+1)\Gamma(3\beta-1)}{\Gamma(n_1+2\beta)\Gamma(1-\beta)}$$
$$=(-1)^{2\beta-1}\cdot\frac{(\beta)_{n_1}(1-\beta)_{n_1}}{(1)_{n_1}(2\beta)_{n_1}}  \cdot \frac{\Gamma(2\beta )\Gamma(3\beta-1)}{\Gamma( \beta )\Gamma(4\beta-1)},$$
where we used Gauss's identity in the penultimate line, and the equalities $\Gamma(N+1)=(1)_N$, $\Gamma(\beta)(\beta)_N=\Gamma(\beta+N)=\Gamma(1-\beta)(1-\beta)_{n_1}$, $(1-2\beta)_{1-2\beta}=(-1)^{2\beta-1}\Gamma(2\beta)/\Gamma(4\beta-1)$, and $(2\beta)_{1-2\beta}=1/ \Gamma(2\beta)$ in the last line. Note that the last line differs from (\ref{A1@N}) by a factor $(-1)^{2\beta-1}$. If $2\beta$ is odd, then they are equal because the factor is $1$. If $2\beta$ is even, they are also equal because both equal $0$ as can be seen from the zero factor $(1-\beta)_{n_1}$ since $n_1\ge 2\beta$ and $\beta\in\N$. Thus, we obtain (\ref{A1@N}) in Case 2.

Now we prove (\ref{AL}). The equality in (\ref{AL}) holds by the symmetry of $A$. Set $$V_{n_1}:=\sum_{n_2=0}^\infty |A^0_{(n_1,n_2)}|,\quad n_1\in\N_0.$$ It suffices to show that $\sum_{n_1} V_{n_1} L^{n_1}<\infty$ for any $L\in (0,1)$. 
We write
$$V_{n_1}= \sum_{n_2=0}^{\lfloor n_1/2 \rfloor}  |A^0_{(n_1,n_2)}| +  \sum_{n_2=\lfloor n_1/2 \rfloor+1}^{2n_1}  |A^0_{(n_1,n_2)}| + \sum_{n_2=2n_1+1}^\infty |A^0_{(n_1,n_2)}| =: I_1+I_2+I_3.$$

We first estimate $I_3$. If $n_2\ge 2n_1$, then  $n_2-n_1\asymp n_2$, and so by (\ref{Stirling}) we get
\BGE |A^0_{(n_1,n_2)}|\lesssim  (n_1+1)^{\beta-1} (n_2+1)^{-3\beta}. \label{est3}\EDE
Since $-3\beta<-1$, from (\ref{est3}) we get  $I_3\lesssim   (n_1+1)^{-2\beta}$.

To estimate $I_1$, we note that when $n_1\ge 2n_2$, we have $n_1-n_2\asymp n_1$,  and so by (\ref{Stirling}),
\BGE |A^0_{(n_1,n_2)}| \lesssim (n_1+1)^{-3\beta} (n_2+1)^{\beta-1},\label{est1}\EDE
which implies that $I_1\lesssim   (n_1+1)^{-2\beta}$ since $\beta>0$.

To estimate $I_2$,  since $n_1\asymp n_2$ and $4\beta>1$, we get by (\ref{Stirling}),
\BGE  |A^0_{(n_1,n_2)}| \lesssim (n_1+1)^{2\beta-2}(|n_1-n_2|+1)^{1-4\beta}\le (n_1+1)^{2\beta-2}. \label{est2}\EDE
Since the middle range contains \(O(n_1+1)\) terms, (\ref{est2}) gives
\(I_2\lesssim (n_1+1)^{2\beta-1}\).

Combining the estimates for $I_1,I_2,I_3$, we get $V_{n_1}\lesssim (n_1+1)^{2\beta-1}$. Thus, $\sum_{n_1} V_{n_1} L^{n_1}<\infty$ if $0<L<1$. This shows (\ref{AL}).

By (\ref{AL}) and Weierstrass M-test, for any $L\in(0,1)$,   $F_0$ is continuous on $[0,L]\times [0,1]$ and $[0,1]\times [0, L]$. Since this holds for any $L\in (0,1)$, $F_0$ is continuous on $[0,1]^2\sem \{(1,1)\}$. The continuity  allows us to evaluate the boundary value $F_0(x,1)$, and (\ref{A1@N}) gives (\ref{F1@N}). The positivity of $F_0$ on $[0,1)\times\{1\}$ and $\{1\}\times [0,1)$  follows from (\ref{2F1>0}).

(ii) Suppose $\beta>1/2$. We now prove that $F_0 $ further extends continuously to $\u 1:= (1,1)$.
Once this is established, we may let \(x\uparrow1\) in (\ref{F1@N}); Gauss's
identity then gives   (\ref{F11@N}), which implies that $F_0(\u 1)>0$, and so $F_0$ is positive on $[0,1]^2$.

To prove the convergence of $F_0(\u r)$ as $\u r\to \u 1$, it suffices to show
\BGE \int\!\!\int_{(0,1)^2}W(\u s) |\Phi(\u s,\u r)-\Phi(\u s, \u 1)|d\u s\to 0,\quad\text{as }\u r\to \u 1.\label{WPhi}\EDE

Let $u_j=u_j(s_j)=1-s_j$ and  $t_j=t_j(r_j)=1-r_j$ for $j\in [2]$, $U=U(\u s)=u_1\vee u_2$, and $T=T(\u r)=t_1\vee t_2$.
Then \BGE 1-s_j \le U,\quad 1-r_j \le T,\quad  1-s_j r_j\le  T+U,\quad 1-s_1s_2r_1r_2\ge T\vee U. \label{sx+ssxx}\EDE

Fix $\delta\in (0,1/2)$. Suppose $\u r\in [0,1)^2$ is sufficiently close to $\u 1$ so that $T(\u r)\le \delta$.
Decompose $(0,1)^2$ into
\BGE \Omega_\delta^{\text L}:=U^{-1}((\delta,1)),\quad  \Omega^{\text M}_{\delta,\u r}:=U^{-1}([T(\u r),\delta]),\quad \Omega^{\text S}_{\u r}:=U^{-1}((0,T(\u r))).\label{OmegaLMS}\EDE For $\u s\in \Omega^{\text L}_\delta$,  by (\ref{sx+ssxx}) and that $U(\u s)\ge \delta$, and $2\beta>1$, we get $W(\u s) \Phi(\u s,\u r)\le \delta^{1-4\beta}W(\u s)$ for any $\u r\in [0,1)^2$. Since $W$ is integrable on $(0,1)^2$,  the Dominated  Convergence Theorem gives
\BGE \lim_{\u r\to \u 1} \int\!\!\int_{\Omega^{\text L}_\delta} W(\u s) |\Phi(\u s,\u r)-\Phi(\u s, \u 1)|d\u s= 0.\label{integral-A*}\EDE

For $\u s\in \Omega^{\text M}_{\delta,\u r}\cup\Omega^{\text S}_{\u r}$, since $U(\u s)=(1-s_1)\vee (1-s_2)\le \delta\le 1/2$, we have $1/2\le s_j\le 1$, $j=1,2$, and so
$ W(\u s)\lesssim  (u_1(s_1)u_2(s_2))^{\beta-1}$. By a change of variables, this estimate implies that,  for any   interval $I\subset (0,\delta]$ with endpoints $a<b$, and $\gamma\in\R$,
\BGE \int\!\!\int_{U^{-1}(I)} W(\u s) U(\u s)^\gamma d\u s \lesssim 2\int_a^b U^{\gamma+\beta-1} \int_0^U u^{\beta-1} du  \lesssim \int_I U^{\gamma+2\beta-1} dU .\label{dT}\EDE

If $\u s\in \Omega^{\text M}_{\delta,\u r}\cup \Omega^{\text S}_{\u r}$, by  (\ref{sx+ssxx}) and $2\beta>1$,
$\Phi(\u s, \u r)\lesssim (U+T)^{-1}$, which together with (\ref{dT}) implies that, for any $\u r\in [0,1]^2$ with $T(\u r)\le \delta$,
$$  \int\!\!\int_{\Omega^{\text M}_{\delta,\u r}}  W(\u s) \Phi(\u s,\u r)d\u s \lesssim    \int_{T(\u r) }^{\delta} U ^{2\beta-2} dU\lesssim {\delta^{2\beta-1}} ;$$ 
$$ \int\!\!\int_{\Omega^{\text S}_{\u r}}  W(\u s) \Phi(\u s,\u r) d\u s \lesssim     T(\u r)^{-1} \int_0^{T(\u r) } U^{2\beta-1}dU \lesssim T(\u r)^{2\beta-1} \le  {\delta^{2\beta-1}} .$$ 
The same estimates hold for \(\u r=\u 1\), with \(T(\u 1)=0\) and
\(\Omega_{\u 1}^S=\varnothing\).
Combining the above estimates  with (\ref{integral-A*}), we obtain
$$\limsup_{\u r\to \u 1}\int\!\!\int_{(0,1)^2} W(\u s) |\Phi(\u s,\u r)-\Phi(\u s, \u 1)|d\u s\lesssim \delta^{2\beta-1}.$$
Since $\delta\in (0,1/2)$ is arbitrary and $2\beta>1$, the LHS of the above formula equals $0$. So we get (\ref{WPhi}), as desired.
\end{proof}

\begin{Remark}
In the above proof, we should not expect $\sum_{\u n\in\N_0^2}|A^0_{\u n}|<\infty$. On the contrary, by (\ref{AB=@N}) and (\ref{Stirling}),  we get $A^0_{(n,n)}\asymp n^{2\beta-2}$, which implies  $\sum_{n=1}^\infty |A^0_{(n,n)}|=\infty$ for $\beta\ge 1/2$. \label{Rem:comp}
\end{Remark}

\subsection{Degenerate partition functions and boundary Green's functions} \label{section-degenerate}
In this subsection, we discuss how the degenerate partition functions studied above are
heuristically related to boundary Green's functions for SLE.

Suppose $\eta$ is an SLE$_\kappa$-type curve in a domain $D$, and $\xi$ is a point on the boundary or in the bulk of $D$. The Green's function of $\eta$ at $\xi$ is the renormalized probability that $\eta$ comes within distance $r$ of $\xi$ as $r\to 0$.  Specifically, if, for some $\delta>0$, the limit
$$\lim_{r\to 0^+} r^{-\delta} \PP[\dist(\xi,\eta)<r]$$
exists and lies in $(0,\infty)$, then the limit is called the Green's function of $\eta$ at $\xi$, and $\delta$ is called the scale exponent. 

For $\kappa \in (0,8)$, the Green's function of a single chordal SLE$_\kappa$ is known to exist at both bulk and boundary points lying on analytic boundary arcs, with corresponding scale exponents of $1-\frac{\kappa}{8}$ \cite{LR} and $\frac{8}{\kappa}-1$ \cite{Mink-real}, respectively, and the explicit formulas  are known in both cases.  For $\kappa\in (0,4]$, an explicit formula for the Green's function at a bulk point of an SLE$_\kappa(2)$ curve is derived in \cite{LV}. Here, an SLE$_\kappa(2)$ curve is one of the two components in a degenerate $2$-SLE$_\kappa$ configuration, where the two curves share a common endpoint. The explicit formulas of two-curve Green's functions of $2$-SLE$_\kappa$ were derived for both bulk points (\cite{2-curve-Green-interior}) and boundary points (\cite{2-curve-Green-boundary}), where the two-curve Green's function of a $2$-SLE$_\kappa$  at $\xi$ requires that both curves get close to $\xi$.

Here we give a heuristic discussion of the boundary Green's function for one curve in a
non-degenerate $2$-SLE$_\kappa$.  The martingale argument below is therefore informal; however, once the candidate PDE is obtained, the verification that our proposed formula satisfies it is a direct computation. Let $\kappa\in (0,8)$. Suppose $\eta_1$ and $\eta_2$ form a $2$-SLE$_\kappa$ in $\HH$ with link pattern $\alpha$ and marked points $x_1<x_2<x_3<x_4$. Let $\xi\in \R\sem \{x_1,x_2,x_3,x_4\}$ be such that $\eta_1$ does not disconnect it from $\eta_2$ in $\HH$. We will study the boundary Green's function of $\eta_2$ at $\xi$. By a M\"obius transformation, after relabeling if necessary, we may assume that $\alpha=\{\{1,2\},\{3,4\}\}$, i.e., $\eta_j$ joins $x_{2j-1}$ to $x_{2j}$, $j=1,2$, and $\xi\in (x_2,x_3)\cup (x_3,x_4)\cup (x_4,\infty)$.

It is known that the marginal law of $\eta_{2}$ is that of an hypergeometric SLE$_\kappa$ (hSLE$_\kappa$, for short) curve in $\HH$ from $x_{3}$ to $x_{4}$ with force points $x_{1}$ and $x_{2}$ (\cite{Wu-hyper}; this particular hSLE family appeared earlier in \cite{reversal-kappa-rho}  under the name ``intermediate SLE$_\kappa(\rho)$'').  We orient $\eta_{2}$ from $x_{3}$ to $x_{4}$, truncate it at the first time $T$ that the curve disconnects $x_{4}$ from $\infty$, and parameterize the truncated curve, denoted by $\til \eta$, by the half-plane capacity so that $\til\eta$ becomes a chordal Loewner curve. Let $W_t$, $g_t$, and $K_t$, $0\le t<T$, be the corresponding driving function, maps, and hulls, respectively. Then for some Brownian motion $B$, $W$ satisfies the SDE:
\BGE dW_t=\sqrt\kappa dB_t+\sum_{j=1}^2 \frac{2(-1)^j dt}{W_t-g_t(x_j)} +\frac{(\kappa-6)dt}{W_t-g_t(x_4)}
+\Lambda(R_t) \Big(\frac {1}{W_t-g_t(x_2)}-\frac {1}{W_t-g_t(x_1)}\Big),\label{dWt}\EDE
with initial value $W_0=x_{3}$, where $\Lambda(z):=\kappa z\frac{f_\beta'(z)}{f_\beta(z)}$, $R_t:=\frac{(W_t-g_t(x_2))(g_t(x_4)-g_t(x_1))}{(W_t-g_t(x_1))(g_t(x_4)-g_t(x_2))}$, and   $f_\beta$ is as in (\ref{2F1>0}) with $\beta=4/\kappa$.

Let $\xi\in (x_2,x_3)\cup (x_3,x_4)\cup (x_4,\infty)$. We assume that the Green's function of $\eta_2$ at $\xi$ exists with scale exponent $\frac8\kappa -1$, the exponent of the boundary Green's function for a single chordal SLE$_\kappa$. Let $\mathcal{G}(x_1,x_2,x_3,x_4,\xi)$ denote this Green's function. Let $\eps>0$ be small. Then
$$\PP[\eta_{2}\text{ hits }B(\xi,\eps)]\approx \mathcal{G}(x_1,x_2,x_3,x_4,\xi) \eps^{\frac 8\kappa -1}.$$
Let $\F=(\F_t)_{t\ge 0}$ be the filtration generated by $B$. Let $\tau$ be an $\F$-stopping time with $\tau<T$. Conditionally on $\F_\tau$, by the Markov property of hSLE$_\kappa$, the part of $\eta_{2}$ after $\tau$, i.e., $\eta_{2}(\tau+\cdot)$ is an hSLE$_\kappa$ curve in $\HH\sem K_\tau$ from $\eta_{2}(\tau)$ to $x_{4}$ with force points $x_{1}$ and $x_{2}$. After the conformal map $g_\tau$, the conditional law of $g_\tau(\eta_{2 }(\tau+\cdot))$ then equals that of an hSLE$_\kappa$ curve in $\HH $ from $W_\tau$ to $g_\tau(x_{4})$ with force points $g_\tau(x_{ 1})$ and $g_\tau(x_{2})$. Thus, conditionally on $\F_\tau$, the probability that $\eta_{2}$ hits $B(\xi,\eps)$ is the probability that the above hSLE$_\kappa$ curve hits $g_\tau(B(\xi,\eps))$, which is approximated by the disk $B(g_\tau(\xi),g_\tau'(\xi)\eps)$.
Thus,
$$\PP[\eta_{2 }\text{ hits }B(\xi,\eps)|\F_\tau]\approx \mathcal{G}(g_\tau(x_{1}),g_\tau(x_{2}),W_\tau,g_\tau(x_{4}),g_\tau(\xi)) (g_\tau'(\xi) \eps)^{\frac 8\kappa -1}.$$
So it is reasonable to expect that the process
\BGE M_t:=g_t'(\xi)^{\frac 8\kappa -1}\mathcal{G}(g_t(x_{ 1}),g_t(x_{2}),W_t,g_t(x_{4}),g_t(\xi)),\quad 0\le t<T,\label{M-mtgl}\EDE
is a local martingale.

 We further assume that \(\mathcal G\) is smooth, as is the case in the known
explicit examples of SLE Green's functions.
It\^o's formula shows that the local-martingale condition for \(M\) leads to
the PDE
$$\Big[\frac\kappa 2\pa_{x_3}^2+\sum_{j\in [4]\sem \{3\}} \frac{2\pa_{x_j}}{x_j-x_3}+  \frac{2\pa_\xi}{\xi -x_3}+ \frac{2\pa_{x_3}}{x_3-x_2}-\frac{2\pa_{x_3}}{x_3-x_1}+\frac{(\kappa-6)\pa_{x_3}}{x_3-x_4} $$
\BGE +\Big(
\frac{1}{x_3-x_2}
-
\frac{1}{x_3-x_1}
\Big)\Lambda\Big(\frac{(x_3-x_2)(x_4-x_1)}{(x_3-x_1)(x_4-x_2)}\Big)\pa_{x_3}+\frac{-2(\frac 8\kappa -1)}{(\xi-x_3)^2}\Big]\mathcal{G} =0. \label{Green-PDE}\EDE
Here we used $\frac d{dt} g_t'(\xi)=\frac{-2g_t'(\xi)}{(g_t(\xi)-W_t)^2}$.

We now derive a candidate solution of (\ref{Green-PDE}) heuristically. The intuition is that, if we condition $\eta_2$ to pass through $\xi$, then we get two curves connecting $\xi$ with $x_3$ and $x_4$, which together with $\eta_1$ form a degenerate $3$-SLE$_\kappa$ with link pattern $\{\{x_1,x_2\},\{x_3,\xi\},\{\xi,x_4\}\}$. The word ``degenerate'' means that two curves share a common endpoint. We may also obtain a degenerate $3$-SLE$_\kappa$ by merging two endpoints of a non-degenerate $3$-SLE$_\kappa$.

\begin{Remark}
Here  we use a mild abuse of notation for link patterns. Strictly speaking, a link
pattern is a pairing of $[2N]$ together with an ordering of the marked boundary points. When
we write expressions such as
\(
\{\{x_1,x_2\},\{x_3,\xi'\},\{\xi'',x_4\}\},
\)
we mean the pairing of the displayed marked boundary points, with the understanding that
the displayed marked points are relabeled increasingly along \(\mathbb R\)
to obtain the actual labels in \([2N]\). Likewise, an
expression such as
\(
\{\{x_1,x_2\},\{x_3,\xi\},\{\xi,x_4\}\}
\)
is only a shorthand for the corresponding degenerate limit in which two adjacent marked points
coalesce to $\xi$.
\end{Remark}

We define the pure partition function for the above degenerate $3$-SLE$_\kappa$ by
$${\mathcal Z}_3(x_1,x_2,x_3,x_4,\xi):=|(x_2-x_1)(x_3-\xi )(x_4-\xi )|^{\frac{\kappa-6}\kappa} ( C_{\{x_1,x_2\},\{\xi ,x_3\}} C_{\{x_1,x_2\},\{\xi ,x_4\}}  )^{\frac 2\kappa} $$ \BGE \cdot \frac{(x_4-x_3)^{\frac 2\kappa}}{|x_4-\xi|^{\frac 2\kappa}|x_3-\xi|^{\frac 2\kappa}}  \cdot F_0(C_{\{x_1,x_2\},\{\xi,x_3\}},C_{\{x_1,x_2\},\{\xi,x_4\}}),\label{Z3*}\EDE
where $F_0$ is as in Proposition \ref{Appell} (for the rainbow patterns) if $\xi\in (x_2,x_3)\cup (x_4,\infty)$ or in Proposition \ref{Horn} (for the neighbor patterns) if $\xi\in (x_3,x_4)$.


Formally, the formula for \(\mathcal Z_3\) is obtained by splitting the
degenerate point \(\xi\) into two nearby ordered points \(\xi',\xi''\), chosen
on the two sides of \(\xi\) so that the resulting non-degenerate link pattern
is the corresponding rainbow or neighbor pattern, applying (\ref{Z3}), dividing
by \(|\xi''-\xi'|^{2/\kappa}\), and then letting \(\xi',\xi''\to\xi\).

Let ${\mathcal Z}_2$ denote the unnormalized  pure partition function for $2$-SLE$_\kappa$ with link pattern $\{\{x_1,x_2\},\{x_3,x_4\}\}$. We use the formula \BGE {\mathcal Z}_2(x_1,x_2,x_3,x_4)=((x_2-x_1) (x_4-x_3))^{\frac{\kappa-6}\kappa} \Big(\frac{(x_3-x_2)(x_4-x_1)}{(x_3-x_1)(x_4-x_2)}\Big)^{\frac 2\kappa} f_\beta \Big(\frac{(x_3-x_2)(x_4-x_1)}{(x_3-x_1)(x_4-x_2)}\Big).\label{Z2Green}\EDE
We now check formally that, for any $C\in (0,\infty)$,
\BGE \mathcal{G}(x_1,x_2,x_3,x_4,\xi):=C\cdot \frac{{\mathcal Z}_3(x_1,x_2,x_3,x_4,\xi)}{{\mathcal Z}_2(x_1,x_2,x_3,x_4)}\label{GZZ}\EDE
satisfies (\ref{Green-PDE}).

Since (\ref{Green-PDE}) is a linear PDE, we may assume $C=1$. We work on three cases separately. Case 1. $\xi\in (x_4,\infty)$. In this case, $F_0$ is as in Proposition \ref{Appell}, $C_{\{x_1,x_2\},\{\xi ,x_3\}}=\frac{(x_3-x_2)(\xi-x_1)}{(x_3-x_1)(\xi-x_2)}$, and $C_{\{x_1,x_2\},\{\xi ,x_4\}}=\frac{(x_4-x_2)(\xi-x_1)}{(x_4-x_1)(\xi-x_2)}$. So we have
\BGE \mathcal{G}=\Big(\frac{(\xi-x_3)(\xi-x_4)}{x_4-x_3}\Big)^{\frac{\kappa-8}\kappa} \Big(\frac{(x_4-x_2)(\xi-x_1)}{(x_4-x_1)(\xi-x_2)}\Big)^{\frac 4\kappa}\cdot \frac{F_{0}\Big( \frac{(x_3-x_2)(\xi-x_1)}{(x_3-x_1)(\xi-x_2)}, \frac{(x_4-x_2)(\xi-x_1)}{(x_4-x_1)(\xi-x_2)}\Big)}{f_\beta\Big(\frac{(x_3-x_2)(x_4-x_1)}{(x_3-x_1)(x_4-x_2)}\Big)}.
\label{G=F/F}\EDE

Since $F_0$  satisfies (\ref{rainbow1'''}) and (\ref{rainbow2'''}), by eliminating $G_0$, we get
 $$(r_1-r_2) r_1(1-r_1)\pa_{r_1}^2 F_0+[(r_1-r_2)(2\beta-2r_1)+\beta r_1(1-r_2)]\pa_{r_1} F_0$$ \BGE -\beta r_2(1-r_2) \pa_{r_2} F_0-\beta(1-\beta)(r_1-r_2) F_0=0.\label{FG3@R}\EDE
A direct computation, using (\ref{FG3@R}) and $\L_r f_\beta(r)=0$, where $\L_r$ is as in (\ref{hyperF21}),  shows that the $\mathcal{G}$ in (\ref{G=F/F}) satisfies (\ref{Green-PDE}).

Case 2. $\xi\in (x_2,x_3)$. Then  $F_0$ is as in Proposition \ref{Appell},  $C_{\{x_1,x_2\},\{\xi ,x_3\}}=\frac{(x_3-x_1)(\xi-x_2)}{(x_3-x_2)(\xi-x_1)}$, and $C_{\{x_1,x_2\},\{\xi ,x_4\}}=\frac{(x_4-x_1)(\xi-x_2)}{(x_4-x_2)(\xi-x_1)}$. So we have
\BGE \mathcal{G}=\Big(\frac{(x_3-\xi)(x_4-\xi)}{x_4-x_3}\Big)^{\frac{\kappa-8}\kappa} \Big(\frac{(x_3-x_1)(\xi-x_2)}{(x_3-x_2)(\xi-x_1)}\Big)^{\frac 4\kappa}\cdot \frac{F_{0}\Big( \frac{(x_3-x_1)(\xi-x_2)}{(x_3-x_2)(\xi-x_1)}, \frac{(x_4-x_1)(\xi-x_2)}{(x_4-x_2)(\xi-x_1)}\Big)}{f_\beta\Big(\frac{(x_3-x_2)(x_4-x_1)}{(x_3-x_1)(x_4-x_2)}\Big)}.
\label{G=F/F*}\EDE
Then we may also use (\ref{FG3@R}) and $\L_r f_\beta(r)=0$ to show that  the $\mathcal{G}$ in (\ref{G=F/F*}) satisfies (\ref{Green-PDE}).

Case 3. $\xi\in (x_3,x_4)$. Then $F_0$ is as in Proposition \ref{Horn}, $C_{\{x_1,x_2\},\{\xi ,x_3\}}=\frac{(x_3-x_2)(\xi-x_1)}{(x_3-x_1)(\xi-x_2)}$, and $C_{\{x_1,x_2\},\{\xi ,x_4\}}=\frac{(x_4-x_1)(\xi-x_2)}{(x_4-x_2)(\xi-x_1)}$.
So we get
\BGE \mathcal{G}=\Big(\frac{(\xi-x_3)(x_4-\xi)}{x_4-x_3}\Big)^{\frac{\kappa-8}\kappa}  \cdot \frac{F_{0}\Big( \frac{(x_3-x_2)(\xi-x_1)}{(x_3-x_1)(\xi-x_2)}, \frac{(x_4-x_1)(\xi-x_2)}{(x_4-x_2)(\xi-x_1)}\Big)}{f_\beta\Big(\frac{(x_3-x_2)(x_4-x_1)}{(x_3-x_1)(x_4-x_2)}\Big)}.
\label{G=F/F**}\EDE

Since $F_0$  satisfies (\ref{rainbow1'''}) and (\ref{neighbor2'''}),  by eliminating $G_0$, we get
 $$( 1-r_1r_2) r_1(1-r_1)\pa_{r_1}^2 F_0+[( 1-r_1r_2)(2\beta-2r_1)+\beta r_1(1-r_2)]\pa_{r_1} F_0$$ \BGE +\beta r_2(1-r_2) \pa_{r_2} F_0+[\beta^2(1-r_2)-\beta(1-\beta) (1-r_1r_2)] F_0=0.\label{FG3@N}\EDE
 Then we may  use (\ref{FG3@N}) and (\ref{hyperF21}) to show that  the $\mathcal{G}$ in (\ref{G=F/F**}) satisfies (\ref{Green-PDE}).

Heuristically,  (\ref{GZZ}) should give the boundary Green's function for $\eta_2$ at $\xi$. The constant $C$ in (\ref{GZZ}) is closely related to the unknown constant $\hat c$ in \cite{Mink-real}. Theorem 1 of \cite{Mink-real} together with a M\"obius transformation implies  that, if $\eta_2$ is a chordal SLE$_\kappa$ in $\HH$ from $x_3$ to $x_4$, then the boundary Green's function of $\eta_2$ at $\xi$ equals
$$\hat c \Big(\frac{|\xi-x_3||\xi-x_4|}{|x_4-x_3|}\Big)^{\frac{\kappa-8}\kappa} .$$

Intuitively, when $x_2\to x_1$, the curve $\eta_2$ in the  $2$-SLE$_\kappa$   with link pattern $\{\{x_1,x_2\},\{x_3,x_4\}\}$ that joins $x_3$ to $x_4$ tends to a chordal SLE$_\kappa$ in $\HH$ from $x_3$ to $x_4$. We also observe that, when $x_2\to x_1$,  the functions in (\ref{G=F/F}), (\ref{G=F/F*}), and (\ref{G=F/F**}) all converge to
$$  \Big(\frac{|\xi-x_3||\xi-x_4|}{|x_4-x_3|}\Big)^{\frac{\kappa-8}\kappa} \cdot \frac{F_0(1,1)}{f_\beta(1)} .$$
Thus, it is reasonable to expect that the $C$ in (\ref{GZZ}) equals $\hat c \cdot \frac{f_\beta(1)}{F_0(1,1)}$. Recall $f_\beta(1)=\frac{\Gamma(2\beta)\Gamma(2\beta-1)}{\Gamma(\beta)\Gamma(3\beta-1)}$. In Cases 1 and 2, corresponding to the rainbow degeneration, $F_0(1,1) =\frac{\Gamma(3\beta)\Gamma(2\beta-1)}{\Gamma(\beta)\Gamma(4\beta-1)}$ by (\ref{F11@R}). In Case 3, corresponding to the neighbor degeneration, $F_0(1,1) = \frac{\Gamma(2\beta)^2\Gamma(2\beta-1)}{\Gamma(\beta)^2 \Gamma(4\beta-1)}$ by (\ref{F11@N}).

 \begin{Remark}
The same degeneration heuristic also suggests a formula for the ordered two-point boundary
Green's function of a chordal SLE$_\kappa$ in $\mathbb H$ from $0$ to $\infty$, namely for the
event that the curve approaches two marked boundary points $\xi_1$ and $\xi_2$ in the prescribed
order $\xi_1$ then $\xi_2$. More generally, ordered boundary visit amplitudes (or boundary zig-zags)  go back to Bauer and Bernard \cite{BB04} and were studied
systematically in \cite{JJK16}, where multi-point boundary Green's functions were expressed by
Coulomb gas integrals.

When $\xi_1$ and $\xi_2$ lie on the same side of $0$, the existence of the corresponding
two-point boundary Green's function was proved by Lawler in \cite{Mink-real}, while an explicit
formula was obtained by Schramm and Zhou in \cite{SZ}; see also \cite{JJK16}. Our
computation of the pure partition function for the degenerate $3$-SLE$_\kappa$ with neighbor
pattern $\{\{0,\xi_1\},\{\xi_1,\xi_2\},\{\xi_2,\infty\}\}$ recovers this formula up to a
multiplicative constant.

When $\xi_1$ and $\xi_2$ lie on opposite sides of $0$, the existence of the corresponding
ordered boundary Green's function is known from \cite{Rami}. Our computation of the pure
partition function for the degenerate $3$-SLE$_\kappa$ with rainbow pattern
$\{\{0,\xi_1\},\{\xi_1,\xi_2\},\{\xi_2,\infty\}\}$ suggests that the ordered Green's function
should be given by
\[
\hat c^2 |\xi_1|^{\frac{\kappa-8}{\kappa}} |\xi_2|^{\frac{4}{\kappa}}
|\xi_2-\xi_1|^{\frac{\kappa-12}{\kappa}}
\frac{\Gamma(2\beta)\Gamma(4\beta-1)}{\Gamma(3\beta)\Gamma(3\beta-1)^2}
\,{}_2F_1\!\left(1-\beta,\beta;3\beta;\frac{\xi_2}{\xi_2-\xi_1}\right),
\]
where $\hat c>0$ is the one-point normalizing constant from \cite{Mink-real}. In contrast to the general
integral formulas of \cite{JJK16}, this expression is comparatively simple.
\end{Remark}

\section{Analytic Solutions} \label{section:analytic}
In this section, we return from the two-dimensional degenerations to the full three-variable problem. For the rainbow patterns, the coefficient estimates yield direct convergence and boundary regularity; for the neighbor patterns, we first obtain a local analytic solution near the origin and then use Pfaff systems to continue it to a larger domain. Recall the $f_\beta$ and $\L_\cdot$ defined respectively in (\ref{2F1>0}) and (\ref{hyperF21}).

\subsection{The rainbow case}
Let $A_{\u n}$, $\u n\in\N_0^3$, be the coefficient array from Theorem \ref{Theorem-AB@R}.
Define for $\u r\in\N_0^3$,
$$a_{\u n}:=\max\{n_1,n_2,n_3\},\quad b_{\u n}:=\med\{n_1,n_2,n_3\},\quad c_{\u n}:=\min\{n_1,n_2,n_3\},$$
$$N_{\u n}:=|\u n|=n_1+n_2+n_3,\quad A^{(-1)}_{\u n}:=A_{\u n-e_j},\quad A^{(-2)}_{\u n}:=A_{\u n-e_j-e_k},$$
where $(j,k,{\ell})$ is a permutation of $\{1,2,3\}$ such that $n_j\ge n_{\ell}\ge n_k$. When there are ties, the value of $A^{(-1)}_{\u n}$ and $A^{(-2)}_{\u n}$ do not depend on the choice of  $(j,k,{\ell})$ since $A$ is symmetric. Recall the zero-extension convention: if $\u n-e_j\not\in\N_0^3$, then $A^{(-1)}_{\u n}=0$.

Replacing the $\u n$ by $\u n-e_j$ in (\ref{+-@R}), we get
   $$ (P^jT A)_{\u n}=((P^k+Q^j)M^k A)_{\u n-e_j}+(Q^kM^{\ell} A)_{\u n-e_j-e_k}.$$
  Using (\ref{PQ}) and (\ref{DT@R}) to expand the above formula, we get
\BGE A_{\u n}=w^{(1)}_{\u n} A^{(-1)}_{\u n}+w^{(2)}_{\u n} A^{(-2)}_{\u n}=w^{(1)}_{\u n} A _{\u n-e_j}+w^{(2)}_{\u n} A^{ }_{\u n-e_j-e_k},\label{AAA**}\EDE
where
\BGE w_{\u n}^{(1)}:=\frac{(a_{\u n}+c_{\u n}+\beta-1)(a_{\u n}+b_{\u n}-c_{\u n}-\beta)}{a_{\u n}(N_{\u n}+3\beta-1)},\quad w_{\u n}^{(2)}:=\frac{(c_{\u n}+\beta-1)(a_{\u n}-b_{\u n}+c_{\u n}-\beta-1)}{a_{\u n}(N_{\u n}+3\beta-1)}.\label{w1w2**}\EDE
Let $\delta=1-w^{(1)}-w^{(2)}$.  A direct calculation shows that
  \BGE  \delta =\frac{(a+c+\beta-1)(2\beta+1)}{a(N+3\beta-1)}.\label{del*}\EDE
If $N_{\u n}>6\beta+1$ and $c_{\u n}\ge 1$, then since $a_{\u n}\ge N_{\u n}/3$, we get $w^{(1)}_{\u n}>0$ and $w^{(2)}_{\u n} $ has the same sign as $a_{\u n}-b_{\u n}+c_{\u n}-\beta-1$.

\begin{Lemma}
(i) For any   $p\in (0,3\beta+1-(\beta\vee 1))$, there exists $ C_p>0$ such that $|A_{\u n}|\le C_p(N_{\u n}+1)^{-p}$  for all $\u n\in\N_0^3$.

(ii) If $\beta\in (1/2,1)$, then for any $s\in (1, (\beta+1/2)\wedge (2-\beta))$, there exists $ C_s>0$ such that $|A_{\u n}|\le C_s(a_{\u n}+1)^{-s}(b_{\u n}+1)^{-1}(c_{\u n}+1)^{-s}$  for all $\u n\in\N_0^3$.
\label{N-p}
\end{Lemma}
\begin{proof}
(i) Fix $p\in (0,3\beta+1-(\beta\vee 1))$. By (\ref{A0@R<*}), there is $C_1>0$ such that $|A_{\u n}|\le C_1 (N_{\u n}+1)^{-p}$ if $c_{\u n}=0$.  Let  $m_1=6\beta+2$. By the Taylor expansion of $(1+x)^p$ and $(1-x)^p$ at $0$, there are $m_2,C_2,m_3,C_3>0$ such that
\BGE \Big(1+\frac 1{n-1}\Big)^p \le 1+\frac p{n+1} +\frac{C_2}{(n+1)^2},\quad \text{if }n\ge m_2;\label{N1Cp*}\EDE
\BGE 1-\frac p{n+1}\le \Big(1-\frac 1{n+1}\Big)^p \Big(1+\frac{C_3}{(n+1)^2}\Big),\quad \text{if }n\ge m_3.\label{N1Cp**}\EDE
Since $p<3\beta+1-(\beta\vee 1)\le 2\beta+1$, there is $m_4\ge 2$ such that
\BGE \frac{2\beta+1}{n+3\beta-1}-\frac p{n+1}\ge \frac{12\beta^2}{n(n-1)}+\frac{C_2}{(n+1)^2},\quad \text{if }n\ge m_4.\label{N3C2*}\EDE

For $m\in\N_0$, let $S_m=\{\u n\in\N_0^3:N_{\u n}=m\}$ and $\Psi_m=\max\{|A_{\u n}|:\u n\in S_m\}$. Then it suffices to prove that there is $C_p>0$ such that $\Psi_m\le C_p(m+1)^{-p}$ for all $m\in\N_0$.
Pick $m\in\N$ such that $m\ge m_0:=\left\lceil\max_{j\in [4]} m_j\right\rceil$. If the maximum in \(\Psi_m\) is attained at some \(n\) with \(c_n=0\), then the desired bound follows from the boundary estimate above. Hence it remains to consider a maximizer \(n\in S_m\) with \(c_n\ge1\). Take $\u n\in S_m$ such that $c_{\u n}\ge 1$. Assume by symmetry that $\u n=(a,b,c)$ with $a\ge b\ge c\ge 1$.
By (\ref{AAA**}) and that $\u n-e_1\in S_{m-1}$ and $\u n-e_1-e_3\in S_{m-2}$, we get \BGE |A_{\u n}|\le |w^{(1)}_{\u n}|\Psi_{m-1}+|w^{(2)}_{\u n}| \Psi_{m-2} .\label{APsi-bd*}\EDE
Since $N_{\u n}=m\ge 6\beta+1$ and $c_{\u n}\ge 1$, we have $w^{(1)}_{\u n}>0$, and $w^{(2)}_{\u n} $ has the same sign as $a -b +c -\beta-1$.

If $w^{(2)}_{\u n}<0 $, then $c <\beta+1$, which implies
$$|c +\beta-1|=c +\beta-1<2\beta,\quad |a -b +c -\beta-1|\le \beta+1-c \le \beta,$$
and so from $a\ge N_{\u n}/3=m/3$, we get $$|w^{(2)}_{\u n}|\le \frac{2\beta^2}{a (N_{\u n}-1)}\le \frac{6\beta^2}{m(m-1)}<\frac{\beta}{m}\le \frac{c +\beta-1}{m} .$$
If $w^{(2)}_{\u n}\ge 0 $, then $|a -b +c -\beta-1|=a  -b +c -\beta-1<a -1$, and so
$$|w^{(2)}_{\u n}|\le \frac{(a -1)(c +\beta-1)}{a  (N_{\u n}+3\beta-1)}< \frac{(a -1)(c +\beta-1)}{a  (m -1)}\le \frac{c +\beta-1}{m}.$$
Thus, we always have $|w^{(2)}_{\u n}|\le \frac{c +\beta-1}{m}<\frac 12$ because  $m\ge (3c )\vee  (6\beta) $.

From (\ref{del*}) and $|w^{(2)}_{\u n}|\le \frac{c +\beta-1}{m}\le\frac{c  +\beta-1}{a }$, we get
\BGE \delta\ge \frac{2\beta+1}{m+3\beta-1}(1+|w^{(2)}_{\u n}|).\label{delta>*}\EDE
Since $|w^{(2)}_{\u n}|\le  \frac{6\beta^2}{m(m-1)}$ when $w^{(2)}_{\u n}<0$, we get $|w^{(2)}_{\u n}|\le w^{(2)}_{\u n}+\frac{12\beta^2}{m(m-1)}$, and so
\BGE 1-|w^{(1)}_{\u n}|-|w^{(2)}_{\u n}| \ge \delta_{\u n}-2  |w^{(2)}_{\u n}|\ge  \delta_{\u n}-\frac{12\beta^2}{m(m-1)}.\label{|w1|*}\EDE
Since $m\ge m_4$, by (\ref{N3C2*}), (\ref{delta>*}), (\ref{|w1|*}), and $1>|w_2|\ge 0$,
$$ 1-|w^{(1)}_{\u n}|-|w^{(2)}_{\u n}| \ge  \frac{p}{m+1}(1+|w^{(2)}_{\u n}|) +\frac{C_2}{(m+1)^2}\ge  \frac{p}{m+1}(1+|w^{(2)}_{\u n}|) +\frac{C_2|w^{(2)}_{\u n}|}{(m+1)^2}.$$
Rearranging the above inequality, we get
$$|w^{(1)}_{\u n}|+|w^{(2)}_{\u n}|\Big(1+  \frac{p}{m+1}+\frac{C_2 }{(m+1)^2}\Big)\le 1-\frac p{m+1} .$$
Since $m\ge m_2\vee m_3$, by (\ref{N1Cp*}) and (\ref{N1Cp**}),
$$|w^{(1)}_{\u n}|+|w^{(2)}_{\u n}| \Big(1+\frac 1{m-1}\Big)^p\le \Big(1-\frac 1{m+1}\Big)^p  \Big(1+\frac{C_3}{(m+1)^2}\Big).$$
Multiplying both sides by $m^{-p}$, we obtain
\BGE |w^{(1)}_{\u n}|m^{-p}+|w^{(2)}_{\u n}|(m-1)^{-p}\le  (m+1)^{-p}  \Big(1+\frac{C_3}{(m+1)^2}\Big).\label{w+w<*}\EDE

Combining (\ref{APsi-bd*}) with (\ref{w+w<*}), we get
$$|A_{\u n}|\le (m+1)^{-p}  \Big(1+\frac{C_3}{(m+1)^2}\Big)  \max_{s\in [2]}\{(m-s+1)^p\Psi_{m-s} \}.$$
This inequality holds for any $\u n\in S_m$ with $c_{\u n}\ge 1$. Combining it with $|A_{\u n}|\le C_1 (N_{\u n}+1)^{-p}$ when $c_{\u n}=0$, we get
$$(m+1)^p \Psi_m \le  \Big(1+\frac{C_3}{(m+1)^2}\Big)\Big(C_1\vee (m^p \Psi_{m-1})\vee  ((m-1)^p \Psi_{m-2})\Big),\quad \text{if }m\ge m_0.$$
Let $C_4=\prod_{m=m_0}^\infty (1+\frac{C_3}{(m+1)^2})\in (0,\infty)$. Let $ C_p=C_4(C_1\vee  (m_0^p \Psi_{m_0-1})\vee  ((m_0-1)^p \Psi_{m_0-2}))$. Then $\Psi_m\le   C_p (m+1)^{-p}$ for all $m\ge m_0$. By possibly increasing $C_p$, we can make the inequality hold for all $m\in\N_0$ since there are only finitely many $m\in\N_0$ with $m<m_0$.

(ii) Let $\beta\in (1/2,1)$ and $s\in (1, (\beta+1/2)\wedge (2-\beta))$. Define
$$\Phi_{\u n}=(a_{\u n}+1)^s(b_{\u n}+1)(c_{\u n}+1)^s,\quad W_{\u n}=A_{\u n}\Phi_{\u n},\quad \u n\in \N_0^3.$$
Then we need to show that $W$ is bounded on $\N_0^3$. Since $s<\beta+1/2<2\beta$, by (\ref{A0@R<*}), if $c_{\u n}=0$,
$$|W_{\u n}|\lesssim (a_{\u n}+1)^{s-3\beta}(b_{\u n}+1)^{\beta}\le (a_{\u n}+1)^{s-2\beta}\le 1.$$
Thus, $W$ is bounded on $\{\u n:c_{\u n}=0\}$. We now show that $W$ is bounded on $\{\u n:c_{\u n}=1\}$. Let $\u n\in\N_0^3$ satisfy $c_{\u n}=1$. By symmetry, we assume $\u n=(a,b,1)$ with $a\ge b\ge 1$. Let $\u n'=(a,b,0)$. Applying (\ref{Ljk=}) to $j=3$, $k=1$, and $\u n=\u n'$, we get
$$2\beta A_{\u n} =(a+1)(a+2\beta) A_{\u n'+e_1}-2a(a+\beta) A_{\u n'} +(a-1+\beta)(a-\beta) A_{\u n'-e_1}.$$
By (\ref{AB=@R}),
$$\frac{A_{\u n'+e_1}}{A_{\u n'}}=\frac{(a+\beta)(a+b+1-\beta)}{(a+1)(a+b+3\beta)},\quad \frac{A_{\u n'-e_1}}{A_{\u n'}}= \frac{a(a+b-1+3\beta)}{(a-1+\beta)(a+b-\beta)}.$$
Thus,
$$\frac{ A_{\u n}}{A_{\u n'}}= \frac{(a+2\beta)(a+\beta)(a+b+1-\beta)}{2\beta(a+b+3\beta)}+ \frac{(a-\beta)a(a+b-1+3\beta)}{ 2\beta (a+b-\beta)} -\frac{2a(a+\beta)}{2\beta}$$
$$=\frac{ \beta (a^2+b^2)+2 (1-3\beta)ab+ \beta (1-2\beta)(a+b)+ \beta^2(\beta-1)}{(a+b+3\beta)(a+b-\beta)}.$$
Since the numerator is a polynomial of degree $2$, while $(a+b+3\beta)(a+b-\beta)\asymp (a+b+1)^2$, we get $|\frac{ A_{\u n}}{A_{\u n'}}|\lesssim 1$. Since $c_{\u n'}=0$, we have
$$|W_{\u n}|=|\Phi_{\u n}| |A_{\u n}|\lesssim  |\Phi_{\u n}| |A_{\u n'}|=2^s  |\Phi_{\u n'}| |A_{\u n'}|\lesssim |W_{\u n'}|\lesssim 1.$$
So $W$ is bounded on $\{\u n:c_{\u n}=1\}$.

Let $\u n\in\N_0^3$ satisfy $c_{\u n}\ge 2$. Assume by symmetry that $\u n=(a,b,c)$ with $a\ge b\ge c\ge 2$. Let $N=N_{\u n}=a+b+c$. Since $\beta\in (1/2,1)$ and $c\ge 2$, by (\ref{w1w2**}), $w^{(j)}_{\u n}>0$ for $j\in [2]$. By (\ref{AAA**}),
$$|W_{\u n}|\le   w^{(1)}_{\u n}\frac{\Phi_{\u n}}{\Phi_{\u n-e_1}} |W_{\u n-e_1}|+ w^{(2)}_{\u n}\frac{\Phi_{\u n}}{\Phi_{\u n-e_1-e_3}} |W_{\u n-e_1-e_3}|)$$
$$\le \Big(1+\frac 1a\Big)^s \Big[  w^{(1)}_{\u n} |W_{\u n-e_1}|+ w^{(2)}_{\u n}  \Big(1+\frac 1c\Big)^s   |W_{\u n-e_1-e_3}|\Big],$$
where we used the inequalities
$$\frac{\Phi_{\u n}}{\Phi_{\u n-e_1}}\le \Big(\frac{a+1}a\Big)^s,\quad \frac{\Phi_{\u n}}{\Phi_{\u n-e_1-e_3}}\le \Big(\frac{a+1}a\Big)^s \Big(\frac{c+1}c\Big)^s ,$$
which can be checked by considering the cases $a>b>c$, $a=b>c$, $a>b=c$, and $a=b=c$, separately, and using the assumption $s>1$. Let $M=\max\{ |W_{\u n-e_1}|,|W_{\u n-e_1-e_3}|\}$. Since $1-\delta_{\u n}= w^{(1)}_{\u n}+ w^{(2)}_{\u n}$, we get
$$\Big(1+\frac 1a\Big)^{-s} |W_{\u n}|\le M\Big[ w^{(1)}_{\u n}  + w^{(2)}_{\u n}   \Big(1+\frac 1c\Big)^s\Big]\le M\Big[ 1-\delta_{\u n}   + w^{(2)}_{\u n}   \Big(\Big(1+\frac 1c\Big)^s-1\Big)\Big].$$
Since $0<s-1<1-\beta<1$, we have
$$\Big(1+\frac 1c\Big)^s-1=s\int_0^{1/c} (1+x)^{s-1}dx< \frac s c \Big(1+\frac 1c\Big)^{s-1}< \frac s c \Big(1+\frac{s-1} c\Big)$$
$$< \frac s c \Big(1+\frac{1-\beta} c\Big)< \frac s c \Big(1-\frac{1-\beta} c\Big)^{-1}=\frac{s}{c+\beta-1}.$$
From (\ref{w1w2**}), (\ref{del*}) and $(1+\frac 1a)^s\le 1+\frac s a+\frac{s(s-1)}{2a^2}$, we get
$$  |W_{\u n}|\le M \Big(1+\frac s a+\frac{s(s-1)}{2a^2}\Big) \Big[1-\frac{(a+c+\beta-1)(2\beta+1)}{a(N+3\beta-1)}+\frac{(a-b+c-\beta-1)s}{a(N+3\beta-1)}\Big]$$
$$\le M  \Big[1+\frac{s(s-1)}{2a^2}+\frac s a-\frac{(a+c+\beta-1)(2\beta+1)}{a(N+3\beta-1)}+\frac{(a-b+c-\beta-1)s}{a(N+3\beta-1)}\Big]$$
$$= M  \Big[1+\frac{s(s-1)}{2a^2}- \frac{(a+c+\beta-1)(2\beta+1-2s)}{a(N+3\beta-1)} \Big]\le M  \Big[1+\frac{9 s^2}{2N^2}- \frac{ 2\beta+1-2s }{  N+3\beta } \Big] ,$$
where in the second line we used $(a+c+\beta-1)(2\beta+1)>(a-b+c-\beta-1)s$, and in the third line we used $a\ge N/3$ and $2\beta+1>2s$. Suppose $N\ge  m_0:=\left\lceil(3\beta)\vee \frac{9s^2}{2\beta+1-2s}\right\rceil$. Then
$$\frac{9 s^2}{2N^2}- \frac{ 2\beta+1-2s }{  N+3\beta } \le \frac{9 s^2}{2N^2}- \frac{ 2\beta+1-2s }{ 2N } =\frac{9 s^2-(2\beta+1-2s)N}{2N^2}\le 0,$$
which implies $ |W_{\u n}|\le M=\max\{ |W_{\u n-e_{j_{\u n}}}|,|W_{\u n-e_{j_{\u n}}-e_{k_{\u n}}}|\}$. This inequality holds for any $\u n\in\N_0^3$ with $N_{\u n}\ge m_0$ and $c_{\u n}\ge 2$. Combining it with the boundedness of $W$ on $\{\u n: c_{\u n}\in\{0,1\}\}$ and an induction on $N_{\u n}$, we conclude that $W$ is bounded on $\N_0^3$.
\end{proof}

\begin{Theorem}
  Let $A_{\u n}$ and $B_{\u n}$ be given by Theorem \ref{Theorem-AB@R}. Let $F(\u r)=\sum_{\u n\in\N_0^3} A_{\u n} \u r^{\u n}$ and $G(\u r)=\sum_{\u n\in\N_0^3} B_{\u n} \u r^{\u n}$, where $\u r^{\u n}:=r_1^{n_1} r_2^{n_2} r_3^{n_3}$. The following hold.
  \begin{enumerate}
    \item [(i)] If $\beta>1/2$, $\sum_{\u n\in\N_0^3} |A_{\u n}|<\infty$, and so $F$ is (defined and) continuous on $[-1,1]^3$, and
         \BGE F(r_1,r_2,1)=\frac{\Gamma(3\beta)\Gamma(3\beta-1)}{\Gamma(2\beta)\Gamma(4\beta-1)}\cdot \prod_{k\in [2]} f_\beta (r_{k}) .\label{F=1@R}\EDE
    \item [(ii)] For all $\beta>0$,  $\sum_{\u n\in\N_0^3}  A_{\u n} \u r^{\u n} $ and $\sum_{\u n\in\N_0^3}  B_{\u n}\u r^{\u n}$ converge locally uniformly in $(-1,1)^3$, and so $F$ and $G$ are analytic in $(-1,1)^3$ and satisfy (\ref{rainbow1}) and (\ref{rainbow2}).
  \end{enumerate}
\label{convergence@R*}
\end{Theorem}
\begin{proof}
  (i) For the continuity of $F$, by Weierstrass M-test, it suffices to show $\sum_{\u n\in\N_0^3} |A_{\u n}|<\infty$.
  If $\beta=1$, by Theorem \ref{Theorem-AB@R}, $A_{\u n}= 0$ if $\u n\ne \u 0$, and so $\sum_{\u n\in\N_0^3} |A_{\u n}|=|A_{\u 0}|<\infty$. Suppose $\beta>1$. We pick $p\in (3,2\beta+1)$. By Lemma \ref{N-p} (i), $|A_{\u n}|\lesssim (N+1)^{-p}$ if $\u n\in S_N=\{\u n\in\N_0^3:|\u n|=N\}$. Thus,
  $$\sum_{\u n\in\N_0^3} |A_{\u n}|\le\sum_{N=0}^\infty |S_N| (N+1)^{-p} = \sum_{N=0}^\infty \frac{(N+1)(N+2)}2 (N+1)^{-p}  \le   \sum_{N=0}^\infty (N+1)^{2-p}  <\infty.$$
  Suppose $\beta\in (1/2,1)$. We pick $s\in (1, (\beta+1/2)\wedge (2-\beta))$. By Lemma \ref{N-p} (ii),  $|A_{\u n}|\lesssim (a_{\u n}+1)^{-s}(b_{\u n}+1)^{-1} (c_{\u n}+1)^{-s}$, and so
  $$\sum_{\u n\in\N_0^3} |A_{\u n}|\lesssim 6\sum_{a=0}^\infty \sum_{b=0}^a\sum_{c=0}^b (a+1)^{-s} (b+1)^{-1} (c+1)^{-s} $$ $$
  \lesssim \sum_{a=0}^\infty \sum_{b=0}^a (a+1)^{-s} (b+1)^{-1}\le  \sum_{a=0}^\infty (a+1)^{-s} (1+\log(a+1))<\infty.$$

By the above absolute convergence, $F(r_1,r_2,1)=\sum_{(n_1,n_2)\in\N_0^2} A^\Sigma_{(n_1,n_2)}r_1^{n_1} r_2^{n_2}$, where
  $$A^\Sigma_{(n_1,n_2)}:=\sum_{n_3=0}^\infty A_{(n_1,n_2,n_3)},\quad n_1,n_2\in\N_0.$$
Since  ${\mathcal H}_1 A={\mathcal H}_2 A$, we also have  ${\mathcal H}_1 A^\Sigma={\mathcal H}_2 A^\Sigma$. By Lemma \ref{construction} (i), $A^\Sigma$ is determined by the values of $A^\Sigma$ on $\N_0\times \{0\}$. By  Proposition \ref{Appell} (ii), \BGE A^\Sigma_{(n_1,0)}=\frac{\Gamma(3\beta)\Gamma(3\beta-1)}{\Gamma(2\beta)\Gamma(4\beta-1)} \frac{(1-\beta)_{n_1}(\beta)_{n_1}}{(1)_{n_1}(2\beta)_{n_1}},\quad n_1\in\N_0.\label{bdry-A}\EDE
Using the one-dimensional hypergeometric recurrence in each coordinate, one
checks that
  $$\til A^\Sigma_{(n_1,n_2)}:=\frac{\Gamma(3\beta)\Gamma(3\beta-1)}{\Gamma(2\beta)\Gamma(4\beta-1)} \prod_{j=1}^2  \frac{(1-\beta)_{n_j}(\beta)_{n_j}}{(1)_{n_j}(2\beta)_{n_j}},\quad n_1,n_2\in\N_0,$$
  satisfies ${\mathcal H}_1 \til A^\Sigma=0={\mathcal H}_2 \til A^\Sigma$ and (\ref{bdry-A}). Thus, $A^\Sigma=\til A^\Sigma$, and so we get (\ref{F=1@R}),

  (ii) For a general $\beta>0$, we   pick $p< 3\beta+1-(\beta\vee 1) $. Let $L\in (0,1)$. Then
  $$\sum_{\u n\in\N_0^3} |A_{\u n}| L^{|\u n|}\lesssim \sum_{N=0}^\infty |S_N| (N+1)^{-p} L^{N}\lesssim \sum_{N=0}^\infty  (N+1)^{2-p} L^{N}<\infty.$$
  Thus,  $\sum_{\u n\in\N_0^3} A_{\u n} \u r^{\u n}$ converges uniformly in $[-L,L]^3$, and so $F$ is analytic in $(-L,L)^3$. Since $L\in (0,1)$ is arbitrary, $F$ is analytic in $(-1,1)^3$.

  By (\ref{all=B@R}) and the estimate $|A_{\u n}|\lesssim (|\u n|+1)^{-p}$, we get $|\nabla_1 B_{\u n}|\lesssim, (|\u n|+1)^{2-p}$. Since $B$ vanishes outside $\N_0^3$, we have $B_{\u n}=\sum_{k=0}^{n_1} \nabla_1 B_{\u n-k e_1}$, which implies $|B_{\u n}|\lesssim (|\u n|+1)^{ 3-p}$. A similar argument shows $\sum |B_{\u n}| L^{|\u n|}<\infty$ for each $L\in (0,1)$.  Thus, $\sum_{\u n\in\N_0^3} B_{\u n} \u r^{\u n}$ converges absolutely and uniformly in $(-L,L)^3$ for any $L\in (0,1)$, and so $G$ is analytic in $(-1,1)^3$.

  Finally, since we may differentiate the series $\sum A_{\u n} \u r^{\u n}$ and $\sum B_{\u n} \u r^{\u n}$ term by term, f  the coefficient relations (\ref{AB1@R}) and (\ref{rainbow2'}) imply that $F$ and $G$ satisfy (\ref{rainbow1}) and (\ref{rainbow2}).
\end{proof}

\begin{Remark}
  If $\beta\in (0,1/2]$, then by (\ref{AB=@R}) and (\ref{Stirling}) , $$\sum_{\u n\in\N_0^3} |A_{\u n}|\ge\sum_{\u n\in\N_0^2} |A^0_{\u n}|\gtrsim \sum_{(n_1,n_2)\in\N_0^2} (n_1+1)^{\beta-1} (n_2+1)^{\beta-1}  (n_1+n_2+1)^{1-4\beta}  $$
  $$\gtrsim \sum_{b=0}^\infty \sum_{a=b}^\infty (b+1)^{\beta-1} (a+1)^{-3\beta}\gtrsim  \sum_{b=0}^\infty (b+1)^{-2\beta}=\infty.$$
  If $\beta\in (1/3,1/2]$, taking $p\in (1,3\beta+1-(\beta\vee 1))$ and using $|A_{\u n}|\lesssim (N_{\u n}+1)^{-p}$,  one can show $$\sum_{\u n\in\N_0^3} |A_{\u n}| L^{n_1+n_2}\lesssim \sum_{n_1,n_2=0}^\infty L^{n_1+n_2} \sum_{n_3=0}^\infty (n_1+n_2+n_3+1)^{-p}$$ $$\lesssim \sum_{n_1,n_2=0}^\infty (n_1+n_2+1)^{1-p} L^{n_1+n_2}  = \sum_{m=0}^\infty (m+1)^{2-p} L^{m}  <\infty,\quad L\in (0,1).$$
  This implies that  $\sum_{\u n\in\N_0^3} A_{\u n} \u r^{\u n}$ converges uniformly and absolutely on $[-L,L]^2\times [0,1]$ for any $L\in (0,1)$, which implies that $F$ is continuous on $(-1,1)^2\times [0,1]$. A similar argument yields (\ref{F=1@R}) for $\beta\in (1/3,1/2]$ and $(r_1,r_2)\in (-1,1)^2$.
\end{Remark}

\begin{Corollary}
Assume \(\beta>1/2\). Let \(F\) and \(G\) be as in
Theorem \ref{convergence@R*}. Let $\alpha\in\LP_3$ be a rainbow pattern. Then the ${\mathcal Z}_\alpha$ defined by (\ref{Z3}) with
\BGE
c_\alpha
 :=
\frac{\Gamma(2\beta)\Gamma(4\beta-1)}
{\Gamma(3\beta)\Gamma(3\beta-1)}
\left(
\frac{\Gamma(\beta)\Gamma(3\beta-1)}
{\Gamma(2\beta)\Gamma(2\beta-1)}
\right)^2 \label{calpha}
\EDE is positive and satisfies (PDE) (\ref{PDE}), (COV) (\ref{COV}), and (ASY) (\ref{ASY}). In particular,
\(Z_\alpha\)  has the defining properties of  the corresponding pure partition function for
\(\kappa=4/\beta\in(0,8)\).
  \label{victory@R}
\end{Corollary}
\begin{proof}
By Theorem~\ref{convergence@R*} (ii) and Theorem \ref{Prop1} (i),  ${\mathcal Z}_\alpha$ satisfies (PDE) (\ref{PDE}) and (COV) (\ref{COV}).

By Theorem \ref{convergence@R*} (i), $F$ is continuous on $[0,1]^3$ and satisfies (\ref{boundary-F}) with
\BGE c_F:=\frac{\Gamma(3\beta)\Gamma(3\beta-1)}{\Gamma(2\beta)\Gamma(4\beta-1)} \cdot f_\beta(1).\label{cF}\EDE
By Gauss's identity, the $c_\alpha$ in (\ref{calpha}) satisfies $c_\alpha=c_F^{-1} {}2F_1(1-\beta,\beta;2\beta;1)^{-1}$. Thus, by Theorem \ref{Prop1} (iii), ${\mathcal Z}_\alpha$ satisfies (ASY) (\ref{ASY}). By Proposition \ref{Appell} (iii), $F>0$ on the faces $r_j=0$. By (\ref{F=1@R}) and (\ref{2F1>0}), $F>0$ on the faces $r_j=1$.  Thus, by Theorem \ref{Prop1} (iv), ${\mathcal Z}_\alpha$ is positive.
\end{proof}

\subsection{The neighbor case: local convergence}
We now turn to the neighbor case. Unlike the rainbow case, the argument in this subsection yields only a local analytic solution near the origin. The global continuation and boundary analysis will be carried out in Sections~\ref{subsection-Pfaff} and~\ref{subsection-analytic-extension}.

For $z_0\in\C$ and $R>0$, let $B(z_0,R)=\{z\in\C:|z-z_0|<R\}$. For $S\subset\C$ and $R>0$, let $B(S,R)=\bigcup_{z\in S} B(z,R)$.

\begin{Proposition}
   Let $A_{\u n}$ and $B_{\u n}$, $\u n\in\N_0^3$, be given by Theorem \ref{Theorem-AB@N}. Then for any $R\in (0,\frac 1{10})$, $\sum |A_{\u n}| R^{|\u n|}<\infty$ and  $\sum |B_{\u n}| R^{|\u n|}<\infty$, which implies that  $\sum A_{\u n} \u r^{\u n}$ and $\sum B_{\u n} \u r^{\u n}$ converge absolutely and uniformly in $B(0,R)^3$. So (\ref{FG}) defines  two analytic functions $F$ and $G$ on $B(0,\frac 1{10})^3$, which satisfy the PDEs (\ref{rainbow1}) and (\ref{neighbor2}). \label{convergence@N}
\end{Proposition}
\begin{proof}
Let $N\in\N $ and $\u n\in S_N$. Replacing the $\u n$ in (\ref{+-@N}) by $\u n-e_j$, we get
  $$  (P^jM^{\ell} A)_{\u n }=-((P^k+Q^j)M^jA)_{\u n-e_j}+(Q^kTA)_{\u n-e_j-e_k}.$$
  Using (\ref{PQ}) and (\ref{DT@N}) to expand the above formula, we get
  \BGE n_j(|\u n|-2n_{\ell}+2\beta-1) A_{\u n}=-(n_j+n_k+\beta-1)(|\u n|-2n_j+2\beta)A_{\u n-e_j} +(n_k+\beta-1)(|\u n|-1)A_{\u n-e_j-e_k}.\label{factors}\EDE
Suppose $\{j, k, {\ell}\}=[3]$ and $n_j\ge n_k\ge n_{\ell}$. If $n_k=0$ then $n_{\ell}=0$, and so $|\u n|=n_j$. From (\ref{AB=@N}), (\ref{Stirling}) and the symmetry of $A$ we find that there exists $C_*>0$ such that
\BGE |A_{\u n}|\lesssim C_*(|\u n|+1)^{-3\beta},\quad \text{if }n_k=0.\label{trivial}\EDE

  Suppose $n_k\ge 1$. Then all factors other than $|\u n|-2n_j+2\beta$ in (\ref{factors}) are positive. We consider two cases. Case 1. $|\u n|-2n_j+2\beta\le 0$. In this case, we have
  $$n_j(|\u n|-2n_{\ell}+2\beta-1) |A_{\u n}| \le (n_j+n_k+\beta-1)(2n_j-|\u n|-2\beta)\Psi_{N-1}+(n_k+\beta-1)(|\u n|-1) \Psi_{N-2}$$
  $$\le [ (n_j+n_k+\beta-1)(2n_j-|\u n|-2\beta)+(n_k+\beta-1)(|\u n|-1) ]\max\{\Psi_{N-1},\Psi_{N-2}\}.$$
  A direct calculation gives that the coefficient of $\max\{\Psi_{N-1},\Psi_{N-2}\}$ is no more than that of $|A_{\u n}|$. So we get
  \BGE |A_{\u n}|\le \max\{\Psi_{N-1},\Psi_{N-2}\},\quad\text{ if }\u n\in S_N\text{, }n_k\ge 1\text{ and }2n_j\ge |\u n|+2\beta.\label{Good}\EDE
  Case 2.  $|\u n|-2n_j+2\beta> 0$. In this case, we have
  $$n_j(|\u n|-2n_{\ell}+2\beta-1) |A_{\u n}| \le (n_j+n_k+\beta-1)(|\u n|+2\beta-2n_j)\Psi_{N-1}+(n_k+\beta-1)(|\u n|-1) \Psi_{N-2}$$
  $$\le [ (n_j+n_k+\beta-1)(n_k+n_{\ell}+2\beta-n_j)+(n_k+\beta-1)(|\u n|-1) ]\max\{\Psi_{N-1},\Psi_{N-2}\}.$$
 Since $n_k,n_{\ell}\le n_j$ and $|\u n|\le 3 n_j$, the coefficient of $\max\{\Psi_{N-1},\Psi_{N-2}\}$ is bounded above by
 \BGE (2n_j+\beta-1)(n_j+2\beta)+(n_j+\beta-1)(3 n_j-1) =5 n_j^2+(8\beta-5) n_j+(2\beta-1)(\beta-1). \label{ubpd}\EDE
Suppose $n_j\ge 2\beta+1$. Then $n_j\ge 2$ since $\beta>0$. Note that
$$5 n_j^2+(8\beta-5) n_j=5n_j^2+4(2\beta+1)n_j-9n_j\le 9n_j^2-9n_j.$$
 When $\beta\ge 1$, $(2\beta-1)(\beta-1)\le (2\beta-1)(2\beta-2)\le n_j(n_j-1)$; when $\beta\in (\frac 12,1)$, $(2\beta-1)(\beta-1)<0<n_j(n_j-1)$; when $\beta\in (0,1/2)$, $(2\beta-1)(\beta-1)=(1-2\beta)(1-\beta)<1<n_j(n_j-1)$ since $n_j\ge 2$.
Combining these estimates with (\ref{ubpd}) we see that the coefficient of $\max\{\Psi_{N-1},\Psi_{N-2}\}$ is bounded above by $10n_j^2-10 n_j$.  On the other hand, since $\beta>0$, the coefficient of $|A_{\u n}|$ is $n_j(n_j+n_k-n_{\ell}+2\beta-1)\ge n_j(n_j-1)$. Thus,
  \BGE  |A_{\u n}|\le 10 \max\{\Psi_{N-1},\Psi_{N-2}\},\quad\text{ if }\u n\in S_N\text{, }n_j\ge 2\beta+1\text{, }n_k\ge 1\text{ and }2n_j\le |\u n|+2\beta.\label{Bad}\EDE
  Let $\Phi_N=\max\{\Psi_N,\Psi_{N-1}\}$. Combining (\ref{trivial}), (\ref{Good}) and (\ref{Bad}) and using that $n_j\ge |\u n|/3$, we see that, when $N\ge 6\beta+3$,  $\Phi_N\le 10 \max\{C_* , \Phi_{N-1}\}$.  Let $N_1=\lceil 6\beta+3\rceil$.
 An induction shows that $\Phi_N\le  10^{N-N_1+1}(C_*+\Phi_{N_1-1})$  if $N\ge N_1$. Thus, if $R<\frac 1{10}$,   $\sum_{\u n\in \N_0^3} |A_{\u n}| R^{|\u n|} \le \sum_{N=0}^\infty |S_N| \Phi_N R^N  <\infty$.
 Using (\ref{all=B@N}), we first obtain the same polynomial-times-$10^N$ bound for
the discrete derivatives $\nabla_{\ell} B$ on each shell. Since $B=0$ on
$\mathbb Z^3\setminus \mathbb N_0^3$, summing these differences along coordinate
lines shows that $10^{-N}\max_{n\in S_N}|B_n|$ also has polynomial growth in $N$. Thus, we also have $\sum |B_{\u n}| R^{|\u n|}<\infty$ for each $R\in (0,\frac 1{10})$. The remaining statements of the theorem follow easily from the above convergence and the discussion after (\ref{FG}).
\end{proof}

\begin{Remark}
A little more work could improve the lower bound for the local polyradius
from \(1/10\) to \(1/5\).  In general, we do not expect that the power series converges in $B(0,1)^3$. In fact, when $\beta=1$, the power series does not converge in $B(0,R)^3$ if $R> 1/3$. See Remark \ref{Rem:extension}.
\end{Remark}

\subsection{Pfaff systems}\label{subsection-Pfaff}

To continue the local solution in Proposition \ref{convergence@N} beyond the initial polydisc $B(0,\frac 1{10})^3$, we now derive Pfaff systems for the neighbor pattern.

We will use subscripts to denote partial derivatives.
Suppose $F$ and $G$ are symmetric smooth functions satisfying  (\ref{rainbow1}) and (\ref{neighbor2}). We will show that, for any ${\ell}\in[3]$, $G_{\ell}$  can be expressed as a linear combination of $F$, $G$, and $F_k$, $k\in [3]$,  with coefficients rational in $r_1,r_2,r_3$. By symmetry, we consider the case ${\ell}=3$.

Define the following discriminant polynomial for the neighbor pattern:
\BGE \Delta (\u r):=(1-r_1r_2-r_1r_3-r_2r_3)^2-4r_1r_2r_3(r_1+r_2+r_3-2). \label{Delta-nb}\EDE
It is symmetric in $r_1,r_2,r_3$ and has boundary values:
\BGE \Delta |_{r_j=0}=(1-r_{j-1}r_{j+1})^2,\quad \Delta |_{r_j=1}=(1-r_{j-1})^2(1-r_{j+1})^2,\quad j\in [3],\label{Delta-nb-boundary}\EDE
where $r_0:=r_3$ and $r_4:=r_1$.
For ${\ell}\in [3]$, define
\BGE K_{\ell}(\u r)=r_{\ell}(r_{{\ell}+1}+r_{{\ell}-1})-(1+r_{{\ell}+1}r_{{\ell}-1}),\quad K_{\ell}^+(\u r)=K_{\ell}(\u r)+2(1-r_{\ell}).\label{Kl}\EDE

Let $t_j=1-r_j$, $j\in [3]$. Then
\BGE \Delta=t_1^2 t_2^2+t_2^2 t_3^2+t_3^2 t_1^2+2(2-t_1-t_2-t_3) t_1t_2 t_3  \label{Delta-tj}\EDE
$$=\frac 12\sum_{j\in [3]} t_j^2 (t_{j+1}-t_{j-1})^2 +(4-\sum_{j\in [3]} t_j ) \prod_{j\in [3]} t_j  >0,\quad \text{if }\u r\in [0,1)^3.$$
\BGE K_{\ell}^+= t_{\ell}(t_{{\ell}+1}+t_{{\ell}-1})-t_{{\ell}+1}t_{{\ell}-1},\quad K_{\ell}=K_{\ell}^+-2t_{\ell}, \quad {\ell}\in [3].\label{K-tj}\EDE

Setting ${\ell}=2$ in (\ref{rainbow1}) and then differentiating it w.r.t.\ $r_3$, we get
\BGE r_2t_2F_{223}-t_1t_3 G_3=-t_1G+\beta(1-\beta)F_3-(2\beta-2r_2) F_{23}. \label{2231@R}\EDE
Setting ${\ell}=1$ in (\ref{neighbor2}) and then differentiating it w.r.t.\ $r_2$, we get
\BGE r_2r_3 F_{223}-\frac {K_1}2 G_2=\frac 12(r_1 -r_3)G-\frac\beta 2 r_2 F_{22}-\frac{\beta+\beta^2} 2 F_2-\frac{\beta+2}2 r_3 F_{23} -\frac\beta2 r_1 F_{12}.\label{2232@N}\EDE
$t_2\times (\ref{2232@N}) - r_3\times (\ref{2231@R})$  yields:
$$ r_3t_1 t_3 G_3-\frac {K _1}2 t_2 G_2= [r_3t_1-\frac 12t_2(r_3-r_1)]G -\beta(1-\beta) r_3 F_3$$
\BGE-\frac{\beta+\beta^2} 2t_2F_2-\frac\beta2r_1t_2F_{12}-\frac\beta 2 r_2t_2 F_{22} +[2(\beta-r_2)r_3-\frac{\beta+2}2 r_3t_2]F_{23}.
\label{G231@N}\EDE
Since $F$, $G$, and $K _1$ are symmetric in $r_2$ and $r_3$, by swapping the roles of $r_2,t_2$ with those of $r_3,t_3$, we get  $$ r_2t_1t_2 G_2-\frac {K _1}2 t_3 G_3= [r_2t_1-\frac 12t_3(r_2-r_1)]G -\beta(1-\beta) r_2 F_2$$
\BGE-\frac{\beta+\beta^2} 2t_3F_3-\frac\beta2r_1t_3F_{13}-\frac\beta 2 r_3t_3 F_{33} +[2(\beta-r_3)r_2-\frac{\beta+2}2 r_2t_3]F_{23}.
\label{G232@N}\EDE
$-4r_2t_1\times (\ref{G231@N}) -2  {K _1}  \times (\ref{G232@N})$ together with the identity $K_1^2-4r_2 r_3 t_1^2=\Delta$ yields
\begin{align}
  &t_3\Delta G_3\nonumber\\
  =&-{t_3} ((r_1-r_2)^2 r_3 -r_1^2r_2-r_1r_2^2+4r_1r_2-r_1-r_2)G\nonumber\\
   &-2{\beta r_2}((\beta-1) r_3(r_1-r_2)+(\beta+1)(r_1+r_2) -2r_1r_2 -2\beta)F_2\nonumber\\
  &-\beta [(\beta+1)(1-r_1r_2-r_3-r_1r_3+r_1r_3^2-r_2r_3^2)+(5-3\beta)r_1r_2r_3+(5\beta-3) r_2 r_3]F_3\nonumber\\
  &+2\beta  r_1r_2t_1t_2F_{12}+\beta  r_1t_3(r_1r_2+r_1r_3-r_2r_3-1)F_{13}\nonumber\\
  &+ 2\beta t_1r_2^2t_2F_{22}+\beta  r_3t_3(r_1r_2+r_1r_3-r_2r_3-1) F_{33}\nonumber\\
  &+{r_2} [(\beta-2)(r_1r_2r_3+r_2r_3^2-r_1r_3^2)+(3\beta-2)(r_1r_3-r_1r_2+1)\nonumber\\
  &\qquad+(\beta+2)r_2r_3-(5\beta-2)r_3 ] F_{23}. \label{G3=1@N}
\end{align}

Then we use (\ref{rainbow1}) and (\ref{neighbor2}) to express $F_{22}$, $F_{33}$, $F_{12}$, $F_{13}$, and $F_{23}$ in terms of $G$, $F$, and $F_j$, $j\in [3]$, and substitute them back into (\ref{G3=1@N}) to obtain, for ${\ell}=3$,
\BGE   G_{\ell}=  \sum_{f\in \mathcal S^{\text{Pf}}} C^{({\ell})}_f f ,\label{G3=2@N}\EDE
where $\mathcal S^{\text{Pf}}:=\{F_1,F_2,F_3,F,G\}$, and for $\{j,k\}=[3]\sem \{{\ell}\}$,
\begin{align}
 C^{({\ell})}_{F_j}&:= -\beta(1-\beta)\frac{r_j K _k}{r_{\ell} \Delta}=2\beta(1-\beta)\frac{r_jt_k }{r_{\ell} \Delta} -\beta(1-\beta)\frac{r_j K _k^+}{r_{\ell} \Delta} ,\label{CnbF1}\\
 C^{({\ell})}_{F_{\ell}}&:=-2 \beta(1-\beta) \frac{r_{\ell} t_j t_k}{t_{\ell} \Delta} ,\label{CnbF3}\\
 C^{({\ell})}_F&:= \beta^2(1-\beta)\frac{K_{\ell}^+}{t_{\ell} r_{\ell} \Delta} ,\label{CnbF}\\
 C^{({\ell})}_G&:=-\frac \beta {r_{\ell} t_{\ell}} -\frac{Q^{({\ell})}}{r_{\ell} \Delta}, \label{CnbG}
\end{align}
where
\BGE Q^{({\ell})}:=2r_{\ell}^2(r_{j}-r_{k})^2-r_{\ell}(r_{j}+r_{k})(1+r_{j} r_{k})+4 r_{\ell} r_j r_k-(1-r_{j}r_{k})^2\label{CnbQ}\EDE
\BGE =t_{\ell}(2t_{\ell} -3)(t_j-t_k)^2+t_lt_jt_k(2-t_j-t_k)-t_jt_k(3(2-t_j-t_k)+t_j t_k). \label{Q-tj}\EDE
By the symmetry of $F$ and $G$, (\ref{G3=2@N}) holds for any $\{j,k,{\ell}\}=[3]$.

For rational functions $R_k$, $k\in[5]$, in $\u r$,  define   ${Y}=[\begin{array}{ccccc} R_1F_1 & R_2 F_2 & R_3F_3 & R_4F  & R_5G
\end{array}]^{\text T}$. Using (\ref{rainbow1}), (\ref{neighbor2}) and (\ref{G3=2@N}), we see that, for any ${\ell}\in[3]$, we get a Pfaff equation $\pa_{r_{\ell}} {Y}=M_lY$, where $M_{\ell}$ is a $5\times 5$ matrix function whose entries are rational functions of $\u r$. Together these form a Pfaff system (see, e.g., \cite{Aomoto-Kita}).

\begin{Remark}
 If  $F$ and $G$ satisfy  (\ref{rainbow1}) and (\ref{rainbow2}) for the rainbow pattern,  we may derive a formula for $G_{\ell}$ of the same form as (\ref{G3=2@N}) but with different formulas for the coefficients $C^{({\ell})}_f$: if $\{j,k,{\ell}\}=[3]$,
 $$ C^{({\ell})}_{F_j} =-\beta(1-\beta)\frac{r_jK _k}{r_{\ell}\Delta^{[\text{R}]} },\,  C^{({\ell})}_{F_{\ell}} =-2 \beta(1-\beta) \frac{r_lt_jt_k}{t_{\ell}\Delta^{[\text{R}]} }, \,  C^{({\ell})}_F = \beta^2(1-\beta) \frac{K^+_{\ell} }{t_{\ell} r_{\ell}\Delta^{[\text{R}]} },\,  C^{({\ell})}_G =-\frac\beta {r_lt_{\ell}}- \frac{Q^{({\ell})}}{r_{\ell}\Delta^{[\text{R}]}},$$
where $\Delta^{[\text{R}]}$ is the discriminant polynomial for the rainbow pattern defined by
$$
\Delta^{[\text{R}]}(r_1,r_2,r_3)
:=
(r_1r_2r_3-r_1-r_2-r_3+2)^2
-
4(1-r_1)(1-r_2)(1-r_3).  
$$
and $K_j$, $K^+_j$, and $Q^{({\ell})}$ now have expressions:
$$  K_j=r_1r_2r_3-r_1-r_2-r_3+2r_j;\quad K^+_j=K_j+2 t_j-2 t_{j+1} t_{j-1};$$
$$Q^{({\ell})}=2r_{\ell}^2(1-r_{{\ell}+1}r_{{\ell}-1})^2-r_{\ell}(r_{{\ell}+1}+r_{{\ell}-1})(1+r_{{\ell}+1} r_{{\ell}-1})-(r_{{\ell}+1}-r_{{\ell}-1})^2+4 r_1 r_2 r_3.$$
Such $\Delta^{[\text{R}]} $ has zeros in $(0,1)^3$, and so the Pfaff system for the rainbow pattern is not as useful as that for the neighbor pattern. The sign of $\Delta^{[\text{R}]}$ has the following physical meaning.
  Let
\(
\alpha=\{\{a_j,b_j\}:j\in[3]\}\in\mathrm{LP}_3
\)
be a rainbow pattern. Then $$(
C_{\{x_{a_1},x_{b_1}\},\{x_{a_3},x_{b_3}\}},
C_{\{x_{a_2},x_{b_2}\},\{x_{a_3},x_{b_3}\}},
C_{\{x_{a_1},x_{b_1}\},\{x_{a_2},x_{b_2}\}}
)\in \{\u r\in(0,1)^3: \Delta^{[\text{R}]}(\u r)\ge 0\}.$$
In fact, for $\u r\in [0,1)^3$, $\Delta(\u r)>0$ iff $\log(\frac{1+\sqrt{r_j}}{1-\sqrt{r_j}})$, $j\in [3]$, form the three sides of a triangle.
\end{Remark}

\subsection{Analytic extension to a neighborhood of $[0,1)^3$} \label{subsection-analytic-extension}


With the Pfaff systems from the previous subsection in hand, we now
analytically continue the local power-series solution.  More precisely, we
will show that \(F\) and \(G\) are analytic on \([0,1)^3\).  Here and below,
we use the convention that a function \(f\) is said to be analytic on a
subset \(S\subset\mathbb R^3\) if \(f\) extends analytically to an open
subset \(U\subset\mathbb C^3\) containing \(S\).

\begin{Theorem}
  The analytic functions $F$ and $G$ on $B(0,1/10)^3$ defined in Proposition \ref{convergence@N} extend analytically to a domain in $\C^3$ containing $[0,1)^3$. Moreover, the PDEs (\ref{rainbow1}) and (\ref{neighbor2}) hold in this domain. \label{extension}
\end{Theorem}
\begin{proof}
By the identity theorem, it suffices to show that, for any $L\in (0,1)$, $F$ and $G$ extend analytically to a domain in $\C^3$ containing $[0,L]^3$.
Fix some $L\in (0,1)$. For $S\subset \C$, we will use $B(S,r)$ to denote $\bigcup_{z\in S} B(z,r)$.

Using (\ref{rainbow1}), (\ref{neighbor2}), (\ref{G3=2@N}) and the symmetry of $F$ and $G$, we construct the following Pfaff vector: ${Y}:=[\begin{array}{ccccc} F_1 & F_2 & F_3 & F & G
\end{array}]^{\text T}$, which satisfies, for any ${\ell}\in [3]$,
\BGE \pa_{\ell} {Y}=M^{({\ell})} {Y}\quad \text{on }B(0,1/10)^3\sem\Big(\{\text{zeros of }\Delta\}\cup\bigcup_{j=1}^3\{r_j=0\}\Big).\label{pa3Y}\EDE
where $M^{({\ell})}=(M^{({\ell})}_{jk})_{j,k=1}^5$  is a $5\times 5$ matrix function  such that each $M_{jk}$  can be written with denominator dividing  $(1-r_{\ell})r_1r_2r_3 \Delta(\u r)$.

Since $\Delta>0$ on $[0,1)^3$ and $L\in (0,1)$, there is $\eps\in (0,(1/10)\wedge (1-L))$ such that $\Delta$ has no zero in $B([0,L],\eps)^3$. Let $B_*(z,r)=B(z,r)\sem \{z\}$. Let $U=B_*(0,\eps) ^2$ and $V=B([0,L],\eps)\sem [-\eps,0]$. Since $\prod_{j\in[3]}(r_j(1-r_j)^2) \Delta(\u r) $ has no zeros in $U\times V$, $M^{(3)}$ is analytic in $U\times V$. Let $z_0=\eps/2\in V$ and $\phi(\cdot,\cdot)={Y}(\cdot,\cdot,z_0)$. Then $\phi$ is analytic on $U$. Since $V$ is a simply connected domain, by Proposition \ref{ODE-complex}, the equation
 $\pa_3\til {Y}=M^{(3)}\til {Y}$ with $\til {Y}(\cdot,z_0)=\phi$
has a unique analytic solution in $U\times V$. By the uniqueness part in Proposition \ref{ODE-complex} applied to $U':=B_*(0,\eps)^2$ and $V':=B(0,\eps)\sem [-\eps,0]$, we get $\til Y=Y$ on $U'\times V'$.   Since $\til Y$ and $Y$ are analytic on the domain $D:= (U\times V)\cap B(0,1/10)^3\supset U'\times V'$,  the identity theorem implies that $\til Y=Y$ on $D$.  Thus,  ${Y}$ extends analytically to $(U\times V)\cup B(0,1/10)^3$, which contains $B(0,\eps)^3\cup(B_*(0,\eps)^2\times B([0,L],\eps))$. Since $F$ and $G$ are components of $Y$, they also extend analytically to $B(0,\eps)^3\cup (B_*(0,\eps)^2\times B([0,L],\eps))$.

At $r_1=0$, $F$ has Laurent series expansion: $\sum_{n=-\infty}^\infty a_n(r_2,r_3) r_1^n$, where each $a_n$ is analytic in $B_*(0,\eps)\times B([0,L],\eps)$.
Since \(F\) is analytic in \(B(0,\epsilon)^3\), the negative Laurent
coefficients vanish for \((r_2,r_3)\in B_*(0,\epsilon)\times B(0,\epsilon)\).
By the identity theorem in the variables \((r_2,r_3)\), they vanish on
\(B_*(0,\epsilon)\times B([0,L],\epsilon)\). Thus, $F$ extends analytically to $B(0,\eps)^3\cup (B(0,\eps)\times B_*(0,\eps)\times B([0,L],\eps))$. A similar argument shows that $F$ further extends analytically to $\Omega_1:=B(0,\eps)^2\times B([0,L],\eps)$. Similarly, $G$ also extends analytically to $\Omega_1$.

By the identity theorem, $F$ and $G$ also satisfy  the PDEs (\ref{rainbow1}) and (\ref{neighbor2})  in $\Omega_1$. Define $Y=[\begin{array}{ccccc}F_1 & F_2 & F_3 & F & G
\end{array}]^{\text T}$ on $\Omega_1$. Then $Y$ satisfies $\pa_2 Y=M^{(2)} Y$ in $\Omega_1$.

We now repeat the preceding continuation argument cyclically. For the equation
$\pa_2Y=M^{(2)}Y$, we first work on a punctured domain on which the entries of
$M^{(2)}$ are analytic, apply Proposition~\ref{ODE-complex}  to the \(r_2\)-equation, with \((r_1,r_3)\) treated as parameters, and then remove the apparent singularities along
the coordinate hyperplanes $r_1=0$ and $r_3=0$ by Laurent expansions and the identity
theorem. In this way, $Y$, and hence $F$ and $G$, extend analytically to
\(
\Omega_2:=B(0,\eps)\times B([0,L],\eps)^2
\).
By the identity theorem, $F$ and $G$ still satisfy (\ref{rainbow1}) and
(\ref{neighbor2}) in $\Omega_2$. Applying the same punctured-domain argument once
more to the equation $\pa_1Y=M^{(1)}Y$, and then removing the singularities in $r_2=0$ and $r_3=0$, we
finally obtain analytic extensions of $F$ and $G$ to
\(
\Omega_3:=B([0,L],\eps)^3
\).
The proof is now complete.
\end{proof}


\begin{Remark}
The following observation will be used in Section \ref{section:integer}.
Since the functions $F$ and $G$ from Theorems \ref{convergence@R*} and \ref{extension} satisfy (\ref{rainbow1}),  the function $F$  satisfies
\BGE {\mathcal H}_j F={\mathcal H}_k F,\quad \forall j,k\in[3],\label{Ljk-PDE}\EDE
where now $ {\mathcal H}_j:=(1-r_j)\L_{r_j}$,
where $\L_{r_j}$ is defined in (\ref{hyperF21}) with $x=r_j$.
On the other hand, if $F$ is analytic at the origin, then the coefficients of its power series, extended by zero outside $\N_0^3$,  satisfy (\ref{Ljk=}) iff $F$ satisfies (\ref{Ljk-PDE}).
\label{Remark-Hj}
\end{Remark}



\section{Continuous Extension to $[0,1]^3$ in the Neighbor Case} \label{section:continuous}
In this section, \(F\) and \(G\) denote the functions obtained in
Theorem~\ref{extension}, restricted to \([0,1)^3\). We use subscripts to
denote partial derivatives of \(F\) and \(G\), and   use the convention
$t_j=1-r_j$, $j\in [3]$.
Our goal is to prove that \(F\) extends continuously to \([0,1]^3\)  {for every $\beta>1/2$}.
 {In Sections \ref{Section-6.1}-\ref{Section-6.3},}
 {a}part from a few preliminary estimates that are stated in the larger range \(\beta>1/2\), we work under the standing assumption \(\beta\ge 2/3\).  {The remaining range $1/2<\beta<2/3$ is treated separately in Section \ref{Section-6.4}.}
All implicit constants in \(\lesssim\), \(\gtrsim\), or $O(\cdot)$ are allowed to depend on \(\beta\). When an implicit constant depends on an additional parameter, say \(L\), we write \(\lesssim_L\), \(\gtrsim_L\), or $O_L(\cdot)$. Recall the $f_\beta$ defined in (\ref{2F1>0}). We also set
\BGE \alpha:=2\beta-1>0,\quad \gamma:=\alpha\wedge 1\in (0,1]. \label{gamma}\EDE

\begin{Theorem}
Let \(\beta\ge 2/3\). Then \(F\) extends continuously to \([0,1]^3\), and for each \({\ell}\in[3]\), its restriction of this continuous extension to the side face \(\{r_{\ell}=1\}\) is given by \BGE
F|_{r_{\ell}=1}
=
\frac{\Gamma(2\beta)\Gamma(3\beta-1)}
{\Gamma(\beta)\Gamma(4\beta-1)}
\prod_{j\in[3]\setminus\{{\ell}\}}
f_\beta (r_j).
\label{F-boundary}
\EDE
\label{Thm-coninuity-neighbor}
\end{Theorem}

The section is organized as follows. In Section \ref{Section-6.1} we prove non-corner estimates and continuity away from the corner point. In Section \ref{Section-6.2} we derive the side-face equations and regular normal estimates. In Section \ref{Section-6.3} we propagate these estimates into the corner region and complete the final stitching argument.  {Section~\ref{Section-6.4} treats the remaining range
$1/2<\beta<2/3$.}

\subsection{Non-corner estimates and continuity} \label{Section-6.1}
For $\{j,k,{\ell}\}=[3]$ and $L\in (0,1)$, define
$$\mathcal D^{({\ell})}_L=\{\u r\in [0,1)^3: r_{\ell}\le L\},\quad \mathcal D^{({\ell})}_{L*}=\{\u r\in [0,1)^3: 0<r_{\ell}\le L\};$$
$$\tau_{jk}=\tau_{kj}=1-r_jr_k ,\quad w_{jk}=\tau_{jk}^{\gamma-1},\quad\text{and}\quad \Lambda_{jk}= \log\frac{\tau_{jk} }{t_j}.$$
Note that $0<t_j,t_k\le \tau_{jk}\le 1$, $w_{jk}\ge 1$, and $\Lambda_{jk}\ge  0$.

\begin{Proposition}
Suppose $\beta\ge 2/3$. Let $\{j,k,{\ell}\}=[3]$ and $L\in (0,1)$. Then the following estimates hold on $\mathcal D^{({\ell})}_L$:
\BGE |F|+|F_{\ell}|\lesssim_L 1,\quad |F_j|+|F_k|+\tau_{jk}|G|\lesssim_L w_{jk},\quad\text{if }\beta>2/3;\label{F+G>2/3}\EDE
$$ |F|+|F_{\ell}|\lesssim_L 1,\quad \tau_{jk}|G|\lesssim_L w_{jk}(1+\Lambda_{jk}+\Lambda_{kj}),$$
\BGE |F_j|\lesssim_L w_{jk}(1+\Lambda_{jk}),\quad  |F_k|\lesssim_L w_{kj}(1+\Lambda_{kj}),  \quad\text{if }\beta=2/3.\label{F+G=2/3}\EDE
In particular, for every $L\in (0,1)$, $F$, $t_j F_j$, $j\in [3]$, and $t_1t_2 t_3 G$ are bounded on $\bigcup_{{\ell}\in [3]}  \mathcal D^{({\ell})}_{L*}$. Moreover,  $F$ extends continuously to $[0,1]^3\sem \{(1,1,1)\}$ for all $\beta\ge 2/3$.
\label{prop-non-corner}
\end{Proposition}

We first establish two auxiliary lemmas needed for the proof of Proposition~\ref{prop-non-corner}.  By symmetry, it suffices to prove them on \(\mathcal D_L^{(3)}\).  We write $\tau=\tau_{12}$, $w=w_{12}$, $\Lambda_1=\Lambda_{12}$,  $\Lambda_2=\Lambda_{21}$, and $\Lambda=\Lambda_1+\Lambda_2$.
We will use the following simple estimates: for all $x_1,x_2\in [0,1]$,
\BGE |x_1-x_2|\le 1-x_1x_2\le (1-x_1)+(1-x_2)\le 2(1-x_1x_2);\label{x1x2-est1}\EDE
\BGE (1-x_1)(1-x_2)\le (1-x_1x_2)^2.\label{x1x2-est2}\EDE
Let $p=p(r_1,r_2)=\frac{r_1-r_2}{\tau }$ and $q=q(r_1,r_2)=\frac{t_1 t_2}{\tau}$.
The above estimates imply that $\tau_{12}\asymp t_1+t_2$, $|p|\le 1$, and $0<q\le \tau_{12}\le 1$. Moreover, $0<qw(1+\Lambda)=\tau^\gamma\cdot \frac{t_1t_2}{\tau^2}(1+\log\frac{\tau^2}{t_1t_2})\le 1$.

 Define
\BGE F^{[3]}=F_3,\quad S=(r_1F_1+r_2 F_2)/w,\quad U=(r_1 F_1-r_2 F_2)/w,\quad V=\tau G/w 
.\label{SUVW}\EDE
Setting ${\ell}=3$ in (\ref{rainbow1}), we get
\BGE  r_3 t_3 \pa_{r_3} F^{[3]}+(2\beta-2 r_3) F^{[3]}-\beta(1-\beta)F=q w V=t_1 t_2 G.\label{FGH-1}\EDE
Summing the two formulas obtained by applying (\ref{neighbor2}) with ${\ell}=1$ and ${\ell}=2$,   we obtain:
\BGE r_3w\pa_{r_3} S+\beta w S+\beta r_3  F^{[3]}+\beta^2   F+ w V=0.\label{FGH-2}\EDE
For $X\in\Sigma:=\{F,F^{[3]},S,U,V
\}$,  define $X_0=X|_{r_3=0}$ and $\lin X=X-X_0$.
By a slight abuse of notation,
we regard \(X_0\) as a function on \([0,1)^2\times[0,1]\) by setting
\(
X_0(r_1,r_2,r_3):=X_0(r_1,r_2)
\),
so that it is independent of \(r_3\).  Note that the $F_0$  here agrees with the $F_0$ in Proposition \ref{Horn}.
Setting $r_3=0$ in (\ref{FGH-1}) and (\ref{FGH-2}), we get
\BGE 2\beta F^{[3]}_0-\beta(1-\beta) F_0=q w V_0.\label{FGH0-1}\EDE
\BGE \beta w S_0+\beta^2 F_0   + w V_0=0. \label{FGH0-2}\EDE

\begin{Lemma}
Suppose $\beta\ge 2/3$. The following estimates hold on $[0,1)^2$.
$$|F_0|+|F^{[3]}_0|+|S_0|+|U_0|+|V_0|\lesssim 1,\quad \text{if }\beta>2/3;$$
$$|F_0|+|F^{[3]}_0|\lesssim 1,\quad  |S_0|+|U_0|+|V_0|\lesssim 1+ \Lambda,\quad \text{if }\beta=2/3.$$
Recall that $\Lambda=\Lambda_1+\Lambda_2=\Lambda_{12}+\Lambda_{21}=\log\frac{\tau^2}{t_1t_2}$.
\label{1.G}
\end{Lemma}
\begin{proof}
By Proposition~\ref{Horn} (ii), since $\beta\ge 2/3>1/2$, $|F_0|\lesssim 1$.

Suppose $\beta>2/3$. We first prove $|\pa_1 F_0|\lesssim w$.
When $\beta=1$, by (\ref{AB=@N}), we have
\[
        F_0(r_1,r_2)
        =
        \sum_{n=0}^\infty(r_1r_2)^n
        -
        \frac12\sum_{n=0}^\infty r_1^{n+1}r_2^n
        -
        \frac12\sum_{n=0}^\infty r_1^nr_2^{n+1}=
        \frac{2- (r_1+r_2)}{2\tau}.
\]
Therefore (recall $t_j=1-r_j\le \tau=1-r_1 r_2$)
\(
       | \partial_1F_0|
        =
        \frac{t_2^2}{2\tau^2}\le 1=w\).

Suppose   $\beta\in (2/3,1)\cup(1,\infty)$.  Differentiating the integral representation (\ref{F=int-K}) under the integral sign and using \(s_1^\beta\le1\), we estimate
$ |\pa_{r_1} F_0(\u r)|\lesssim I_1+I_2$,
where
$$I_1:=\int_0^1\!\!\int_0^1  s_2^{\beta-1} (1-s_1)^{\beta-1}(1-s_2)^{\beta-1} (1-s_1 r_1)^{2\beta-2} (1-s_2r_2)^{2\beta-1} (1-s_1s_2r_1r_2)^{1-4\beta} ds_1ds_2;$$
$$I_2:=\int_0^1\!\!\int_0^1 s_2^\beta  (1-s_1)^{\beta-1}(1-s_2)^{\beta-1}(1-s_1 r_1)^{2\beta-1} (1-s_2r_2)^{2\beta-1}   (1-s_1s_2r_1r_2)^{-4\beta} ds_1ds_2.$$
If $r_1\wedge r_2\le 1/2$, then   $1-s_1s_2r_1r_2\asymp 1$. Using separation of variables, we can estimate upper bounds for $I_1$ and $I_2$. Using further $s_1^\beta,(1-s_2r_2)^{2\beta-1}\le 1$, $s_2^\beta\le s_2^{\beta-1}$, and $(1-s_1 r_1)^{2\beta-1}\le (1-s_1r_1)^{2\beta-2}$, we find that, for $j\in [2]$,
$$I_j\lesssim \int_0^1    (1-s_1)^{\beta-1}(1-s_1 r_1)^{2\beta-2} ds_1 \cdot \int_0^1  s_2^{\beta-1}(1-s_2)^{\beta-1} ds_2$$
\BGE \lesssim \int_0^1    (1-s_1)^{\beta-1}(1-s_1 r_1)^{2\beta-2} ds_1\le \int_0^1  (1-s_1) ^{(\beta-1)+((2\beta-2)\wedge 0)} ds_1\lesssim 1,\label{Ij-2/3}\EDE
where the last integral is finite because $(\beta-1)+((2\beta-2)\wedge 0)>-1$ since $\beta>2/3$. So we get $ | \partial_1F_0| \lesssim 1\le w$.

In the following, we assume $r_1,r_2\ge 1/2$. We decompose the integral for $I_j$ into
$$I_j=I_j'+I_j'':= {\int\!\!\int_{[1/2,1]^2} \cdot ds_1 ds_2} +  {\int\!\!\int_{[0,1]^2\sem [1/2,1]^2} \cdot ds_1 ds_2} ,\quad j\in [2].$$
On $[0,1]^2\sem [1/2,1]^2$, we have  $1-s_1s_2r_1r_2\asymp 1$. Then we may use the same separation-of-variables argument to get $I_j''\lesssim 1$. So it suffices to show that $I_j'\lesssim w$, $j\in [2]$.

We now use a change of variables: $u_j=1-s_j$, $j\in [2]$. When $s_j,r_j\in [1/2,1]$, $j\in [2]$, we have by (\ref{x1x2-est1}), $1-s_jr_j\asymp u_j+t_j$ and $1-s_1r_1s_2r_2\asymp u_1+t_1+u_2+t_2$. Using further $(u_2+t_2)^{2\beta-1}\le (u_1+u_2+t_1+t_2) ^{2\beta-1}$ because $\beta>1/2$, we get
\BGE I_j'\lesssim \int \!\!\int_{[0,1/2]^2} u_1^{\beta-1}u_2^{\beta-1}(u_1+t_1)^{2\beta-3+j} (u_1+u_2+t_1+t_2) ^{-(2\beta -1+j)}du_1du_2.\label{Ij'}\EDE

Note that $[0,1/2]^2\subset\triangle:=\{(u_1,u_2)\in \R_+^2:u_1+u_2\le 1\}$. We use the elementary estimate:
\BGE
        \int\!\!\int_{\triangle }
        u_1^{\lambda_1-1}u_2^{\lambda_2-1}(u_1+u_2+v)^{-m} du_1du_2
        \lesssim  v^{(\lambda_1+\lambda_2-m)\wedge 0},\quad v>0.
\label{elementary}
\EDE
which is valid for $\lambda_1,\lambda_2>0$, $m\ge 0$, and $\lambda_1+\lambda_2\ne m$.  This follows by the change of variables:
$u_1=v s \theta$ and $u_2=vs(1-\theta)$ for $s\in (0,1/v)$ and $\theta\in [0,1]$.

We estimate $I_j'$ in two cases. Case 1. $\beta\in (2/3,1)$ and $j=1$. Then  $2\beta-3+j=2\beta-2 <0$, and so $\gamma-1=2\beta-2$. From (\ref{Ij'}) we get
\BGE I_1'\le \int\!\!\int_\triangle u_1^{3\beta-3} u_2^{ \beta-1} (u_1+u_2+t_1+t_2)^{-2\beta}du_1du_2.\label{Ij'-2/3}\EDE
Applying (\ref{elementary}) with $\lambda_1=3\beta-2>0$, $\lambda_2= \beta>0$ and $m=2\beta$, since $\lambda_1+\lambda_2-m=2\beta-2<0$, we get $I_j'\lesssim (t_1+t_2)^{2\beta-2}=(t_1+t_2)^{\gamma-1}$.

Case 2. Either $\beta>1$ or $j=2$. Then $2\beta-3+j>0$. From (\ref{Ij'}), we get
$$I_j' \le \int\!\!\int_\triangle u_1^{ \beta-1} u_2^{ \beta-1} (u_1+u_2+t_1+t_2)^{-2}du_1du_2.$$
Applying (\ref{elementary}) with $\lambda_1=\lambda_2= \beta$ and $m=2$, we get $I_j'\lesssim (t_1+t_2)^{\gamma-1}$.

Thus, in both cases, we have $I_j'\lesssim (t_1+t_2)^{\gamma-1}$, $j\in[2]$, which together with $I_j''\lesssim 1$ and (\ref{x1x2-est1}) implies that $|\pa_1 F_0|\lesssim w$. Similarly, $|\pa_2 F_0|\lesssim w$. The two estimates together imply that $|S_0|,|U_0|\lesssim 1$.
From $|F_0|,|S_0|\lesssim 1$, $w\ge 1$, and (\ref{FGH0-2}), we get $|V_0|\lesssim 1$. By $q\le \tau$,   $|F_0|,|V_0| \lesssim 1$,  and (\ref{FGH0-1}), we get $|F^{[3]}_0|\lesssim 1+\tau w\lesssim 1$.

We now treat the case $\beta=2/3$. Most of the above argument still goes through except for two places. First, (\ref{Ij-2/3}) no longer holds, since the last integral in (\ref{Ij-2/3}) equals $\infty$. Now the first integral in (\ref{Ij-2/3}) becomes $\int_0^1 (1-s_1)^{-1/3} (1-r_1 s_1)^{-2/3}$. We use the change of variables: $u=(\frac{1-s_1}{1-r_1 s_1})^{1/3}$ and find that it equals
$$\int_0^1 \frac{3u du}{1-r_1 u^3}\lesssim \int_0^1 \frac{du}{1-r_1^{1/3} u}=-r_1^{-1/3} \log(1-r_1^{1/3})\lesssim 1+\Lambda_1,$$
where  the ``$\le$'' holds because $1-x^3\ge 3x(1-x)$ for any $x\in [0,1)$, and the ``$\lesssim$'' holds because (i) if $r_1\le 1/2$, $-r_1^{-1/3} \log(1-r_1^{1/3})\lesssim 1\le 1+ \Lambda_1$; and (ii) if $r_1> 1/2$, then $r_2\le 1/2$ by the assumption $r_1\wedge r_2\le 1/2$, and so $\tau\ge t_2\ge 1/2$, and
$$-r_1^{-1/3} \log(1-r_1^{1/3})\lesssim -\log(1-r_1^{1/3})\le \log 3-\log t_1\lesssim 1+\Lambda_1  .$$
Thus, we get $|\pa_1 F_0|\lesssim w(1+\Lambda_1) $ in the case $r_1\wedge r_2\le 1/2$.

The second place is that we cannot apply (\ref{elementary}) to bound the integral in (\ref{Ij'-2/3}), which now equals $\infty$. We may bound the integral in (\ref{Ij'}) by first integrating with respect to $u_2$ for every fixed $u_1$:
$$\int_0^{1/2} u_2^{-1/3} (u_1+u_2+t_1+t_2)^{-4/3} du_2\lesssim (u_1+t_1+t_2)^{-2/3}.$$
Thus, using $\tau\asymp t_1+t_2$, we see that the integral in (\ref{Ij'}) is
$$\lesssim \int_0^{\infty} u_1^{-1/3}(u_1+t_1)^{-2/3}  (u_1+\tau)^{-2/3} du_1=:I_1+I_2+I_3,$$
where $I_1,I_2,I_3$ are respectively the integral  over $[0,t_1]$, $[t_1,\tau]$, and $[\tau,\infty)$. On $[0,t_1]$, the integrand is bounded by $\tau^{-2/3} u_1^{-1/3}(u_1+t_1)^{-2/3}$, and so $I_1\lesssim \tau^{-2/3}$. On $[\tau,\infty)$, the integrand is bounded by $u_1^{-5/3}$, and so $I_3\lesssim \tau^{-2/3}$.  On $[t_1,\tau]$, the integrand is bounded by $\tau^{-2/3} u_1^{-1}$, and so $I_2\le \tau^{-2/3}\Lambda_1$. Thus, we get $|\pa_1 F_0|\lesssim w( 1+\Lambda_1) $ in the case $r_1\wedge r_2\ge 1/2$.

Symmetrically,  $|\pa_2 F_0|\lesssim w( 1+\Lambda_2) $. The two estimates together imply that $|S_0|,|U_0|\lesssim 1+\Lambda$.
From $|F_0|,|S_0|\lesssim 1+\Lambda $, $w\ge 1$, and (\ref{FGH0-2}), we get $|V_0|\lesssim  1+\Lambda$. By $|F_0|\lesssim 1$, $|V_0| \lesssim1+ \Lambda$, $qw(1+\Lambda)\le  1$,    and (\ref{FGH0-1}), we get $|F^{[3]}_0|\lesssim   1$.
\end{proof}


Since $F_3=F^{[3]}$, we have
\BGE r_3  \pa_{r_3}\lin{F }=r_3(\lin{F^{[3]}}+F^{[3]}_0).\label{F-row}\EDE
Subtracting (\ref{FGH0-1}) from (\ref{FGH-1}), we get
\BGE r_3  \pa_{r_3}\lin{ F^{[3]}} =({\beta(1-\beta)}\lin F+ {qw}\lin V-2\beta \lin{F^{[3]}})+\frac{r_3}{t_3} (\beta(1-\beta) \lin F+2(1-\beta) \lin{F^{[3]}}+ q w\lin V+2F^{[3]}_0) . \label{F3-row}\EDE
Subtracting (\ref{FGH0-2}) from (\ref{FGH-2}), we get
\BGE r_3\pa_{r_3} \lin S= -{\beta^2} \lin F/w-\beta \lin S-\lin V-\beta r_3 (\lin{ F^{[3]}}+ F^{[3]}_0)/w.\label{S-row}\EDE
Taking the difference of the two formulas we obtain by applying (\ref{neighbor2}) with ${\ell}=1$ and ${\ell}=2$,   we obtain:
\BGE r_3\pa_{r_3}\lin U= -pr_3 \lin V-p r_3 V_0.\label{U-row}\EDE


\begin{Lemma}
\label{lem:weighted-edge-cancellation}
Let $\beta\ge 2/3$. Let $\Lambda^*=0$ if $\beta>2/3$, and $\Lambda^*=\Lambda$ if $\beta=2/3$. Fix $L\in (0,1)$.
On \( \mathcal D^{(3)}_{L*}\), there exist functions $E_V$ and $R_{V,X}$, $X\in \Sigma$, with $|E_V|\lesssim_L 1+\Lambda^*$ and $ |R_{V,X}|
        \lesssim_L
        1 $, $X\in\Sigma$,  such that
\begin{equation}
       r_3 \partial_{r_3}\lin V = {\beta^2(1-\beta)}  \lin F/w+\beta(1-\beta) \lin S +(1-\beta) \lin V +r_3\sum_{X\in \Sigma} R_{V,X} \lin X
        + r_3 E_V.
\label{V-row}
\end{equation}
 \end{Lemma}
\begin{proof}
Recall the $C^{({\ell})}_f$, $f\in\{F_1,F_2,F_3,F,G\}$, and $Q^{({\ell})}$ defined in (\ref{CnbF1})-(\ref{CnbQ}).
Define
\BGE
        B_S :=  \frac12  ( r_1^{-1}C^{(3)}_{F_1}+ r_2^{-1}C^{(3)}_{F_2} )=\frac{ \beta(1-\beta)}{r_3\Delta}\,(1-r_1r_2),\label{BS}\EDE
\BGE
         B_U :=  \frac12 ( r_1^{-1}C^{(3)}_{F_1}- r_2^{-1}C^{(3)}_{F_2}) = \frac{ \beta(1-\beta)}{\Delta}\,(r_1-r_2). \label{BU}\EDE
Since $F_1=\frac{w(S+U)}{2r_1}$ and
        $F_2=\frac{w(S-U)}{2r_2}$,
we have
$ C^{(3)}_{F_1}F_1+C^{(3)}_{F_2}F_2
        =
    w(    B_S S+B_U U)$. We first prove the identity on \((0,1)^2\times(0,L]\) and then extend it to \(D^{(3)}_{L*}\) by continuity.
        By (\ref{G3=2@N}), we get
\BGE     r_3 \pa_{r_3}   \lin V  = r_3\tau ( C^{(3)}_F F/w+C^{(3)}_{F_3} F^{[3]}/w+ B_S S+B_U  U )+r_3C^{(3)}_G  V. \label{paW}\EDE
Define $b_F=r_3\tau C^{(3)}_F/w$, $b_{F^{[3]}}=r_3\tau C^{(3)}_{F_3}/w$, $b_S=r_3\tau B_S$, $b_U=r_3\tau B_U$,  $b_V=  r_3C^{(3)}_G$, $b_X^0 =b_X|_{r_3=0}$ (in the sense of limits) and $R_{V,X}=(b_X-b_X^0)/r_3$ for $X\in \Sigma$, and $E_V=\frac 1{r_3} \sum_{X\in\Sigma}  b_X X_0$.
Then
\(
       r_3 \partial_{r_3} \lin V
        =  \sum_{X\in\Sigma} (b^0_X+r_3R_{V,X}) \lin X+ r_3 E_V\).
Thus it suffices to show $b_F^0=\beta^2(1-\beta)/w$, $b_S^0=\beta(1-\beta)$, $b_V^0=1-\beta$, $b_{F^{[3]}}^0=b_U^0=0$, $ |R_{V,X}|\lesssim_L 1$, and $|E_V|\lesssim_L 1+ \Lambda^*$ on \( \mathcal D^{(3)}_{L*} \). We restrict all subsequent estimates to $\mathcal D^{(3)}_{L*}$.

Since $t_3\ge 1-L \gtrsim_L 1$, we get $\Delta\gtrsim_L t_1^2+t_2^2$ by  (\ref{Delta-tj}).
By (\ref{CnbF3}), (\ref{BU}) and (\ref{x1x2-est1}),  $b_X^0 =0$ and $|R_{V,X}|=|r_3^{-1}b_X|\lesssim_L 1$ for $X\in \{ {F^{[3]}},U\}$. By  (\ref{CnbG}), $b_V-(1-\beta)=-\frac{\beta r_3}{t_3}-\frac{ \Delta+Q^{(3)} }{\Delta}$.  By (\ref{Delta-nb}) and (\ref{CnbQ}),  $ |\Delta+Q^{(3)}|=3r_3(t_1t_2(r_1+r_2)+(t_1-t_2)^2 t_3)\lesssim r_3(t_1^2+t_2^2)$. Thus, $b_V^0=1-\beta$ and $|R_{V,V} |  \lesssim_L \frac 1{t_3}+ \frac{t_1^2+t_2^2}\Delta\lesssim_L 1$. By (\ref{BS}), $b_S=\beta(1-\beta)\frac{\tau^2}\Delta$. Since $\Delta^0:=\Delta|_{r_3=0}=\tau^2$ by (\ref{Delta-nb-boundary}), we get $b_S^0= \beta(1-\beta)$. By (\ref{x1x2-est1}),
$$|R_{V,S}|\lesssim_L \frac{ \tau^2}{r_3}\cdot \frac{|\Delta^0-\Delta|}{\Delta\Delta^0}=\frac{|2 (r_1 t_2^2+r_2t_1^2)-r_3 (t_1-t_2)^2|}\Delta\lesssim \frac{t_1^2+t_2^2 }{\Delta} \lesssim_L 1.
$$
By (\ref{CnbF}) and (\ref{Kl}), $b_F=\beta^2(1-\beta)\tau \cdot \frac{\tau-r_3(t_1+t_2)}{w t_3 \Delta}$. Thus, $$b_F^0=\beta^2(1-\beta) \tau\cdot \frac{\tau}{w\Delta^0}=\frac{\beta^2(1-\beta)}w .$$
$$|R_{V,F}|\lesssim_L \frac{\tau}{r_3w}\Big|\frac{\tau-r_3(t_1+t_2)}{t_3 \Delta}-\frac 1{\tau}\Big|=\frac{ |P |}{ wt_3 \Delta} ,$$
where
$P := -\tau t_1t_2+2t_3(r_2 t_1^2+r_1t_2^2)-r_3t_3 (t_1-t_2)^2=O_L(t_1^2+t_2^2)$. Thus, from $w\ge 1$ and $t_3\Delta \gtrsim_L t_1^2+t_2^2$, we get $|R_{V,F}|\lesssim_L 1$.
From the values of the $b_X^0$'s we get $\sum_{X\in\Sigma} b_X^0 X_0=0$. Thus, $E_V=\sum_{X\in\Sigma} R_{V,X} X_0$.
By Lemma \ref{1.G} and the estimate $|R_{V,X}|\lesssim_L 1$, we get $|E_V|
        \lesssim_L 1+ \Lambda^*$.
\end{proof}

\begin{proof}[Proof of Proposition \ref{prop-non-corner}]
By symmetry, it suffices to consider \({\ell}=3\). We restrict  estimates in this proof to $\mathcal D^{(3)}_{L}$ unless otherwise stated.
First, suppose $\beta>2/3$. We adopt the notation in (\ref{SUVW}) and
define a Pfaff factor $Y=[\begin{array}{ccccc} \lin F  & \lin{F^{[3]}} & \lin S & \lin U & \lin V
\end{array}]^{\text T}$. Let $\Lambda^*$ be as in Lemma \ref{lem:weighted-edge-cancellation}. From (\ref{F-row})-(\ref{V-row}), we get the Pfaff system:
\BGE \pa_{r_3} Y=r_3^{-1}M Y+ RY+E, \label{Pfaff-RE}\EDE
where
$$M=\left[ \begin{array}{ccccc} 0 & 0 & 0 & 0 & 0\\ \beta(1-\beta) & -2\beta & 0 & 0 & q w\\   -{\beta^2}/w  &0& -\beta & 0 & -1 \\ 0 & 0 & 0 & 0 & 0\\   {\beta^2(1-\beta)}/w & 0& \beta(1-\beta)  & 0 & 1-\beta
\end{array}
\right],$$
$R$ is a matrix function satisfying  $\|R\| \lesssim_L 1$, and $E$ is a vector function satisfying $\|E\|\lesssim_L 1+ \Lambda^*$.

We note that $M$ has five  eigenvalues: $0,0,0,-2\beta,1-2\beta$  and five corresponding linearly independent eigenvectors:
$$e_4,\quad e_1+\frac{1-\beta}2 e_2-\frac \beta{w} e_3,\quad \frac{qw}{2\beta} e_2-\frac 1\beta e_3+e_5,\quad e_2,\quad -(1-\beta)qw e_2+ e_3-(1-\beta) e_5.$$
Let $\Phi$ be the matrix formed by the above eigenvectors. From $w\ge 1$ and $|qw|\le 1$, we get $\|\Phi\|\lesssim 1$. We calculate $|\det(\Phi)|=\frac{2\beta-1}\beta\gtrsim 1$. So we also have $\|\Phi^{-1}\|\lesssim 1$. Since all eigenvalues are $\le 0$, we get $\|t^M\|\lesssim 1$ for $t\ge 1$.

Fix $(r_1,r_2)\in [0,1)^2$ and view $M$ as a constant matrix and $Y,R,E$ as functions of $r_3\in [0,L]$.
The variation-of-constants formula for (\ref{Pfaff-RE}) gives, for $0<\eps<r_3\le L$,
$$Y(r_3)=\Big(\frac{r_3}\eps\Big)^M Y(\eps)+\int_\eps^{r_3} \Big(\frac{r_3}s\Big)^M (  R(s) Y(s)+ E(s)) ds.$$
Letting $\eps\to 0^+$, we get $Y(r_3)= \int_0^{r_3} (\frac{r_3}s )^M ( R(s) Y(s)+ E(s)) ds$ since   $Y(\eps)\to 0$  as $\eps\to 0^+$.
From $\|(\frac{r_3}s)^M\|\lesssim 1$, $\|R(s)\|\lesssim_L 1$, and  $\|E(s)\|\lesssim_L 1+ \Lambda^*$, we see that there exists $C>0$ depending only on $\beta,L$ such that  $ \|Y(r_3)\|\le C \int_0^{r_3} (\|Y(s)\|+1+\Lambda^*)ds$ for $r_3\in (0,L]$. By Gronwall  inequality, $\|Y\|\lesssim_L r_3(1+\Lambda^*)$, which together with Lemma \ref{1.G} implies that $|X|\lesssim_L 1+ \Lambda^*$ for $X\in\Sigma$. By (\ref{SUVW}) and the continuity of $F,F_j,G$ at $[0,1)^2\times \{0\}$, the bound of $X$  further implies  that $|F|+|F_3|\lesssim_L 1+\Lambda^*$ and $r_1|F_1|+r_2|F_2|+\tau |G|\lesssim w (1+\Lambda^*) $.

Suppose $\beta>2/3$. Then $\Lambda^*=0$, and so we get the desired bounds of $F$, $F_3$, and $\tau G$.
We now show $|F_1| \lesssim_L w $. We write $\mathcal D^{(3)}_{L}=\Omega_\ge\cup \Omega<$   such that $r_1\ge 1/2$ in $\Omega_\ge$ and $r_1<1/2$ in $\Omega_<$.  In $\Omega_\ge $, we have $|F_1|\lesssim_L r_1^{-1} w\lesssim w$ since $r_1\ge 1/2$. It remains to show that $|F_1|\lesssim_L w$ in $\Omega_<$. In $\Omega_<$, we have $t_1,\tau,w\asymp_L 1$, and so $|G|\lesssim_L 1$. Applying (\ref{rainbow1}) with ${\ell}=1$, we get
\BGE\pa_{r_1} F_1+\frac{(2\beta-2 r_1)}{r_1t_1} F_1=H:=\frac{t_2 t_3 G+\beta(1-\beta) F}{r_1t_1}.\label{H=G}\EDE
By the above estimates, $|H|\lesssim_L r_1^{-1}$ in $\Omega_<$. Solving the above equation using the integration factor
$r_1^{2\beta}(1-r_1)^{2-2\beta}$, we see that, in $\Omega_<$,
$$  F_1(\u r)=\Big(\frac \eps{r_1}\Big)^{2\beta}\Big(\frac{1-\eps}{1-r_1}\Big)^{2-2\beta} F_1(\eps,r_2,r_3)+ \int_\eps^{r_1}\Big(\frac s{r_1}\Big)^{2\beta} \Big(\frac{1-s}{1-r_1}\Big)^{2-2\beta}  H(s,r_2,r_3)ds .$$
Letting $\eps\to 0^+$ in the above equality and using $|F_1(\eps,r_2,r_3)|\lesssim \eps^{-1} w \lesssim \eps^{-1}$ and $2\beta>1  $, we get
\BGE |F_1(\u r)|\le \int_0^{r_1}\Big(\frac s{r_1}\Big)^{2\beta} \Big(\frac{1-s}{1-r_1}\Big)^{2-2\beta}  |H(s,r_2,r_3)|ds\label{FH}\EDE
 $$\lesssim_L \int_0^{r_1}\Big(\frac s{r_1}\Big)^{2\beta} \frac{ds}s\lesssim_L 1\le w.$$
A similar argument shows $|F_2| \lesssim_L w $ on $\mathcal D^{(3)}_{L} $.

Suppose $\beta=2/3$. Then $\Lambda^*=\Lambda$, and so $|F|+|F_3|\lesssim_L 1+\Lambda$,  $\tau|G|\lesssim w(1+\Lambda)$.
To prove $|F|+|F_3|\lesssim_L 1+\Lambda$, consider a two-component  Pfaff vector $\til Y:=[\begin{array}{cc } \lin F  & \lin{F^{[3]}}
\end{array}]^{\text T}$. Since $|V|\lesssim 1+\Lambda$ and $qw(1+\Lambda)\le 1$, by (\ref{F-row}) and (\ref{F3-row}), we get the Pfaff system:
$$  \pa_{r_3} \til Y=r_3^{-1}\til M \til Y+ \til R\til Y+\til E,$$
where $\til M=\Big[\begin{array}{cc} 0 & 0 \\ \beta(1-\beta) & -2\beta
\end{array}\Big]$, $\til R$ is a matrix function satisfying  $\|\til R\| \lesssim_L 1$, and $\til E$ is a vector function satisfying $\|E\|\lesssim_L 1$. Since $\til M$ is a constant matrix and has two distinct nonpositve eigenvalues: $0$ and $-2\beta$, we have $\|t^{\til M}\|\lesssim 1$ for any $t\ge 1$. Then the same argument as in the proof of $\|Y\|\lesssim_L 1+ \Lambda^*$ can be used here to show that $\|\til Y\|\lesssim_L 1$, which then implies that $|F|,|F_3|\lesssim_L 1$.

We now show that $|F_1| \lesssim_L w (1+\Lambda_1)$. Consider three cases. Case 1. $t_1\ge 1/2$, i.e., $r_1\le 1/2$. In this case, $\Lambda_1\le \log 2$, and $\tau\asymp w\asymp 1$. Thus, $|G|\lesssim 1+\Lambda_2\le 1+\log(1/t_2)$. Then $|t_2t_3 G|\le t_2(1+\log(1/t_2))\lesssim 1$. This together with $|F|\lesssim_L 1$ implies that $|H|\lesssim r_1^{-1}$. Applying the integral estimate (\ref{FH}) with the above bound on \(H\), we obtain  $|F_1|\lesssim w\le w(1+\Lambda_1)$. Case 2. $t_1\le 1/2$ and $t_1\le t_2$. In this case $r_1\ge 1/2$ and $\Lambda_2\le \Lambda_1$, and so $|F_1|\lesssim w(1+2\Lambda_1)\lesssim w(1+\Lambda_1)$. Case 3. $t_2\le t_1\le 1/2$. Then $r_1\ge 1/2$ and $\tau\le 2 t_1$. We break the integral in (\ref{FH}) into integrals over $[0,1/2]$ and $[1/2,{r_1}]$. When $s\in[0,1/2]$, by the argument in Case 2, $|H(s,r_2,r_3)|\lesssim_L s^{-1}$. When $s\in[1/2,r_1]$,  $\Lambda_1(s,r_2)=\log(\frac{1-s r_2}{1-s})\le \log(2)$ and $\Lambda_2(s,r_2)=\log\frac{1-s r_2}{t_2}\lesssim 1+ \log\frac{1-s}{t_2}$. Thus, from $|G|\lesssim \tau^{-1} w(1+\Lambda)$, $|F|\lesssim \tau^{-1} w(1+\Lambda)$,  and $\tau(s,r_2)\asymp 1-s$, we get
 $|G(s,r_2,r_3)|\lesssim_L  (1-s)^{\gamma-2}(1+\log\frac{1-s}{t_2})$.  Since  $|F|\lesssim_L 1$, by (\ref{H=G}), we get with $u:=1-s$,
$$|H(s,r_2,r_3)|\lesssim_L u^{-1} (t_2|G| +|F|)\lesssim u^{-1}+ u^{\gamma-3}  t_2  (1+\log(u/t_2))   .  $$
Thus, by (\ref{FH}), $$|F_1(\u r)|\lesssim \int_0^{1/2}\frac{s^{2\beta-1}}{t_1^{2-2\beta}} ds +\int^{1/2}_{t_1} \Big(\frac u{t_1}\Big)^{2-2\beta}\Big[u^{-1}+u^{\gamma-3} t_2\Big(1+\log\Big(\frac u{t_2}\Big)\Big)\Big] du $$
$$\lesssim t_1^{2\beta-2}+t_1^{2\beta-2}\Big[\int_{t_1}^{1/2} u^{1-2\beta} du+\int_{t_1}^{1/2}\frac{ t_2}{ u^{2}} \Big(1+\log\Big(\frac u{t_2}\Big)\Big) du\Big]$$
$$\lesssim  t_1^{2\beta-2}\Big[1+\int_{2t_2}^{t_2/t_1} (1+\log(1/v))dv\Big]\lesssim t_1^{\gamma-1}\lesssim w\le w(1+\Lambda).$$
Thus, we have $|F_1| \lesssim_L w (1+\Lambda_1)$ in all cases. Symmetrically, $|F_2| \lesssim_L w (1+\Lambda_2)$.

Since $|F|+|F_3|\lesssim_L 1$,  $F$ and $t_3 F_3$ are bounded on $\mathcal D^{(3)}_{L}$. For $j\in [2]$, since $t_j w(1+\Lambda_j)\le t_j^\gamma (1+\log(1/t_j))\lesssim 1$ and $|F_j|\lesssim_L w(1+\Lambda_j)$, we see that $t_j F_j$ is bounded on $\mathcal D^{(3)}_{L}$.  Since $|G|\lesssim_L \tau^{-1} w(1+\Lambda)\lesssim (t_1+t_2)^{\gamma-2}(1+\log(1/t_1)+\log(1/t_2))$, from $t_1t_2 t_3(t_1+t_2)^{\gamma-2}\le  t_1^{\gamma/2} t_2^{\gamma/2}$, and $t_j^{\gamma/2} \log(1/t_j)\lesssim 1$ we see that $t_1t_2t_3 G$ is bounded on $\mathcal D^{(3)}_{L}$. By symmetry, $F$, $t_j F_j$, $j\in [3]$, and $t_1t_2 t_3 G$ are bounded on $\mathcal D^{({\ell})}_{L}$ for each ${\ell}\in [3]$, and so is bounded on their union.

Finally, we prove the continuous extension of $F$. Since $[0,1]^3\sem \{(1,1,1)\}=\bigcup_{{\ell}\in [3]}\bigcup_{L\in (0,1)} \lin{D^{({\ell})}_L}$, it suffices to show that $F$ extends continuously to $ \lin{D^{({\ell})}_L}$ for each ${\ell}\in [3]$ and $L\in (0,1)$. By the symmetry, we may assume ${\ell}=3$.  Let $L\in (0,1)$. On $\mathcal D^{(3)}_{L}$, $|F_3|\lesssim_L 1$ and for $j\in [2]$,
$$|F_j|\lesssim_L w(1+\Lambda_j)\le \tau^{\gamma-1}(1+\log(1/(1-r_j)))\le (1-r_j)^{\gamma-1}(1+\log(1/(1-r_j))).$$ Since $\gamma>0$, $\int_0^1 (1-r )^{\gamma-1}(1+\log(1/(1-r )))dr<\infty$. Thus, on $\mathcal D^{(3)}_{L}$, $F$ is uniformly continuous in each variable, uniformly with respect to all other variables. So $F$ is jointly uniformly continuous on $\mathcal D^{(3)}_{L}$, which implies that it extends continuously to  $ \lin{D^{(3)}_L}$.
\end{proof}

\subsection{Side-face equations and regular normal estimates}  \label{Section-6.2}
Recall the $f_\beta$ in (\ref{2F1>0}), which extends continuously to a positive function on $[0,1]$ since $\beta>1/2$. We have $f_\beta'(r)=\frac{1-\beta}2 {}_2F_1(2-\beta,1+\beta;1+2\beta;r)$ by the standard differentiation formula for ${}_2F_1$. If $\beta\ge 1$, $f_\beta'$ is bounded on $[0,1)$ by Gauss's identity. If $\beta\in (1/2,1)$, by the connection formula of ${}_2F_1$ at $1$, we get $|f_\beta'(r)|\lesssim (1-r)^{2\beta-2})$  on $[0,1)$. So we conclude that
\BGE  |f_\beta'(r)|\lesssim (1-r)^{\gamma-1}. \label{f'-bd}\EDE
When $r\in [1/2,1)$, using \(\mathcal L_rf_\beta=0\) and (\ref{f'-bd}), we get
 \BGE  |f_\beta''(r)|\lesssim (1-r)^{\gamma-2}. \label{f''-bd}\EDE
 This formula holds for all $r\in [0,1)$ since on \([0,1/2]\), \(f_\beta''\) is bounded by analyticity.

\begin{Proposition}  (i) Let $\beta> 2/3$. Then $r_1F_1,r_2 F_2,F_3,G$  extend continuously to $[0,1)^2\times [0,1]$. Recall that, by Proposition~\ref{prop-non-corner}, \(F\) already extends continuously to
\([0,1)^2\times[0,1]\).  After the continuous extension, we define $\F,\G,\HF$ on $[0,1)^2$ such that
   \BGE \F =F(\cdot,\cdot,1),\quad \G =G(\cdot,\cdot,1) ,\quad \HF =F_3(\cdot,\cdot,1).\label{FGH}\EDE
 Then  for $j\in [2]$, the following hold in $(0,1)^2$:
 \BGE \L_{r_j}\F=0 ;\label{F-bdry-hyper}\EDE
    \BGE \F(r_1,r_2)=\frac{\Gamma(2\beta)\Gamma(3\beta-1)}{\Gamma(\beta)\Gamma(4\beta-1)}   f_\beta(r_1) f_\beta(r_2);\label{F1&N*}\EDE
    \BGE  |
        \partial_{r_1}^{n_1}\partial_{r_2}^{n_2}\F
         | \lesssim t_1^{(\gamma-n_1)\wedge 0} t_2^{(\gamma-n_2)\wedge 0} \quad \text{for any }n_1,n_2\in\N_0\text{ with }n_1\vee n_2\le 2; \label{paF}\EDE
    \begin{equation}
          |
        \partial_{r_1}^{n_1}\partial_{r_2}^{n_2}\HF
         |
        \lesssim t_1^{\gamma-1-n_1}t_2^{\gamma-1-n_2}  \quad \text{for any }n_1,n_2\in\N_0\text{ with }n_1+n_2\le 2;
\label{eq:HF-side-jet-bound}
\end{equation}
\BGE 2( \beta-1)\HF -\beta(1-\beta) \F =t_1t_2\G.\label{FGH1}\EDE

(ii) If $\beta\in (1/3,2/3]$, then $F$, $r_1F_1$, $r_2F_2 $, $r_1r_2 F_{12}$ and $\beta F_3+t_1t_2 G$  extend continuously to $[0,1)^2\times [0,1]$, and (\ref{F-bdry-hyper})   still holds. If $\beta\in [1/2,2/3]$, we also have (\ref{F1&N*}).
\label{Extension1}\end{Proposition}
\begin{proof}
Define a Pfaff vector:
\BGE Y=[\begin{array}{ccccc} r_1 F_1 & r_2 F_2 & F_3 & F & t_1t_2 G
\end{array}]^{\text T}.\label{Pfaff1}\EDE  By (\ref{rainbow1}), (\ref {neighbor2}) and (\ref{G3=2@N}), we have the Pfaff system: $\pa_3 Y=MY$, where
\BGE M=\left[\begin{array}{ccccc}
  -\dfrac{\beta}{2r_3} &-\dfrac{ \beta}{2r_3}&-\dfrac\beta 2 & -\dfrac{\beta^2}{2r_3} & \dfrac{K _2}{2t_1t_2r_3}\\
  -\dfrac{\beta}{2r_3} & -\dfrac{\beta}{2r_3} &-\dfrac\beta 2 & -\dfrac{\beta^2}{2r_3} &\dfrac{K _1}{2t_1t_2r_3} \\
  0 & 0 & \dfrac{2(r_3-\beta)}{r_3t_3} & \dfrac{\beta(1-\beta)}{r_3t_3} & \dfrac1{r_3t_3}\\
  0 & 0 & 1 & 0 & 0\\
  t_1t_2 r_1^{-1} C^{(3)} _{F_1} & t_1t_2 r_2^{-1} C^{(3)} _{F_2} & t_1t_2 C ^{(3)}_{F_3} & t_1t_2 C^{(3)} _{F} &  C^{(3)}_{G}
\end{array}
\right],\label{Paff1-M}\EDE
in which $K_1,K_2$ are as in (\ref{Kl}), and $C^{(3)}_f$, $f\in\{F_1,F_2,F_3,F,G\}$, are as in (\ref{CnbF1})-(\ref{CnbG}).

Fix $L,J\in (0,1)$. We restrict all subsequent estimates to $R_{L,J}:= {[0,L]^2\times [J,1)}$. For $j\in[2]$, we have $r_3\in [J,1)$ and $t_j\in [1-L,1]$,  and so $r_3 ,t_j\asymp_{L,J} 1$. By (\ref{Delta-nb-boundary}) and (\ref{Delta-tj}), $\Delta>0$ on the compact set $\lin{R_{L,J}}$, and so $\Delta\asymp_{L,J} 1$. By (\ref{K-tj}), $|K_j^+|\le 3$ on $[0,1]^3$ for $j\in [3]$.

We define
\BGE u:=[\begin{array}{ccccc} 0 & 0 & 1 & 0 & -\beta
\end{array}]^{\text T},\quad v:=[\begin{array}{ccccc}0 & 0 &2(1-\beta)&\beta(1-\beta)&1
\end{array}]^{\text T}\label{uv}\EDE
and claim that
\BGE  N(\u r):=M(\u r)-\frac{u v^{\text T}}{t_3} \label{uvM}\EDE
is uniformly bounded.
The boundedness of the entries $N_{m,n}$ for $m\in[4]$ and $n\in [5]$ follows easily from the bounds on $r_3^{-1}$ and $t_j^{-1}$, $j\in [2]$. By (\ref{CnbF1}), the boundedness of $N_{5,n}$, $n\in[2]$, follows from the bounds on $K_{3-n}^+$, $r_3^{-1}$ and $\Delta^{-1}$. Let $C_{5,3}=-2\beta(1-\beta)$. By (\ref{CnbF3}) and (\ref{Delta-tj}),
$$N_{5,3}= t_1t_2 C ^{(3)}_{F_3}-\frac{C_{5,3}}{t_3}=C_{5,3}\cdot \frac{r_3t_1^2t_2^2 - \Delta}{ t_3\Delta}=  \frac{P_{5,3}(t_1,t_2,t_3)}{ \Delta}$$
for some polynomial $P_{5,3}$, where the last equality holds because by (\ref{Delta-tj}), $\Delta-t_1^2t_2^2$ is divisible by $t_3$. Let $C_{5,4}=-\beta^2(1-\beta)$. By (\ref{CnbF}), (\ref{K-tj}) and (\ref{Delta-tj}),
$$N_{5,4}= t_1t_2 C ^{(3)}_{F }-\frac{C_{5,4}}{t_3}=-C_{5,4}\cdot\frac{t_1t_2K_3^++(1-t_3)\Delta}{r_3 t_3\Delta}= \frac{P_{5,4}(t_1,t_2,t_3)}{r_3 \Delta}$$
for some polynomial $P_{5,4}$, where the last step holds because by (\ref{K-tj}) and (\ref{Delta-tj}), $\Delta+t_1 t_2K_3^+$ is divisible by $t_3$. So the boundedness of $N_{5,3}$ and $N_{5,4}$ follows from the bounds on  $r_3^{-1}$ and $\Delta^{-1}$.  By (\ref{CnbG}),  $N_{5,5}=-\frac\beta{r_3}-\frac{Q^{(3)}}{r_3\Delta}$, whose boundedness  follows from the bounds on $r_3^{-1}$, $\Delta^{-1}$ and $Q^{(3)}$.

Let $a=3\beta-2$. Define two scalar functions of  $\u r$ by \BGE y=v^{\text T}Y=2(1-\beta)F_3+\beta(1-\beta)F+t_1t_2 G\label{y=vY}\EDE
and $y_N=v^{\text T}NY$.   Fix $(r_1,r_2)\in (0,L]^2$ and view $Y,y,y_N$ as functions of $r_3\in [J,1)$.  From $\pa_3Y=NY+\frac{u v^{\text T} Y}{t_3}$ and $v^{\text T} u=-a$, we get
\BGE \pa_{r_3} y=\frac{-a }{t_3}\cdot y+y_N.  \label{dy}\EDE

Suppose $\beta>2/3$; then $a>0$.
 By Proposition \ref{prop-non-corner}, $ |F|+|F_1|+|F_2|+|F_3|+|G|\lesssim_{L,J}  1$, 
which implies that $\|Y\|\lesssim_{L,J} 1$, and so $\|y\|+\|y_N\|\lesssim_{L,J} 1$. Solving (\ref{dy}) using the integration factor $(1-r_3)^{-a}=t_3^{-a}$, we see that, for $r_3\in [J,1)$,
$$ |y(r_3)|\lesssim_{L,J}  \Big(\frac{t_3}{1-J}\Big)^a+ \int_J^{r_3}  \Big(\frac{t_3}{1-s}\Big)^a ds\lesssim_J \delta(t_3):=\begin{cases} t_3^a, & \text{if }a<1,\\ t_3\log(1/t_3),&  \text{if }a=1,\\ t_3, & \text{if }a>1.
\end{cases}$$
Then we have $\delta(t_3)\to 0$ as $t_3\to 0$, and $\int_J^1 \frac{\delta(1-s)}{1-s}ds<\infty$.
By $|y|\lesssim \delta$ and (\ref{y=vY}),
\BGE y= 2(1-\beta)F_3+\beta(1-\beta)F+t_1t_2 G\to 0,\quad \text{as }r_3\to 1.\label{y->0}\EDE
From $\pa_3 Y= NY+\frac{u y}{t_3}$ and $\|Y\|\lesssim_{L,J}1$ we see that, for any  $r_3<r_3'\in [J,1)$,
$$\|Y(r_3')-Y(r_3)\|\lesssim_{L,J}  \int_{r_3}^{r_3'}\Big(1+ \frac{\delta(1-s)}{1-s}\Big) ds .$$
Since $\int_J^1 \frac{\delta(1-s)}{1-s}ds<\infty$,  $Y(r_1,r_2,r_3)$ converges as $r_3\to 1^-$ uniformly in $(r_1,r_2)\in (0,L]^2$. Since $Y$ is continuous on $[0,1)^3$, the uniform convergence holds for $(r_1,r_2)\in [0,L]^2$. Thus, $Y$ extends continuously to $\lin{R_{L,J}}=[0,L]^2\times [J,1]$. Since $F$, $r_1F_1$, $r_2F_2$, $F_3$, and $t_1t_2G$  are components of $Y$, they all extend continuously to $\lin{R_{L,J}}$.  Since $t_1t_2>0$ on $\lin{R_{L,J}}$, $G$ also extends continuously to $\lin{R_{L,J}}$.
Since $F$, $r_1F_1$, $r_2F_2$, $F_3$ and $G$ are continuous on $[0,1)^3$, and $L,J\in (0,1)$ are arbitrary, these functions extend continuously to $[0,1)^2\times [0,1]$. From (\ref{y->0}) we get (\ref{FGH1}).

From (\ref{neighbor2}) we see that $r_1r_2 F_{12}$, $r_1r_3 F_{13}$, and $r_2r_3 F_{23}$ also extend continuously to $[0,1)^2\times [0,1]$. It follows that $F_1$, $F_2$,  $F_{12}$, $F_{13}$, $F_{23}$  extend continuously to $(0,1)^2\times [0,1]$. Applying (\ref{rainbow1}) with ${\ell}\in[2]$, we see that $F_{11},F_{22}$ extend continuously to $(0,1)^2\times [0,1]$. It follows from the fundamental theorem of calculus that, for each $j,k\in[2]$, the continuous extensions of $F_j$ and $F_{jk}$ to $(0,1)^2\times\{1\}$ coincide with $\pa_{r_j}\F$ and $\pa_{r_j}\pa_{r_k}\F$, and the continuous extensions of $F_{j3}$ to $(0,1)^2\times\{1\}$ coincide with $\pa_{r_j}\HF$. Setting ${\ell}=3$ in (\ref{neighbor2}) and sending $r_3\to 1$, we get  \BGE 2r_1r_2 \pa_{r_1}\pa_{r_2} \F +\beta r_1 \pa_{r_1}\F+\beta r_2 \pa_{r_2}\F+\beta \HF+\beta^2 \F= -t_1 t_2 \G.\label{FGH2}\EDE
Summing  (\ref{FGH1}) and (\ref{FGH2}), we get
 \BGE (2r_1r_2\pa_{r_1}\pa_{r_2}    +\beta r_1 \pa_{r_1}  +\beta r_2 \pa_{r_2} +\beta(2\beta-1))\F+(3\beta-2)\HF=0.\label{FGH3}\EDE

Since $t_1t_2 G$ is the $5$-th component of $Y$, from $\pa_{r_3} Y=MY$ and $M=N+\frac{ u v^{\text T}}{t_3}$, we get
$$t_1t_2t_3 G_3=t_3(NY)_5+ ( u v^{\text T}Y)_5=t_3 (NY)_5-\beta y\to 0,\quad \text{as }r_3\to 1.$$
 Thus, $t_3 G_3\to 0$ as $r_3\to 1^-$ for every fixed $(r_1,r_2)\in [0,1)^2$. Setting ${\ell}\in [2]$ in (\ref{rainbow1}) and then differentiating w.r.t.\ $r_3$, we get
$\L_{r_{\ell}} F_3=t_3t_{3-{\ell}} G_3-t_{3-{\ell}} G$. Sending $r_3\to 1^-$ and using the convergence of $G$ and $t_3 G_3$ as $r_3\to 1^-$, we see that $F_{\ell\ell 3}$ also extends continuously to $(0,1)^2\times [0,1]$, whose values on $(0,1)^2\times \{1\}$ agree with those of $\HF_{\ell\ell}$ by the fundamental theorem of calculus. Then we get
\BGE \L_{r_j} \HF= - t_{3-j} \G,\quad j\in [2],\quad \text{in }(0,1)^2. \label{FGH5}\EDE
For $j\in [2]$, (\ref{FGH1})+$t_j\times$ (\ref{FGH5}) yields
\BGE r_j t_j^2  \HF_{jj}+(2\beta-2r_j)t_j \HF_j- (\beta t_j+2)(1-\beta) \HF-\beta(1-\beta)\F=0  .\label{FGH6}\EDE

Applying (\ref{rainbow1}) with ${\ell}\in [2]$  and letting $r_3\to 1^-$,    by the continuity of $F_{jj},F_j,G$ on $(0,1)^2\times [0,1]$ and the fact that $t_3G_3\to 0$, we get (\ref{F-bdry-hyper}).
Fix $r_2\in (0,1)$. Then $\F(\cdot,r_2)$ satisfies the Gauss hypergeometric equation (\ref{F-bdry-hyper}) on $(0,1)$. Hence it is a linear combination of two linearly independent local solutions near $r_1=0$. One of them is  $f_\beta(r_1)$, which is continuous at $0$. Since $\beta>1/2$, the other solution has Frobenius exponent $1-2\beta<0$, and therefore cannot be continuous at $0$. Since $\F$ is continuous on $[0,1)^2$, there is a function $\til y $ on $[0,1)$ such that $\F(r_1,r_2)=f_\beta(r_1)\til y(r_2)$. Swapping the roles of $r_1$ and $r_2$, we see that $\til y=Cf_\beta$ for some $C\in\R$. By (\ref{F1@N}) and the symmetry of $F$, we get $C=F(0,0,1)=F(1,0,0)=\frac{\Gamma(2\beta)\Gamma(3\beta-1)}{\Gamma(\beta)\Gamma(4\beta-1)}$. So we get  (\ref{F1&N*}). Now (\ref{paF}) follows from Gauss's identity,  (\ref{f'-bd}), (\ref{f''-bd}) and (\ref{eq:HF-side-jet-bound}) in the case $n_1\vee n_2\le 1$ follows from (\ref{FGH3}), (\ref{F1&N*}) and (\ref{paF}). When $(n_1,n_2)=(2,0)$ or $(0,2)$, (\ref{eq:HF-side-jet-bound}) follows from (\ref{FGH6}) and the estimates proved above.

We now turn to Part (ii). Assume $\beta\in (1/3,2/3]$. Then $a\in (-1,0]$.   Since  $v^{\text T}u=-a$, we get
\BGE \|t_3^{\pm uv^{\text T}}\|=\begin{cases}
  \|I_5-\frac{t_3^{\mp a}-1}{ a} uv^{\text T}\| {\lesssim} t_3^{-\max\{0,\pm a\}},&\text{if }a\ne 0;\\
 \| I_5\pm \log(t_3) u v^{\text T}\|\lesssim 1+\log(1/t_3),&\text{if }a=0.
\end{cases}\label{Epm}\EDE
Let $\til Y=t_3^{uv^{\text T}} Y$ and $\til N=t_3^{uv^{\text T}} N t_3^{ {-uv^{\text T}}}$.
Since $\|N\|\lesssim_{L,J} 1$, we have
$$\|\til N\|\lesssim_{L,J} h(r_3):=\begin{cases} (1-r_3)^{-|a|},&\text{if }a\ne 0\\ (1+\log(1/(1-r_3)))^2,&\text{if }a= 0. \end{cases}$$ Since $|a|<1$, we have $\int_ J^1 h(s) ds<\infty$.
Solving   $\pa_{r_3} Y-\frac{u v^{\text T}}{1-r_3} Y=NY$ using the integration factor $(1-r_3)^{uv^{\text T}}$, we get
\[\pa_{r_3} \til Y= \til N\til Y.\]
 {This renormalized system will be used again in Section~6.4, where
$r_1,r_2$ are allowed to approach $1$ and quantitative control of $N$ near the corner is needed.}

 {It follows that}
$$ \|\til Y(r_3)\|\le  \|\til Y(J)\|+\int_{J}^{r_3} \|\til N(s)\| \|\til Y(s)\| ds,\quad \text{if }J\le r_3 <1.$$
By Gronwall's inequality, for $r_3\in [J,1)$,
$$\|\til Y(r_3)\|\le \|\til Y(J)\|\exp\Big(\int_J^{r_3} \|\til N(s)\|ds\Big)\lesssim_{L,J}\|\til Y(J)\| \exp\Big(\int_J^{1}h(s) ds\Big)\lesssim_{L,J} 1 .$$ Thus, if $r_3\le r_3'\in [J,1)$ and $r_3,r_3'\to 1$,
$$ \|\til Y(r_3')-\til Y(r_3)\|\le \int_{r_3}^{r_3'} \|\til N(s)\|\| \til Y(s)\| ds\lesssim_{L,J} \int_{r_3}^{r_3'} h(s) ds\to 0,$$
Hence, $\til Y(r_1,r_2,r_3)$ converges as $r_3\to 1^-$, uniformly in $(r_1,r_2)\in [0,L]^2$, which implies that $\til Y$ extends continuously to $\lin{R_{L,J}}$.

Define vectors $w_1:=e_1$, $w_2:=e_2$, $w_3:=e_4$, and $w_4:=\beta e_3+e_5$. Then, for $s\in [4]$,  $w_s{\text{T}} u=0$, and so $w_s^{\text{T}} \til Y=w_s^{\text{T}}(1-r_3)^{uv^{\text T}} Y=w_s^{\text{T}} Y$, which implies that $w_1^{\text{T}} Y=r_1 F_1$, $w_2^{\text{T}} Y=r_2 F_2$, $w_3^{\text{T}} Y=F$, and $w_4^{\text{T}} Y=\beta F_3+t_1t_2\G$  extends continuously to $\lin{R_{L,J}}$.  Since these functions are continuous on $[0,1)^3$, and $L,J\in (0,1)$ are arbitrary, they extend continuously to $[0,1)^2\times [0,1]$. 

By (\ref{Epm}) and the fact that $|a|<1$, $(1-r_3)  \| (1-r_3)^{-uv^{\text T}}\|\to 0$ as $r_3\to 1^-$. Since $Y =(1-r_3)^{-uv^{\text T}}  \til Y $,  as $r_3\to 1$, $t_3 Y\to 0$. Since  $t_1t_2 G$ is a component of $Y$ and $t_1t_2>0$ on $[0,1)^2$, we conclude that  $t_3 G\to 0$  as $r_3\to 1$ locally uniformly on $[0,1)^2$. Applying (\ref{neighbor2}) to ${\ell}=3$ and using (\ref{K-tj}), we get
 $$-2r_1r_2 F_{12}=\beta r_1 F_1+\beta r_2 F_2+\beta^2 F+r_3(\beta F_3+t_1t_2 G)+ (2-(t_1+t_2)+t_1 t_2) t_3 G.$$
 Since each term on the RHS extends continuously to $[0,1)^2\times [0,1]$, so does $r_1r_2 F_{12}$.   Applying (\ref{rainbow1}) with ${\ell}\in[2]$ and letting $r_3\to 1$, we see that $F_{11}$ and $F_{22}$  extend continuously to $(0,1)^2\times [0,1]$, and (\ref{F-bdry-hyper}) holds. When $\beta\in [1/2,2/3]$, any solution of (\ref{F-bdry-hyper}) on $(0,1)$ other than $f_\beta$ is not continuous at $0$. For \(\beta>1/2\), the second Frobenius exponent is \(1-2\beta<0\);
for \(\beta=1/2\), the second local solution is logarithmic. So the same argument as  that used in the case $\beta>2/3$ again shows $\F$ satisfies (\ref{F1&N*}).
\end{proof}

\begin{Remark} Assume $\beta> 1/2$. Since $\F$ is continuous on $[0,1)^2$, (\ref{F1&N*}) actually holds on $[0,1)^2$. Since $f_\beta$ is continuous and positive on $[0,1]$,  $\F$ extends to a positive continuous function on $[0,1]^2$. 
By (\ref{paF}),
\BGE |\F(r_1,r_2)-\F(1,1)|\lesssim t_1^\gamma+t_2^\gamma. \label{F-11}\EDE
\end{Remark}

\begin{Lemma}
Let $\beta=2/3$. Let $y$ be as in (\ref{y=vY}), and $\A=\A(r_1,r_2)=\lim_{r_3\to 1^-} y(r_1,r_2,r_3)$. Then $\A=-\frac 29 C_{\beta} t_1^{-2/3} t_2^{-2/3}$, where $C_\beta:=\frac{\Gamma(2\beta)\Gamma(3\beta-1)}{\Gamma(\beta)\Gamma(4\beta-1)}$. Let $L(x)=\log(1/x)$. Then for any $J\in (0,1/2]$, on the slice $\Omega_J:=\{\u r\in [0,1)^3: 0<t_3\le t_2\le t_1=J\}$, we have
\BGE |y -\A |\lesssim_J t_2^{-\frac 53} t_3(1+L(t_3)) ,\quad  |F_3-L(t_3)\A|\lesssim_J t_2^{-\frac 23}(1+L(t_2)).\label{y-A}\EDE
\label{critial-est}
\end{Lemma}
\begin{proof}
The existence of $\A$ follows from Proposition \ref{Extension1}.
Applying (\ref{neighbor2}) to ${\ell}=3$ and letting $r_3\to 1$, we get
$$2r_1r_2 \F_{12}+\beta r_1 \F_1+\beta r_2\F_2+(\beta^2/2) \F+\mathcal A=0.$$
Thus, by (\ref{F1&N*}),
$$\mathcal A=-2C_\beta (r_1 f_{\beta}'(r_1)+(\beta/2) f_{\beta}(r_1))(r_2 f_{\beta}'(r_2)+(\beta/2) f_{\beta}(r_2)).$$
Since $\beta=2/3$, $\beta/2=1-\beta=2\beta-1$, and so
$$r f_{\beta}'(r )+(\beta/2) f_{\beta}(r )=\sum_{n=0}^\infty \frac{(1-\beta)_n(\beta)_n}{(1)_n(2\beta)_n} (n+\beta/2)r^n  =\frac\beta 2 \sum_{n=0}^\infty \frac{ (\beta)_n}{(1)_n }   r^n =\frac\beta 2 (1-r)^{-\beta}.$$
Therefore, $\mathcal A=-\frac 29 C_{\beta}  t_1^{-2/3} t_2^{-2/3}$.

By (\ref{dy}) and the fact that $a=3\beta-2=0$, $\pa_{r_3} y=y_N=v^{\text T}NY$, where $v$, $N$ and $Y$ are respectively as in (\ref{uv}), (\ref{uvM}), and (\ref{Pfaff1}). By (\ref{uvM}) and that $v^{\text T} u=-a=0$, $v^{\text T}N=v^{\text T}M$, where $M$ is as in (\ref{Paff1-M}). Thus, $\pa_{r_3} y =v^{\text T}MY$. We restrict all estimates in this proof to $\Omega_J$. For example, since $t_1=J$, we have $|\A|\lesssim_J t_2^{-\frac 23}$.

Applying Proposition \ref{prop-non-corner} to ${\ell}=1$ and $L=1-J$, we get
\BGE|F|+|F_1|\lesssim_J 1,\quad |F_2|\lesssim_J t_2^{-\frac 23},\quad |F_3|+t_2|G|\lesssim_J t_2^{-\frac 23}\Big(1+\log \frac{t_2}{t_3}\Big),\label{LJFG}\EDE
  where we used $\gamma=(2\beta-1)\wedge 1=\frac 13$, $\tau_{23}\asymp t_2+t_3\asymp t_2$, $\Lambda_{23}\lesssim 1$, and $\Lambda_{32}\lesssim  (1+\log \frac{t_2}{t_3} )$. Thus, $\|Y\|\lesssim t_2^{-\frac 23} (1+\log \frac{t_2}{t_3} )$. Using (\ref{Paff1-M}), (\ref{uv}), (\ref{CnbF1})-(\ref{Q-tj}),  (\ref{Delta-tj}) and (\ref{K-tj}), we get  $\|v^{\text T} M\|\lesssim_J t_2^{-1}$. For example, the third and fifth entries are given by
  $(v^{\text T} M)_3=\frac 29 -\frac 89\cdot\frac 1{ r_3} +\frac 49\cdot\frac{t_1^2 t_2^2}\Delta +\frac 49 \cdot \frac{\Delta-t_1^2t_2^2}{t_3\Delta}$, and $(v^{\text T} M)_5=-\frac{Q^{(3)}}{r_3\Delta}$.
  Then we use $r_3\asymp 1$, $\Delta\gtrsim t_2^2$ and $|\Delta-t_1^2t_2^2|\lesssim_J t_2t_3$ (by (\ref{Delta-tj}))  to get $(v^{\text T} M)_3=O(t_2^{-1})$; and use $r_3\asymp 1$, $\Delta\gtrsim t_2^2$, and $|Q^{(3)}|\lesssim_J t_2$ (by (\ref{Q-tj})) to get $(v^{\text T} M)_5=O(t_2^{-1})$. When estimating $(v^{\text T} M)_j$ for $j\in [2]$, we also use $|K^+_{3-j}|\lesssim_J t_2$ (by (\ref{K-tj})); and when estimating $(v^{\text T} M)_4$, we also use $\Delta+t_1t_2 K_3^+=O(t_2t_3)$ (by (\ref{Delta-tj}) and (\ref{K-tj})).

  Since $\pa_{r_3} y =v^{\text T}MY$, we then get $|\pa_{r_3} y|\lesssim_J t_2^{-\frac 53}  (1+\log \frac{t_2}{t_3} )$, which implies the first inequality in (\ref{y-A}) because
$$ |y(\u r)-\A|\le \int_{r_3}^1 (1-r_2)^{-\frac 53}  (1+\log \frac{1-r_2}{1-s} )ds=t_2^{-\frac 53} \int_0^{t_3} (1+\log \frac{t_2}u ) du$$ $$=t_2^{-\frac 53} t_2 \int_0^{\frac{t_3}{t_2}} (1+\log \frac 1v)dv=t_2^{-\frac 53} t_3(2+\log \frac {t_2}{t_3}) \lesssim  t_2^{-\frac 53} t_3(1+L(t_3)) .$$ 
Applying (\ref{rainbow1}) to ${\ell}=3$, we get $r_3 t_3 F_{33}+2 t_3 F_3=y$, which implies $\pa_{r_3}(r_3^2 F_3)=\frac{r_3}{t_3} y$. Thus,
  $$r_3^2 F_3(\u r)-r_2^2 F_3(r_1,r_2,r_2)=\int_{r_2}^{r_3} \frac{s}{1-s}\cdot y(r_1,r_2,s) ds.$$
  Since $\int_{r_2}^{r_3} \frac s{1-s} ds= L(t_3)-L(t_2)+t_3-t_2$, by the above   formula,
  $$ r_3^2 F_3(\u r)- L(t_3)\A= r_2^2 F_3(r_1,r_2,r_2)-(L(t_2)+t_2-t_3)\A+\int_{r_2}^{r_3} \frac s{1-s} ( {y(r_1,r_2,s)-\A}) ds,$$
  which together with (\ref{LJFG}), $|\A|\lesssim_J t_2^{-\frac 23}$, and the first inequality in (\ref{y-A}) implies that
  $$|F_3- L(t_3)\A|\lesssim_J t_3|F_3|+r_2^2|F_3(r_1,r_2,r_2)|+(1+L(t_2))|\A|+t_2^{-\frac 53}\int_{r_2}^{r_3} (1+L(1-s))ds$$
  $$\lesssim_J t_2^{-\frac 23} t_3(1+L(t_3))+t_2^{-\frac 23}+t_2^{-\frac 23}(1+L(t_2))+t_2^{-\frac 53}\cdot t_2 (1+L(t_2))\lesssim_J  t_2^{-\frac 23}(1+L(t_2)).$$
This proves the second inequality in (\ref{y-A}).
\end{proof}

\subsection{Corner propagation and final stitching} \label{Section-6.3}
We now complete the analysis near the corner \((1,1,1)\).  The estimates from the previous subsection provide the necessary boundary input on the side faces.  The remaining task is to propagate these estimates into the interior corner region.  We first work in the ordered region \(t_3\le t_2\le t_1\le 1/2\), where the relative size of \(t_3\) compared with \(t_1t_2\) separates the analysis into two regimes.  After obtaining the required bounds in this ordered region, symmetry will give the corresponding estimates in the other regions and will complete the proof of continuity at the corner.

Define $\Omega=\{\u r\in [0,1)^3:t_3\le t_2\le t_1\le 1/2\}$. Define $\chi$ on $[0,1)^3$ by $\chi(\u r)=t_1t_2/t_3$. Define $\mathcal O=\{\u r\in \Omega:\chi(\u r)\ge 1/2\}$. Recall the $\alpha,\gamma$ from (\ref{gamma}).

\begin{Proposition}
\label{lem:direct-inner-core}
If $\beta\ge 2/3$, then on $\mathcal O$,
$$|F(\u r)-\F(1,1)|+\sum_{j=1}^3 t_j |F_j(\u r)| + t_1t_2t_3| G(\u r)| \lesssim  t_1^\gamma.$$
 \end{Proposition}
\begin{proof}
First, suppose $\beta>2/3$.
Let $\F,\G,\HF$ be as in (\ref{FGH}), and $y$ be as in (\ref{y=vY}). We regard \(\mathcal F,\mathcal G,\mathcal F^{[3]}\) as functions of
\((r_1,r_2,r_3)\) which are independent of \(r_3\) and use subscripts to denote partial derivatives of them.
Define $R:=F-\F+t_3\HF$,
\BGE
       X_j:=t_j  R_j,\quad j\in [2],\quad   P:=t_1t_2 R_3,\quad Z:=t_1t_2y.\label{RXQZ}\EDE
Here the subscript in $R_j$ means the partial derivative  with respect to $r_j$, $j\in [3]$, but the subscripts in  $X_1$ and $X_2$ do not have this meaning. We will use
\BGE F=R+\F-t_3\HF,\quad F_j=R_j+\F_j-t_3\HF_j,\quad j\in [2],\quad F_3=R_3+\HF.\label{FR}\EDE

We restrict $R,X_j,P,Z$ to $\mathcal O$ and view them as functions of $(\rho,\vt,s)$, where $\rho\in (0,1]$, $\vt\in (0,2]$, and $s\in [s_0:=\log 2,\infty)$, via the correspondence:
\BGE t_1=e^{-s},\quad t_2=\rho t_1=\rho e^{-s},\quad t_3=\vt t_1t_2=\vt \rho e^{-2s}.\label{corresp}\EDE
We use $\pa_s$ to denote the partial derivative with respect to $s$ with $\rho,\vt$ fixed. In terms of $\pa_{r_j}$, $j\in[3]$, $\pa_s$ is expressed by
\(
        \partial_s
        =
        t_1\partial_{r_1}
        +
        t_2\partial_{r_2}
        +
        2t_3\partial_{r_3}
\).

All estimates in this proof are restricted to $\mathcal O$. Note that we have   $0<t_3\le t_2\le t_1\le 1/2$ and $t_3\le 2 t_1t_2$ in $ \mathcal O$, which implies that $r_j\asymp 1$ and $r_j^{-1}=1+O(t_1)$, $j\in [3]$.
 For \(s\ge s_0\), define
\BGE
       \mathcal N(s)
        :=
        \sup_{\u r\in \mathcal O_s}
       \max\{ |X_1(\u r)|,|X_2(\u r)|,|P(\u r)|,|Z(\u r)|\},\quad   \mathcal O_s:=\{\u r\in  \mathcal O: 1-r_1=e^{-s}\} . \label{SM}
\EDE
By (\ref{F+G>2/3}) applied to ${\ell}=1$ and $L=1-e^{-s}$, (\ref{paF}) and (\ref{eq:HF-side-jet-bound}), we have that on $ \mathcal O_s$,
$$|F|+|F_1|+|\F_1|\lesssim_s 1,\quad |F_2|+|F_3|+|\F_2|+|\HF|+|\HF_1|\lesssim_s t_2^{\gamma-1},\quad |G|+|\HF_2|\lesssim_s t_2^{\gamma-2}.$$
These estimates together imply that $\mathcal N(s)<\infty$ for any $s\ge s_0$.

Since $\F$ and $\HF$ are respectively the continuous extension of $F$ and $F_3$ to $[0,1)^2\times \{1\}$, we have $R\to 0$ as $t_3\to 0^+$.   Thus, using $t_3\le 2t_1t_2$ and $P=t_1t_2 R_3\le \mathcal N(s)$, we get
\BGE
        |R(\u r)|
        \le
        \int_{r_3}^{1}|R_3(r_1,r_2,\xi)|d\xi \le 2 t_1 t_2 \max_{\xi\in [r_3,1)}  |R_3(r_1,r_2,\xi)|
        \le
        2 \mathcal N(s). \label{eq:65-rough-normal}
\EDE

Let $\delta=(t_1t_2)^\gamma$. We write $\err$ for the combined error $O(t_1\mathcal N(s)) +O(\delta)$, which may change values from line to line. Using (\ref{paF})-(\ref{eq:HF-side-jet-bound}), (\ref{RXQZ})-(\ref{FR}) and (\ref{SM})-(\ref{eq:65-rough-normal}), we get
\begin{equation}\label{err-F}
\begin{aligned}
&
t_j(F_j-\mathcal F_j)=X_j+\mathscr E,\quad j\in[2],\quad
t_1t_2F_3=P+\mathscr E,
\quad t_3F_3 =\vartheta P+\mathscr E, \\
& t_1(F-\mathcal F) =\mathscr E,\quad
t_1t_2F=\mathscr E,\quad
t_1t_2F_j=\mathscr E,\quad
t_3F_j=\mathscr E,\quad j\in[2],\quad  t_1 t_3 F_3=\err.
\end{aligned}
\end{equation}
Using (\ref{y=vY}) and the above estimates, we get
\BGE  t_1^2 t_2^2 G  =Z-2(1-\beta)P +\err,\quad t_1t_2t_3 G=\vt Z-2(1-\beta)\vt P+\err,\quad t_1^2t_2t_3 G=\err.\label{GZ}\EDE


We compute the \(\partial_s\)-derivatives of \(X_1,X_2,P,Z\) and express them as linear combinations of \(X_1,X_2,P,Z\), up to the error $\err$.
The bounds (\ref{SM})-(\ref{eq:65-rough-normal}), estimates (\ref{err-F})-(\ref{GZ}), and basic facts: $0<t_3\le t_2\le t_1\le 1/2$, $t_3=\vt t_1 t_2=O(t_1t_2)$, $r_j\asymp 1$, and $r_j^{-1}=1+O(t_1)$, $j\in [3]$, will be repeatedly used without being further mention.

Since $X_1=t_1 R_1$ and $\pa_s=t_1\pa_{r_1}+t_2\pa_{r_2}+2t_3\pa_{r_3}$,    using (\ref{paF}) and (\ref{eq:HF-side-jet-bound}), we get
\BGE \pa_s X_1=-X_1+t_1^2(F_{11}-\F_{11}) +t_1t_2 F_{12}  +2t_1t_3 F_{13} +\err.  \label{pasX1}\EDE
Applying (\ref{rainbow1})  with ${\ell}=1$ and (\ref{F-bdry-hyper}) with $j=1$, we get
\begin{align}
  t_1^2 (F_{11}-\F_{11})&=\frac{t_1}{r_1} [(2r_1-2\beta)(F_1-\F_1)+\beta(1-\beta)(F-\F)+t_2 t_3 G ] \label{F11-pre}\\
 & =(2-2\beta)X_1-2(1-\beta) \vt P+\vt Z+\err.\label{F11}
\end{align}
We will need the following estimates by (\ref{K-tj}):
\BGE K_j=-2 t_j+t_1t_2+O(t_1t_3),\quad j\in[2],\quad \text{and}\quad K_3=-2t_3-t_1t_2 +O(t_1t_3).\label{K23}\EDE
Applying (\ref{neighbor2}) with ${\ell}=3$ and ${\ell}=2$, respectively, we get
\begin{align}
  t_1t_2 F_{12}&=-\frac{t_1t_2}{2r_1r_2} (\beta r_1F_1+\beta r_2F_2+\beta r_3 F_3+\beta^2 F- K_3 G ) \label{F12-pre}\\
 & =((1-\beta)(2\vt +1)-\frac\beta 2)P-(\vt+\frac 12)Z+\err;\label{F12}
\end{align}
\begin{align}
  t_1^2 t_2 F_{31}&=-\frac{t_1^2 t_2}{2r_3 r_1}(\beta r_1F_1+\beta r_2F_2+\beta r_3 F_3+\beta^2 F- K_2 G )\label{F31-pre}\\
   & =2(1-\beta) P- Z+\err.\label{F31}
\end{align}
A similar computation gives the same estimate for $ t_1 t_2^2F_{32}$.
Multiplying (\ref{F31}) by $2\vt$, we get
\BGE
  2t_1t_3 F_{13}   =4(1-\beta)\vt P-2\vt Z+\err.\label{F13}
\EDE
Substituting  (\ref{F11}), (\ref{F12}) and (\ref{F13}) into (\ref{pasX1}),  we get 
\BGE \pa_s X_1=-\alpha X_1+(1-\frac32{\beta} +4(1-\beta)\vt) P-(\frac 12+2\vt)Z +\err.\label{pasX1*}\EDE
A similar computation yields
\BGE \pa_s X_2=-\alpha X_2+(1-\frac32{\beta} +4(1-\beta)\vt) P-(\frac 12+2\vt)Z +\err.\label{pasX2*}\EDE

Now we compute $\pa_s P$.
Since $P=t_1t_2 R_3$ and $R_3=F_3-\HF$, by (\ref{eq:HF-side-jet-bound}), we get
\BGE \pa_s P=-2P+t_1t_2 (t_1F_{31} +t_2 F_{32}+2t_3 F_{33})+\err. \label{pasQ-pre}\EDE
Applying (\ref{rainbow1}) with ${\ell}=3$, we get
\begin{align}
  2t_1t_2 t_3 F_{33}&=\frac{2t_1t_2}{r_3}[(2r_3-2\beta) F_3+\beta(1-\beta) F+t_1 t_2 G] = 2Z+\err ,\label{F33}
\end{align}
which together with (\ref{F31}) and its $j=2$ analogue  implies
\BGE \pa_s P=-2\alpha P +\err.\label{pasQ*}\EDE

It remains to compute $\pa_s Z$.  Since $Z=t_1 t_2 y$, by (\ref{y=vY}), with $(\lambda_1,\lambda_2,\lambda_3):=(1,1,2)$, we have
\begin{align}
  \pa_s Z&= -2 Z-2t_1^2 t_2^2 G+ t_1t_2 \sum_{j\in [3]} \lambda_j  t_j(2(1-\beta) F_{3j}+\beta(1-\beta)F_j+ t_1t_2 G_j)\label{pasZ-pre}\\
  &=4(1-\beta)(3-2\beta) P-4 Z +t_1^2 t_2^2 (t_1 G_1+t_2 G_2+2t_3 G_3)+\err. \label{pasZ-pre2}
\end{align}

We use the Pfaff system (\ref{G3=2@N}) to estimate the last term   in (\ref{pasZ-pre2}). By (\ref{Delta-tj}), (\ref{K-tj}) and (\ref{Q-tj}),
\BGE
\begin{aligned}
  & \Delta\gtrsim  {t_1^2 t_2^2},\quad  \Delta^{-1}= t_1^{-2} t_2^{-2} (1+4\vt)^{-1} +O(t_1^{-2} t_2^{-3} t_3),\quad K_{\ell}^+=O(t_1t_2),\quad {\ell}\in [3];\\
  & t_3Q^{(3)}= t_1^2t_2^2(-6\vt+O(\vt t_1)),\quad t_j Q^{(j)}=t_1^2 t_2^2 (-3-6\vt+ O(  t_1)),\quad j\in [2].
\end{aligned}
\label{err-DQ}
\EDE
Applying  (\ref{G3=2@N})-(\ref{CnbG}) to ${\ell}=1$ and using the above estimates, we get
\begin{align}
  t_1^3t_2^2  G_1=&-2\beta(1-\beta)\frac{  t_1^2 t_2^2} \Delta r_1 t_2 t_3 F_1 -\beta(1-\beta) \frac{ t_1^2 t_2^2 }{r_1\Delta} {r_2 t_1  K_3}  F_2  +\beta^2(1-\beta) \frac{t_1^2 t_2^2 }{r_1\Delta}K_1^+ F \nonumber\\
 & -\beta(1-\beta) \frac{  t_1^2 t_2^2 }{r_1\Delta} r_3 t_1 K_2 F_3 -\Big(\frac{\beta}{r_1} +\frac{t_1 Q^{(1)}}{r_1\Delta}\Big) t_1^2 t_2^2 G \label{t13G-pre}\\
 =&\Big[ \frac{2\beta(1-\beta)}{1+4\vt}-2(1-\beta)\Big(\frac{3+6\vt}{1+4\vt}-\beta\Big)\Big]P+\Big[\frac{3+6\vt}{1+4\vt}-\beta\Big] Z+\err,\label{t12G}
\end{align}
where the terms in (\ref{t13G-pre}) are absorbed into the error $\err$ except for the last two. A similar computation shows that $t_1^2t_2^3  G_2$ can also be estimated by (\ref{t12G}).
Applying  (\ref{G3=2@N})-(\ref{CnbG}) to ${\ell}=3$, we get
\begin{align}
2t_1^2 t_2^2 t_3 G_3=&-2\beta(1-\beta)\frac{t_1^2t_2^2}{r_3\Delta} r_1 t_3 K_2 F_1 -2\beta(1-\beta)\frac{t_1^2t_2^2}{r_3\Delta} r_2 t_3 K_1 F_2+2\beta^2(1-\beta) \frac{t_1^2t_2^2}{r_3\Delta} K_3^+ F \nonumber\\
&-4\beta(1-\beta) \frac{t_1^2t_2^2}{ \Delta} r_3 t_1 t_2 F_3-2\Big(\frac\beta{r_3}+\frac{t_3 Q^{(3)}}{r_3\Delta}\Big) t_1^2 t_2^2 G \label{t31G-pre}\\
=&\Big[- \frac{4\beta(1-\beta)}{1+4\vt}-2(1-\beta)\Big(\frac{12 \vt}{1+4\vt}-2\beta\Big)\Big]P+\Big[\frac{12 \vt}{1+4\vt}-2\beta\Big] Z+\err,\label{t31G}
\end{align}
where we absorb the terms in (\ref{t31G-pre}) except for the last two into $\err$.
Combining the above estimates, we get
\BGE \pa_s Z=-2\alpha Z+ \err. \label{pasZ}\EDE

Define a Pfaff vector:
\BGE Y=[\begin{array}{ccccc}X_1 & X_2 & P & Z
\end{array}]^{\text T}.\label{Pfaff-XPZ}\EDE By (\ref{pasX1*}), (\ref{pasX2*}), (\ref{pasQ*}) and (\ref{pasZ}), we obtain the Pfaff system: \BGE \pa_s Y=MY+\mathcal N(s) D+E,\label{YMDE}\EDE where
\BGE M:=\Big[\begin{array}{cc} -\alpha I_2 & N \\ 0 & -2\alpha I_2
\end{array}\Big],\quad N:=\Big[\begin{array}{cc} 1-\frac 32\beta+4(1-\beta)\vt & -\frac 12-2\vt \\1-\frac 32\beta+4(1-\beta)\vt & -\frac 12-2\vt
\end{array}\Big],\label{MN}\EDE
and $D,E$ are vector function of length $4$ that satisfy $\|D\|\lesssim t_1$ and $\|E\|\lesssim \delta$.

We fix $\rho,\vt$ for a moment and view $M$ as a constant matrix, and $Y,D,E$ as vector functions of $s$. Solving $ Y'=MY+ \mathcal N D+E$, we get
$$Y(s)=e^{(s-s_0)M} Y(s_0)+\int_{s_0}^s\mathcal N(r)  e^{(s-r) M} D(r) dr+\int_{s_0}^s e^{(s-r)M} E(r) dr,\quad s\ge s_0.$$
For $t\ge 0$, since $\|N\|\lesssim 1$, we get
$e^{tM}=O(e^{-\alpha t})$. Taking the norm  and then taking a supremum over all such $(\rho,\vt)\in (0,1]\times (0,2]$, we get by $\|D\|\lesssim t_1=e^{-s}$ and $\|E\|\lesssim \delta\le e^{-2\gamma s}$.

 \BGE e^{\alpha s} \mathcal N(s)\lesssim e^{ \alpha s_0 }\mathcal N(s_0)+  \int_{s_0}^s  e^{ \alpha r}(e^{-r} \mathcal N(r)  +e^{-2\gamma r}) dr  ,\quad s\ge s_0.\label{N-exp}\EDE
Let $f(s)$ be equal to the integral in (\ref{N-exp}). Then for some $C >0$ depending on $\beta$,
$$ f'(s)=e^{(\alpha-2\gamma)s}+e^{\alpha s-s} \mathcal N(s)\le e^{(\alpha-2\gamma)s} + e^{-s}( C e^{\alpha s_0}  \mathcal N(s_0)  + C f(s)),$$
which implies
$$\frac{d}{ds}\Big(\exp(C e^{-s}) f(s)\Big)\le \exp(C e^{-s})\Big(e^{(\alpha-2\gamma)s} + C e^{\alpha s_0}  \mathcal N(s_0) e^{-s}\Big)\lesssim e^{(\alpha-2\gamma)s},$$
where we used $\alpha-2\gamma\ge -1$ and the finiteness of $\mathcal N(s_0)$. Since $f(s_0)=0$ and $\exp(C e^{-s})\asymp 1$, the above inequality implies that $f(s)\lesssim \int_{s_0}^s  e^{(\alpha-2\gamma)r} dr$. By (\ref{N-exp}), we   get
\BGE \mathcal N(s)\lesssim e^{-\alpha s}\int_{s_0}^s  e^{(\alpha-2\gamma)r} dr. \label{N-upper}\EDE
By considering cases $\alpha>2\gamma$ ($\alpha>2$), $\alpha=2\gamma$ ($\alpha=2$) and $\alpha<2\gamma$ ($\alpha<2$) separately, we find that $\mathcal N(s)\lesssim e^{-\gamma s}$.
 Since $t_1=e^{-s}$,  $|X_1|+|X_2|+|P|+|Z|+|R|\lesssim t_1^\gamma$.

Since $P=t_1t_2(F_3-\HF)$ and  $t_1t_2\HF=O(\delta)$, we get $t_1 t_2 F_3=O(t_1^\gamma)$, which implies $t_3 F_3=O(t_1^\gamma)$. For $j\in [2]$, since $X_j=t_j(F_j-\F_j+t_3\HF_j)$,   $t_j\F_j=O(t_j^\gamma)$, and $t_j t_3\HF_j =O(\delta)$, we get $t_j F_j=O(t_1^\gamma)$. Since $F-\F=R-t_3\HF$, $|R| \lesssim t_1^\gamma$, and $t_3|\HF|\lesssim t_1t_2 |\HF|\lesssim \delta$, we get $|F(\u r)-\F(r_1,r_2)|=O(t_1^\gamma)$, which together with (\ref{F-11}) implies that $|F(\u r)- \F(1,1)|=O(t_1^\gamma)$. Finally, by (\ref{GZ}) and $t_3\le 2 t_1 t_2$, we get $t_1t_2 t_3 |G(\u r)|= O(t_1^\gamma)$.

Now we assume $\beta=2/3$. Then $\alpha=\gamma=1/3$. Let $y$ be as in (\ref{y=vY}), and let $L(t)$ and $\A$ be as in Lemma \ref{critial-est}. We use $\A_j$ to denote  $\pa_{r_j}\A$ for $j=1,2$. Then $\A_j=\beta t_j^{-1}\A$.
Set $L_3=L(t_3)=\log(1/t_3)$.
Define $R:=F-\F+t_3(1+L_3)\mathcal A$ and then define $X_1,X_2,P,Z$ by (\ref{RXQZ}), with the present definition of $R$. Then (\ref{FR}) is replaced by
\BGE F=R+\F-t_3(1+L_3)\mathcal A,\quad F_j=R_j+\F_j-t_3(1+L_3)\mathcal A_j,\quad j\in [2],\quad F_3=R_3+L_3\mathcal A.\label{FR*}\EDE
We restrict the above functions to $\mathcal O$ and view them as functions of $(\rho,\vt,s)$ for $\rho\in (0,1]$, $\vt\in (0,2]$, and $s\ge s_0=\log(2)$,   using the correspondence (\ref{corresp}).


We  define $\mathcal N(s)$ using (\ref{SM}) for $s\ge s_0$, with the present definitions of $X_1,X_2,P,Z$. By Lemma \ref{critial-est}, $R\to 0$ as $t_3\to 0^+$,  hence (\ref{eq:65-rough-normal}) still holds in the present setting. We now show that $X_1,X_2,P,Z$ are bounded on $\mathcal O(s)$ for any $s\ge s_0$, which implies that $\mathcal N(s)<\infty$. We fix $s\ge s_0$ and restrict the estimates in this paragraph to $\mathcal O(s)$. Applying Proposition \ref{prop-non-corner} to ${\ell}=1$, $j=2$,  $L=1-e^{-s}$, and using $\tau_{23}\asymp t_2+t_3\asymp t_2$, we get $|F_1|\lesssim_s 1$ and $t_2|F_2|\lesssim_s t_2^\gamma\lesssim 1$. By (\ref{paF}), $t_j|\F_j|\lesssim 1$ for $j\in [2]$. Since $t_3\le t_2$, $ t_3(1+L_3)|\mathcal A|\lesssim_s  t_3(1+L_3) t_2^{-2/3}\le t_3^{1/3} (1+L_3)\lesssim 1$. For $j\in [2]$, since   $X_j=t_j F_j-t_j\F_j+\beta t_3(1+L_3)\mathcal A$, it is bounded on $\mathcal O(s)$ by these estimates. The boundedness of $P=t_1t_2(F_3-L_3\A)$  on  $\mathcal O(s)$ follows from the second inequality in (\ref{y-A}). The boundedness of $Z$ on  $\mathcal O(s)$ follows from the first  inequality in (\ref{y-A}) because
$$|Z|\le t_1t_2| y-\A|+t_1t_2|\A|\lesssim_s   t_2^{-2/3} t_3(1+L_3)+t_2^{1/3}\le  t_3^{1/3}(1+L_3)+t_2^{1/3}\lesssim 1.$$

Pick $\gamma'\in (\gamma/2,\gamma)=(1/6,1/3)$, e.g., $\gamma'=1/4$, and define $\delta$ by $\delta=(t_1t_2)^{\gamma'}$ instead of   $(t_1t_2)^{\gamma}$. Since $t_3\le 2t_1t_2$ and $|\A|\lesssim (t_1t_2)^{\gamma-1}$, \BGE t_3 L_3 |\A|\lesssim   (t_1t_2/t_3)^{1-\gamma+\gamma'} t_3 L_3 (t_1t_2)^{\gamma-1}\lesssim \delta t_3^{\gamma-\gamma'} L_3 \lesssim \delta. \label{tLA}\EDE
However, $t_1 t_2  L_3 \A$ cannot be absorbed into $O(\delta)$. Also note that $t_1 t_2 A=O(t_1^\gamma t_2^\gamma)$ can be absorbed into $O(\delta)$.
We still define $\err$ to be the combined error $O(t_1\mathcal N(s)) +O(\delta)$.

The estimates in (\ref{err-F}) remain valid in the current setting except for $t_1t_2 F_3$, which should now be replaced by
\BGE t_1t_2 F_3=P+t_1t_2 L_3\A.\label{t1t2F3}\EDE
The estimate in (\ref{GZ}) for $t_1t_2t_3G$ still holds here, but the estimate for $t_1^2t_2^2 G$ should be replaced by
\BGE t_1^2 t_2^2 G= Z-2(1-\beta) P-2(1-\beta)t_1t_2 L_3\A +\err, \label{GZ*}\EDE
The estimates for $t_1t_2t_3G$ and $t_1^2t_2^2 G$ together with (\ref{K23}) imply that
\BGE t_1t_2 K_3 G=(1+2\vt)(2(1-\beta) P-Z) +\beta t_1t_2 L_3\A  +\err;
\label{K3G}\EDE
\BGE t_1 t_2  t_j K_{3-j} G=4(1-\beta) P-2Z+2 (1-\beta)(2-t_j)t_1t_2 L_3\A   +\err,\quad j\in [2].\label{K2G}\EDE

We still have (\ref{pasX1}) for the current setting since the terms containing $\A$ are $O(\delta)$.  Formulas (\ref{F11-pre}), (\ref{K23}), (\ref{F12-pre}) and (\ref{F31-pre}) still hold here since the argument for them does not depend on $\beta$. Since formula (\ref{F11}) does not involve $t_1t_2 F_3$ or $t_1^2 t_2^2 G$, it still holds here. By (\ref{K3G})-(\ref{K2G}), we find that formula (\ref{F12}) still holds here, and $t_1t_2 t_j F_{3j}$, $j\in [2]$, should now be estimated not by (\ref{F31}) and its $j=2$ analogue but by
\BGE t_1t_2 t_j F_{3j}=2(1-\beta) P- Z+ \beta t_1t_2 L_3\A+\err,\quad j\in [2]. \label{F31*}\EDE
where  the extra terms involving $t_1t_2 L_3\A$ are contributed by $-\frac{t_1 t_2 t_j}{2r_3 r_j}(\beta r_3 F_3-K_{3-j} G)$, and the equality $(1-\beta)(2-t_j)-\frac\beta 2t_j=\beta r_j$ is used in the last step since $\beta=2/3$. We have replaced $r_3$ by $1$ since the resulting difference can be absorbed into $\err$.

Multiplying (\ref{F31*}) for $j=1$ by $2\vt$, we see that (\ref{F13}) still holds in the current setting. Since (\ref{F11}), (\ref{F12}) and (\ref{F13}) remain unchanged,  we get (\ref{pasX1*}) in this setting. A similar argument establishes (\ref{pasX2*}) in the current setting.

As for $\pa_s P$, we have the following formula to replace (\ref{pasQ-pre}):
\BGE \pa_sP=-2P+t_1t_2(t_1 F_{31}+t_2 F_{32}+2t_3 F_{33})-2\beta t_1t_2 L_3\A+\err,\label{pasQ-pre*} \EDE
where the extra logarithmic term comes from summing two constributions $t_1t_2t_j \pa_{r_j} (- \L_3\A)$, $j\in [2]$. The estimate for $2t_1t_2t_3 F_{33}$ in (\ref{F33}) still holds here because the extra $t_1t_2 L_3\A$ terms come from $t_1t_2 F_3$ and $t_1^2t_2^2 G$ cancel each other. Substituting (\ref{F33}) and (\ref{F31*}) into (\ref{pasQ-pre*}), we get (\ref{pasQ*}) in the current setting.

As for $\pa_s Z$, we now still have (\ref{pasZ-pre}). We now prove (\ref{pasZ}) for the current setting. At this moment, we can only say that (\ref{pasZ}) holds here with a possible additional  multiple of $t_1t_2 L_3\A$. Thus, we need to carefully evaluate the contribution to the coefficient of $t_1t_2 L_3\A$ from each term in (\ref{pasZ-pre}). First, by (\ref{GZ*}), the coefficient contributed by the term $-2 t_1^2t_2^2 G$ in (\ref{pasZ-pre}) is $4(1-\beta)$. Second, by (\ref{err-F}), all terms $t_1t_2 t_j F_j$ can be absorbed into $\err$. Third, by (\ref{F33}) and (\ref{F31*}), the coefficient contributed by
$t_1t_2\sum_{j\in [3]} 2(1-\beta) \lambda_j t_j F_{3j}$ is $4\beta(1-\beta)$.
 Fourth, we use (\ref{G3=2@N})-(\ref{CnbQ}) to estimate $t_1^2 t_2^2 (t_1 G_1+t_2 G_2+2t_3 G_3)$. Equations (\ref{t13G-pre}) and (\ref{t31G-pre}) still hold here, where only the last two terms of each formula contribute  coefficients of $t_1t_2 L_3\A$, which may be replaced respectively by the expressions in (\ref{t13G-pre-alt}) and (\ref{t31G-pre-alt}) below because (\ref{err-F}), (\ref{K23}), and (\ref{err-DQ}) ensure that the resulting difference can be absorbed into the error $\err$:
\BGE -\beta(1-\beta) \frac{ 1}{r_1 (1+4\vt)}  t_1 (-2t_2+t_1t_2) F_3 -\Big(\frac{\beta}{r_1} +\frac{ -3-6\vt+2t_1 }{r_1 (1+4\vt)}\Big) t_1^2 t_2^2 G,\label{t13G-pre-alt}\EDE
\BGE -4\beta(1-\beta) \frac{1}{ 1+4\vt }   t_1 t_2 F_3 -2\Big(\beta+\frac{   -6\vt }{ 1+4\vt }\Big) t_1^2 t_2^2 G .\label{t31G-pre-alt}\EDE
By (\ref{t1t2F3}) and (\ref{GZ*}), the contributions from (\ref{t13G-pre-alt}) and (\ref{t31G-pre-alt})  are, respectively
$$-\beta(1-\beta) \frac{  -2+t_1  }{r_1 (1+4\vt)} +2(1-\beta)\Big(\frac{\beta}{r_1} +\frac{ (-3-6\vt+2t_1)}{r_1 (1+4\vt)}\Big),$$
$$ -4\beta(1-\beta) \frac{1}{  1+4\vt  } +4(1-\beta) \Big(\beta+\frac{-6\vt}{1+4\vt}\Big).$$
We may further replace $\vt$ by $0$ since the resulting difference is some bounded quantity times $\vt t_1t_2 L_3\A=t_3 L_3\A$, which can also be absorbed into $\err$ by (\ref{tLA}). Using $\beta=2/3$, we see that the contributions to the coefficients of $t_1 t_2 L_3\A$ from $t_1^2 t_2^2 t_1 G_1$ and $2t_1^2 t_2^2 t_3 G_3$ are respectively $-\frac{10}9$ and $0$. A similar argument shows that the contribution by $t_1^2 t_2^2 t_2 G_2$ is also $-\frac{10}9$.
Thus the total coefficient of \(t_1t_2L_3\mathcal A\) in the current version of (\ref{pasZ}) is
\(
4(1-\beta)+4\beta(1-\beta)-\frac{10}{9}-\frac{10}{9}=0
\).
Hence all additional \(t_1t_2L_3\mathcal A\) terms cancel, and (\ref{pasZ}) remains valid in the case $\beta=2/3$.

Define a new Pfaff vector $Y$ using  (\ref{Pfaff-XPZ}) but for the functions $X_1,X_2,P,Z$  defined here. Combining (\ref{pasX1*}), (\ref{pasX2*}), (\ref{pasQ*}) and (\ref{pasZ}), we see that $Y$ satisfies the same Pfaff system (\ref{YMDE}) with the same matrix $M$ as in (\ref{MN}) and the same type of estimates for $D$ and $E$:  $\|D\|\lesssim t_1$ and $\|E\|\lesssim \delta$. Since the $\delta$ here equals $(t_1t_2)^{\gamma'}=O(t_1^{2\gamma'})$,  (\ref{N-upper})   holds in the current setting with $\gamma$  replaced by $\gamma'$. Since $\alpha=\gamma<2\gamma'$, we get $|\mathcal N(s)|\lesssim  e^{-\gamma s}$. Consequently, $|X_1|+|X_2|+|P|+|Z|+|R|\lesssim t_1^{\gamma}$. Since $t_2\le t_1$ and $2\gamma'>\gamma$, we get $\delta\le t_1^{2\gamma'}\le t_1^\gamma$, and so $\err=O(t_1^\gamma)$.

Since $\vt P=t_3(F_3-L_3\A)$  and $t_3 L_3|\A|\lesssim \delta\le t_1^\gamma$, we get $t_3 F_3=O(t_1^{\gamma})$. For $j\in [2]$, since $X_j=t_j(F_j-\F_j+t_3(1+L_3)\A_j)$,   $t_j\F_j=O(t_j^\gamma)$ by (\ref{paF}), and $t_j t_3(1+L_3)\A_j=\beta t_3(1+L_3)\A =O(t_1^\gamma)$, we get $t_j F_j=O(t_1^{\gamma})$. Since $F-\F=R-t_3(1+L_3)\A$,  we get $|F(\u r)-\F(r_1,r_2)|=O(t_1^{\gamma})$, which together with (\ref{F-11}) implies that $|F(\u r)- \F(1,1)|=O(t_1^{\gamma})$. Finally, by the current version of (\ref{GZ}), we get $t_1t_2 t_3 |G(\u r)|= O(t_1^{\gamma})$.
\end{proof}

\begin{proof}[Proof of Theorem \ref{Thm-coninuity-neighbor}]
Given the continuity and estimates away from the corner point $(1,1,1)$ by Proposition \ref{prop-non-corner}, we now only need to focus on estimates near the corner point.

 {We record that the radial Pfaff-flow argument leading to
(\ref{Y<t1}) will be used again in Section~\ref{Section-6.4}.
The argument uses the present parameter range only through
$\alpha>0$ and the initial estimates on the hypersurface
$\chi=1/2$.  In the range $1/2<\beta<2/3$, those initial estimates
will be supplied by Proposition~\ref{prop:subcritical-inner-core} in place of Proposition~\ref{lem:direct-inner-core}.}

Fix $\u r^0\in\Omega$ such that $\chi_0:=\chi(\u r^0)=1/2$. Then $\u r^0\in \mathcal O$. In the following, we evaluate all the relevant functions at $\u r(\zeta):=(r_1(\zeta),r_2(\zeta),r_3(\zeta))$, where  $r_j(\zeta):=1-t_j(\zeta)$ and $t_j(\zeta)=t_j^0 e^{-\zeta}$, $j\in [3]$, and thus view them as functions of $\zeta\in [0,\infty)$. By Proposition \ref{lem:direct-inner-core}, $|F|_{\zeta=0}=|F(\u r^0)|\lesssim 1$. We will frequently use the following basic estimates. For all $\zeta\in [0,\infty)$,   $t_3\le t_2\le t_1\le 1/2$, and so $r_j\asymp 1$ and $r_j^{-1}=1+O(t_j)=1+O(t_1)$ for $j\in [3]$; $\chi(\zeta)=\chi_0 e^{-\zeta}\le 1/2$; and for any $j,k,{\ell}\in [3]$ such that $j\ne k$, $\frac{t_jt_k}{t_{\ell}}\le \chi$. By (\ref{Delta-tj}), (\ref{K-tj}), (\ref{Q-tj}),
\BGE \Delta=4 t_1t_2t_3 (1+O(\chi)),\quad K_{\ell}=-2t_{\ell}(1+O(\chi)),\quad t_{\ell} Q^{({\ell})}=-6 t_1t_2 t_3(1+O(\chi)),\quad {\ell}\in[3].\label{D-K-Q-est}\EDE

Consider a Pfaff vector $${Y}:=[\begin{array}{ccccc} t_1F_1 & t_2F_2 & t_3F_3 & t_1t_2t_3 G
\end{array}]^{\text T},$$
as a function of $\zeta\in [0,\infty)$.
Since $\u r^0\in\mathcal O$, by Proposition \ref{lem:direct-inner-core}, $\|Y(0)\|\lesssim (t_1^0)^\gamma$.
We now evaluate $\pa_{\zeta} Y$ using $\pa_{\zeta}=\sum_{j\in [3]} t_j \pa_{r_j}$. First,
$$\pa_{\zeta} (t_{\ell} F_{\ell})=-t_{\ell} F_{\ell}+t_{\ell}^2 F_{\ell\ell}+ \sum_{j\in [3]\sem \{{\ell}\}} t_{\ell} t_j F_{\ell j},\quad {\ell}\in [3].$$
Let ${\ell}\in [3]$ and pick distinct $j,k\in [3]\sem \{{\ell}\}$. By (\ref{rainbow1}),
$$t_{\ell}^2 F_{\ell\ell}=\frac{t_{\ell}}{r_{\ell}}(t_jt_kG+(2r_{\ell}-2\beta) F_{\ell}+\beta(1-\beta) F)$$
$$=t_1t_2t_3G+(2-2\beta)t_lF_{\ell}+ O(t_1)|t_1t_2t_3G|+ O(t_1)|t_lF_{\ell}|+ O(t_1)|F|.$$
By (\ref{neighbor2}) and (\ref{D-K-Q-est}),
$$t_lt_j F_{j\ell}=\frac{t_lt_j}{2r_{\ell} r_j} (K_k G-\beta r_1 F_1-\beta r_2 F_2-\beta r_3 F_3-\beta^2 F)$$
$$=-t_1t_2t_3 G+O(\chi)|t_1t_2t_3G|+O(\chi)|t_1 F_1|+O(\chi)|t_2F_2|+O(\chi)|t_3 F_3|+O(t_1)|F| .$$
Combining the above displayed formulas, we get
\BGE \pa_{\zeta} (t_{\ell} F_{\ell})=-\alpha  t_lF_{\ell} - t_1t_2t_3 G + \mathbf{O}(\chi) ^{\text{T}}Y+O(t_1)|F|,\quad {\ell}\in [3]. \label{paellFj}\EDE
where $\mathbf{ O}(\chi)$ stands for a four-component vector, whose norm is of the size $O(\chi)$.

We use (\ref{G3=2@N})  to evaluate $\pa_{\zeta} (t_1t_2t_3G)$ and get
$$\pa_{\zeta}(t_1t_2t_3G)=-3 t_1t_2t_3G +t_1t_2 t_3\sum_{{\ell}\in [3]}  t_{\ell} \Big(C^{({\ell})}_{F_{\ell}} F_{\ell} +\sum_{j\in [3]\sem\{{\ell}\}} C^{({\ell})}_{F_j} F_j +C^{({\ell})}_{F} F +C^{({\ell})}_G G\Big).$$
Let ${\ell}\in [3]$ and take distinct $j,k\in [3]\sem \{{\ell}\}$. Using (\ref{CnbF1})-(\ref{CnbG}) and (\ref{D-K-Q-est}), we get
$$t_1t_2t_3 t_{\ell} C^{({\ell})}_{F_{\ell}} F_{\ell}=-2\beta(1-\beta) \frac{r_{\ell} t_j t_k t_{\ell}}{\Delta}\cdot\frac{t_j t_k}{t_{\ell}}\cdot t_{\ell} F_{\ell}=O(\chi)| t_{\ell} F_{\ell}|; $$
$$t_1t_2t_3 t_{\ell} C^{({\ell})}_{F_j} F_j=-\beta(1-\beta)\frac{r_j t_1t_2t_3}{r_{\ell}\Delta} t_{\ell} K_k F_j =O(\chi)| t_j F_j|; $$
$$t_1t_2t_3 t_{\ell} C^{({\ell})}_{F } F = \beta^2(1-\beta) \frac{t_1t_2t_3}{r_{\ell} \Delta}\cdot {K_{\ell}^+} F=O(t_1)| F|; $$
$$t_1t_2t_3 t_{\ell} C^{({\ell})}_{G } G= t_1t_2t_3\Big(-\frac\beta{r_\ell}-\frac{t_{\ell} Q^{({\ell})}}{\Delta}\Big) G=\Big(\frac 32-\beta\Big) t_1t_2 t_3 G+O(\chi)| t_1t_2t_3 G|. $$
Combining the above estimates, we get
\BGE \pa_{\zeta}(t_1t_2t_3G)=-\frac 32 \alpha   t_1t_2t_3G +\mathbf{ O}(\chi)^{\text{T}}Y+O(t_1)|F|.\label{paellG}\EDE
Combining (\ref{paellFj}) and (\ref{paellG}), and using $\chi=\frac 12 e^{-\zeta}$ and $t_1=t_1^0e^{-\zeta}$, we get a Pfaff system:
\BGE \pa_{\zeta} Y=MY+e^{-\zeta} DY+e^{-\zeta}FE,\label{Pfaff-3}\EDE
where
$$M:=\left[ \begin{array}{cccc} -\alpha & 0 & 0 & -1 \\ 0 & -\alpha & 0 & -1 \\  0 & 0  & -\alpha & -1 \\ 0 & 0 & 0 & -\frac 32\alpha
\end{array}
\right],$$
$D$ is a $4\times 4$ matrix function with $\|D\|\lesssim 1$, and $E$ is a $4$-component vector function with $\|E\|\lesssim t_1^0$. The variation-of-constants formula gives
$$Y(\zeta)=e^{\zeta M} Y(0)+\int_0^\zeta e^{-s} e^{(\zeta-s)M} (D(s) Y(s)+F(s) E(s))ds,\quad \zeta\ge 0.$$
Since $M$ has eigenvalues: $-\alpha,-\alpha,-\alpha,-\frac 32\alpha$ and corresponding linearly independent eigenvectors: $e_1$, $e_2$, $e_3$, $e_1+e_2+e_3+\frac\alpha 2 e_4$, for any $s\ge 0$, $\|e^{sM}\|\lesssim e^{-\alpha s}$. Thus,
$$\|Y(\zeta)\|\lesssim e^{-\alpha\zeta} \|Y(0)\|+\int_0^\zeta e^{-s} e^{-\alpha(\zeta-s)} ( \| Y(s)\|+t_1^0|F(s)|)ds,\quad \zeta\ge 0.$$
On the other hand, we have $\pa_{\zeta} F=\sum_{{\ell}\in [3]} t_{\ell} F_{\ell}$, and so $|F(s)|\le |F(0)|+\int_0^s \sqrt 3 \|Y(u)\|du$. Plugging this inequality into the above displayed formula, and using $|F(0)|\lesssim 1$,  $\|Y(0)\|\lesssim (t_1^0)^\gamma$, $e^{-s}\le e^{-\gamma s}$, and $e^{-\alpha(\zeta-s)}\le e^{-\gamma(\zeta-s)}$ when $0\le s\le \zeta$ (because $\gamma=1\wedge \alpha$),  we get
$$\|Y(\zeta)\|\lesssim e^{-\alpha\zeta} \|Y(0)\|+ \int_0^\zeta e^{-s} e^{-\alpha(\zeta-s)} \| Y(s)\|ds+t_1^0|F(0)|\int_0^\zeta e^{-s} e^{-\alpha(\zeta-s)}ds$$
$$+t_1^0\int_0^\zeta e^{-s} e^{-\alpha(\zeta-s)} \int_0^s \sqrt 3 \|Y(u)\|du ds$$
$$\lesssim  e^{-\alpha\zeta} (t_1^0)^\gamma+  e^{-\gamma \zeta}(1+t_1^0\zeta) \int_0^\zeta   \| Y(s)\|ds +t_1^0\zeta e^{-\gamma\zeta}  .$$
Pick any  $\gamma'\in (0,\gamma)$. Then there is a constant $C$ depending only on $\beta,\gamma'$ such that
$$ \|Y(\zeta)\|\le C e^{-\gamma' \zeta} \Big( \int_0^\zeta   \| Y(s)\|ds +  (t_1^0)^\gamma   \Big)  .$$
Using Gr\"onwall's inequality and $t_1=t_1^0 e^{-\zeta}$, we get
\BGE  \|Y(\zeta)\|\le C e^{-\gamma' \zeta} (t_1^0)^\gamma  \exp ( (C/{\gamma'}) (1-e^{-\gamma' \zeta}) )
\lesssim e^{-\gamma' \zeta}  (t_1^0)^{\gamma'}=t_1^{\gamma'}.\label{Y<t1}\EDE

Suppose $\u r\in \Omega\sem \mathcal O$. Then $\chi(\u r)=t_1t_2/t_3<1/2$. Define $\u r^0$ such that
$$t_1^0=\frac{t_3}{2t_2},\quad t_2^0=\frac{t_3}{2t_1},\quad t_3^0= \frac{t_3^2}{2t_1t_2}.$$
Then $\u r^0\in\Omega$ and $\chi(\u r^0)= 1/2$, and the Pfaff flow $Y(\zeta)$, $\zeta\ge 0$, starting from such $\u r^0$ reaches $\u r$ at $\zeta=\log(\frac{t_3}{2t_1t_2})>0$. By (\ref{Y<t1}), we get $t_j |F_j(\u r)|\lesssim t_1^{\gamma'}$, $j\in [3]$, and $t_1t_2t_3 |G|\lesssim  t_1^{\gamma'}$ on $\Omega\sem \mathcal O$. Combining this estimate with Proposition \ref{prop-non-corner} (applied to $L=1/2$), Proposition \ref{lem:direct-inner-core}, and the symmetry of $F$ and $G$, we get
$\sum_{j=1}^3 t_j |F_j | +t_1t_2t_3|G| \lesssim (t_1\vee t_2\vee t_3)^{\gamma'}$ on $[0,1)^3$.

We now prove that $F(\u r)\to \F(1,1)$ as $\u r\to (1,1,1)$. Since $\F(r_1,r_2)\to\F(1,1)$ as $(r_1,r_2)\to (1,1)$, by the symmetry of $F$, it suffices to show that, when $\u r\to (1,1,1)$ within the region $\Omega$, we have $|F(\u r)-\F(r_1,r_2)|\to 0$. By Proposition \ref{lem:direct-inner-core}, this holds if $\Omega$ is replaced by $\mathcal O$.
Define $\phi:\Omega\sem \mathcal O\to \mathcal O$ by $\phi(\u r)=(r_1,r_2,1-2 t_1 t_2)$. Then $\phi(\u r)\to (1,1,1)$ as $\u r\to (1,1,1)$ within $\Omega\sem \mathcal O$. Thus, it suffices to show that $|F(\u r)-F(\phi(\u r))|\to 0$ as  $\u r\to (1,1,1)$, which holds because
$$|F(\u r)-F(\phi(\u r))|\le \int^{1-2t_1t_2}_{r_3} |F_3(r_1,r_2,s)|ds\lesssim t_1^{\gamma'} \int_{2t_1t_2}^{t_3} u^{-1} du=  t_1^{\gamma'}\log \frac{t_3}{2t_1t_2}\le t_1^{\gamma'}\log \frac{1}{ t_1 }\to 0,$$
where the second inequality holds because for any $s\in (r_3,1-2t_1t_2)$, $(r_1,r_2,s)\in \Omega\sem \mathcal O$, and so $(1-s) |F_3(r_1,r_2,s)|\lesssim t_1^{\gamma'}$.

Combining the result in the last paragraph with the continuous extension of $F$ to $[0,1]^3\sem \{(1,1,1)\}$ given by Proposition \ref{prop-non-corner}, we conclude that $F$ extends continuously to $[0,1]^3$. Finally,  (\ref{F-boundary}) follows from (\ref{F1&N*}) and the symmetry of $F$.
\end{proof}

\subsection{The remaining range $1/2<\beta<2/3$} \label{Section-6.4}

We now treat the remaining range $\beta\in (1/2,2/3)$, so that the $\alpha,\gamma$ in (\ref{gamma}) satisfy $\gamma=\alpha=2\beta-1\in(0,1/3)$.
By Proposition~\ref{Extension1} (ii), $F$, $r_j F_j$, $j\in [2]$, $r_1r_2 F_{12}$, and $\beta F_3+t_1t_2 G$ extend continuously to $[0,1)^2\times [0,1]$, and  $\F:=F|_{r_3=1}$ is given by (\ref{F1&N*}). By the fundamental theorem of calculus,
$$(r_jF_j)|_{r_3=1}=r_j\F_j,\quad j\in [2];\quad (r_1r_2 F_{12})|_{r_3=1} =r_1r_2\F_{12}. $$ 
We rewrite (\ref{neighbor2}) with $\ell=3$ as
$$\beta r_3 F_3+r_3 t_1t_2 G=-2r_1r_2F_{12} -\beta r_1 F_1-\beta r_2 F_2-\beta^2 F-t_3 (1+r_1r_2) G.$$
The proof of Proposition~\ref{Extension1} (ii) also shows that
$t_3G\to0$ locally uniformly on $[0,1)^2$ as $r_3\to1^-$. Therefore, sending $r_3\to 1^-$ in the above displayed formula, we get
$$ (\beta F_3+t_1t_2 G)|_{r_3=1}=-2 r_1r_2\F_{12}-\beta r_1\F_1-\beta r_2\F_2-\beta^2\F. $$ 
By (\ref{F1&N*}), (\ref{f'-bd}), and the above formulas, we get, for $j\in [2]$,
\BGE F|_{r_3=1}=O(1);\quad (r_jF_j)|_{r_3=1}=O(t_j^{2\beta-2});\quad (\beta F_3+t_1t_2 G)|_{r_3=1}=O((t_1t_2)^{2\beta-2}).\label{F3G}\EDE

We use the notation $Y,u,v,N$ from (\ref{Pfaff1}), (\ref{uv}) and (\ref{uvM}),   and the vector \(
\widetilde Y=t_3^{uv^{\rm T}}Y
\) and matrix $\til N=t_3^{uv^{\rm T}}Nt_3^{-uv^{\rm T}}$ defined in the proof of Proposition~\ref{Extension1}(ii). Then
\(
\partial_{r_3}\widetilde Y=\til N\widetilde Y
\). 
Set $\nu:=3\beta-2\in (-1/2,0)$. Then $v^{\rm T}u=-\nu$, and so for any $t\in(0,\infty)$,
\BGE
\|t^{u v^{\rm T}}\|=\|I_5+\frac{t^{-\nu}-1}{-\nu} u v^{\rm T}\|\lesssim (t\vee 1)^{-\nu}. \label{Epm*}
\EDE

\begin{Lemma}
Throughout $[0,1)^2\times (0,1)$,
\BGE
\|N\|\lesssim r_3^{-1}(t_1t_2)^{-1};
\label{subcritical-N}
\EDE
and for any $m_1,m_2\in\N_0$,
\BGE
\|\pa_{r_1}^{m_1}\pa_{r_2}^{m_2} N\|\lesssim_{m_1,m_2} r_3^{-1}(t_1t_2)^{-2(m_1+m_2+1)}.
\label{subcritical-N-m}
\EDE
\label{lem:N-bdd}
\end{Lemma}
\begin{proof}
  From (\ref{Delta-tj}) we have
\BGE
\Delta=t_1^2t_2^2+t_1t_2t_3(4-2t_1-2t_2)+t_3^2(t_1-t_2)^2 \ge t_1^2t_2^2 +t_3^2(t_1-t_2)^2 \ge t_1^2t_2^2. \label{subcritical-Delta} \EDE
For (\ref{subcritical-N}) and (\ref{subcritical-N-m}) we need to bound $|N_{i,j}|$ and $|\pa_{r_1}^{m_1}\pa_{r_2}^{m_2} N_{i,j}|$ for each $(i,j)\in [5]^2$. The computations for $(i,j)\in [4]\times [5]$ are simple  because  $N_{1,5}$ and $N_{2,5}$ have the form of a polynomial divided by $r_3 t_1t_2$, and all other entries of the first four rows have the form of some constant or some constant divided by $r_3$.  For the fifth row, by (\ref{CnbF1}), $N_{5,1}$ and $N_{5,2}$ have the form of $\frac{t_1t_2}{r_3\Delta}$ times a polynomial, and so $\pa_{r_1}^{m_1}\pa_{r_2}^{m_2} N_{5,j}$, $j\in [2]$, have the form of ${r_3^{-1}\Delta^{-m_1-m_2-1}}$ times a polynomial. Since $\Delta\ge t_1^2t_2^2$ by (\ref{subcritical-Delta}), we get the desired upper bounds for the $(5,1)$- and $(5,2)$-entries. The proof of Proposition~\ref{Extension1} gives
\[
\frac{N_{5,3}}{-C_{5,3}}= \frac{\Delta-r_3t_1^2t_2^2}{t_3\Delta} =\frac{t_1 t_2 (2-t_1)(2-t_2)+ t_3(t_1-t_2)^2}{ \Delta} ;\]
\[
\frac{N_{5,4}}{-C_{5,4}}=\frac{t_1t_2K_3^++r_3\Delta}{r_3t_3\Delta} = \frac{t_1 t_2 (t_1+t_2-t_1t_2+2r_3(2-t_1-t_2))+r_3  t_3(t_1-t_2)^2}{r_3  \Delta};
\]
\[N_{5,5}+\frac{\beta}{r_3}=-\frac{Q^{(3)}}{r_3\Delta}=\frac{t_1^2t_2^2+t_1t_2(3-t_3)(2-t_1-t_2) +t_3(3-2t_3)(t_1-t_2)^2}{r_3\Delta}.
\]
Thus, for $j\in \{3,4,5\}$, $\pa_{r_1}^{m_1}\pa_{r_2}^{m_2} N_{5,j}$  has the form of $r_3^{-1}\Delta^{-m_1-m_2-1}$ times a polynomial, which together with (\ref{subcritical-Delta}) implies that $|\pa_{r_1}^{m_1}\pa_{r_2}^{m_2} N_{5,j}|\lesssim r_3^{-1}(t_1t_2)^{-2(m_1+m_2+1)}$. By the above displayed formulas, Cauchy-Schwarz inequality, and (\ref{subcritical-Delta}), we also conclude that, for $j\in \{3,4\}$,
$$|N_{5,j}|\lesssim \frac{t_1t_2+t_3(t_1-t_2)^2}{r_3\Delta} =\frac{t_1^2t_2^2+t_1t_2t_3(t_1-t_2)^2}{t_1t_2r_3\Delta} \le \frac{2t_1^2t_2^2+ t_3^2(t_1-t_2)^4}{t_1t_2r_3\Delta}\lesssim r_3^{-1}(t_1t_2)^{-1},$$
and similarly, $|N_{5,5}|\lesssim r_3^{-1}+ r_3^{-1}(t_1t_2)^{-1}\lesssim  r_3^{-1}(t_1t_2)^{-1}$.
This completes the proof.
\end{proof}

We use $P_3$ to denote the map $(r_1,r_2,r_3)\mapsto (r_1,r_2,1)$ from $[0,1)^3$ to $[0,1)^2\times \{1\}$. For any function $f$ on $[0,1)^3$, which extends continuously to $[0,1)^2\times [0,1]$, we define $\delta_3 f=f-f\circ P_3$.

\begin{Lemma}
By the proof of Proposition~\ref{Extension1}(ii), $\til Y$ extends continuously to $[0,1)^2\times [0,1]$. Let $S=v^{\rm T} \til Y|_{r_3=1}$. Then
\BGE
S(r_1,r_2) =s_\beta t_1^{-\beta}t_2^{-\beta},
\qquad
s_\beta:=
\frac{\Gamma(2\beta)\Gamma(3-3\beta)}
{\Gamma(\beta)\Gamma(1-2\beta)},
\label{subcritical-amplitude}
\EDE
and as $r_3\to 1^-$, for any $L\in (0,1)$, uniformly in $(r_1,r_2)\in [0,L]^2$,  we have
\BGE
F_3=-\frac{S}{\nu}t_3^\nu+O_L(1),\qquad
t_1t_2G=\frac{\beta S}{\nu}t_3^\nu+O_L(1),\qquad
\delta_3 F=\frac{S}{\nu(\nu+1)}t_3^{\nu+1}+O_L(t_3).
\label{subcritical-normal-exp}
\EDE
The same asymptotic
estimates remain valid after applying tangential derivatives of order at most two, namely,
$\partial_{r_1}^{m_1}\partial_{r_2}^{m_2}$  for
$m_1+m_2\le2$.
\label{lem:subcritical-amplitude}
\end{Lemma}
\begin{proof}
Fix $L\in (0,1)$ and let $R_{L}=R_{L,1/2}=[0,L]^2\times [1/2,1)$. Throughout the proof, all estimates are restricted to $R_{L}$. By (\ref{subcritical-N-m}) and (\ref{Epm*}), for any $m_1,m_2\in\N_0$ with $m_1+m_2\le 2$,  $\|\pa_{r_1}^{m_1}\pa_{r_2}^{m_2} \til N\|\lesssim_{L} t_3^\nu$.

We now briefly review the proof that $\til Y$ extends continuously to $\lin{R_{L}}$, and give an estimate of $\delta_3\til Y$.
Fix $(r_1,r_2)\in [0,L]^2$ and view $\til Y$ and $\til N$ as functions in $r_3$. Convert \(
\partial_{r_3}\widetilde Y=\til N\widetilde Y
\) into the integral equation: $\til Y(r_3)=\til Y(1/2)+\int_{1/2}^{r_3} \til N(s)\til Y(s) ds$. From $\|\til N\|\lesssim_{L} (1-r_3)^\nu$, we get
$$\|\til Y(r_3)\|\le \|\til Y(1/2)\|+C_{L}\int_{1/2}^{r_3}  (1-s)^\nu \|\til Y(s)\|ds,\quad 1/2\le r_3<1.$$
By Gronwall's inequality, for $1/2\le r_3<1$,
$$\|\til Y(r_3)\|\le \|\til Y(1/2)\|\exp\Big(C_{L} \int_{1/2}^{r_3}(1-s)^\nu ds\Big)\lesssim_{L} \|\til Y(1/2)\|\lesssim_{L} 1,$$
where the first ``$\lesssim_{L}$'' holds because $\nu >-1$, and the second ``$\lesssim_{L}$'' holds because $\til Y$ is bounded on $[0,L]^2\times \{1/2\}$. Thus, $\|\til Y\|\lesssim_{L} 1$. Combining the bounds of $\til Y$ and $\til N$ and the equation  \(
\partial_{r_3}\widetilde Y=\til N\widetilde Y
\), we get $\|\til Y(r_3')-\til Y(r_3)\|\lesssim_{L} \int_{r_3}^{r_3'}  (1-s)^\nu ds$ for $r_3\le r_3'\in [1/2,1)$. Since $\nu>-1$, the RHS tends to $0$ as $r_3,r_3'\to 1^-$. Thus, $\til Y$ converges as $r_3\to 1^-$, uniformly on $[0,L]^2$, and so extends continuously to $\lin{R_{L}}$. In addition, the integral equation implies $\|\delta_3\til Y\|\lesssim_{L} t_3^{\nu +1}$.

For $j\in[2]$, differentiating \(
\partial_{r_3}\widetilde Y=\til N\widetilde Y
\) with
respect to $r_j$ gives
\[
\partial_{r_3}\partial_{r_j}\widetilde Y
=
\widetilde N\,\partial_{r_j}\widetilde Y
+(\partial_{r_j}\widetilde N)\widetilde Y.
\]
Since $\|\partial_{r_j}\widetilde N\|\lesssim_{L}t_3^\nu$,
the same integral-equation and Gronwall argument shows that
$\partial_{r_j}\widetilde Y$ is uniformly bounded,
 converges uniformly as $r_3\to1^-$, and satisfies $\|\delta_3\partial_{r_j}\til Y\|\lesssim_{L} t_3^{\nu+1}$ on $R_{L}$.  Repeating the argument
after further tangential differentiation gives the same conclusion
for $\partial_{r_j}^2\widetilde Y$, $j\in[2]$, and
$\partial_{r_1}\partial_{r_2}\widetilde Y$.

The fundamental theorem of calculus, followed by passage to the
uniform limit as $r_3\to1^-$, shows that the limits of $\partial_{r_1}^{m_1}\partial_{r_2}^{m_2}\widetilde Y$ as $r_3\to 1^-$ are the
corresponding tangential derivatives of
$\widetilde Y|_{r_3=1}$, which is therefore $C^2$ on
$[0,L]^2$. Since $L\in (0,1)$ is arbitrary, $\til Y|_{r_3=1}$ is $C^2$ on $[0,1)^2$. Thus, $S=v^{\rm T} \til Y|_{r_3=1}$ is also $C^2$.  Since $Y=t_3^{-uv^{\rm T}} \til Y$,  and
\(
t_3^{-uv^{\rm T}}=I_5+\frac{1-t_3^\nu}{\nu}uv^{\rm T}
\), we get
\BGE Y=\til Y+\frac{1-t_3^{\nu}}\nu u v^{\rm T} \til Y =-\frac{t_3^\nu} \nu S u+\til Y + \frac{v^{\rm T} \til Y -t_3^\nu v^{\rm T} \delta_3 \til Y  }\nu  u=-\frac{t_3^\nu}\nu S u+ O_{L}(1), \label{Y=S+}
\EDE
where the last equality holds because $\|\til Y\|\lesssim_{L} 1$, $| \delta_3 \til Y |\lesssim_{L} t_3^{\nu+1}$, and $2\nu+1>0$.
Reading the third and fifth components of this equation gives the first two formulas in (\ref{subcritical-normal-exp}), and integration in the $r_3$ direction gives the third.

Applying $D:=\pa_{r_1}^{m_1}\pa_{r_2}^{m_2}$ with $m_1+m_2\le 2$ to the third component of (\ref{Y=S+}), we get
$$\pa_{r_3}D F=-\frac{t_3^\nu}\nu D S +D\til Y_3 + \frac{ v^{\rm T} D  \til Y  -t_3^\nu v^{\rm T}\delta_3( D \til Y) }\nu=-\frac{t_3^\nu}\nu D S +O_{L}(1),$$
where the last estimate holds because $\|D \til Y\|\lesssim_{L} 1$, $|  \delta_3(D\til Y)|\lesssim_{L} t_3^{\nu +1}$, and $2\nu+1>0$. Integration in the $r_3$ direction gives
$$\delta_3 DF=\int_{r_3}^1 \frac{(1-s)^\nu} \nu ds\cdot D S+ O_{L}(t_3)=\frac{t_3^{\nu+1}DS}{\nu(\nu+1)} + O_{L}(t_3),$$
which together with (\ref{rainbow1}) with ${\ell}=j\in[2]$ and (\ref{F-bdry-hyper}) implies that
$$t_{3-j}t_3G=\L_{r_j} F=\L_{r_j}F-\L_{r_j}\F=\delta_3 \L_{r_j}F=\frac{t_3^{\nu+1}\L_{r_j}S}{\nu(\nu+1)} + O_{L}(t_3),\quad j\in [2].$$
Dividing this identity by $t_3^{\nu+1}$, letting $r_3\to 1^-$, and using the asymptotic formula for $t_1t_2 G$ in (\ref{subcritical-normal-exp}),  we obtain
\(
\L_{r_j} S=\frac{\beta(\nu+1)}{1-r_j} S\), $j\in [2]$. For the one-variable equation, $r_j=0$ is a regular singular point. The two Frobenius exponents at $0$ are $0$ and $1-2\beta<0$.
A direct computation shows that $(1-r_j)^{-\beta}$ solves the one-variable equation.  Since $S$ is continuous on $[0,1)^2$, the singular branch is excluded in each variable.  Hence
\(
S(r_1,r_2)= s_\beta t_1^{-\beta}t_2^{-\beta}
\)
for a constant $s_\beta$.

Finally, by (\ref{ABn00@N}),
\(
F(0,0,r)={}_2F_1(\beta,1-2\beta;2\beta;r)
\).
Applying \cite[Formula 15.10.21]{NIST:DLMF} with  $a=\beta$, $b=1-2\beta$ and $c=2\beta$, we obtain
\[F(0,0,r)=\frac{\Gamma(2\beta)\Gamma(3\beta-1)}{\Gamma(\beta)\Gamma(4\beta-1)} {}_2F_1(\beta,1-2\beta;2-3\beta;1-r)\]
\[+\frac{\Gamma(2\beta)\Gamma(1-3\beta)}{\Gamma(\beta)\Gamma(1-2\beta)} (1-r)^{3\beta-1} {}_2F_1(\beta,4\beta-1;3\beta;1-r).\]
Since $\nu+1=3\beta-1\in(0,1)$,
comparing this with the third formula in (\ref{subcritical-normal-exp}) at $r_1=r_2=0$ gives $\frac{s_\beta}{\nu(\nu+1)}=\frac{\Gamma(2\beta)\Gamma(1-3\beta)}{\Gamma(\beta)\Gamma(1-2\beta)}  $, which implies the formula of $s_\beta$ in (\ref{subcritical-amplitude}).
\end{proof}


\begin{Lemma}
Define $\varpi_3$ on $[0,1)^3$ by $\varpi_3(\u r)=\frac{t_3}{t_1t_2}$.
Fix $J\in (0,\infty)$. Let $W_{J}$ denote the region $\{\u r\in [0,1)^3: \varpi_3(\u r)\le J,r_3\ge 1/2\}$. Then, on  $W_{J}$,
we have
\BGE
\|Y\|\lesssim_{J}(t_1t_2)^{2\beta-2}
 \varpi_3^\nu  ,
\label{subcritical-Y-layer}
\EDE
and
\BGE
|\delta_3 F|
\lesssim_{J}(t_1t_2)^\alpha.
\label{subcritical-F-layer}
\EDE
\label{lem:subcritical-layer}
\end{Lemma}
\begin{proof}
Throughout this proof, all estimates are restricted to $W_{J}$.
Keeping $(r_1,r_2)\in [0,1)^2$ fixed and using $\varpi_3$ as the normal variable, which ranges in $(0,J\wedge \frac{1/2}{t_1t_2})$, the equation $\pa_{r_3} Y=MY$ and (\ref{uvM}) imply
\[
\partial_{\varpi_3}Y=-t_1t_2 MY=-\frac{uv^{\rm T}}{\varpi_3}Y-(t_1t_2)NY.
\]
Set $\widehat Y=\varpi_3^{uv^{\rm T}}Y$ and $\ha N=\varpi_3^{uv^{\rm T}}(t_1t_2)N\varpi_3^{-uv^{\rm T}}$. Then $\pa_{\varpi_3} \ha Y=-\ha N\ha Y$.  By (\ref{subcritical-N}) and (\ref{Epm*}),
\(
 \|\ha N\|
\lesssim_{J} \varpi_3^\nu
\).

Since $\til Y= t_3^{uv^{\rm T}}Y$, and $\til Y$ (as a function of $(r_1,r_2,r_3)$) extends continuously to $[0,1)^2\times \{1\}$, we have
$\ha Y=(t_1t_2)^{-uv^{\rm T}}\til Y$, and $\ha Y$ (as a function of $(t_1,t_2,\varpi_3)$) extends continuously to $(0,1]^2\times \{0\}$. In particular, $\ha Y(0)$ exists for any fixed $(t_1,t_2)\in (0,1]^2$.

Take $f_1=e_1$, $f_2=e_2$, $f_3=e_4$, $f_4=\beta e_3+e_5$, and $f_5=v$. For $j\in [4]$, since $f_j^{\rm T} u=0$, we obtain by (\ref{F3G})
$$f_j^{\rm T} \ha Y(0)=\lim_{\varpi\to 0} f_j^{\rm T} \varpi_3^{uv^{\rm T}} Y(\varpi)=\lim_{\varpi\to 0} f_j^{\rm T} Y( {\varpi })=O((t_1t_2)^{2\beta-2}),\quad j\in [4].$$
By Lemma \ref{lem:subcritical-amplitude} and that $v^{\rm T}u=-\nu$,
$$f_5^{\rm T} \ha Y(0)=v^{\rm T} (t_1t_2)^{-uv^{\rm T}}\til Y(0)=(t_1t_2)^\nu v^{\rm T} \til Y(0)=s_\beta (t_1t_2)^{\nu-\beta}=s_\beta (t_1t_2)^{2\beta-2}.$$

The vectors $f_1,f_2,f_3,f_4,f_5$ are linearly independent since the determinant of the square matrix $[\begin{array}{ccccc} f_1 & f_2 & f_3 & f_4 & f_5
\end{array}]$ equals $2-3\beta>0$. Thus, by the above two displayed formulas, $\|\ha Y(0)\|\lesssim (t_1t_2)^{2\beta-2}$. Since $\pa_{\varpi_3} \ha Y=-\ha N\ha Y$ and \(
 \|\ha N\|
\lesssim_{J} \varpi_3^\nu
\), Gronwall's inequality now gives
$$\|\ha Y(\varpi)\|\le \|\ha Y(0)\|\exp\Big(\int_0^\varpi \|\ha N(s)\|ds\Big)\lesssim_{J} (t_1t_2)^{2\beta-2}\exp(C_{J} \varpi_3^{\nu+1})\lesssim_{J} (t_1t_2)^{2\beta-2}.$$
From $Y= \varpi_3^{-uv^{\rm T}}\ha Y$ and (\ref{Epm*}), we then get (\ref{subcritical-Y-layer}).   Since the third component of $Y$ is $F_3$, and
\(
\partial_{\varpi_3}F=-(t_1t_2)F_3
\),
 integration of (\ref{subcritical-Y-layer}) gives (\ref{subcritical-F-layer}).
\end{proof}

We reuse the notation $\Omega$ and $\mathcal O$  defined right before Proposition~\ref{lem:direct-inner-core}. In terms of the $W_{J}$ in the previous lemma, $\mathcal O=\{\u r\in W_{2}:t_2\le t_1\le 1/2\}$.

\begin{Proposition}
If $1/2<\beta<2/3$, then on $\mathcal O$,
\BGE
|F(\u r)-\F(1,1)|+\sum_{j=1}^3t_j|F_j(\u r)|
+t_1t_2t_3|G(\u r)|\lesssim t_1^\alpha.
\label{subcritical-inner-core}
\EDE
\label{prop:subcritical-inner-core}
\end{Proposition}
\begin{proof}
All estimates in this proof are restricted to $\mathcal O$, on which Lemma~\ref{lem:subcritical-layer} applies with $J=2$. Since $Y_3=F_3$ and $Y_5=t_1t_2 G$, from (\ref{subcritical-Y-layer}) we get
\[
t_3|F_3|+t_1t_2t_3|G|\lesssim t_3(t_1t_2)^{2\beta-2}\varpi_3^\nu=(t_1t_2)^{2\beta-1}\varpi_3^{\nu+1} \lesssim (t_1t_2)^\alpha\le t_1^\alpha.
\]
Let $j\in [2]$. From (\ref{K-tj}) and $t_3\le2t_1t_2$, we obtain $|K_j|\lesssim t_j$. Since $\pa_{\varpi_3} Y= -t_1t_2 MY$, and $Y_j=r_j F_j$,   from the first two rows of (\ref{Paff1-M}), we obtain
\(
|\partial_{\varpi_3}(r_jF_j)|\lesssim t_{3-j}\|Y\|\).
Integrating from $\varpi_3=0$ and using (\ref{subcritical-Y-layer}) and $r_j\ge 1/2$ gives
\(
 |\delta_3 (t_j F_j)|\lesssim t_j |\delta_3 (r_j F_j)|\lesssim t_j t_{3-j}(t_1t_2)^{2\beta-2}= (t_1t_2)^{\alpha}\). By (\ref{f'-bd}) and (\ref{F1&N*}), $|t_j\F_j|\lesssim t_j^\gamma=t_j^\alpha\le t_1^\alpha$, and so $|t_j F_j|\lesssim  t_1^\alpha$.

Finally, (\ref{subcritical-F-layer})  and (\ref{F-11}) give
\(
|F-\F(1,1)|\lesssim t_1^\alpha
\),
and the proposition follows.
\end{proof}

\begin{Theorem}
The conclusion of Theorem~\ref{Thm-coninuity-neighbor} also holds for $1/2<\beta<2/3$.
\label{Thm-continuity-neighbor-full}
\end{Theorem}
\begin{proof}
We first prove that $F$ extends continuously to
$[0,1]^3\sem\{(1,1,1)\}$. By
Proposition~\ref{Extension1}(ii) and symmetry, $F$ extends
continuously to the relative interiors of the three side faces: $\{r_j=1\}$, $j\in [3]$.
It remains to consider intersections of two side faces.

Fix, for example, a point $(p,1,1)$ with $p<1$, and choose
$c\in(0,1-p)$. In a sufficiently small neighborhood of this point
we have $t_1\ge c$ and $r_2,r_3\ge1/2$. By symmetry, we may restrict
to the region $t_3\le t_2$. Then
\(
\varpi_3=\frac{t_3}{t_1t_2}\le\frac1c
\).
Lemma~\ref{lem:subcritical-layer} therefore gives
\[
|F(r_1,r_2,r_3)-F(r_1,r_2,1)|
\lesssim_c (t_1t_2)^\alpha
\to 0
\qquad\text{as }\u r\to(p,1,1).
\]
Since
\(
F(r_1,r_2,1)=\F(r_1,r_2),
\)
and $\F$ extends continuously to $[0,1]^2$ by
(\ref{F1&N*}), we also have
\[
F(r_1,r_2,1)\to\F(p,1) \qquad\text{as }\u r\to(p,1,1).
\]
Thus $F$ has a continuous extension at $(p,1,1)$.
By symmetry, the same holds at every intersection of two side
faces. Hence $F$ extends continuously to
$[0,1]^3\sem\{(1,1,1)\}$.

It remains to prove that
\(
F(\u r)\to\F(1,1)\) as $\u r\to (1,1,1)$.
By symmetry, it suffices to consider
\(
\Omega
=
\{\u r\in[0,1)^3:t_3\le t_2\le t_1\le1/2\}
\).
As noted in the proof of Theorem~\ref{Thm-coninuity-neighbor},
the radial Pfaff-flow argument leading to (\ref{Y<t1}) applies
here with Proposition~\ref{prop:subcritical-inner-core} in place
of Proposition~\ref{lem:direct-inner-core}. Since
$\gamma=\alpha=2\beta-1>0$, for every $\gamma'\in(0,\alpha)$
we obtain
\BGE
t_1|F_1|+t_2|F_2|+t_3|F_3|+t_1t_2t_3|G|
\lesssim_{\gamma'} t_1^{\gamma'}
\qquad\text{on }\Omega\setminus\mathcal O. \label{Y<t1'}
\EDE

For $\u r\in\Omega\sem\mathcal O$, set
\(
\phi(\u r):=(r_1,r_2,1-2t_1t_2)\in\mathcal O
\).
Then $\phi(\u r)\to(1,1,1)$ as $\u r\to(1,1,1)$.
For $r_3\le s<1-2t_1t_2$, writing $u=1-s$ gives
\(
2t_1t_2<u\le t_3\le t_2\le t_1\) and
\(
\frac{t_1t_2}{u}<\frac12
\).
Hence the interior of the line segment connecting $\u r$ with
$\phi(\u r)$ lies in $\Omega\sem\mathcal O$, and
(\ref{Y<t1'}) gives
\[
|F(\u r)-F(\phi(\u r))|
 \le
\int_{r_3}^{1-2t_1t_2}|F_3(r_1,r_2,s)|\,ds \lesssim_{\gamma'}
t_1^{\gamma'}\int_{2t_1t_2}^{t_3}\frac{du}{u} \le
t_1^{\gamma'}\log\frac1{t_1}
\to 0.
\]

On the other hand, Proposition~\ref{prop:subcritical-inner-core}
and $\phi(\u r)\in\mathcal O$ imply that
\(
F(\phi(\u r))\to\F(1,1)
\).
Therefore
\(
F(\u r)\to \F(1,1)\)
as $\u r\to(1,1,1)$ within $\Omega$.
This proves continuity at $(1,1,1)$, and hence $F$ extends
continuously to $[0,1]^3$. Formula~(\ref{F-boundary}) follows
from (\ref{F1&N*}) and symmetry.
\end{proof}

\begin{Corollary}
  When $\beta>1/2$, the $F$ in Theorem \ref{extension}  satisfies the hypotheses of Theorem \ref{Prop1} (ii), (iii), and (iv) with $c_F=\frac{\Gamma(2\beta)\Gamma(3\beta-1)}{\Gamma(\beta)\Gamma(4\beta-1)} f_\beta(1)$. Thus, the ${\mathcal Z}_\alpha$ defined by (\ref{Z3}) with
  $$C=c_F^{-1}f_\beta(1)^{-1} =\frac{\Gamma(\beta)\Gamma(4\beta-1)} {\Gamma(2\beta)\Gamma(3\beta-1)}\cdot\Big(\frac{\Gamma(\beta)\Gamma(3\beta-1)} {\Gamma(2\beta)\Gamma(2\beta-1)}\Big)^2   $$
  satisfies (PDE) (\ref{PDE}), (COV) (\ref{COV}), (ASY) (\ref{ASY}), and positivity for a neighbor pattern $\alpha$. In particular, $\mathcal{Z}_\alpha$ is the corresponding pure partition function for $\kappa=4/\beta\in (0,8)$.
  \label{victory@N-full}
\end{Corollary}
\begin{proof}
By Theorem \ref{extension}, $F$ together with another function $G$ satisfies (\ref{rainbow1}) and (\ref{neighbor2}), and so satisfies the hypotheses of Theorem \ref{Prop1} (ii). By Theorem \ref{Thm-coninuity-neighbor} when $\beta\ge 2/3$ and Theorem \ref{Thm-continuity-neighbor-full} when $1/2<\beta<2/3$, $F$ extends continuously to $[0,1]^3$ and satisfies (\ref{F-boundary}).  Hence, $F$ satisfies
the conditions in Theorem \ref{Prop1} (iii). By (\ref{2F1>0}), $F>0$  on $\{r_{\ell}=1\}$, ${\ell}\in [3]$. By Proposition \ref{Horn}, $F>0$  on $\{r_{\ell}=0\}$, ${\ell}\in [3]$. Hence, $F$ satisfies the conditions in Theorem \ref{Prop1} (iv). With the normalization constant \(C\) chosen above, Theorem~\ref{Prop1} shows that the corresponding $\mathcal Z_\alpha$ satisfies (PDE) (\ref{PDE}), (COV) (\ref{COV}), (ASY) (\ref{ASY}), and positivity, and Gauss's identity yields the stated value of $C$.
\end{proof}

\section{Integer Values of $\beta$ in the Neighbor Case} \label{section:integer}
In this section, we aim to find explicit formulas for the function $F$ from Theorem \ref{extension}  in the neighbor case when $\beta\in\N$. Let $A_{\u n}$, $\u n\in\N_0^3$, be from Theorem \ref{Theorem-AB@N}.

For any fixed $\beta\in\N$, we define $\til f_+(n) =n(n+1-2\beta)$ and $\til f_-(n)= (n+1-\beta)(n+2-3\beta)$. For each $j\in [3]$, we define $\til {\mathcal H}_j :\R^{\N_0^3}\to \R^{\N_0^3}$ such that
\BGE (\til {\mathcal H}_j A)_{\u n}=\til f_+(n_j+1) A_{\u n+e_j}-(\til f_+(n_j)+\til f_-(n_j))A_{\u n} +\til f_-(n_j-1) A_{\u n-e_j},\quad \u n \in\Z^3,\label{difference}\EDE
where $A$ is extended by zero outside $\N_0^3$.
Since by Proposition \ref{convergence@N}, $\sum_{\u n\in\N_0^3} A_{\u n} \u r^{\u n}$ converges to  $F$  in a neighborhood of $\u 0$, $\sum_{\u n\in\N_0^3} (\til {\mathcal H}_jA)_{\u n} \u r^{\u n}$ converges to $\til {\mathcal H}_j F$ in the same neighborhood, where $\til {\mathcal H}_j$ now denotes the differential operator:
\BGE \til {\mathcal H}_j= (1-r_j)[r_j(1-r_j)\pa_j^2+(2-2\beta-(4-4\beta)r_j)\pa_j-(1-\beta)(2-3\beta)].\label{differential}\EDE
Define   $\sigma_3(\u r)=r_1r_2r_3$. The relation between $\til{\mathcal H}_j$ and the ${\mathcal H}_j$ in Remark \ref{Remark-Hj}  is
\BGE \til{\mathcal H}_j M_{\sigma_3^{2\beta-1}} =M_{\sigma_3^{2\beta-1}} {\mathcal H}_j,\quad j\in [3], \label{HFFH}\EDE
where $M_{\sigma_3^{2\beta-1}}$ is the multiplication operator: $f\mapsto \sigma_3^{2\beta-1} f$.
Thus, for an analytic function $F$   on $[0,1)^3$, if $\til F=\sigma_3^{2\beta-1} F$, then $F$ lies in the intersection of kernels of ${\mathcal H}_j-{\mathcal H}_k$, $j,k\in[3]$, if and only if $\til F$ lies in the intersection of the kernels of $\til {\mathcal H}_j-\til {\mathcal H}_k$, $j,k\in[3]$.


For $\kappa=2$ and $\kappa=4$, which correspond respectively to $\beta=2$ and $\beta=1$ in our notation, explicit formulas or constructions for multiple-SLE partition functions are
available in the literature for general \(N\) in these two cases
(\cite{LERW-partition,Global}). These formulas, however, are not written in the form of the function $F$ considered here. Indeed, $F$ is obtained only after normalizing the partition functions and applying the transformations in Theorem \ref{Prop1}. Rewriting the known formulas in our framework for $N=3$ then requires a further nontrivial sequence of algebraic manipulations. We omit these details and present the resulting formulas directly.

Recall the polynomial $\Delta$ as defined in (\ref{Delta-nb}).  By (\ref{Delta-tj}) there is a symmetric domain $\Omega\subset\C^3$ satisfying $\Omega\supset [0,1)^3$ and $\Omega\cap\{\u r\in \C^3:r_j=0 \text{ and } r_k=1\}=\emptyset$ for any distinct $j,k\in [3]$, such that $\Ree \Delta>0$ on $\Omega$. Since \(\Delta(0)=1\), we choose the analytic branch of \(\sqrt{\Delta}\)
on \(\Omega\) satisfying
\(
\sqrt{\Delta(0)}=1
\).

Consider the symmetric function
\BGE F_{\beta=1}  (\u r) :=\frac{\sqrt{\Delta(\u r)}-(1-r_1r_2-r_2r_3-r_3r_1)}{4\sigma_3 (\u r)},\quad \u r\in \Omega \sem\sigma_3^{-1}\{0\} .\label{GFF-neighbor}\EDE
Rationalizing the numerator in (\ref{GFF-neighbor}) and using  (\ref{Delta-nb}), we obtain, wherever the denominator is nonzero,
 \BGE F_{\beta=1} (\u r) = \frac{2-r_1-r_2-r_3}{\sqrt{\Delta(\u r)}+(1-r_1r_2-r_2r_3-r_3r_1)}.\label{GFF-neighbor*}\EDE  By (\ref{Delta-nb-boundary}), for any $j\in [3]$, when $r_j=0$,   $\sqrt{\Delta(\u r)}+(1-r_1r_2-r_2r_3-r_3r_1)= {2(1-r_{j-1}r_{j+1})}$, which is not zero if $(r_{j+1},r_{j-1})\in [0,1]^2\sem\{ (1,1)\}$.
Thus, $F_{\beta=1}$ can be extended analytically to $\Omega$ such that $F_{\beta=1}|_{r_j=0}=\frac{2-r_{j-1}-r_{j+1}}{2(1-r_{j-1}r_{j+1})}$. Setting $j=3$ and $r_2=0$, we get $F_{\beta=1}(r_1,0,0)=1-\frac{r_1}2$.

Let $\beta=1$. Then $\til{\mathcal H}_j$ reduces to $r_j(1-r_j)^2\pa_j^2$. Since
$$\pa_j^2\sqrt\Delta=\frac 14\Delta^{-3/2}(2\Delta\pa_j^2\Delta-(\pa_j \Delta)^2)=-4\Delta^{-3/2}  r_{j+1}r_{j-1} (1-r_{j-1})^2(1-r_{j+1})^2,$$
we see that $\sqrt\Delta$ lies in the  kernel  of $\til{\mathcal H}_j-\til{\mathcal H}_k$ for any $j,k\in [3]$, which implies that $\sqrt{\Delta(\u r)}-(1-r_1r_2-r_2r_3-r_3r_1)$ lies in the same kernels since $(1-r_1 r_2-r_2 r_3-r_3 r_1)$ lies in the kernel of  $\til {\mathcal H}_j$ for any $j\in [3]$. This further implies that $F_{\beta=1}$ lies in the kernel  of $ {\mathcal H}_j- {\mathcal H}_k$ for any $j,k\in [3]$, which means that the coefficients $(A_{\u n})$ of the Taylor expansion of $F_{\beta=1}$ at the origin, extended by zero outside $\N_0^3$, satisfy (\ref{Ljk=}). Since $F_{\beta=1}(r_1,0,0)=1-\frac{r_1}2$,  the values of $A$ at $(n_1,0,0)$, $n_1\in\N_0$, agree with those of the $A$ from Theorem \ref{Theorem-AB@N} when $\beta=1$. By Lemma \ref{construction}, the $A$ here agrees with the one in Theorem \ref{Theorem-AB@N} when $\beta=1$. So the $F_{\beta=1}$ here agrees with the $F$ from Theorem \ref{extension} when $\beta=1$.

\begin{Remark}
  For $\beta=1$, the power series $\sum A_{\u n} \u r^{\u n}$ from Proposition \ref{convergence@N} does not converge in $B(0,R)^3$ for any $R>1/3$. Indeed, if this power series converged in \(B(0,R)^3\) for some \(R>1/3\),
then by the identity theorem it would give an analytic extension of
\(F_{\beta=1}\) to a neighborhood of
\(
\u s:=\left(-\frac13,-\frac13,-\frac13\right)
\), and so would $\sqrt\Delta$ since $\sqrt\Delta=4\sigma_3(\u r) F_1+1- (r_1r_2+r_2r_3+r_3r_1)$. Direct computation shows $\Delta(\u s)=0$ and $\pa_3 \Delta(\u s)=\frac{64}{27}\ne 0$. From $\Delta=(\sqrt{\Delta})^2$, we would then obtain
\(
\partial_3\Delta(s)=2\sqrt{\Delta}(s)\,\partial_3\sqrt{\Delta}(s)=0\),
a contradiction.\label{Rem:extension}
\end{Remark}

Now let $\beta=2$. Define a symmetric function
\BGE F_{\beta=2}(\u r)=\frac{P(\u r)\sqrt{\Delta(\u r)}-Q(\u r)}{40 \sigma_3(\u r)^3},\quad \u r\in \Omega \sem\sigma_3^{-1}\{0\}, \label{F2}\EDE
where
{\small
$$P (\u r):=  1- (r_1 r_2 + r_2 r_3 + r_3 r_1) + 2 r_1 r_2 r_3- (r_1^2 r_2^2 +
   r_2^2 r_3^2 + r_3^2 r_1^2)+
 2 (r_1^2 r_2^2 r_3 + r_2^2 r_3^2 r_1 + r_3^2 r_1^2 r_2)  $$\BGE - (r_1^3 r_2^2 r_3 + r_1^3 r_3^2 r_2 +
   r_2^3 r_3^2 r_1 + r_2^3 r_1^2 r_3 + r_3^3 r_1^2 r_2 + r_3^3 r_2^2 r_1)+(r_1^3 r_2^3 + r_2^3 r_3^3 + r_3^3 r_1^3) ;\label{P2}\EDE
$$Q (\u r):=  1-
 2 (r_1 r_2 + r_2 r_3 + r_3 r_1)+6 r_1 r_2 r_3 +
 2 (r_1^3 r_2^3 + r_2^3 r_3^3 + r_3^3 r_1^3) -
 6 (r_1^3 r_2^3 r_3 + r_2^3 r_3^3 r_1 + r_3^3 r_1^3 r_2)  $$ \BGE +
 2 (r_1^4 r_2^3 r_3 + r_1^4 r_3^3 r_2 + r_2^4 r_3^3 r_1 + r_2^4 r_1^3 r_3 +
    r_3^4 r_1^3 r_2 + r_3^4 r_2^3 r_1)-(r_1^4 r_2^4 + r_2^4 r_3^4 + r_3^4 r_1^4).\label{Q2}\EDE
}

A direct symbolic computation verifies that \(F_{\beta=2}\) satisfies (\ref{Ljk-PDE})
and has the initial coefficients prescribed by Theorem~\ref{Theorem-AB@N} for \(\beta=2\);
hence, by Lemma~\ref{construction} and Theorem~\ref{extension}, it agrees with the \(F\) with $\beta=2$ (see Remark \ref{CAS}). Together with the case $\beta=1$, this suggests
that for positive integer values of $\beta$ the function $F$ in Theorem~\ref{extension} should admit a representation
of the form
\[
C_\beta \sigma_3(\u r)^{2\beta-1}F(r)=P(\u r)\sqrt{\Delta(\u r)}-Q(\u r),
\]
where $P$ and $Q$ are symmetric polynomials with $P(0)=Q(0)=1$, and $C_\beta$ is a positive constant. We now formulate a more precise version of this
pattern.

Motivated by the cases $\beta=1$ and $\beta=2$, we now describe a more precise conjectural
form for the function $F$ in Theorem~\ref{extension}. The following construction gives a candidate for
the polynomial $Q$ and leads to the conjectural formula stated below.

\begin{Construction}
  Let $\beta\in\N$. Define $E^0$ on $\N_0^2\times\{0\}$ such that $E^0_{(m,m,0)}= \frac{(1-\beta)_m (2-3\beta)_m}{(1)_m (2-2\beta)_m}$ for $0\le m\le \beta-1$ and $E^0_{\u n}$ vanishes elsewhere. A direct calculation gives $\til {\mathcal H}_1 E^0=\til {\mathcal H}_2 E^0$. \footnote{In fact, we only need to check the equality at those $(n_1,n_2,0)$, which have distance $\le 1$ from the support of $E^0$: $\{(m,m,0):0\le m\le \beta-1\}$. By the symmetry of $E^0$, we only need to show that $\til {\mathcal H}_1 E^0=\til {\mathcal H}_2 E^0$ at $(n,n+1,0)$ for $-1\le n\le \beta-1$, which holds because $\til f_+(n+1) E_{(n+1,n+1,0)}=\til f_-(n) E_{(n,n,0)}$.}
  Modify $\til f_\pm$ to define $\ha f_\pm$ such that $\ha f_-=\til f_-$,  $\ha f_+(n)=\til f_+(n)$ if $n\ne 2\beta-1$, and $\ha f_+(2\beta-1)=1$. This modification removes the zero of \(\widetilde f_+\) at \(2\beta-1\),
which is needed for the propagation argument below. Then we define $\ha{\mathcal H}_j$ using (\ref{difference}) with $\ha f_\pm$ in place of $\til f_\pm$. Since $E_{(n_1,n_2,0)}=0$ if $n_1$ or $n_2$ equals $2\beta-1$ (because $2\beta-1>\beta-1$), we see that $E^0$ lies in the kernels of  $\ha{\mathcal H}_1-\ha{\mathcal H}_2$. By the proof of Lemma \ref{construction} (ii), $E^0$ extends uniquely to a symmetric function $E$ on $\N_0^3$, which lies in the kernels of $\ha{\mathcal H}_j -\ha{\mathcal H}_k$, $j,k\in [3]$. \footnote{Here we use $\ha{\mathcal H}_j$ instead of $\til{\mathcal H}_j$ because $\ha f_+(n)\ne 0$ for any $n\in\N$ while $\til f_+(2\beta-1)=0$.}
  Since $E$ vanishes on $\Z_{\ge \beta}\times \N_0\times \{0\}$, from $\ha{\mathcal H}_3 E=\ha{\mathcal H}_2 E$ we see that $E$ vanishes on $\Z_{\ge \beta}\times \N_0^2$. Indeed, applying the recursion given by
\(\widehat{\mathcal H}_3E=\widehat{\mathcal H}_2E\) inductively in the
third coordinate propagates the vanishing from the face \(n_3=0\) to all
\(n_3\ge0\). Symmetrically, $E$ also vanishes on $\N_0^2\times \Z_{\ge \beta}$ and $\N_0\times \Z_{\ge \beta}\times \N_0$. So the support of $E$ is contained in $\Z_{[0,\beta-1]}^3$. Since $\ha f_\pm$ and $\til f_\pm$ only differ at $2\beta-1$, and $E_{\u n}=0$ when $n_j=2\beta-1$, we get $\til {\mathcal H}_jE=\ha{\mathcal H}_jE$  for any $j\in [3]$. Thus, $E$ lies in the kernels of $\til{\mathcal H}_j -\til{\mathcal H}_k$, $j,k\in [3]$.

  For $m\in\N_0$ and $j\in [3]$, define the involution $T^{(m)}_j:\Z^3\to \Z^3$ such that
\BGE T^{(m)}_j(n_1,n_2,n_3)=n_je_j+\sum_{k\in [3]\sem\{j\}} (m-n_k)e_k.\label{Tmj}\EDE
For each $j\in [3]$, define $E^{(j)}\in\R^{\Z^3}$ by $E_{\u n}^{(j)}=E_{T^{(3\beta-2)}_j(\u n)}$. Then the support of $E^{(j)}$ is contained in $T^{(3\beta-2)}_j(\Z_{[0,\beta-1]}^3)\subset\N_0^3$.  Note that the supports of $E$ and $E^{(j)}$, $j\in [3]$, are mutually disjoint.
Since $E$ lies in the kernels of ${\til {\mathcal H}}_j-{\til {\mathcal H}}_k$, $j,k\in[3]$, and $\til f_\pm (3\beta-2-n)=\til f_\mp (n)$,   each $E^{(j)}$ also lies in these kernels.  Define
$ Q=E-\sum_{j\in [3]} E^{(j)}$. 
Then $Q$ has a finite support in ${\N_0^3}$, is symmetric and lies in these kernels. 
Define the symmetric polynomial $Q(\u r):=\sum_{\u n\in \N_0^3} Q_{\u n} \u r^{\u n}$. 
We observe that $Q(\u 0)=Q_{\u 0}=E_{\u 0}=1$.

Let $P^*=Q /\sqrt{\Delta}$ and let $(P^*_{\u n})_{\u n\in\N_0^3}$ be the coefficients of the Taylor expansion of $P^*$ at $\u 0$. Define $P\in\R^{\N_0^3}$ by
$P =\ind_{\min\{n_1,n_2,n_3\}\le 2\beta-2}  P^* $
and let $P(\u r)=\sum_{\u n\in\N_0^3} P_{\u n}\u r^{\u n}$. Since $Q(\u 0)=\sqrt\Delta(\u 0)=1$, $P(\u 0)=P_{\u 0}=P^*_{\u 0}=P^*(\u 0)=1$. Since $Q$ and $\Delta$ are symmetric in $r_1,r_2,r_3$, so are $P^*$ and $P$.
\end{Construction}

\begin{Conjecture}
The coefficient array $P$ has finite support, equivalently, $P(\u r)$  is a  polynomial. Moreover, the $F $ from Theorem \ref{extension} is given by
\BGE F  (\u r)=\frac{P  (\u r)\sqrt{\Delta(\u r)}-Q  (\u r)}{C_\beta (\sigma_3(\u r))  ^{2\beta-1}},\quad C_\beta:=\frac{4(3\beta-2)!(4\beta-3)!((\beta-1)!)^3}{((2\beta-1)!)^2((2\beta-2)!)^3}.\label{Cbeta}\EDE
\label{Conj1}
\end{Conjecture}

The formula in Conjecture \ref{Conj1} recovers the above two examples when $\beta\in\{1,2\}$, with $C_1=4$ and $C_2=40$; $P_{\beta=1}(\u r)=1$ and $Q_{\beta=1}(\u r)=1-r_1r_2-r_2r_3-r_3r_1$; $P_{\beta=2}$ and $Q_{\beta=2}$ are given by (\ref{P2}) and (\ref{Q2}). Moreover, if we define  $E(\u r)=\sum_{\u n\in \N_0^3} E_{\u n} \u r^{\u n}$, then $E_{\beta=1}(\u r)=1$ and $E_{\beta=2}(\u r)=1- 2 (r_1 r_2 + r_2 r_3 + r_3 r_1)+6 r_1 r_2 r_3$.

Now we explain how to use CAS to rigorously verify Conjecture \ref{Conj1} for any specific $\beta\in\N$. We define $\til P\in \R^{\N_0^3}$ by
$$ \til P_{\u n}=\ind_{\max\{n_1,n_2,n_3\}\le 3\beta-3}(\u n) P_{\u n}=\ind_{\max\{n_1,n_2,n_3\}\le 3\beta-3}(\u n)\ind_{\min\{n_1,n_2,n_3\}\le 2\beta-2}(\u n) P^*_{\u n},$$ 
and let $\til P(\u r)=\sum_{\u n\in\N_0^3} \til P_{\u n}\u r^{\u n}$. Since $Q(\u 0) =\Delta(\u 0)=1$, we get $P^*(\u 0)=1$, and so $P^*_{\u 0}=P_{\u 0}=1$, which implies $P(\u 0)=1$.  Let $U_{\u n}$, $\u n\in\N_0^3$, be the coefficients of the Taylor expansion of $\til P\sqrt\Delta-Q$ about $\u 0$. For each $j\in[3]$, define
\BGE  \til {\mathcal H}^\Delta_j : =\Delta^{3/2} \til{\mathcal H}_j ( \sqrt\Delta\cdot)\label{H-Delta}\EDE
We use CAS to check three criteria:
\begin{enumerate} [label=(C\arabic*), ref=C\arabic*]
  \item \label{C1*} $\til P^2\Delta-Q^2$ is divisible by $\sigma_3^{2\beta-1}$;
  \item \label{C2*}  $\til P$  lies in the kernels of $ \til {\mathcal H}^\Delta_j- \til {\mathcal H}^\Delta_k$ for any $j,k\in [3]$.
  \item \label{C3*} $U_{(2\beta-1+n,2\beta-1,2\beta-1)}=C_\beta\frac{(\beta)_{n }(1-2\beta)_{n }}{(2\beta)_{n }(1)_{n }}$ for $0\le n\le 2\beta-1$, where $C_\beta$ is given by (\ref{Cbeta}).
\end{enumerate}

These criteria can be checked by CAS because (i) we only need to compute finitely many terms to get $\til P$, and $\til P^2\Delta-Q^2$ is a polynomial; (ii) $\til P$ is a polynomial, and each $ \til {\mathcal H}^\Delta_j$ is a second-order differential operator with polynomial coefficients; (iii) the coefficients $U_{\u n}$ are computable, and we only need to check finitely many terms.

With a straightforward sparse implementation, the computational cost of checking (\ref{C1*})-(\ref{C3*}) for a fixed $\beta$ grows polynomially in $\beta$. The dominant step is (\ref{C1*}): the number of potentially relevant coefficients of $\widetilde P$ is at most
\(
(3\beta-2)^3-(\beta-1)^3 = O(\beta^3),
\)
so forming $\widetilde P^2\Delta-Q^2$ and checking its divisibility by $\sigma_3^{2\beta-1}$ requires $O(\beta^6)$ arithmetic operations. The checks of (\ref{C2*}) and (\ref{C3*}) are of lower order.
Theorem \ref{check} below shows that the three criteria are enough to verify Conjecture \ref{Conj1} for a specific $\beta\in \N$.

We expand $\Delta$ w.r.t.\ $r_3$:
\BGE \Delta(\u r)=a_0(r_1,r_2)^2+2a_1(r_1,r_2) r_3+a_2(r_1,r_2) r_3^2,\label{Delta-a12}\EDE
where $$a_0(r_1,r_2):=1-r_1r_2,\quad a_1(r_1,r_2):=4 r_1 r_2 -(r_1+r_2)(1+r_1 r_2),\quad a_2(r_1,r_2):=(r_1-r_2)^2.$$ Suppose $\Delta(\u r)^{\pm 1/2}=\sum_{m=0}^\infty d_m^\pm(r_1,r_2) r_3^m$.  From (\ref{Delta-a12}) we get
  $$1=(a_0^2+2 a_1 r_3+a_2 r_3^2)\cdot \Big(\sum_{m=0}^\infty d_m^-  r_3^m\Big)^2 = (a_0^2+2 a_1 r_3+a_2 r_3^2)\cdot \Big(\sum_{m=0}^\infty \Big( \sum_{k=0}^m d^-_k  d^-_{m-k} \Big) r_3^m \Big) $$
  $$=  \sum_{m=0}^\infty \Big(a_0^2 \sum_{k=0}^m d^-_k d^-_{m-k}+2 a_1  \sum_{k=0}^{m-1} d^-_k d^-_{m-1-k} + a_2  \sum_{k=0}^{m-2} d^-_k d^-_{m-2-k}\Big) r_3^m,$$
where $d^-_k:=0$ if $k\le -1$. By the above formula,
\BGE d^-_0=\frac 1{a_0},\quad d^-_m=-\frac{a_0}2 \sum_{k=1}^{m-1} d^-_k d^-_{m-k}-\frac{a_1}{a_0} \sum_{k=0}^{m-1} d^-_k d^-_{m-1-k} -\frac{a_2}{2a_0} \sum_{k=0}^{m-2} d^-_k d^-_{m-2-k},\quad m\ge 1.\label{dm}\EDE

For $k\in\N_0$, a two-dimensional coefficient array $A\in \R^{\N_0^2}$ is called $k$-diagonal if $A_{(n_1,n_2)}=0$ whenever $|n_1-n_2|>k$. A function $f$ in $r_1,r_2$ that is analytic at $(0,0)$ is called $k$-diagonal if the coefficients of the power series expansion of $f$ at $\u 0$ are $k$-diagonal. The product of a $k_1$-diagonal function and a $k_2$-diagonal function is $(k_1+k_2)$-diagonal. We observe that $a_j$ is $j$-diagonal for $j=0,1,2$. Thus, by (\ref{dm}), every $d^-_k$ is $k$-diagonal. Since $1=(\sum_m d^+_m r_3^m)(\sum_m d^-_m r_3^m)$, every $d^+_k$ is also $k$-diagonal.

We expand $Q$ in terms of $r_3$: $Q(\u r)=\sum_{m\in\N_0} Q_m(r_1,r_2) r_3^m$, where $Q_m$, $m\in\N_0$, are two-variable symmetric functions. We similarly define $P^*_m$, $P_m$,  and $U_m$, $m\in\N_0$.

\begin{Lemma}
For any $m\in\N_0$,   $Q_m$, $P^*_m$,   $P_m$, and $U_m$ are $m$-diagonal. \label{Q-diagonal}
\end{Lemma}
\begin{proof}
  From the construction of $Q$ we know that $Q_0$ is $0$-diagonal. Assume that $Q_k$ is $k$-diagonal for $0\le k\le m$. Let $\u n=(n_1,n_2,m+1)\in\N_0^3$ satisfy $|n_1-n_2|\ge m+2$. We need to show $Q_{\u n}=0$. By symmetry we may assume $n_1\ge n_2$. Let $\u n'=(n_1,n_2,m)$. Since $Q_m$ and $Q_{m-1}$ are $m$-diagonal,  $Q$ vanishes at $\u n'$, $\u n'-e_3$ and $\u n'\pm e_1$. Using $(\til H_3 Q)_{\u n'}=(\til H_1 Q)_{\u n'}$   we get $\til f_+(m+1)Q_{\u n}=0$. If $m+1\ne 2\beta-1$, then $\til f_+(m+1)\ne 0$, and so $Q_{\u n}=0$.

Assume now $m+1= 2\beta-1$. Then $\u n=(n_1,n_2,2\beta-1)$ and $n_1-n_2\ge 2\beta$.  Recall that $Q=E-\sum_{j\in [3]} E^{(j)}$, the support of $E$ is contained in $\Z_{[0,\beta-1]}^3$, and the support of $E^{(j)}$ is contained in $T_j^{3\beta-2}(\Z_{[0,\beta-1]}^3)$, $j\in [3]$. If $\u n$ does not lie in any of these supports, then $Q_{\u n}=0$. Now we assume that $\u n$ lies in one of these supports. Since $n_3=2\beta-1>\beta-1$, $\u n$ does not lie in the support of $E$ or $E^{(3)}$. If $\u n$ lies in the support of $E^{(1)}$, then $n_2\ge (3\beta-2)-(\beta-1)=2\beta-1>\beta-1\ge n_1$, which contradicts that $n_1>n_2$. Thus, $\u n$ lies in the support of $E^{(2)}$. So $Q_{\u n}=-Q_{\lin{\u n}}$, where $\lin{\u n}=(\lin n_1,\lin n_2,\lin n_3)\in\Z_{[0,\beta-1]}^3$, $\lin n_1=3\beta-2-n_1 $, $\lin n_2=n_2$, and $\lin n_3=3\beta-2-n_3=\beta-1$. From $n_1-n_2\ge 2\beta $, we get $\lin n_3-\lin n_1> \lin n_2$. Since $\lin n_2\le \beta-1$, by the symmetry of $Q$ and the fact that $Q_m$ is $m$-diagonal for $m\le \beta-1$, we get $Q_{\lin{\u n}}=0$, which again implies that $Q_{\u n}=0$. The conclusion follows from induction.

Since $P^*=Q \Delta^{-1/2}$, we get
\BGE P_m^*=\sum_{k=0}^{m} Q_k d^-_{m-k},\quad m\in\N_0.\label{Km-sum*}\EDE
Since $d^-_k$ is $k$-diagonal for any $k\in\N_0$, by Lemma \ref{Q-diagonal} and (\ref{Km-sum*}), each $P_k^*$ is $k$-diagonal, and so is $P_k$ because $P^*_{\u n}=0$ implies that $P_{\u n}=0$.

Let $Q^*=P\sqrt\Delta$ and $Q^*(\u r)=\sum_m Q^*_m(r_1,r_2) r_3^m$. Since $\sqrt\Delta(\u r)=\sum_m d^+_m(r_1,r_2)$, and for each $k$, $P_k$ and $d^+_k$ are $k$-diagonal, each $Q^*_m$ is also $m$-diagonal. Thus,  $U_m=Q^*_m-Q_m$ is $m$-diagonal for all $m\in\N_0$.
\end{proof}

\begin{Proposition}
  If (\ref{C1*}), (\ref{C2*}) and (\ref{C3*}) all hold for some $\beta\in\N$, then Conjecture \ref{Conj1} holds for this $\beta$. \label{check}
\end{Proposition}
\begin{proof}
We write $f\equiv g$ $(\text{mod }\sigma_3^{2\beta-1})$ if $f-g$ equals $\sigma_3^{2\beta-1}$ times an analytic function in a neighborhood of $\u 0$.
As elements in $\R^{\N_0^3}$, we have $P^*-P=\ind_{\min\{n_1,n_2,n_3\}\ge 2\beta-1} P^*$. This implies that, as  analytic functions, $P^*\equiv P$ $(\text{mod }\sigma_3^{2\beta-1})$, and so $P^2\Delta\equiv (P^*)^2\Delta=Q^2$ $(\text{mod }\sigma_3^{2\beta-1})$. Since (\ref{C1*}) is true,  $(\til P)^2\Delta\equiv Q^2\equiv P^2\Delta$ $(\text{mod }\sigma_3^{2\beta-1})$.    Since $\Delta(\u 0)=P(\u 0)=1$, this implies that $\til P^2\equiv P^2$ $(\text{mod }\sigma_3^{2\beta-1})$.
 Since \(P(0)=\widetilde P(0)=1\), the analytic function
\(\widetilde P+P\) is a unit near the origin. Hence
$\til P^2\equiv P^2$ $(\text{mod }\sigma_3^{2\beta-1})$
implies
$\til P-P\equiv 0$ $(\text{mod }\sigma_3^{2\beta-1})$,
 Thus, $\til P_{\u n}-P_{\u n}=0$ if $\min\{n_1,n_2,n_3\}\le 2\beta-2$. By the definitions of $P$ and $\til P$, if $\min\{n_1,n_2,n_3\}> 2\beta-2$,  $P_{\u n}=\til P_{\u n}=0$, and so $\til P_{\u n}-P_{\u n}=0$. Thus, $\til P_{\u n}=P_{\u n}$ holds for all $\u n\in\N_0^3$, which implies that $P$ equals the polynomial $\til P$.

Since $\til P=P$, we get $\til P^2\Delta\equiv  Q^2$ $(\text{mod }\sigma_3^{2\beta-1})$.
Since \(\til P(0)\sqrt{\Delta(0)}+Q(0)=2\), the function
\(\til P\sqrt\Delta+Q\) is a unit near the origin. Thus
$\til P^2\Delta-  Q^2=0$ $(\text{mod }\sigma_3^{2\beta-1})$
implies
$\til P\sqrt\Delta-  Q=0$ $(\text{mod }\sigma_3^{2\beta-1})$.
So  there is an analytic function $V$ in $(-\eps,\eps)^3$ for some $\eps>0$ such that $\til P\sqrt\Delta-  Q =\sigma_3^{2\beta-1}V$ there. Since (\ref{C2*}) is true, $\til P\sqrt\Delta$ lies in the kernels of $ \til {\mathcal H} _j- \til {\mathcal H} _k$ for any $j,k\in [3]$. Since $Q$ lies in the same kernels, so does $\til P\sqrt\Delta-Q$. Thus, by (\ref{HFFH}), $V$ lies in the kernels of $ {\mathcal H} _j-  {\mathcal H} _k$ for any $j,k\in [3]$.

Let $V_{\u n}$, $\u n\in\N_0^3$, be the coefficients of the Taylor expansion of $V$ at $\u 0$. Then $V_{\u n}=U_{(n_1+2\beta-1,n_2+2\beta-1,n_3+2\beta-1)}$ for any $\u n\in\N_0^3$. Since (\ref{C3*}) holds, for any $0\le n\le 2\beta-1$,
\BGE V_{(n,0,0)}=C_\beta\frac{(\beta)_{n }(1-2\beta)_{n }}{(2\beta)_{n }(1)_{n }}=C_\beta A_{(n,0,0)}.\label{VCbeta}\EDE

Since $U=\sigma_3^{2\beta-1} V$,  $U_{2\beta-1}=(r_1r_2)^{2\beta-1}V(\cdot,\cdot,0)$. Since $U_{2\beta-1}$  is $(2\beta-1)$-diagonal by Lemma \ref{Q-diagonal}, the same holds for $V(\cdot,\cdot,0)$. So $V_{(n,0,0)}=0$ for $n\ge 2\beta$. Thus, (\ref{VCbeta}) holds for any $n\in\N_0$. By Lemma \ref{construction}, $V_{\u n}=C_\beta A_{\u n}$ for all $\u n\in\N_0^3$. Thus, $V\equiv C_\beta F$ in a neighborhood of $\u 0$, which implies that
$$P\sqrt\Delta-Q= \til P\sqrt\Delta-Q=U=\sigma_3^{2\beta-1} V=C_\beta \sigma_3^{2\beta-1} F$$
in a neighborhood of $\u 0$. By the identity theorem, the equality holds in $[0,1)^3$.
\end{proof}

\begin{Remark}
Using exact computer algebra, we have verified \textup{(C1)}--\textup{(C3)}, and hence Conjecture~\ref{Conj1}, for all integers $1\le \beta\le 16$. In particular, for $\beta=2$ we checked that the explicit formulas (\ref{F2})-(\ref{Q2}) indeed reproduce the function $F$ from Theorem~\ref{extension}. We also computed the polynomials $Q$ and $\widetilde P$ for $\beta=17,18,19,20$; under the conclusion of Conjecture \ref{Conj1}, $\til P$ agrees with $P$, although in these four cases we did not check \textup{(C1)}--\textup{(C3)}. Looking at the first twenty positive integers, we found that the number of nonzero coefficients of $Q$ appears to be given by $2\beta(\beta^2+1)$, with a single exception at $\beta=20$, where the actual number is smaller by $24$ (hence by less than one percent). For $\widetilde P$, the number of nonzero coefficients appears to be given by
\[
N_{\widetilde P}(\beta)=
\begin{cases}
9\beta^3-18\beta^2+2\beta+16, & \beta \text{ even},\\[2mm]
9\beta^3-18\beta^2+14\beta-4, & \beta \text{ odd},
\end{cases}
\]
again with a single exception at $\beta=9$, where the actual number is smaller by $48$ (also by less than one percent). These data suggest that the conjectural algebraic structure is remarkably stable, with only rare extra cancellations at special integer values of $\beta$.
\label{CAS}
\end{Remark}

\begin{Remark}
Write
\(
P(r)=\sum_{m\ge0}P_m(r_1,r_2)r_3^m
\).
A partial coefficient analysis, which we do not include here, suggests that
the slices \(P_m\) are polynomials in \((r_1,r_2)\) for
\(0\le m\le \beta-1\), while for \(\beta\le m\le 2\beta-2\) the same
analysis only gives that
\(
(1-r_1r_2)^{2(m+1-\beta)}P_m(r_1,r_2)
\)
is a polynomial.  The remaining obstruction is to prove that \(P_m\) is a polynomial for
every \(m\) in this range; once this is done, the symmetry of \(P\)  would imply that \(P\) is a polynomial in
\((r_1,r_2,r_3)\).
\end{Remark}

\begin{appendices}
\section{A Linear ODE with Analytic Parameters} \label{section:appendix}
\begin{Proposition}
  Let $m,n\in\N$. Let $U$ be an open set in $\C^m$. Let $V\subset\C$ be a simply connected domain, and let $z_0\in V$. Let $A:U\times V\to \C^{n\times n}$ be an analytic matrix-valued function.  Let $\phi :U\to \C^n$ be   analytic. Then there is a unique $\C^n$-valued analytic function $f$ on $U\times V$ such that ($\pa_{m+1}$ stands for the partial complex derivative with respect to the $(m+1)$-th variable)
  \BGE \begin{cases}
    \pa_{{m+1}} f= Af&\text{on }U\times V;\\
    f(\cdot,z_0) = \phi &\text{on }U.
  \end{cases}\label{ODE-complex}\EDE
\end{Proposition}
\begin{proof}
  For each fixed $\u u$, by \cite[Theorem 15.3]{IY}, there exists a unique $\C^n$-valued analytic function $f(\u u,\cdot)$ on $V$, which satisfies $\pa_{m+1} f(\u u,\cdot)=A(\u u,\cdot) f(\u u,\cdot)$ and $f(\u u, z_0)=\phi(\u u)$. Putting these solutions together we get a function $f$ which solves (\ref{ODE-complex}) pointwise in the parameter $u$. \cite[Theorem 15.3]{IY} also implies the uniqueness of the solution.

  It remains to prove the (joint) analyticity of $f$ in $U\times V$. Fix an open set $U_1\Subset U$. It suffices to show that $f$ is analytic on $U_1\times V$. Let $S_1$ be the set of $z\in V$ such that for some $r>0$, $B(z,r)\Subset V$, and $f$ is bounded and analytic on $U_1\times B(z,r)$. Let $S_2$ be the set of $z\in V$ such that $f(\cdot,z)$ is bounded and analytic on $U_1$. Then $S_1\subset S_2\subset V$, and $z_0\in S_2$. For the analyticity of $f$ in $U\times V$, it suffices to show $S_1=V$.

  We claim   that
  \BGE \forall K\Subset V\quad\exists r>0\quad \forall z_1\in S_2\cap K:\quad B(z_1,r)\subset S_1.\label{KVrS}\EDE
  Suppose (\ref{KVrS}) has been proved. Then (i) $S_2\subset S_1$, which combined with $S_1\subset S_2$ implies $S_1=S_2$; (ii) $S_1$ is open; and (iii) $z_0\in S_1$, and so $S_1\ne\emptyset$. Let $(w_n)_{n\in\N}$ be a sequence in $S_1$ that converges to $w_0\in V$. Let $K= \{w_n:n\in\N_0\}$. Then $K\Subset V$. We pick $r>0$ as in (\ref{KVrS}). Since $w_n\to w_0$, there is $n_0\in\N$ such that $w_0\in B(w_{n_0},r)$. Since $w_{n_0}\in S_1\cap K=S_2\cap K$, by (\ref{KVrS}), $w_0\in B(w_{n_0},r)\subset S_1$, i.e., $w_0\in S_1$. This shows that $S_1$ is relatively closed in $V$. Since $V$ is connected, we get $S_1=V$, as desired.

  It remains to prove (\ref{KVrS}). Let $K\Subset V$. There is $R>0$ such that $V_1:=B(K,R)\Subset V$. Then $U_1\times V_1\Subset U\times V$. Since $A$ is continuous in $U\times V$, there is $M\in (0,\infty)$ such that $\|A\|\le M$ on $U_1\times V_1$. Let $r=R\wedge \frac 1{2M}$, $z_1\in S_2\cap K$, and $\phi_1=f(\cdot,z_1)$. Since $z_1\in S_2$, $\phi_1$ is bounded and analytic on $U_1$. Since $z_1\in K$, $B(z_1,r)\subset B(z_1,R)\subset V_1$, we have $\|A\|\le M$ on $U_1\times B(z_1,r)$.

  Define a sequence of $\C^n$-valued functions $(f_k)_{k\in\N_0}$ on $U_1\times B(z_1,r)$ by $f_0(\u u,z)=\phi_1(\u u)$, and
\BGE f_{k+1} (\u u,z)=\phi_1(\u u)+\int_{[z_1,z]} A(\u u, w) f_k(\u u, w) dw, \quad\text{for any }(\u u, z)\in U_1\times B(z_1,r),\label{induction}\EDE
where $[z_1,z]$ is the line segment from $z_1$ to $z$ parametrized by $\gamma(t)=z_1+t(z-z_1)$, $0\le t\le 1$.
Since $\phi_1$ is analytic and bounded on $U_1$, and $A$ is analytic and bounded on $U_1\times B(z_1,r)$,  all $f_k$'s are analytic and bounded on $U_1\times B(z_1,r)$. We use the norm $\|f\|_r:=\sup\{|f(\u u,z)|:(\u u,z)\in U_1\times B(z_1,r)\}$. Since for any $z\in B(z_1,r)$,
$$|f_{k+2} (\u u,z)-f_{k+1} (\u u,z)|=\Big|\int_{[z_1,z]}  A(\u u, w) (f_{k+1}(\u u, w)-f_k(\u u, w)) dw\Big|$$
$$\le M r \|f_{k+1}-f_k\|_r\le \frac 12 \|f_{k+1}-f_k\|_r,$$
we get $\|f_{k+2}-f_{k+1}\|_r\le \frac 12 \|f_{k+1}-f_k\|_r$. Thus, $(f_k)$ converges uniformly on $U_1\times B(z_1,r)$. Let $f_{\phi_1,z_1}$ be the limit.
By the Weierstrass theorem, \(f_{\phi_1,z_1}\) is analytic on \(U_1\times B(z_1,r)\), and it is bounded
there as a uniform limit of bounded functions.   From (\ref{induction}) we know that
$$ f_{\phi_1,z_1}(\u u,z)=\phi_1(\u u)+\int_{[z_1,z]} A(\u u,w) f_{\phi_1,z_1}(\u u, w) dw, \quad\text{for any }(\u u, z)\in U_1\times B(z_1,r). $$ 
This shows that $f_{\phi_1,z_1}$ solves $\pa_{m+1} f_{\phi_1,z_1} = Af_{\phi_1,z_1}$ on $U_1\times B(z_1,r)$ with  $f_{\phi_1,z_1}(\cdot,z_1)=\phi_1$. For each fixed \(u\in U_1\), uniqueness for the linear ODE system in the
\(z\)-variable, see \cite[Theorem 15.3]{IY}, implies that, $f(u,\cdot)$ and $f_{\phi,z_1}(u,\cdot)$ coincide in $B(z_1,r)$ since they solve the same linear system on $B(z_1,r)$ with the same initial value at $z_1$,
Thus \(f\) is bounded and analytic on \(U_1\times B(z_1,r)\). Hence every
\(z\in B(z_1,r)\) has a smaller disk \(B(z,\rho)\subset B(z_1,r)\) on which
the defining condition of \(S_1\) holds. Therefore \(B(z_1,r)\subset S_1\).  The proof of (\ref{KVrS}) is now complete.
\end{proof}

\end{appendices}

\section*{Declaration on the use of AI tools}

In the preparation of this paper, the author used AI tools as an aid in discussion, brainstorming, and checking
possible approaches to some proofs. The final mathematical arguments and proofs presented in the paper were
written by the author, who carefully checked them line by line.

\end{document}